\documentclass[12pt,a4paper,right=4cm,titlepage]{report}
\usepackage[english]{babel}
\usepackage[utf8]{inputenc}

\usepackage[margin=2.5cm, right=2cm, left=3cm]{geometry}
\usepackage{amsmath}
\usepackage{amsfonts}
\usepackage{amsthm}
\usepackage{amssymb}
\usepackage{tikz}
\usepackage{graphicx}
\usepackage{caption}
\usepackage{subfig}
\usepackage{wasysym}
\usepackage{xcolor}
\usetikzlibrary{calc,intersections}
\newtheorem{theorem}{Theorem}[chapter]
\newtheorem{lemma}[theorem]{Lemma}
\newtheorem{prop}[theorem]{Proposition}
\newtheorem{corollary}[theorem]{Corollary}
\newtheorem{definition}[theorem]{Definition}
\newtheorem{Oss}[theorem]{Remark}
\newtheorem{Prob}[theorem]{Question}
\newtheorem*{es}{Example}

\newtheorem{claim}{Claim}
\usepackage{pgfplots}
\pgfplotsset{compat=1.15}
\usepackage{mathrsfs}
\usetikzlibrary{arrows}

\usepackage{quoting}
\quotingsetup{font=normalsize}

\begin{document}
	\definecolor{ududff}{rgb}{0.30196078431372547,0.30196078431372547,1.}
	\definecolor{qqqqff}{rgb}{0.,0.,1.} 
	\definecolor{zzttqq}{rgb}{0.6,0.2,0.}
	
	\definecolor{ffzztt}{rgb}{1.,0.6,0.2}
	\definecolor{ffqqqq}{rgb}{1.,0.,0.}
	\definecolor{qqffqq}{rgb}{0.,1.,0.}
	\definecolor{yqqqqq}{rgb}{0.5019607843137255,0.,0.}
	\definecolor{ffqqff}{rgb}{1.,0.,1.}
	\definecolor{ffffww}{rgb}{1.,1.,0.4}
	\definecolor{qqffff}{rgb}{0.,1.,1.}
	
	\definecolor{ffqqtt}{rgb}{1.,0.,0.2} 
	\definecolor{ffffff}{rgb}{1.,1.,1.} 
	\definecolor{ffwwzz}{rgb}{1.,0.4,0.6} 
	\definecolor{ffttww}{rgb}{1.,0.,1} 
	\definecolor{ffcqcb}{rgb}{1.,0.7529411764705882,0.796078431372549} 
	\definecolor{ccqqqq}{rgb}{0.8,0.,0.} 
	\definecolor{cczzqq}{rgb}{0.8,0.6,0.}
	\definecolor{yqqqqq}{rgb}{0.5019607843137255,0.,0.} 
	\definecolor{ffccww}{rgb}{1.,0.8,0.4}
	\definecolor{ccffcc}{rgb}{0.8,1.,0.8}
	\definecolor{zzffqq}{rgb}{0.6,1.,0.}
	\definecolor{zzccqq}{rgb}{0.6,0.8,0.}
	\definecolor{qqzzqq}{rgb}{0.,0.6,0.}
	
	\definecolor{ffxfqq}{rgb}{1.,0.4980392156862745,0.}
	\definecolor{xfqqff}{rgb}{0.4980392156862745,0.,1.}
	\definecolor{ffffqq}{rgb}{1.,1.,0.}
	
	  \begin{figure}[htbp]
	  	\begin{minipage}{0.3\textwidth}
	  		\begin{flushleft}
	  			\includegraphics[scale=0.2]{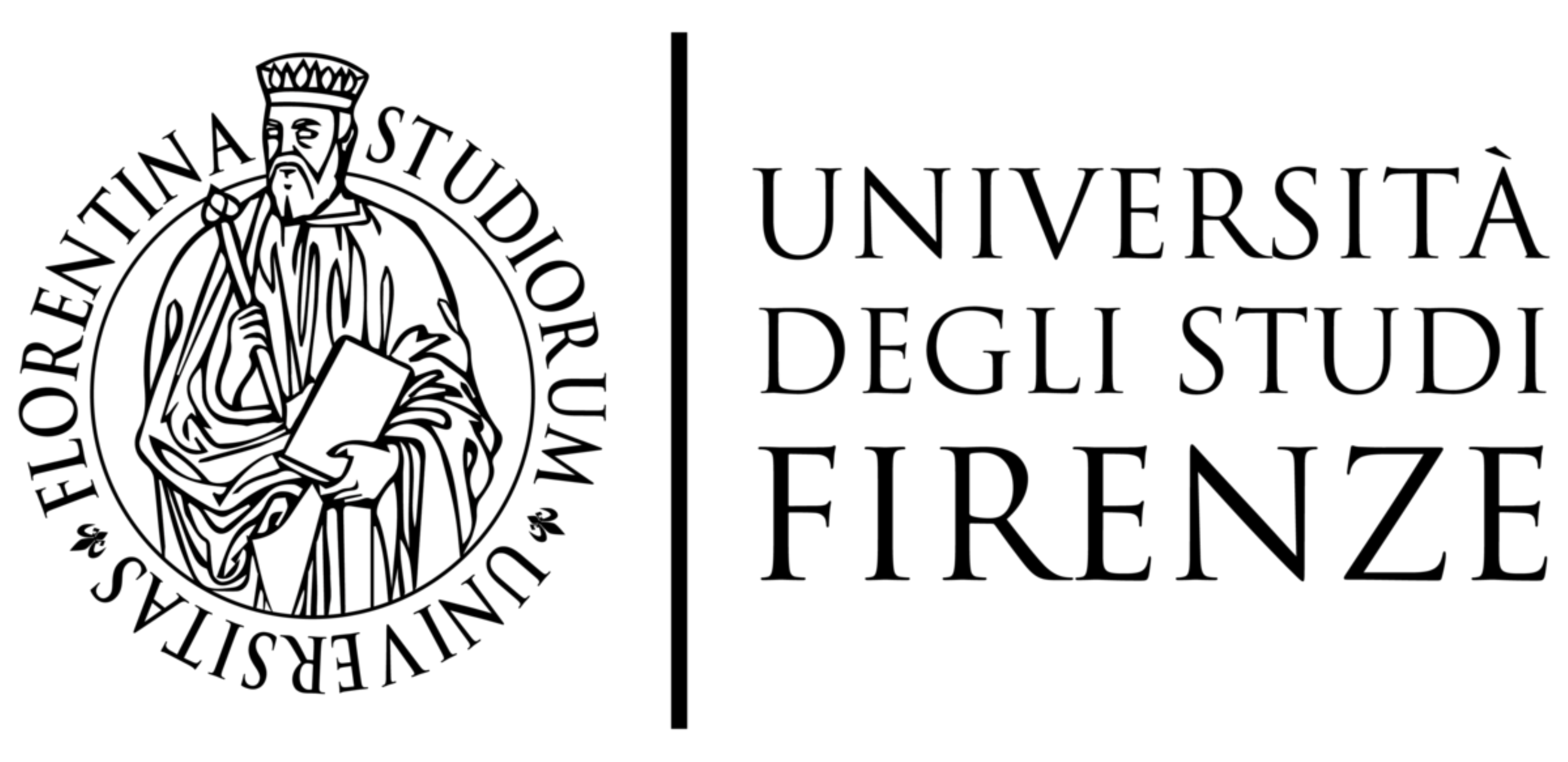}
	  		\end{flushleft}
	  	\end{minipage}~
	  	\begin{minipage}{0.7\textwidth}
	  		\vspace {18mm}
	  		\begin{flushright}
	  			{\Large \textbf{Scuola\\
	  					di Scienze\\ Matematiche,\\ Fisiche e Naturali \\}
	  			}
	  			
	  			{\Large {Corso di Laurea Magistrale \\ }}
	  			\vspace{1mm}
	  			{\Large {in Matematica}
	  			}
	  		\end{flushright}
	  	\end{minipage}
	  \end{figure}
	  \vspace{30mm}
	  \begin{center}

	  	{\LARGE{\bf{ Power Graphs of Finite Groups\\}}}

	  \end{center}
	  \vspace {30mm}
	  
	  \begin{minipage}{0.8\textwidth}
	  	\begin{flushleft}
	  		{\Large \textbf{Candidate:\\}}
	  		\vspace{3mm}
	  		{\LARGE {Nicolas Pinzauti}}
	  	\end{flushleft}
	  	\begin{flushleft}
	  		{\Large \textbf{Advisor:\\}}
	  		\vspace{3mm}
	  		{\LARGE {Professor Daniela Bubboloni}}
	  	\end{flushleft}
	  \end{minipage}
	  \vfill
	  
	  \begin{minipage}{0.8\textwidth}
	  	\begin{flushleft}
	  		{ \large {Academic Year 2020/2021}}
	  	\end{flushleft}
	  \end{minipage}
	  \thispagestyle{empty}
	
	\tableofcontents
	
	\chapter*{Acknowledgement}
	May a not Italian reader forgive me if I write this section in my language.
	
	Dopo cinque anni di studi di questa magnifica materia che è la matematica sono finalmente arrivato a termine di un percorso ricco di esperienze e di incontri, perciò molte sono le persone che devo ringraziare.
	
	Ringrazio la mia relatrice, la professoressa Daniela Bubboloni, non solo per avermi proposto questo argomento, ma anche, e soprattutto, per essere stata in grado di riaccendere in me la stessa curiosità e voglia di scoprire cose nuove che ha caratterizzato il mio primo anno universitario.   
	
	Ringrazio la mia famiglia, che nonostante non voglia più ascoltare i miei vaneggiamenti matematici, mi ha sempre sostenuto.
	
	Un grazie speciale ai \textquotedblleft ragazzi dell'uni". Le partite a Bang, le vacanze e le serate insieme mi hanno regalato dei momenti indimenticabili. E devo anche aggiungere che senza di loro non sarei stato in grado di terminare questo viaggio, o almeno non con questi risultati.
	
	Se con loro forse terminano i ringraziamenti per un aiuto diretto al raggiungimento di questo obiettivo, cominciano adesso quei ringraziamenti verso tutti coloro che hanno reso unici questi cinque anni. Troppi però sarebbero i nomi da citare e poche le parole con cui mi saprei esprimere. Ma non posso dimenticarvi, perciò a voi dedico questa frase: 
	
	\begin{quoting}
		« Non sapremo mai il bene che può fare un semplice sorriso. » \\
		S.ta Madre Teresa di Calcutta.
		
	\end{quoting}
	
	Perchè tanti e spettacolari sono i sorrisi che mi avete dedicato e fatto nascere! Grazie.
	
	Infine, ultimi, ma non per importanza, forse anzi per ringraziarli ancora di più, non uno, bensì mille grazie a Lara e Tommaso, semplicemente per tutto.
	
	Grazie.
	
	\chapter*{Introduction}
	
	Contamination between branches of mathematics is often the starting point of interesting researches and home of beautiful results. The union of two branches such as Group theory and Graph theory is known since 1878 when Cayley's graph was firstly defined. In the following years various graphs have been associated to groups. In 1955, Brauer and Fowler introduced the commuting graph in \cite{1955Commuting}, in early 2000 the directed power graphs have been related to semigroups in \cite{semigroups,semigroups_2}, and only in 2009, Chackrabarty, Gosh and Sen, in \cite{CGS_DefPowerGraph}, defined the undirected power graph, that we will call power graph.
	Approaching a group through a graph associated with it allows to focus on some specific properties of the group. For example, the commuting graph highlights the commutativity among the elements, the generating graph underlines the couples of elements that generate the group, if there is any; finally, as its name suggests, the power graph focuses on the powers of the elements.
	This subject of research it is not only useful for theoretical purpose, but there are clear connections with computer science, see \cite{Automata1, Automata2, DataMining}.

	In this thesis we will focus on directed power graphs and power graphs defined on finite groups.
	Since its appearance, the power graphs, both directed and undirected, stimulates numerous researches. The study of power graphs is still a hot topic, proven by the amount of papers released on the subject recently. 
	Just to give an idea we briefly illustrate some of them released in the years 2020-2021.
	For the reader interested in the field of graphs defined on groups we cite \cite{GraphsOnGroups} where all the graphs mentioned above with the exception of Cayley's graph are discussed. In \cite{ForbiddenGraphs}, Cameron et al. have studied for which groups the power graph is a cograph, a chordal graph or a threshold graph seeing if the power graph contains some particular induced subgraphs. Acharyya and Williams, in \cite{PowerGraphsofPowerGraoups}, have discussed properties of cliques, cycles, paths, and colouring in power graphs of finite groups, together with the construction of the longest directed path in power graphs of cyclic groups. In \cite{ConnectivityandIndependenceNumber2020} characterizations of some groups whose proper power graphs are connected are given. Bo\v{s}njak, Madar\'{a}sz and Zahirovi\'{c}, in \cite{SomeNewResults2020}, have proved that two groups with isomorphic power graphs also have isomorphic enhanced power graphs. Some results regarding finite groups with the same power graph are given in \cite{SamePG2020}. Results on normal subgroups and power graphs are given in \cite{NormalSubgrouspPG2021}. In \cite{LineGraph2021}, Bera has written about line graphs and nilpotent groups, giving a classification of all finite nilpotent groups whose power graphs are line graphs. Finally, in \cite{lambda-number}, Mishra and Sarkar have studied the $\lambda$-number for power graphs of finite $p$-groups. We can see with those recent articles that actual research on power graph has two main fields: the study of groups with particular power graph and the study of properties of the power graphs.  
	
	Among the mathematicians that mostly contributes to the development of the study of power graphs we cannot forget P. J. Cameron. Two of his articles on the subject, precisely \cite{Cameron_2, Cameron_1}, are quite omnipresent in the references of the papers produced later. For example, in our references, the following articles cite at least one of those two Cameron's articles: \cite{KelarevSurvey, PowerGraphsofPowerGraoups, SomeNewResults2020, LineGraph2021, ForbiddenGraphs, GraphsOnGroups, ResultsOnFiniteGroups, Ma et al, CyclicSubgroupGraphofFiniteGroup, SamePG2020, NormalSubgrouspPG2021, lambda-number}. So which better starting point if not this two articles to begin the study of power graph? \\
	We deeply studied the papers \cite{Cameron_2, Cameron_1} and decide to rewrite them in order to produce more detailed proofs. Indeed in many parts the proofs in those papers are just sketched. In this work of rewriting we have corrected some errors and we explained some not clear steps. We put particular effort into the correction of a crucial result in \cite{Cameron_2}. That result is strongly used to get the final conclusion, but it seems to be not completely proven. Indeed a particular case was not covered by Cameron's proof. This new case was not easy to manage until, in the article \cite{Ma et al} (released six years after \cite{Cameron_2} publication) we have found a way to extend Cameron's idea and conclude the proof including the new case.
	Cameron's result focus on the distinction between two distinct type of equivalence classes for a certain equivalence relation using an appropriate argument based on the sizes of the class studied and of an object related to it. The proof provided in the paper shows the existence of what we will refer to as \emph{critical classes} for which the Cameron's argument does not distinguish the type. 
	Also in \cite{Cameron_1}, in the proof of the main theorem, there was a problem. The provided proof claim to cover all possible primes but missed $p=2$.
	These two mistakes have a common nature, they are both originated by a missing formalization of an intuition, but they are completely different in the outcomes.
	In \cite{Cameron_1}, the missing case needed nothing more that an adjustment to make the proof work properly so it will not echoes in the future research. Instead, the introduction of critical classes could be an interesting starting point for some new research related to power graphs of finite groups.
	
	The work of rewriting takes four of the six chapters of this thesis. It allows us a deep study of some equivalence relations defined on the vertex set of a graph. This study was what naturally brings us to quotient graphs, and then to maximal path and cycles.
	Quotient graphs are useful in the research of maximal path and cycles. We extend the idea given in \cite{PowerGraphsofPowerGraoups} and we found a lower bound for the length of a maximum cycle of a power graph searching for a maximal path in a specific quotient graph. The research in this field is far from be over since we simply introduce the players and give some easy results on a subject of research that could offer newsworthy results.  
	
	
	
	After some preliminaries and the setting of a clear notation in Chapter \ref{Preliminaries}, we transcribe the article \cite{Cameron_2, Cameron_1} in a more detailed way, and give the right path to follow to get the main results, in Chapters \ref{UPGeDPG} and \ref{SectionReconstruction}.
	In Chapter \ref{RelatedResults} we announce one of the major result related to the main theorem of the previous chapter. With this chapter we conclude the rewriting of \cite{Cameron_2, Cameron_1}.
	
	We then continue with a brief glance on quotient graphs (Chapter \ref{QuotientGraphs}). 
	For this and the following chapters, some new definitions and notation are needed and we present them at the beginning of the chapter.
	Finally, in Chapter \ref{MaximalCycles}, we conclude the thesis with an introductory study of maximal cycles and paths, both directed and undirected, in the power graphs. Quotient graphs play a crucial role in getting some of these results.  
	In all this chapters some definitions and results are followed by examples and representative images. 
	    
	We hope that this thesis allows a better comprehension on the articles \cite{Cameron_2, Cameron_1} and that it will stimulate new research on this beautiful and rich area. 
	To that purpose we propose in Section \ref{OpenProblems} some open questions about critical classes. Furthermore also maximal path and cycles have to be deeply investigated, we only scratched the surface.

	\newpage
	
	\chapter{Preliminaries and background information}
	\label{Preliminaries}	
		Many graphs have been associated with groups. In this chapter we are going to define the ones on which we focus our research. In particular the power graph and the directed power graph of a group. We also recall some definitions and results and we set the notation. 
		
		\section{Recalls}
		With $\mathbb{N}$ we refer to the set of all positive integers. We denote with $\mathbb{N}_0$ the set $\mathbb{N} \cup \left\lbrace 0 \right\rbrace $.
		For $k\in \mathbb{N}_0$ we set $[k]:= \{n \in \mathbb{N} \, | \, n\leq k\} $. So $[0]=\emptyset$, and then $|[k]|=k$.
		We also use $[k]_0=\{n\in \mathbb{N}_0 \, | \, n \leq k\}$, note that $[0]_0=\{0\}$.
		
		We recall some specific results and notation for groups. The groups considered in this thesis are all finite.\\
		Let $G$ be a group. The identity element of $G$, will be simply denoted with $1$. \\ If $H$ is a subgroup of $G$, then we write $H\leq G$ and if $H$ is a proper subgroup, then we write $H<G$.
		

		Given $X\subseteq G$, $\langle X \rangle$ denotes the subgroup generated by $X$, that is the minimum subgroup of $G$ containing $X$. $\langle X \rangle$ is given by the intersection of all the subgroups of $G$ containing $X$.
		Let $g\in G$. Then we write $\langle g \rangle$ instead of $\langle \{g\} \rangle$. Recall that 
		$$\langle g \rangle=\left\lbrace g^k \, | \, k\in \mathbb{Z}\right\rbrace$$  and that the \emph{order of $g$} is defined by $$o(g):=|\langle g \rangle|.$$ Since $G$ is finite there exists $k\in \mathbb{N}$ such that $g^k=1$ and $o(g)= min \left\lbrace k\in \mathbb{N} \, | \, g^k=1 \right\rbrace $.  As a consequence, if $o(g)=m$, then we have that 
		$$\langle g \rangle=\left\lbrace g^k \, | \, 0\leq k\leq m-1\right\rbrace.$$
		
		If $o(g)=2$, then $g$ is called an \emph{involution}.
		
		$G$ is called  \emph{cyclic} if there exists $g\in G$ such that  $G=\langle g \rangle$. Remember that $G$ is cyclic  if and only if, for all $k \mid |G|$, there exists a unique $H\leq G$ such that $|H|=k$. \\
		If  $G=\langle g \rangle$, then we have that 
		$$\left\lbrace g\in G \, |\, o(g)=|G| \right\rbrace=\left\lbrace  g\in G \, | \, \langle g \rangle = G\right\rbrace. $$ Moreover $g^k$ is a generator  of $G$ if and only if $\gcd(k, o(g))=1$.\\
		If $G$ is cyclic of size $n$ it is denoted $C_n$.
		We also use the notation $D_n$ to refer to a \emph{dihedral group} of size $2n$, that is the group with the following presentation:
		\begin{equation}
			 \label{Dn_presentation}
			 D_n:=\langle a,b \, | \, a^n=1=b^2, a^b=a^{-1}\rangle
		\end{equation}
		Let $p$ be a prime number. If, for all $g\in G$, there exists $k\in \mathbb{N}$ such that $o(g)=p^k$, then $G$ is called a \emph{$p$-group}. Since we are dealing with finite groups, equivalently, $G$ is a $p$-group if $|G|$ is a power of $p$.\\
		Remember also that, for $n \in \mathbb{N}$, the \emph{generalised quaternions} are defined by the presentation:
		
		$$Q_{4n}:=\langle x,y \,| \, x^n=y^2, x^{2n}=y^4=1, y^{-1}xy=x^{-1} \rangle .$$ 
		
		It is easily seen that $|Q_{4n}|=4n$. Hence if $Q_{4n}$ is a $p$-group, then $p=2$ and the group is given by:
		$$ Q_{2^{n+1}}:=\langle x, y \, |\, x^{2^{n-1}}=y^2, x^{2^{n}}=y^4=1, y^{-1}xy=x^{-1}\rangle .$$
		Note that for $n=2$ we obtain the \emph{quaternion group} $Q_8$.  
		
		Remember that, for $g\in G$, the \emph{centralizer of $g$ in $G$} is the set $$C_G(g):=\{h\in G \, | \, gh=hg\}.$$ 
		We now state some classic results about the structure of particular groups.
		
		\begin{theorem}{\rm \cite[Theorem 2.1.3]{GroupTheory}}
			Every Abelian group is the direct product of cyclic groups.
		\end{theorem}
		
		\begin{corollary}
			\label{structureTheoremOfAbelianpGroups} Let $p$ be a prime number. Every finite abelian $p$-group is isomorphic to direct product of cyclic $p$-groups.		
		\end{corollary}
	
		Let $p$ be a prime number and $G$ an abelian $p$-group. By the previous corollary, if we know the number of elements of each order in $G$ we also know $G$ up to group isomorphism.
		
		Another well-known result is the following.
		 
		\begin{theorem}{\rm \cite[Theorem 5.1.4]{GroupTheory}}
			\label{structureTheoremOfNilpotentGroup} Every nilpotent group is isomorphic to a direct product of its Sylows subgroups.
		\end{theorem}	 
		
		We  are going to freely use  some of the results stated in this section without further references.
		
		\section{Graphs and digraphs}
		In this section we give some very preliminary definitions about graph theory.
		
			\begin{definition}
				\rm We define a \emph{graph} $\Gamma$ as a couple $\Gamma=(V_{\Gamma}, E_{\Gamma})$ where $V_{\Gamma}$ is a not empty finite set of elements called \emph{vertices} and $E_{\Gamma}$ is a subset of subsets of $V_{\Gamma}$ of size $2$. If $e=\left\lbrace x,y\right\rbrace $ for some distinct $x$ and $y$ in $V_{\Gamma}$, we also say that such elements, $x$ and $y$, are \emph{joined} or \emph{adjacent}. The elements of $E_{\Gamma}$ are called \emph{edges}.
			\end{definition}

			\begin{definition}
				\rm We define a \emph{directed graph} $\vec{\Gamma}$, or \emph{digraph}, as a couple $\vec{\Gamma}=(V_{\vec{\Gamma}}, A_{\vec{\Gamma}})$ where $V_{\vec{\Gamma}}$ is a not empty finite set of elements called \emph{vertices} and $A_{\vec{\Gamma}}$, whose elements are called \emph{arcs}, is a subset of $(V_{\vec{\Gamma}}\times V_{\vec{\Gamma}}) \setminus \Delta$ where $\Delta=\left\lbrace (x,x) \, |\, x \in V_{\vec{\Gamma}}\right\rbrace $. Therefore each $a \in A_{\vec{\Gamma}}$ is of the form $a=(x,y) $ for some $x$ and $y$ in $V_{\vec{\Gamma}}$ such that $x \not= y$. 
				We say that, if $(x,y)$ or $(y,x)$ is in $A_{\vec{\Gamma}}$, then $x$ and $y$ are \emph{joined}.
				We also say that $a=(x,y)\in A_{\vec{\Gamma}}$ \emph{is directed} from $x$ to $y$ and also that $a$ has \emph{direction} from $x$ to $y$. If $(x,y)\in A_{\vec{\Gamma}}$, then it said $x$ \emph{dominates} $y$.
				Given $X,Y \subseteq V_{\vec{\Gamma}}$, we say that $B\subseteq A_{\vec{\Gamma}}$ is \emph{composed by arcs from $X$ to $Y$} if, for all $b\in B$, we have $b=(x, y)$ for some $x\in X$ and $y \in Y$.
			\end{definition}
		
			If $e=\left\lbrace x,y \right\rbrace $ is an edge of a graph, then $x$ and $y$ are called the \emph{end vertices} of $e$. In the same way, if $a=(x,y)$ is an arc of a digraph, then $x$ and $y$ are called the \emph{end vertices} of $a$.
			Given two distinct subsets $X$ and $Y$ of $V_{\Gamma}$, for $\Gamma$ a graph, we say that $X$ and $Y$ are \emph{joined} if there is an edge with an end vertex in $X$ and the other in $Y$.
			Given two distinct subsets $X$ and $Y$ of $V_{\vec{\Gamma}}$, for $\vec{\Gamma}$ a digraph, we say that $X$ and $Y$ are \emph{joined} if there is an arc with an end vertex in $X$ and the other in $Y$.
			For graphs or digraphs, if $\Gamma$ or $\vec{\Gamma}$ is clear from the context, then we write $(V,E)$ instead of $(V_{\Gamma}, E_{\Gamma})$ and $(V,A)$ instead of $(V_{\vec{\Gamma}}, A_{\vec{\Gamma}})$.
			
			The cardinality of the vertex set of both graphs and digraphs is called the \emph{order} of the graph/digraph.
			
			A digraph $\vec{\Gamma}$ is \emph{transitive} if given $x,y,z \in V$ such that $(x,y), (y,z)\in A$, then we have that $(x,z)\in A$.
			
			Given a graph $\Gamma$ we call \emph{subgraph} of $\Gamma$ a graph $\Gamma'$ with $V_{\Gamma'}\subseteq V_{\Gamma}$ and $E_{\Gamma'}\subseteq E_{\Gamma}$. A subgraph $\Gamma'$ is \emph{induced} by his vertex set if, for all $x,y \in V_{\Gamma'}$, $\left\lbrace x,y \right\rbrace \in E_{\Gamma'}$ if and only if $\left\lbrace x,y \right\rbrace \in E_{\Gamma}$. Given $X\subseteq V_{\Gamma}$, we denote by $\Gamma_X$ the induced subgraph of $\Gamma$ with $V_{\Gamma_X}=X$.\\
			Given a digraph $\vec{\Gamma}$ we call \emph{subdigraph} a digraph $\vec{\Gamma}'$ with $V_{\vec{\Gamma}'}\subseteq V_{\vec{\Gamma}}$ and $A_{\vec{\Gamma}'}\subseteq A_{\vec{\Gamma}}$. We also say that a subdigraph $\vec{\Gamma}'$ is \emph{induced} by his vertex set if, for all $x,y \in V_{\vec{\Gamma}'}$, $(x,y) \in A_{\vec{\Gamma}'}$ if and only if $(x,y) \in A_{\vec{\Gamma}}$.
			Given $X\subseteq V_{\vec{\Gamma}}$, we denote by $\vec{\Gamma}_X$ the induced subdigraph of $\vec{\Gamma}$ with $V_{\vec{\Gamma}_X}=X$.
			
			It is called \emph{complete graph} of $n$ vertices, denoted by $K_n$, a graph where $\left\lbrace x,y \right\rbrace \in E_{K_n} $ for all $x,y \in V_{K_n}$. We call \emph{complete digraph} of $n$ vertices, denoted by $\vec{K_n}$, the digraph such that $A_{\vec{K_n}}=V_{\vec{K_n}} \times V_{\vec{K_n}} \setminus \Delta$.

			Usually, in graph theory, each vertex of a graph $\Gamma$ is represented as a point; an edge $\{x,y\}\in E$ is a line joining vertex points $x$ and $y$.\\
			For a digraph $\vec{\Gamma}$ the representation is quite similar, with the only exception that an arc $(x,y)\in A_{\vec{\Gamma}}$ it is no more a line between the two end vertices, but an arrow that is directed from $x$ to $y$.\\
			In Figure \ref{Complete_graph_digraph} an example of a complete graph and a complete digraph.
			
			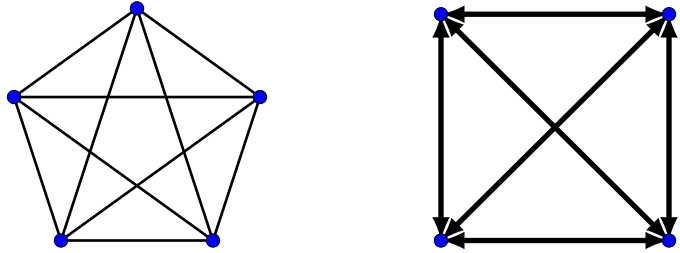
\begin{figure}
				\centering
				\begin{tikzpicture}[line cap=round,line join=round,>=triangle 45,x=1.0cm,y=1.0cm]
					\clip(-4.,-2.) rectangle (7.,3.);
					
					\draw [line width=1.pt] (-2.,-1.)-- (0.,-1.);
					\draw [line width=1.pt] (0.,-1.)-- (0.6180339887498949,0.9021130325903065);
					\draw [line width=1.pt] (0.6180339887498949,0.9021130325903065)-- (-1.,2.077683537175253);
					\draw [line width=1.pt] (-1.,2.077683537175253)-- (-2.618033988749895,0.9021130325903073);
					\draw [line width=1.pt] (-2.618033988749895,0.9021130325903073)-- (-2.,-1.);
					\draw [line width=1.pt] (-1.,2.077683537175253)-- (0.,-1.);
					\draw [line width=1.pt] (0.,-1.)-- (-2.618033988749895,0.9021130325903073);
					\draw [line width=1.pt] (-2.618033988749895,0.9021130325903073)-- (0.6180339887498949,0.9021130325903065);
					\draw [line width=1.pt] (-1.,2.077683537175253)-- (-2.,-1.);
					\draw [line width=1.pt] (-2.,-1.)-- (0.6180339887498949,0.9021130325903065);
					\draw [latex-latex, line width=2.pt] (3.,-1.) -- (3.,2.);
					\draw [latex-latex, line width=2.pt] (3.,-1.) -- (6.,2.);
					\draw [latex-latex, line width=2.pt] (3.,-1.) -- (6.,-1.);
					\draw [latex-latex, line width=2.pt] (6.,-1.) -- (6.,2.);
					\draw [latex-latex, line width=2.pt] (6.,2.) -- (3.,2.);
					\draw [latex-latex, line width=2.pt] (6.,-1.) -- (3.,2.);
					\begin{scriptsize}
						\draw [fill=qqqqff] (-2.,-1.) circle (2.5pt);
						\draw [fill=qqqqff] (0.,-1.) circle (2.5pt);
						\draw [fill=qqqqff] (0.6180339887498949,0.9021130325903065) circle (2.5pt);
						\draw [fill=qqqqff] (-1.,2.077683537175253) circle (2.5pt);
						\draw [fill=qqqqff] (-2.618033988749895,0.9021130325903073) circle (2.5pt);
						\draw [fill=qqqqff] (3.,-1.) circle (2.5pt);
						\draw [fill=qqqqff] (3.,2.) circle (2.5pt);
						\draw [fill=qqqqff] (6.,2.) circle (2.5pt);
						\draw [fill=qqqqff] (6.,-1.) circle (2.5pt);
					\end{scriptsize}
				\end{tikzpicture}
				\caption{$K_5$ and $\vec{K_4}$}
				\label{Complete_graph_digraph}
			\end{figure}
			
			\begin{definition}
				\rm \label{Def_orientation}Let $\Gamma$ be a graph. The digraph obtained by replacing each edge of $\Gamma$ by just one of the two possible arcs with the same ends is called an \emph{orientation  of $\Gamma$}. 
			\end{definition}
			
			In Figure \ref{orientation} there are some possible orientation of $K_4$.
			
			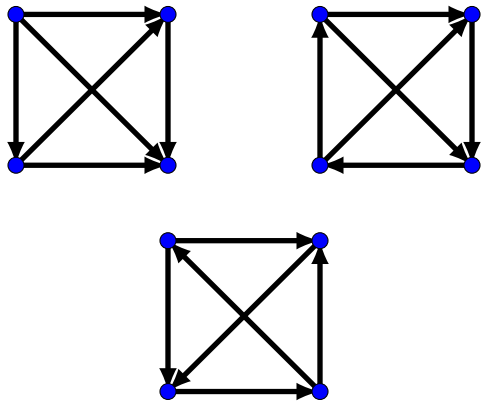
\begin{figure}
				\centering
				\begin{tikzpicture}[line cap=round,line join=round,>=triangle 45,x=1.0cm,y=1.0cm]
					\clip(-1.5,-0.5) rectangle (5.5,5.5);
					\draw [-latex,line width=2.pt] (-1.,5.) -- (1.,5.);
					\draw [-latex,line width=2.pt] (-1.,5.) -- (1.,3.);
					\draw [-latex,line width=2.pt] (-1.,5.) -- (-1.,3.);
					\draw [-latex,line width=2.pt] (-1.,3.) -- (1.,3.);
					\draw [-latex,line width=2.pt] (1.,5.) -- (1.,3.);
					\draw [-latex,line width=2.pt] (-1.,3.) -- (1.,5.);
					\draw [-latex,line width=2.pt] (3.,3.) -- (3.,5.);
					\draw [-latex,line width=2.pt] (3.,5.) -- (5.,5.);
					\draw [-latex,line width=2.pt] (5.,5.) -- (5.,3.);
					\draw [-latex,line width=2.pt] (5.,3.) -- (3.,3.);
					\draw [-latex,line width=2.pt] (3.,3.) -- (5.,5.);
					\draw [-latex,line width=2.pt] (3.,5.) -- (5.,3.);
					\draw [-latex,line width=2.pt] (1.,2.) -- (3.,2.);
					\draw [-latex,line width=2.pt] (3.,2.) -- (1.,0.);
					\draw [-latex,line width=2.pt] (1.,0.) -- (3.,0.);
					\draw [-latex,line width=2.pt] (3.,0.) -- (1.,2.);
					\draw [-latex,line width=2.pt] (1.,2.) -- (1.,0.);
					\draw [-latex,line width=2.pt] (3.,0.) -- (3.,2.);
					\begin{scriptsize}
						\draw [fill=qqqqff] (1.,5.) circle (3.pt);
						\draw [fill=qqqqff] (1.,3.) circle (3.pt);
						\draw [fill=qqqqff] (-1.,3.) circle (3.pt);
						\draw [fill=qqqqff] (-1.,5.) circle (3.pt);
						\draw [fill=qqqqff] (3.,5.) circle (3.pt);
						\draw [fill=qqqqff] (5.,5.) circle (3.pt);
						\draw [fill=qqqqff] (5.,3.) circle (3.pt);
						\draw [fill=qqqqff] (3.,3.) circle (3.pt);
						\draw [fill=qqqqff] (1.,2.) circle (3.pt);
						\draw [fill=qqqqff] (3.,2.) circle (3.pt);
						\draw [fill=qqqqff] (3.,0.) circle (3.pt);
						\draw [fill=qqqqff] (1.,0.) circle (3.pt);
					\end{scriptsize}
				\end{tikzpicture}
				\caption{Some possible orientation of $K_4$.}
				\label{orientation}
			\end{figure}
			
			\begin{definition}
				\label{DefPath}\rm A \emph{path} is a graph whose vertices can be arranged in an ordered sequence in such a way that two vertices are adjacent if and only if they are consecutive in the sequence.
				
				A \emph{directed path} is an orientation of a path in which each vertex dominates its successor in the sequence.
			\end{definition}

			For a path $\mathcal{W}$, $|E_{\mathcal{W}}|$ is the \emph{length} of $\mathcal{W}$, denoted as $\ell(\mathcal{W})$. A \emph{$k$-path} is a path with length $k$. Likewise $\ell(\vec{\mathcal{W}}):=|A_{\vec{\mathcal{W}}}|$ is the \emph{length} of a directed path $\vec{\mathcal{W}}$ and a \emph{directed $k$-path} is a directed path of length $k$.
			
			By definition, the vertices of a path $\mathcal{W}$ can be arranged in an ordered sequence such that two vertices are adjacent if and only if they are consecutive in the sequence. There exist two of such arrangements, both of them determine the path $\mathcal{W}$. If, for $k\in \mathtt{N}$, $V_{\mathcal{W}}=\{x_1, \dots, x_k\}$ and $E_{\mathcal{W}}=\{\{x_i, x_{i+1}\} \, | \, i\in [k-1]\}$, then $\mathcal{W}$ is uniquely determined by both the ordered sequences
			$x_1,x_2,\dots, x_k$ and $x_k, x_{k-1}, \dots, x_1$. So we represent the path $\mathcal{W}$ as one of those arrangements. Hence we write $\mathcal{W}=x_1, x_2 \dots, x_k$.
			In such a representation $x_1$ and $x_k$ are called the \emph{end vertices of $\mathcal{W}$}.
		
			It is similar for directed paths, but there are some differences.
			For a directed path $\vec{\mathcal{W}}$ there exists only one ordered sequence that determine the path. \\ $\vec{\mathcal{W}}=x_1, \dots, x_k$ is a directed path with $V_{\vec{\mathcal{W}}}=\{x_1, \dots, x_k\}$ and $A_{\vec{\mathcal{W}}}=\{(x_i, x_{i+1}) \, | \, i\in [k-1]\}$. $x_1$ and $x_k$ are the \emph{end vertices} of $\vec{\mathcal{W}}$. We also say that $\vec{\mathcal{W}}$ starts in $x_1$ and ends in $x_k$.
			
			In Figures \ref{Esempio3-path} and \ref{Esempi_directed_3-path} there are three examples of path and directed path. Note that $\mathcal{W}_1$ and $\mathcal{W}_2$ are two distinct orientation of $\mathcal{W}$ in Figure \ref{Esempio3-path}.
			
			\begin{figure}
				\centering
				\begin{tikzpicture}[line cap=round,line join=round,>=triangle 45,x=1.0cm,y=1.0cm]
					\clip(-1.5,2.5) rectangle (5.5,3.5);
					\draw [line width=2.pt] (-1.,3.)-- (1.,3.);
					\draw [line width=2.pt] (1.,3.)-- (3.,3.);
					\draw [line width=2.pt] (3.,3.)-- (5.,3.);
					\draw (-1.,3.) node[anchor=north west] {$a$};
					\draw (1.,3.) node[anchor=north west] {$b$};
					\draw (3.,3.) node[anchor=north west] {$c$};
					\draw (5.,3.) node[anchor=north west] {$d$};
					\begin{scriptsize}
						\draw [fill=qqqqff] (-1.,3.) circle (3.pt);
						\draw [fill=qqqqff] (1.,3.) circle (3.pt);
						\draw [fill=qqqqff] (3.,3.) circle (3.pt);
						\draw [fill=qqqqff] (5.,3.) circle (3.pt);
					\end{scriptsize}
				\end{tikzpicture}
				\caption{A $3$-path: $\mathcal{W}=a,b,c,d$}
				\label{Esempio3-path}
			\end{figure}
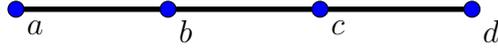
			
			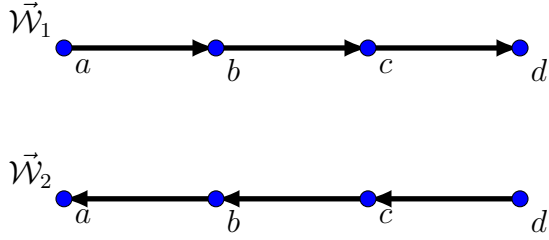
\begin{figure}
				\centering
				\begin{tikzpicture}[line cap=round,line join=round,>=triangle 45,x=1.0cm,y=1.0cm]
					\clip(-2,0.5) rectangle (5.5,3.6);
					\draw [-latex,line width=2.pt] (-1.,3.) -- (1.,3.);
					\draw [-latex,line width=2.pt] (1.,3.) -- (3.,3.);
					\draw [-latex,line width=2.pt] (3.,3.) -- (5.,3.);
					\draw (-1.,3.) node[anchor=north west] {$a$};
					\draw (1.,3.) node[anchor=north west] {$b$};
					\draw (3.,3.) node[anchor=north west] {$c$};
					\draw (5.,3.) node[anchor=north west] {$d$};
					\draw (-1.,3.) node[anchor=south east] {$\vec{\mathcal{W}}_1$};
					
					\draw [latex-,line width=2.pt] (-1.,1.) -- (1.,1.); 
					\draw [latex-,line width=2.pt] (1.,1.) -- (3.,1.);
					\draw [latex-,line width=2.pt] (3.,1.) -- (5.,1.);
					\draw (-1.,1.) node[anchor=north west] {$a$};
					\draw (1.,1.) node[anchor=north west] {$b$};
					\draw (3.,1.) node[anchor=north west] {$c$};
					\draw (5.,1.) node[anchor=north west] {$d$};
					\draw (-1.,1.) node[anchor=south east] {$\vec{\mathcal{W}}_2$};
					\begin{scriptsize}
						\draw [fill=qqqqff] (-1.,3.) circle (3.pt);
						\draw [fill=qqqqff] (1.,3.) circle (3.pt);
						\draw [fill=qqqqff] (3.,3.) circle (3.pt);
						\draw [fill=qqqqff] (5.,3.) circle (3.pt);
						\draw [fill=qqqqff] (-1.,1.) circle (3.pt);
						\draw [fill=qqqqff] (1.,1.) circle (3.pt);
						\draw [fill=qqqqff] (3.,1.) circle (3.pt);
						\draw [fill=qqqqff] (5.,1.) circle (3.pt);
					\end{scriptsize}
				\end{tikzpicture}
				\caption{Two distinct directed $3$-paths: $\vec{\mathcal{W}}_1=a,b,c,d$ and $\vec{\mathcal{W}}_2=d,c,b,a$ }
				\label{Esempi_directed_3-path}
			\end{figure}
			
			\begin{definition}
				\label{DefCycle}\rm A \emph{cycle} on three or more vertices is a graph whose vertices can be arranged in a cyclic sequence so that two vertices are adjacent if and only if they are consecutive in the sequence. 
				
				A \emph{directed cycle} is an orientation of a cycle in which each vertex dominates its successor in the cyclic sequence.
			\end{definition}
			
			Let $\mathcal{C}$ be a cycle and $\vec{\mathcal{C}}$ a directed cycle. $|E_{\mathcal{C}}|$ is the \emph{length} of the cycle and $|A_{\vec{\mathcal{C}}}|$ is the \emph{length} of the directed cycle, denoted, respectively, $\ell(\mathcal{C})$ and $\ell(\vec{\mathcal{C}})$. A \emph{$k$-cycle} and a \emph{directed $k$-cycle} are, respectively, a cycle of length $k$ and a directed cycle of length $k$.
			
			As in the case of paths, we represent a cycle with an ordered list of vertices. Note that, for a cycle $\mathcal{C}$, there exist more then two possible arrangements that determine the same cycle, indeed $\mathcal{C}$ does not have ending points.
			In fact, a cycle $\mathcal{C}$ with $V_{\mathcal{C}}=\{x_1, \dots, x_k\}$ and $E_{\mathcal{C}}=\{\{x_i,x_{i+1}\} \, | \, i\in [k-1]\}\cup \{\{x_1, x_k\}\}$ can be represented, for all $i\in[k]$, by one of the following sequences:
			$$x_i, x_{i+1}, \dots, x_k,x_1, x_2, \dots, x_{i-1},x_i$$
			$$x_i, x_{i-1}, \dots, x_1,x_k,x_{k-1}, \dots, x_{i+1},x_i$$
			Then we write $\mathcal{C}=x_1,x_2 \dots, x_k,x_1$.
			
			Also for a directed cycle $\vec{\mathcal{C}}$ there exist more than two possible ordered sequences of vertices that determine the cycle. Consider $\vec{\mathcal{C}}$ with $V_{\vec{\mathcal{C}}}=\{x_1, \dots, x_k\}$ and $A_{\vec{\mathcal{C}}}=\{(x_i, x_{i+1})\, | \, i\in [k-1]\} \cup \{(x_k,x_1)\}$. The sequences $$x_i, x_{i+1}, \dots, x_k,x_1, x_2, \dots, x_{i-1},x_i$$ for all $i \in [k]$ determine uniquely the directed cycle $\vec{\mathcal{C}}$.\\
			Hence we write $\vec{\mathcal{C}}=x_1,x_2 \dots, x_k,x_1$. 
			
			There are three examples of cycles and directed cycles in Figures \ref{Esempio5-cycle} and \ref{EsempioDirected5-cycle}. Note that $\vec{\mathcal{C}}_1$ and $\vec{\mathcal{C}}_2$ are distinct orientations of $\mathcal{C}$.
		
			\begin{figure}
				\centering
				\begin{tikzpicture}[line cap=round,line join=round,>=triangle 45,x=1.0cm,y=1.0cm]
					\clip(0.,2.5) rectangle (4.,6.5);
					\draw [line width=1.pt] (0.3819660112501053,4.902113032590307)-- (1.,3.);
					\draw [line width=1.pt] (1.,3.)-- (3.,3.);
					\draw [line width=1.pt] (3.,3.)-- (3.618033988749895,4.902113032590306);
					\draw [line width=1.pt] (3.618033988749895,4.902113032590306)-- (2.,6.077683537175253);
					\draw [line width=1.pt] (2.,6.077683537175253)-- (0.3819660112501053,4.902113032590307);
					\draw (1.,3.) node[anchor=north west] {$a$};
					\draw (3.,3.) node[anchor=north west] {$b$};
					\draw (3.618033988749895,4.902113032590306) node[anchor=west] {$c$};
					\draw (2.,6.077683537175253) node[anchor=south west] {$d$};
					\draw (0.3819660112501053,4.902113032590307)
					node[anchor=east] {$e$};
					\begin{scriptsize}
						\draw [fill=qqqqff] (1.,3.) circle (3pt);
						\draw [fill=qqqqff] (3.,3.) circle (3pt);
						\draw [fill=qqqqff] (3.618033988749895,4.902113032590306) circle (2.5pt);
						\draw [fill=qqqqff] (2.,6.077683537175253) circle (3pt);
						\draw [fill=qqqqff] (0.3819660112501053,4.902113032590307) circle (3pt);
						\draw [fill=qqqqff] (3.,2.) circle (3pt);
						\draw [fill=qqqqff] (1.,2.) circle (3pt);
						\draw [fill=qqqqff] (0.3819660112501051,0.09788696740969349) circle (3pt);
						\draw [fill=qqqqff] (2.,-1.0776835371752531) circle (3pt);
						\draw [fill=qqqqff] (3.618033988749895,0.09788696740969272) circle (3pt);
					\end{scriptsize}
				\end{tikzpicture}
				\caption{A $5$-cycle $\mathcal{C}=a,b,c,d,e,a$}
				\label{Esempio5-cycle}
			\end{figure}
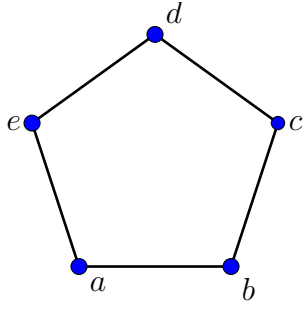
			
			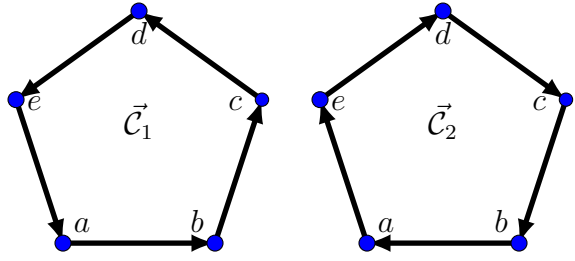
\begin{figure}
				\centering
				\begin{tikzpicture}[line cap=round,line join=round,>=triangle 45,x=1.0cm,y=1.0cm]
					\clip(0.,2.5) rectangle (8.,6.5);
					\draw [-latex, line width=2.pt] (0.3819660112501053,4.902113032590307)-- (1.,3.);
					\draw [-latex, line width=2.pt] (1.,3.)-- (3.,3.);
					\draw [-latex, line width=2.pt] (3.,3.)-- (3.618033988749895,4.902113032590306);
					\draw [-latex, line width=2.pt] (3.618033988749895,4.902113032590306)-- (2.,6.077683537175253);
					\draw [-latex, line width=2.pt] (2.,6.077683537175253)-- (0.3819660112501053,4.902113032590307);
					\draw (1.,3.) node[anchor=south west] {$a$};
					\draw (3.,3.) node[anchor=south east] {$b$};
					\draw (3.5,4.85) node[anchor= east] {$c$};
					\draw (2.,6.077683537175253) node[anchor=north] {$d$};
					\draw (0.3819660112501053,4.85)
					node[anchor= west] {$e$};
					\draw (2., 5.) node[anchor=north] {$\vec{\mathcal{C}}_1$};
					
					\draw [latex-, line width=2.pt] (0.3819660112501053+4,4.902113032590307)-- (1.0+4,3.0);
					\draw [latex-, line width=2.pt] (1.0+4,3.0)-- (3.0+4,3.0);
					\draw [latex-, line width=2.pt] (3.+4,3.)-- (3.618033988749895+4,4.902113032590306);
					\draw [latex-, line width=2.pt] (3.618033988749895+4,4.902113032590306)-- (2.+4,6.077683537175253);
					\draw [latex-, line width=2.pt] (2.+4,6.077683537175253)-- (0.3819660112501053+4,4.902113032590307);
					\draw (1.+4,3.) node[anchor=south west] {$a$};
					\draw (3.+4,3.) node[anchor=south east] {$b$};
					\draw (3.5+4,4.85) node[anchor= east] {$c$};
					\draw (2.+4,6.077683537175253) node[anchor=north] {$d$};
					\draw (0.3819660112501053+4,4.85)
					node[anchor= west] {$e$};
					\draw (2.+4, 5.) node[anchor=north] {$\vec{\mathcal{C}}_2$};
					\begin{scriptsize}
						\draw [fill=qqqqff] (1.,3.) circle (3pt);
						\draw [fill=qqqqff] (3.,3.) circle (3pt);
						\draw [fill=qqqqff] (3.618033988749895,4.902113032590306) circle (2.5pt);
						\draw [fill=qqqqff] (2.,6.077683537175253) circle (3pt);
						\draw [fill=qqqqff] (0.3819660112501053,4.902113032590307) circle (3pt);
						\draw [fill=qqqqff] (3.,2.) circle (3pt);
						\draw [fill=qqqqff] (1.,2.) circle (3pt);
						\draw [fill=qqqqff] (0.3819660112501051,0.09788696740969349) circle (3pt);
						\draw [fill=qqqqff] (2.,-1.0776835371752531) circle (3pt);
						\draw [fill=qqqqff] (3.618033988749895,0.09788696740969272) circle (3pt);
						
						\draw [fill=qqqqff] (1.+4,3.) circle (3pt);
						\draw [fill=qqqqff] (3.+4,3.) circle (3pt);
						\draw [fill=qqqqff] (3.618033988749895+4,4.902113032590306) circle (2.5pt);
						\draw [fill=qqqqff] (2.+4,6.077683537175253) circle (3pt);
						\draw [fill=qqqqff] (0.3819660112501053+4,4.902113032590307) circle (3pt);
						\draw [fill=qqqqff] (3.+4,2.) circle (3pt);
						\draw [fill=qqqqff] (1.+4,2.) circle (3pt);
						\draw [fill=qqqqff] (0.3819660112501051+4,0.09788696740969349) circle (3pt);
						\draw [fill=qqqqff] (2.+4,-1.0776835371752531) circle (3pt);
						\draw [fill=qqqqff] (3.618033988749895+4,0.09788696740969272) circle (3pt);
					\end{scriptsize}
				\end{tikzpicture}
				\caption{Two distinct directed $5$-cycles: $\vec{\mathcal{C}}_1=a,b,c,d,e,a$ and $\vec{\mathcal{C}}_2=a,e,d,c,b,a$}
				\label{EsempioDirected5-cycle}
			\end{figure}
			
			\begin{definition}
				\rm Let $\Gamma_1$ and $\Gamma_2$ be two graphs. A map $\psi:V_{\Gamma_1} \rightarrow V_{\Gamma_2}$ is a \emph{graph homomorphism} if it maps adjacent vertices to adjacent vertices. Formally $\{x,y\}\in E_{\Gamma_1}$ implies $\{\psi(x), \psi(y)\}\in E_{\Gamma_2}$ for all $x,y\in V_{\Gamma_1}$.
				If $\psi$ is a graph homomorphism between $\Gamma_1$ and $\Gamma_2$, then we use the notation $\psi: \Gamma_1 \rightarrow \Gamma_2$
			\end{definition}
		
			\begin{definition}
				\rm Two graphs $\Gamma_1$ and $\Gamma_2$ are said to be \emph{isomorphic}, in symbols $\Gamma_1 \cong \Gamma_2$, if there exists an homomorphism $\psi$ that is a bijection between $V_{\Gamma_1}$ and $V_{\Gamma_2}$ with the property that $\left\lbrace x,y \right\rbrace \in E_{\Gamma_1} $ if and only if $\left\lbrace \psi(x), \psi(y) \right\rbrace \in E_{\Gamma_2} $. Such a $\psi$ is called \emph{graph isomorphism}. If $\Gamma_1=\Gamma_2$, then $\psi$ is also called \emph{graph automorphism}.
			\end{definition}
		
			Analogous definition can be given for digraphs.
		
			\begin{definition}
				\rm Let $\vec{\Gamma}_1$ and $\vec{\Gamma}_2$ be two graphs. A map $\psi:V_{\vec{\Gamma}_1} \rightarrow V_{\vec{\Gamma}_2}$ is a \emph{digraph homomorphism} if, for all $x,y\in V_{\vec{\Gamma}_1}$, $(x,y)\in A_{\vec{\Gamma}_1}$ implies $(\psi(x), \psi(y))\in A_{\vec{\Gamma}_2}$.
				If $\psi$ is a digraph homomorphism between $\vec{\Gamma}_1$ and $\vec{\Gamma}_2$, then we use the notation $\psi: \vec{\Gamma}_1 \rightarrow \vec{\Gamma}_2$
			\end{definition}
		
			\begin{definition}
				\rm Two digraphs $\vec{\Gamma}_1$ and $\vec{\Gamma}_2$ are said to be \emph{isomorphic}, in symbols $\vec{\Gamma}_1 \cong \vec{\Gamma}_2$, if there exists a bijection $\psi$ between $V_{\vec{\Gamma}_1}$ and $V_{\vec{\Gamma}_2}$ with the property that $(x,y)\in A_{\vec{\Gamma}_1} $ if and only if $(\psi(x), \psi(y)) \in A_{\vec{\Gamma}_2} $. Such a $\psi$ is called \emph{digraph isomorphism}. If $\vec{\Gamma}_1=\vec{\Gamma}_2$, then $\psi$ is also called \emph{digraph automorphism}.
			\end{definition}
			
			\begin{definition}
				\rm Given a vertex $x$ of a graph $\Gamma$, the \emph{closed neighbourhood} of $x$ in $\Gamma$ is defined as the set  $N_{\Gamma}[x]=\left\lbrace y \in V_{\Gamma} \, | \, y=x \; \text{or}\; \left\lbrace y, x \right\rbrace \in E_{\Gamma}\right\rbrace $.\\
				The \emph{open neighbourhood} of $x$ in $\Gamma$ is defined as the set $N_{\Gamma}(x)=N_{\Gamma}[x]\setminus \{x\}$.
				
			\end{definition}
			
			
			$N_{\Gamma}[x]$ is never an empty set as it always contains the vertex $x$, indeed $$1\leq |N_{\Gamma}[x]|\leq |V_{\Gamma}|.$$
			If $|N_{\Gamma}[x]|=|V_{\Gamma}|$, and hence $N_{\Gamma}[x]=V_{\Gamma}$, then $x$ is called a \emph{star vertex}. We denote the set of all star vertices of $\Gamma$ by $\mathcal{S}_{\Gamma}$.
			In Figure \ref{Star} the vertex in the centre is the only star vertex.		
			Note that $\Gamma$ is a complete graph if and only if $\mathcal{S}_{\Gamma}=V_{\Gamma}$. Look at Figure \ref{Complete_graph_digraph} to see the case of $K_5$.
			
			\begin{Oss}
				\label{CardinalityS}Let $\Gamma_1$ and $\Gamma_2$ be two graphs. If $\Gamma_1\cong \Gamma_2$, Then $|\mathcal{S}_{\Gamma_1}|=|\mathcal{S}_{\Gamma_2}|$.
			\end{Oss}
			\begin{proof}
				Let $\psi$ be an isomorphism between the two graphs $\Gamma_1$ and $\Gamma_2$.
				Now let $x,y \in V_{\Gamma_1}$, with $y\not= x$.
				Since $\psi$ is a graph isomorphism, then  we have that $\left\lbrace x,y\right\rbrace \in E_{\Gamma_1}$ if and only if $\left\lbrace \psi(x), \psi(y) \right\rbrace \in E_{\Gamma_2} $.
				Hence $x \in \mathcal{S}_{\Gamma_1}$ if and only if $\psi(x)\in \mathcal{S}_{\Gamma_2}$. Thus $|\mathcal{S}_{\Gamma_1}|=|\mathcal{S}_{\Gamma_2}|$, since $\psi$ is a bijection between $V_{\Gamma_1}$ and $V_{\Gamma_2}$.
			\end{proof}
			
			In general $|N_{\Gamma}[x]|-1$ is called the \emph{degree} of $x$, also denoted by $\deg_{\Gamma}(x)$.
			We use $\deg(x)$ when $\Gamma$ is clear from the context.
		
			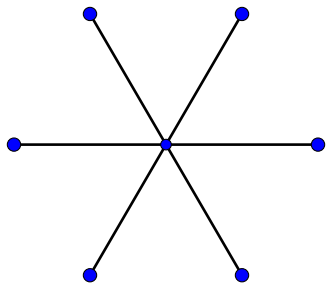
\begin{figure}
				\centering
				\begin{tikzpicture}[line cap=round,line join=round,>=triangle 45,x=1.0cm,y=1.0cm]
					\clip(-1.,-2.) rectangle (5.,3.);
					\draw [line width=1.pt] (1.,2.4641016151377557)-- (3.,-1.);
					\draw [line width=1.pt] (3.,2.4641016151377553)-- (1.,-1.);
					\draw [line width=1.pt] (0.,0.732050807568879)-- (4.,0.7320508075688774);
					\begin{scriptsize}
						\draw [fill=qqqqff] (1.,-1.) circle (2.5pt);
						\draw [fill=qqqqff] (3.,-1.) circle (2.5pt);
						\draw [fill=qqqqff] (4.,0.7320508075688774) circle (2.5pt);
						\draw [fill=qqqqff] (3.,2.4641016151377553) circle (2.5pt);
						\draw [fill=qqqqff] (1.,2.4641016151377557) circle (2.5pt);
						\draw [fill=qqqqff] (0.,0.732050807568879) circle (2.5pt);
						\draw [fill=qqqqff] (2.,0.732050807568878) circle (2.0pt);
					\end{scriptsize}
				\end{tikzpicture}
				\caption{A graph with $|\mathcal{S}_{\Gamma}|=1$}
				\label{Star}
			\end{figure}
			
		\section{The Euler's function}	
			
			Before defining the objects that will be the main players of this thesis, we need to recall the Euler's function $\phi$. 		
			For $n \in \mathbb{N}$, we recall that $$\phi(n):=|\left\lbrace k\in \mathbb{N} \, | \, 1\leq k <n \; \text{with} \; \gcd(k,n)=1\right\rbrace |.$$ 
			
			It is well known that, for $k\in \mathbb{N}$, for $p_i$ a prime number and $a_i$ a positive integer for all $i\in [k]$, with $p_i \not= p_j$ for all $i\not=j$, if  $n=p_1^{a_1}p_2^{a_2}...p_k^{a_k}$, then  $$ \phi(n)= \prod_{i=1}^{k}p_i^{a_i-1}(p_i-1).$$
			
			\begin{Oss}
				$\phi(n)=1$ if and only if $n=1$ or $n=2$. Furthermore $\phi(n)$ is odd if and only if $\phi(n)=1$.
			\end{Oss}
		\begin{proof}
			Clearly $\phi(n)=1$ if $n=1$ or $n=2$.
			Assume that $\phi(n)=1$. Suppose, by contradiction, that $n\ne1$ and $n\ne 2$.\\
			If $n=2^k$ for an integer $k\geq 2$, then $2^{k-1}\mid \phi(n)$, thus $\phi(n)\ne 1$.\\
			If there exists a prime $p\ne 2$ such that $p\mid n$, then $(p-1)\mid \phi(n)$, thus $\phi(n)\ne 1$.\\ 
			This two cases cover all the possible $n\ne 1,2$.
			Note that in both cases $\phi(n)$ is divided by an even number, so it is even.
			Therefore we get both conclusions. 
		\end{proof}
			
			For $n,m \in \mathbb{N}$, it is not true, in general, that if $n\geq m$, then $\phi(n)\geq \phi(m)$. For example pick $n=6$ and $m=5$, then $\phi(6)=2<\phi(5)=4$. The next lemma show us special cases for which $n\geq m$ implies $\phi(n)\geq \phi(m)$.
			
			\begin{lemma}
				\label{phiEulero}Let $m$ and $n$ be two positive integers. If $m \mid n$, then $\phi(m) \mid \phi(n)$ with equality if and only if $m=n$ or $m$ is odd and $n=2m$.
			\end{lemma}
			\begin{proof}
				Since $m\mid n$, we have that there exist distinct primes $p_1, p_2, ..., p_s$ such that $m=p_1^{a_1}p_2^{a_2}...p_k^{a_k}$ and $n=p_1^{b_1}p_2^{b_2}...p_s^{b_s}$ with $k\leq s$, $1\leq a_i\leq b_i$ for all $i \in [k]$ and $a_i,b_j \in \mathbb{N}$ for all $i \in [k]$ and $j \in [s]$.
				
				So $\phi(m)=\prod_{i=1}^k p_i^{a_i-1}(p_i-1)$ and $\phi(n)=\prod_{j=1}^s p_j^{b_j-1}(p_j-1)$. Then
				$$\phi(n)=\phi(m)\prod_{i=1}^k p_i^{b_i-a_i} \prod_{j=k+1}^sp_j^{b_j-1}(p_j-1).$$
				Hence $\phi(m) \mid \phi(n)$ with equality if and only if $b_i=a_i$ for all $i\in[k]$ and $$\prod_{j=k+1}^sp_j^{b_j-1}(p_j-1)=1.$$
				Now $\prod_{j=k+1}^sp_j^{b_j-1}(p_j-1)=1$ happens in two cases:
				
				\begin{enumerate}
					\item \underline{The product is over an empty set}\footnote{As usual a product of number over an empty set of integer is set to be $1$.}.
					This holds only when $n$ and $m$ are products of the same primes. So $n=m$, as $a_i=b_i$ for all $i \in [k]$.
					
					\item \underline{The factors in the product are all equal to 1}, that implies $p_j^{b_j-1}(p_j-1)=1$ for all $k+1\leq j\leq s$. This happens if and only if $s=k+1$, $b_s=1$ and $p_s=2$. As $n$ is product of distinct primes it follows that $m$ is odd. So here we obtain that $n=2m$.
				\end{enumerate}
			\end{proof}
		
		\chapter{Power graph and directed power graph}
			\label{UPGeDPG}
			
			In this chapter we start to explore the connection between Group theory and Graph theory. Definitions, notation and also the results presented here, will be crucial for the following chapters.
			
			\begin{definition}
			\label{defPG}	\rm Let $G$ be a group. The \emph{power graph} of $G$, denoted as $\mathcal{P}(G)$, is defined by $V_{\mathcal{P}(G)}=G$ and $\left\lbrace x,y \right\rbrace \in E_{\mathcal{P}(G)} $ if $x\not=y$ and there exists a positive integer $m$ such that $x=y^m$ or $y=x^m$.
			\end{definition}
			
			By Lagrange's Theorem if $\left\lbrace x,y\right\rbrace $ is an edge of a power graph, then $o(x)\mid o(y)$ or $o(y)\mid o(x)$, in particular $\langle x \rangle \leq \langle y \rangle $ or $\langle y \rangle \leq \langle x \rangle $.
			Moreover, since in a group $G$ every element has a finite order, then, in $\mathcal{P}(G)$, the identity element of $G$ is a star vertex. Hence we have $1\in S_{\mathcal{P}(G)}$. Thus $\mathcal{P}(G)$ is connected.
			
			In Figure \ref{UPG_S_4} there is a representation of $\mathcal{P}(S_4)$. We can recognize the only star vertex as the identity.

			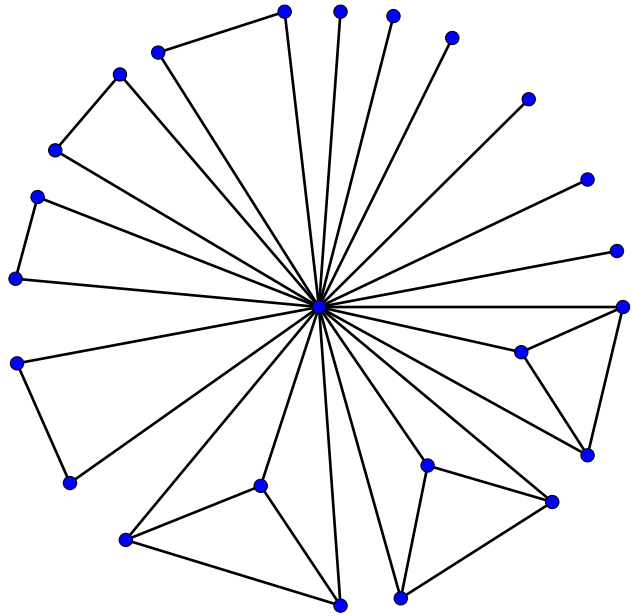
\begin{figure}
				\centering
				\begin{tikzpicture}[line cap=round,line join=round,>=triangle 45,x=1.0cm,y=1.0cm]
					\clip(-6.,2.5) rectangle (4.,12.);
					\draw [line width=1.pt] (-1.4520218260437137,10.915549859085147)-- (-3.1158884964800677,10.373825826850057);
					\draw [line width=1.pt] (-3.6189179549840818,10.083616523866972)-- (-4.470198577067798,9.077557606858948);
					\draw [line width=1.pt] (-4.702366019454266,8.458444427161703)-- (-4.992575322437351,7.374996362691523);
					\draw [line width=1.pt] (-4.276725708412408,4.666376201516074)-- (-4.973228035571812,6.252853724490266);
					\draw [line width=1.pt] (1.661554594802546,6.400870768344779)-- (3.,7.);
					\draw [line width=1.pt] (1.661554594802546,6.400870768344779)-- (2.533519268257322,5.033974651961314);
					\draw [line width=1.pt] (2.533519268257322,5.033974651961314)-- (3.,7.);
					\draw [line width=1.pt] (0.07641383633386878,3.1379405391384996)-- (2.069184383484386,4.414861472264069);
					\draw [line width=1.pt] (0.07641383633386878,3.1379405391384996)-- (0.428937407628609,4.898618571476549);
					\draw [line width=1.pt] (2.069184383484386,4.414861472264069)-- (0.428937407628609,4.898618571476549);
					\draw [line width=1.pt] (-3.541528807521925,3.9118320137600566)-- (-0.7168249251532304,3.041204104810805);
					\draw [line width=1.pt] (-3.541528807521925,3.9118320137600566)-- (-1.7666619570249664,4.628983561782251);
					\draw [line width=1.pt] (-1.7666619570249664,4.628983561782251)-- (-0.7168249251532304,3.041204104810805);
					\draw [line width=1.pt] (-1.4520218260437137,10.915549859085147)-- (-1.,7.);
					\draw [line width=1.pt] (-3.1158884964800677,10.373825826850057)-- (-1.,7.);
					\draw [line width=1.pt] (-3.6189179549840818,10.083616523866972)-- (-1.,7.);
					\draw [line width=1.pt] (-4.470198577067798,9.077557606858948)-- (-1.,7.);
					\draw [line width=1.pt] (-4.702366019454266,8.458444427161703)-- (-1.,7.);
					\draw [line width=1.pt] (-4.992575322437351,7.374996362691523)-- (-1.,7.);
					\draw [line width=1.pt] (-4.276725708412408,4.666376201516074)-- (-1.,7.);
					\draw [line width=1.pt] (-4.973228035571812,6.252853724490266)-- (-1.,7.);
					\draw [line width=1.pt] (-0.7168249251532315,10.915549859085147)-- (-1.,7.);
					\draw [line width=1.pt] (-0.02032259799382734,10.85750799848853)-- (-1.,7.);
					\draw [line width=1.pt] (0.7535688766277328,10.567298695505446)-- (-1.,7.);
					\draw [line width=1.pt] (1.759627793635761,9.754712647152811)-- (-1.,7.);
					\draw [line width=1.pt] (2.533519268257322,8.69061186954817)-- (-1.,7.);
					\draw [line width=1.pt] (2.920465005568102,7.742594813136765)-- (-1.,7.);
					\draw [line width=1.pt] (1.661554594802546,6.400870768344779)-- (-1.,7.);
					\draw [line width=1.pt] (2.533519268257322,5.033974651961314)-- (-1.,7.);
					\draw [line width=1.pt] (3.,7.)-- (-1.,7.);
					\draw [line width=1.pt] (0.07641383633386878,3.1379405391384996)-- (-1.,7.);
					\draw [line width=1.pt] (2.069184383484386,4.414861472264069)-- (-1.,7.);
					\draw [line width=1.pt] (0.428937407628609,4.898618571476549)-- (-1.,7.);
					\draw [line width=1.pt] (-0.7168249251532304,3.041204104810805)-- (-1.,7.);
					\draw [line width=1.pt] (-3.541528807521925,3.9118320137600566)-- (-1.,7.);
					\draw [line width=1.pt] (-1.7666619570249664,4.628983561782251)-- (-1.,7.);
					\begin{scriptsize}
						\draw [fill=qqqqff] (-1.4520218260437137,10.915549859085147) circle (2.5pt);
						\draw [fill=qqqqff] (-3.1158884964800677,10.373825826850057) circle (2.5pt);
						\draw [fill=qqqqff] (-4.470198577067798,9.077557606858948) circle (2.5pt);
						\draw [fill=qqqqff] (-3.6189179549840818,10.083616523866972) circle (2.5pt);
						\draw [fill=qqqqff] (-4.702366019454266,8.458444427161703) circle (2.5pt);
						\draw [fill=qqqqff] (-4.992575322437351,7.374996362691523) circle (2.5pt);
						\draw [fill=qqqqff] (-4.973228035571812,6.252853724490266) circle (2.5pt);
						\draw [fill=qqqqff] (-4.276725708412408,4.666376201516074) circle (2.5pt);
						\draw [fill=qqqqff] (2.533519268257322,5.033974651961314) circle (2.5pt);
						\draw [fill=qqqqff] (1.661554594802546,6.400870768344779) circle (2.5pt);
						\draw [fill=qqqqff] (2.069184383484386,4.414861472264069) circle (2.5pt);
						\draw [fill=qqqqff] (3.,7.) circle (2.5pt);
						\draw [fill=qqqqff] (-0.7168249251532315,10.915549859085147) circle (2.5pt);
						\draw [fill=qqqqff] (-0.7168249251532304,3.041204104810805) circle (2.5pt);
						\draw [fill=qqqqff] (-3.541528807521925,3.9118320137600566) circle (2.5pt);
						\draw [fill=qqqqff] (-0.02032259799382734,10.85750799848853) circle (2.5pt);
						\draw [fill=qqqqff] (2.920465005568102,7.742594813136765) circle (2.5pt);
						\draw [fill=qqqqff] (0.07641383633386878,3.1379405391384996) circle (2.5pt);
						\draw [fill=qqqqff] (-1.7666619570249664,4.628983561782251) circle (2.5pt);
						\draw [fill=qqqqff] (0.428937407628609,4.898618571476549) circle (2.5pt);
						\draw [fill=qqqqff] (1.759627793635761,9.754712647152811) circle (2.5pt);
						\draw [fill=qqqqff] (2.533519268257322,8.69061186954817) circle (2.5pt);
						\draw [fill=qqqqff] (-1.,7.) circle (2.5pt);
						\draw [fill=qqqqff] (0.7535688766277328,10.567298695505446) circle (2.5pt);
					\end{scriptsize}
				\end{tikzpicture}
				\caption{Power graph of $S_4$}
				\label{UPG_S_4}
			\end{figure}
			
			Let observe that, for all $x\in V_{\mathcal{P}(G)}$, we have that $\deg(x)\geq o(x)-1$. This holds because, given $x\in V_{\mathcal{P}(G)}$, $\left\lbrace x,y\right\rbrace \in E_{\mathcal{P}(G)} $ for all $y \in \langle x \rangle \setminus \{x\}$ and we have $|\langle x \rangle|=o(x)$.
			
			In some cases, it is convenient to consider also the subgraph $\mathcal{P}^*(G)$ of $\mathcal{P}(G)$ induced by $G \setminus \left\lbrace 1\right\rbrace $. Such a graph is called the \emph{proper power graph} of $G$. In Figure \ref{PPG_S_4} we see how $\mathcal{P}^*(S_4)$ appears. Look to Figure \ref{UPG_S_4}, that represents $\mathcal{P}(S_4)$, to appreciate the differences. Note that, for all $x\in V_{\mathcal{P}(G)}\setminus \left\lbrace 1\right\rbrace $, $\deg_{\mathcal{P}(G)}(x)=\deg_{\mathcal{P}^*(G)}(x)+1$. Now, in $\mathcal{P}^*(G)$, all the vertices of degree $0$ are the transpositions, the vertices of degree $1$ are the $3$-cycles. The remaining vertices are the $4$-cycles and the product of two transpositions. In each triangle we have two $4$-cycle and a product of two transpositions. For example the vertices of one of the triangles in $\mathcal{P}^*(G)$ are $(1234)$, $(4321)$ and $(13)(24)$.

			\begin{Oss}
			\label{InducedPG}	Let $G$ be a group and $H\leq G$. Then $\mathcal{P}(G)_H=\mathcal{P}(H)$.
			\end{Oss}
			\begin{proof}
				By definition of power graph and induced subgraph we immediately have $V_{\mathcal{P}(H)}=H=V_{\mathcal{P}(G)_H}$. Now $e=\left\lbrace x,y \right\rbrace \in E_{\mathcal{P}(H)} $ if and only if $e\in E_{\mathcal{P}(G)_H}$ since, by definitions, both are equivalent to $x\not=y$, $x,y \in H$ and there exists a positive integer $m$ such that $x=y^m$ or $y=x^m$. Thus $\mathcal{P}(H)=\mathcal{P}(G)_H$.
			\end{proof}
			
			\begin{figure}
				\centering
				\begin{tikzpicture}[line cap=round,line join=round,>=triangle 45,x=1.0cm,y=1.0cm]
					\clip(-6.,2.5) rectangle (4.,12.);
					\draw [line width=1.pt] (-1.4520218260437137,10.915549859085147)-- (-3.1158884964800677,10.373825826850057);
					\draw [line width=1.pt] (-3.6189179549840818,10.083616523866972)-- (-4.470198577067798,9.077557606858948);
					\draw [line width=1.pt] (-4.702366019454266,8.458444427161703)-- (-4.992575322437351,7.374996362691523);
					\draw [line width=1.pt] (-4.276725708412408,4.666376201516074)-- (-4.973228035571812,6.252853724490266);
					\draw [line width=1.pt] (1.661554594802546,6.400870768344779)-- (3.,7.);
					\draw [line width=1.pt] (1.661554594802546,6.400870768344779)-- (2.533519268257322,5.033974651961314);
					\draw [line width=1.pt] (2.533519268257322,5.033974651961314)-- (3.,7.);
					\draw [line width=1.pt] (0.07641383633386878,3.1379405391384996)-- (2.069184383484386,4.414861472264069);
					\draw [line width=1.pt] (0.07641383633386878,3.1379405391384996)-- (0.428937407628609,4.898618571476549);
					\draw [line width=1.pt] (2.069184383484386,4.414861472264069)-- (0.428937407628609,4.898618571476549);
					\draw [line width=1.pt] (-3.541528807521925,3.9118320137600566)-- (-0.7168249251532304,3.041204104810805);
					\draw [line width=1.pt] (-3.541528807521925,3.9118320137600566)-- (-1.7666619570249664,4.628983561782251);
					\draw [line width=1.pt] (-1.7666619570249664,4.628983561782251)-- (-0.7168249251532304,3.041204104810805);
					\begin{scriptsize}
						\draw [fill=qqqqff] (-1.4520218260437137,10.915549859085147) circle (2.5pt);
						\draw [fill=qqqqff] (-3.1158884964800677,10.373825826850057) circle (2.5pt);
						\draw [fill=qqqqff] (-4.470198577067798,9.077557606858948) circle (2.5pt);
						\draw [fill=qqqqff] (-3.6189179549840818,10.083616523866972) circle (2.5pt);
						\draw [fill=qqqqff] (-4.702366019454266,8.458444427161703) circle (2.5pt);
						\draw [fill=qqqqff] (-4.992575322437351,7.374996362691523) circle (2.5pt);
						\draw [fill=qqqqff] (-4.973228035571812,6.252853724490266) circle (2.5pt);
						\draw [fill=qqqqff] (-4.276725708412408,4.666376201516074) circle (2.5pt);
						\draw [fill=qqqqff] (2.533519268257322,5.033974651961314) circle (2.5pt);
						\draw [fill=qqqqff] (1.661554594802546,6.400870768344779) circle (2.5pt);
						\draw [fill=qqqqff] (2.069184383484386,4.414861472264069) circle (2.5pt);
						\draw [fill=qqqqff] (3.,7.) circle (2.5pt);
						\draw [fill=qqqqff] (-0.7168249251532315,10.915549859085147) circle (2.5pt);
						\draw [fill=qqqqff] (-0.7168249251532304,3.041204104810805) circle (2.5pt);
						\draw [fill=qqqqff] (-3.541528807521925,3.9118320137600566) circle (2.5pt);
						\draw [fill=qqqqff] (-0.02032259799382734,10.85750799848853) circle (2.5pt);
						\draw [fill=qqqqff] (2.920465005568102,7.742594813136765) circle (2.5pt);
						\draw [fill=qqqqff] (0.07641383633386878,3.1379405391384996) circle (2.5pt);
						\draw [fill=qqqqff] (-1.7666619570249664,4.628983561782251) circle (2.5pt);
						\draw [fill=qqqqff] (0.428937407628609,4.898618571476549) circle (2.5pt);
						\draw [fill=qqqqff] (1.759627793635761,9.754712647152811) circle (2.5pt);
						\draw [fill=qqqqff] (2.533519268257322,8.69061186954817) circle (2.5pt);
						\draw [fill=qqqqff] (0.7535688766277328,10.567298695505446) circle (2.5pt);
					\end{scriptsize}
				\end{tikzpicture}
				\caption{Proper power graph of $S_4$}
				\label{PPG_S_4}
			\end{figure}
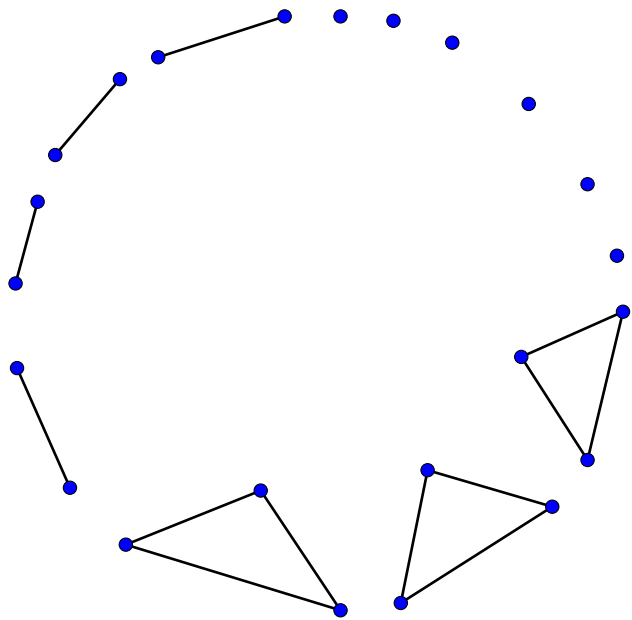
			
			Whenever is clear which group we are referring to, we write $E$ instead of $E_{\mathcal{P}(G)}$.  
			
			We can easily find some simple results related to the power graph of cyclic groups.
			
			\begin{lemma}
				\label{nonGeneratoriInPG}Let $G$ be a cyclic group. If $G$ is not a $p$-group, then for all $x\in G\setminus \left\lbrace 1\right\rbrace $ with $\langle x\rangle\not= G$ there exists $y\in G $ such that $x$ and $y$ are not joined in $\mathcal{P}(G)$.
			\end{lemma}
			\begin{proof}
				Let $G=\langle g \rangle$  and let $o(g)=p_1^{a_1}p_2^{a_2}...p_k^{a_k}$ for an integer $k\geq 2$, $p_1, p_2, ...,p_k$ distinct primes and $a_i\in \mathbb{N}$ for all $i \in [k]$.
				Now let $x\in G$ that satisfies the hypothesis. Then, since $o(x) \mid o(g)$, $o(x)=p_1^{b_1}p_2^{b_2}...p_k^{b_k}$ with, for all $i \in [k]$, $b_i$ an integer such that $0\leq b_i\leq a_i$. Thus, as $\langle x\rangle\not= G$, there is an index $i$ such that $b_i<a_i$. Take $y\in G$ with $o(y)=p_i^{a_i}$ then $o(y) \nmid o(x)$ and $o(x)\nmid o(y)$; therefore $y$ is not joined to $x$ in $\mathcal{P}(G)$.
			\end{proof}
						
			\begin{corollary}
				\label{PGcompleto_1}Let $G$ be a cyclic group. Then $\mathcal{P}(G)$ is a complete graph if and only if $G$ is cyclic of prime power order.
			\end{corollary}
			\begin{proof}
				Assume that $G\cong C_{p^n}$ for some prime $p$ and $n\in \mathbb{N}_0$.
				We proceed by induction on the order of $G$.
				For $|G|=1$ and for $|G|=p$ we have that $\mathcal{P}(G)$ is a complete graph since $G$ contains only the identity and the generators of the group, that are clearly joined to all others elements.
				Now let $|G|=p^n$ for some $n \in \mathbb{N}$. Surely there exists $y\in G$ with order $p^{n-1}$. Then, by inductive hypothesis, $\mathcal{P}(\langle y \rangle)$ is complete.
				Therefore consider two elements $x, z \in G$. If one of them is a generator, then they are joined. Otherwise, since the elements of $G\setminus \langle y \rangle$ are exactly the generators of $G$, we have that $x, z \in \langle y \rangle$, and hence they are joined. Thus $\mathcal{P}(G)$ is a complete graph.
				
				Assume next that $\mathcal{P}(G)$ is complete. By Lemma \ref{nonGeneratoriInPG}, since $G$ is cyclic, it must be of prime power order otherwise there exist at least two vertices that are not joined.
			\end{proof}
					
				In particular Lemma \ref{nonGeneratoriInPG} tells us something about the closed neighbourhoods: \\ Let $G$ be a group. If $x, y \in G\setminus \{1\}$ and $\langle x \rangle < \langle y \rangle$ with $o(y)$ not a prime power, then surely $N_{\mathcal{P}(G)}[x]\not= N_{\mathcal{P}(G)}[y]$. Indeed, by Lemma \ref{nonGeneratoriInPG}, in $\langle y \rangle $ there exists an element not joined to $x$, and such an element is clearly joined to $y$.
				
				With the next proposition we study the star vertices of a power graph.
					
			\begin{prop}{\rm \cite[Proposition 4]{Cameron_2}}
				\label{S>1}Let $G$ be a group. Consider $\mathcal{S}=\mathcal{S}_{\mathcal{P}(G)}$. Suppose that $|\mathcal{S}|>1$. Then one of the following occurs:
				\begin{itemize}
					\item[\rm (a)] $G$ is cyclic of prime power order, and $\mathcal{S}=G$;
					\item[\rm (b)] $G$ is cyclic of non-prime-power order $n$, and $\mathcal{S}$ consists of the identity and the generators of $G$, so that $|\mathcal{S}|=1 + \phi(n)$;
					\item[\rm (c)] $G$ is $2$-group generalised quaternion, and $\mathcal{S}$ contains the identity and the unique involution of $G$, so that $|\mathcal{S}|=2$.
				\end{itemize}
			\end{prop}
			\begin{proof}
				Clearly $\mathcal{S}$ contains the identity.
				\begin{claim}
					\label{p|g}
					\rm	For all $g\in \mathcal{S}\setminus \{1\}$, the order of $g$ is divisible by every prime divisor of $|G|$.
				\end{claim}
				Suppose, by contradiction that there exists a prime $q$ such that $q\mid |G|$ and $q\nmid o(g)$. In particular $o(g)\ne q$. Let $h\in G$ be an element with $o(h)=q$. Then, since $g \in \mathcal{S}$, we have that $\{g,h\}\in E$. Hence $o(g)\mid o(h)$ or $o(h)\mid o(g)$. Now $o(h)=q\nmid o(g)$, then we have that $o(g)\mid o(h)=q$. Therefore $o(g)=1$, that implies $g=1$, against the assumption on $g\in \mathcal{S}\setminus \{1\}$.  
				
				Consider $g\in \mathcal{S}\setminus\{1\}$.
				\begin{claim}
					\rm	Every $h\in G$ with order a prime is a power of $g$. Then $G$ has a unique subgroup of order $p$ for all primes $p \mid |G|$.
				\end{claim}
				Let $p$ be a prime such that $p\mid |G|$. Suppose, by contradiction, that there exists $h$, an element of order $p$, which is not a power of $g$. Since $g\in \mathcal{S}$ then $h$ and $g$ are joined in the power graph. So, for the assumption on $h$, $g$ must be power of $h$; but it is a contradiction: if $o(g)>p$ then $g$ cannot be a power of an element of order $p$, which is $h$, since we would have $o(g)\mid o(h)$, and if $o(g)=p$ then $g$ is power of $h$ if and only if $h$ is power of $g$ against the assumption. Note that $o(g)$ cannot be less than $p$ as $p\mid o(g)$. Hence the claim is true.
				
				Now a $p$-group with a unique subgroup of order $p$ is cyclic or $2$-group generalised quaternion, by \cite[Theorems 8.5, 8.6]{Burnside}. Therefore if $G$ is a $p$-group we obtain conclusions (a) or (c) of the Proposition.
				
				So suppose that the order of $G$ is divisible by at least two distinct primes.
				
				\begin{claim}
					\rm	$G=\langle g \rangle$.
				\end{claim}
				
				Let $z \in G$, since $g \in \mathcal{S}$ holds, we have $\left\lbrace g,z \right\rbrace \in E$ thus we also have $z\in \langle g \rangle$ or $g\in \langle z \rangle$.
				Suppose that $g \in \langle z \rangle$ holds and that $\langle g \rangle \not= \langle z \rangle$; in all the other cases we have $z \in \langle g \rangle$.
				Now $o(g)\mid o(z)$ and if $p$ is a prime such that $p \mid |G|$ then, by Claim \ref{p|g}, $p \mid o(g)$. Hence $o(z)$ is not a prime power and, since $\langle g\rangle \not= 1$, by Lemma \ref{nonGeneratoriInPG}, there exists an element $w\in \langle z \rangle$ not joined to $g$, against the fact that $g \in \mathcal{S}$ holds. Therefore such a $z$ does not exist and then $z\in \langle g \rangle $ thus $G=\langle g \rangle$. 	
				We have obtained the conclusion (b) proving the proposition.
			\end{proof}
			
			We propose two results that follow from this proposition. We omit the proof for both of them because they are straightforward.
						
			\begin{corollary}
				\label{abelUPG_ciclico} If $G$ is an abelian group, then $\mathcal{P}(G)$ has more than one star vertex if and only if $G$ is a non trivial cyclic group. 
			\end{corollary}
		
			\begin{prop}
				\label{PGcompleto}For a group $G$, $\mathcal{P}(G)$ is a complete graph if and only if $G$ is cyclic of prime power order.
			\end{prop}
			
			Note that this last Corollary extends the information of Corollary \ref{PGcompleto_1} to all groups.

			\begin{figure}
				\centering
			\begin{tikzpicture}[line cap=round,line join=round,>=triangle 45,x=1.0cm,y=1.0cm]
				\clip(-1.,-1.5) rectangle (5.,4.5);
				\draw [line width=1.pt] (1.,3.8284271247461903)-- (-0.4142135623730949,2.4142135623730954);
				\draw [line width=1.pt] (1.,3.8284271247461903)-- (-0.41421356237309537,0.4142135623730956);
				\draw [line width=1.pt] (1.,3.8284271247461903)-- (1.,-1.);
				\draw [line width=1.pt] (1.,3.8284271247461903)-- (3.,-1.);
				\draw [line width=1.pt] (1.,3.8284271247461903)-- (4.414213562373095,0.4142135623730949);
				\draw [line width=1.pt] (1.,3.8284271247461903)-- (3.,3.82842712474619);
				\draw [line width=1.pt] (1.,3.8284271247461903)-- (4.414213562373095,2.4142135623730945);
				\draw [line width=1.pt] (3.,3.82842712474619)-- (-0.4142135623730949,2.4142135623730954);
				\draw [line width=1.pt] (3.,3.82842712474619)-- (-0.41421356237309537,0.4142135623730956);
				\draw [line width=1.pt] (3.,3.82842712474619)-- (1.,-1.);
				\draw [line width=1.pt] (3.,3.82842712474619)-- (3.,-1.);
				\draw [line width=1.pt] (3.,3.82842712474619)-- (4.414213562373095,0.4142135623730949);
				\draw [line width=1.pt] (3.,3.82842712474619)-- (4.414213562373095,2.4142135623730945);
				\draw [line width=1.pt] (4.414213562373095,2.4142135623730945)-- (-0.4142135623730949,2.4142135623730954);
				\draw [line width=1.pt] (4.414213562373095,2.4142135623730945)-- (-0.41421356237309537,0.4142135623730956);
				\draw [line width=1.pt] (4.414213562373095,2.4142135623730945)-- (1.,-1.);
				\draw [line width=1.pt] (4.414213562373095,2.4142135623730945)-- (3.,-1.);
				\draw [line width=1.pt] (4.414213562373095,2.4142135623730945)-- (4.414213562373095,0.4142135623730949);
				\draw [line width=1.pt] (4.414213562373095,0.4142135623730949)-- (-0.4142135623730949,2.4142135623730954);
				\draw [line width=1.pt] (4.414213562373095,0.4142135623730949)-- (-0.41421356237309537,0.4142135623730956);
				\draw [line width=1.pt] (4.414213562373095,0.4142135623730949)-- (1.,-1.);
				\draw [line width=1.pt] (4.414213562373095,0.4142135623730949)-- (3.,-1.);
				\draw [line width=1.pt] (3.,-1.)-- (-0.4142135623730949,2.4142135623730954);
				\draw [line width=1.pt] (3.,-1.)-- (-0.41421356237309537,0.4142135623730956);
				\draw [line width=1.pt] (3.,-1.)-- (1.,-1.);
				\draw [line width=1.pt] (1.,-1.)-- (-0.41421356237309537,0.4142135623730956);
				\draw [line width=1.pt] (-0.4142135623730949,2.4142135623730954)-- (1.,-1.);
				\draw [line width=1.pt] (-0.4142135623730949,2.4142135623730954)-- (-0.41421356237309537,0.4142135623730956);
				\draw (1,-1) node[anchor=north east] {$1$};
				\draw (3,-1) node[anchor=north west] {$a$};
				\draw (4.414213562373095,0.4142135623730949) node[anchor=north west] {$a^2$};
				\draw (4.414213562373095,2.4142135623730945) node[anchor=west] {$a^3$};
				\draw ((3.,3.82842712474619) node[anchor=south west] {$a^4$};
				\draw (1.,3.8284271247461903) node[anchor=south east] {$a^5$};
				\draw (-0.4142135623730949,2.4142135623730954) node[anchor= east] {$a^6$};
				\draw (-0.41421356237309537,0.4142135623730956) node[anchor=north east] {$a^7$};
				\begin{scriptsize}
					\draw [fill=qqqqff] (1.,-1.) circle (2.5pt);
					\draw [fill=qqqqff] (3.,-1.) circle (2.5pt);
					\draw [fill=qqqqff] (4.414213562373095,0.4142135623730949) circle (2.5pt);
					\draw [fill=qqqqff] (4.414213562373095,2.4142135623730945) circle (2.5pt);
					\draw [fill=qqqqff] (3.,3.82842712474619) circle (2.5pt);
					\draw [fill=qqqqff] (1.,3.8284271247461903) circle (2.5pt);
					\draw [fill=qqqqff] (-0.4142135623730949,2.4142135623730954) circle (2.5pt);
					\draw [fill=qqqqff] (-0.41421356237309537,0.4142135623730956) circle (2.5pt);
				\end{scriptsize}
			\end{tikzpicture}
				\caption{$\mathcal{P}(C_8)$}
				\label{UPG_C_8}
			\end{figure}
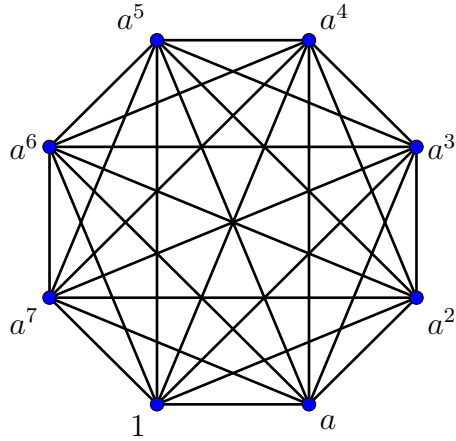
			
			\begin{figure}
				\centering
				\begin{tikzpicture}[line cap=round,line join=round,>=triangle 45,x=1.0cm,y=1.0cm]
					\clip(-1.,-1.5) rectangle (5.,3.);
					\draw [line width=1.pt] (1.,-1.)-- (0.,0.732050807568879);
					\draw [line width=1.pt] (1.,-1.)-- (1.,2.4641016151377557);
					\draw [line width=1.pt] (1.,-1.)-- (3.,2.4641016151377553);
					\draw [line width=1.pt] (1.,-1.)-- (4.,0.7320508075688774);
					\draw [line width=1.pt] (1.,-1.)-- (3.,-1.);
					\draw [line width=1.pt] (0.,0.732050807568879)-- (1.,2.4641016151377557);
					\draw [line width=1.pt] (0.,0.732050807568879)-- (3.,-1.);
					\draw [line width=1.pt] (0.,0.732050807568879)-- (4.,0.7320508075688774);
					\draw [line width=1.pt] (0.,0.732050807568879)-- (3.,2.4641016151377553);
					\draw [line width=1.pt] (3.,-1.)-- (1.,2.4641016151377557);
					\draw [line width=1.pt] (3.,-1.)-- (3.,2.4641016151377553);
					\draw [line width=1.pt] (3.,-1.)-- (4.,0.7320508075688774);
					\draw [line width=1.pt] (1.,2.4641016151377557)-- (4.,0.7320508075688774);
					\draw (0.9,-0.9086580230955188) node[anchor=north east] {$1$};
					\draw (3.0041230403309376,-0.8146375163736965) node[anchor=north west] {$a$};
					\draw (4.16370928990008,1.1754632092382078) node[anchor=north west] {$a^2$};
					\draw (3.145153800413671,3.0245331747673787) node[anchor=north west] {$a^3$};
					\draw (0.9200018079972094,3.118553681489201) node[anchor=north west] {$a^4$};
					\draw (-0.4589656239228511,1.26948371596003) node[anchor=north west] {$a^5$};
					\begin{scriptsize}
						\draw [fill=qqqqff] (1.,-1.) circle (2.5pt);
						\draw [fill=qqqqff] (3.,-1.) circle (2.5pt);
						\draw [fill=qqqqff] (4.,0.7320508075688774) circle (2.5pt);
						\draw [fill=qqqqff] (3.,2.4641016151377553) circle (2.5pt);
						\draw [fill=qqqqff] (1.,2.4641016151377557) circle (2.5pt);
						\draw [fill=qqqqff] (0.,0.732050807568879) circle (2.5pt);
					\end{scriptsize}
				\end{tikzpicture}
				\caption{$\mathcal{P}(C_6)$}
				\label{UPG_C_6}
			\end{figure}
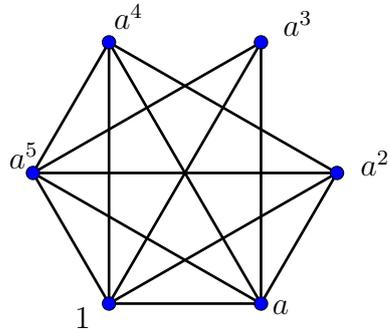
		
			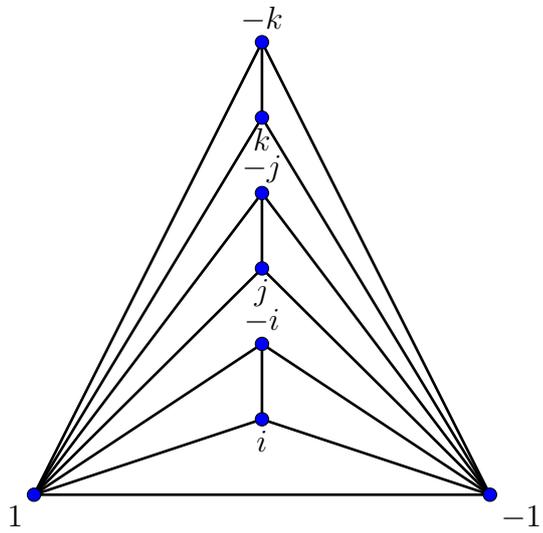
\begin{figure}
				\centering
				\begin{tikzpicture}[line cap=round,line join=round,>=triangle 45,x=1.0cm,y=1.0cm]
					\clip(-6.5,-1.) rectangle (1.,6.5);
					\draw [line width=1.pt] (-3.,2.)-- (0.,0.);
					\draw [line width=1.pt] (-3.,4.)-- (0.,0.);
					\draw [line width=1.pt] (-3.,6.)-- (0.,0.);
					\draw [line width=1.pt] (-3.,5.)-- (0.,0.);
					\draw [line width=1.pt] (-6.,0.)-- (0.,0.);
					\draw [line width=1.pt] (-3.,1.)-- (0.,0.);
					\draw [line width=1.pt] (-3.,3.)-- (0.,0.);
					\draw [line width=1.pt] (-3.,2.)-- (-6.,0.);
					\draw [line width=1.pt] (-3.,4.)-- (-6.,0.);
					\draw [line width=1.pt] (-3.,6.)-- (-6.,0.);
					\draw [line width=1.pt] (-3.,2.)-- (-3.,1.);
					\draw [line width=1.pt] (-3.,4.)-- (-3.,3.);
					\draw [line width=1.pt] (-3.,6.)-- (-3.,5.);
					\draw [line width=1.pt] (-3.,3.)-- (-6.,0.);
					\draw [line width=1.pt] (-3.,1.)-- (-6.,0.);
					\draw [line width=1.pt] (-3.,5.)-- (-6.,0.);
					\draw (-6.,0.) node[anchor=north east] {$1$};
					\draw (0.,0.) node[anchor=north west] {$-1$};
					\draw (-3.,1.) node[anchor=north] {$i$};
					\draw (-3.,2.) node[anchor=south] {$-i$};
					\draw (-3.,3.) node[anchor=north ] {$j$};
					\draw (-3.,4) node[anchor=south] {$-j$};
					\draw (-3.,5) node[anchor=north ] {$k$};
					\draw (-3.,6) node[anchor=south] {$-k$};
					\begin{scriptsize}
						\draw [fill=qqqqff] (-3.,2.) circle (2.5pt);
						\draw [fill=qqqqff] (-3.,4.) circle (2.5pt);
						\draw [fill=qqqqff] (-3.,6.) circle (2.5pt);
						\draw [fill=qqqqff] (-3.,5.) circle (2.5pt);
						\draw [fill=qqqqff] (-3.,3.) circle (2.5pt);
						\draw [fill=qqqqff] (-6.,0.) circle (2.5pt);
						\draw [fill=qqqqff] (-3.,1.) circle (2.5pt);
						\draw [fill=qqqqff] (0.,0.) circle (2.5pt);
					\end{scriptsize}
				\end{tikzpicture}
				\caption{$\mathcal{P}(Q_8)$}
				\label{UPG_Q_8}
			\end{figure}		
					
			From now on we always refer to $\mathcal{S}_{\mathcal{P}(G)}$ as $\mathcal{S}$, if the group is clear from the context.
					
			We propose three examples of power graph in Figures \ref{UPG_C_8}, \ref{UPG_C_6} and \ref{UPG_Q_8}; one for each case of the Proposition \ref{S>1}. Let's analyse briefly these three examples.
			In Figure \ref{UPG_C_8} it is represented the power graph of $C_8$, a cyclic group of prime power order. As we expected from Proposition \ref{PGcompleto}, $\mathcal{P}(C_8)$ is a complete graph. Here, every vertex has the same closed neighbourhood as all the others, in particular $\mathcal{S}=G$.
			
			In the case of Figure $\ref{UPG_C_6}$, where is represented $C_6,$ the hypothesis of Lemma \ref{nonGeneratoriInPG} are satisfied. We have that $a^2, a^3, a^4$ are not star vertices, and, all the generators and the identity are in $\mathcal{S}$. Note that, as stated in Proposition \ref{S>1}, $|\mathcal{S}|=1+\phi(|G|)$.
			 
			Finally, in Figure \ref{UPG_Q_8}, we see an example of a not abelian group, $Q_8$, with a star vertex different to the identity, that is the unique involution of this group. In this case $|\mathcal{S}|=2$.
			
			Let's abandon for a moment the power graph to talk about its directed counterpart.
			
			\begin{definition}
			\label{defDPG}	\rm Let $G$ be a group. The \emph{directed power graph} of $G$, denoted $\vec{\mathcal{P}}(G)$, is defined by $V_{\vec{\mathcal{P}}(G)}=G$ and $(x,y) \in A_{\vec{\mathcal{P}}(G)}$ if $x\not= y$ and there exists a positive integer $m$ such that $y=x^m$.
			\end{definition}
		
			For an example of such a digraph take a look at Figure \ref{DPG_D_4_labeled} where  the group is the dihedral group of $8$ elements: $$D_4=\left\lbrace 1, a, a^2, a^3, b, ab, a^2b, a^3b  \right\rbrace $$ and $$A_{\vec{\mathcal{P}}(D_4)}=\left\lbrace (a,1), (a^2,1), (a^3,1), (b,1), (ab,1), (a^2b,1), (a^3b,1), (a,a^3), (a^3,a), (a,a^2), (a^2,a) \right\rbrace.   $$
			
			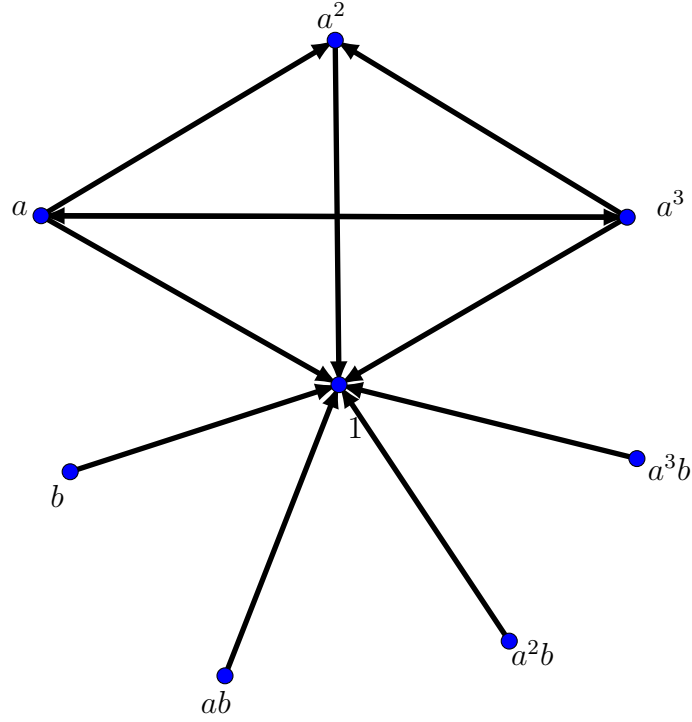
\begin{figure}
				\centering
				\begin{tikzpicture}[line cap=round,line join=round,>=triangle 45,x=1.0cm,y=1.0cm]
					\clip(-6.,-6.) rectangle (4.,4.);
					\draw (-1.0336131263136343,-1.366590941457429) node[anchor=north west] {1};
					\draw (-5.453247351358401,1.4941177569351773) node[anchor=north west] {$a$};
					\draw (-1.4353980558631587,4.1298268947800505) node[anchor=north west] {$a^2$};
					\draw (3.032450360727551,1.6387603315730057) node[anchor=north west] {$a^3$};
					\draw (-4.93896264153501,-2.2665891836483616) node[anchor=north west] {$b$};
					\draw (-2.962180788151351,-4.9987267045851205) node[anchor=north west] {$ab$};
					\draw (1.1199540960718155,-4.323728022941921) node[anchor=north west] {$a^2b$};
					\draw (2.9199505804536843,-1.864804254098838) node[anchor=north west] {$a^3b$};
					\draw [-latex,line width=2.pt] (-4.92,1.16) -- (2.7913794029978334,1.1405470189315967);
					\draw [-latex,line width=2.pt] (2.791379402997834,1.1405470189315967) -- (-4.92,1.16);
					\draw [-latex,line width=2.pt] (-1.049684523495616,3.4869710075008133) -- (-1.,-1.08);
					\draw [-latex,line width=2.pt] (-4.92,1.16) -- (-1.0496845234956158,3.486971007500814);
					\draw [-latex,line width=2.pt] (2.791379402997834,1.1405470189315967) -- (-1.0496845234956163,3.486971007500813);
					\draw [-latex,line width=2.pt] (-4.92,1.16) -- (-1.,-1.08);
					\draw [-latex,line width=2.pt] (2.791379402997834,1.1405470189315967) -- (-1.,-1.08);
					\draw [-latex,line width=2.pt] (-4.537177711985486,-2.2344463892843995) -- (-1.,-1.08);
					\draw [-latex,line width=2.pt] (-2.5,-4.94) -- (-1.,-1.08);
					\draw [-latex,line width=2.pt] (1.24,-4.48) -- (-1.,-1.08);
					\draw [-latex,line width=2.pt] (2.92,-2.06) -- (-1.,-1.08);
					\begin{scriptsize}
						\draw [fill=qqqqff] (-4.92,1.16) circle (3pt);
						\draw [fill=qqqqff] (-1.,-1.08) circle (3pt);
						\draw [fill=qqqqff] (-1.049684523495616,3.4869710075008133) circle (3pt);
						\draw [fill=qqqqff] (2.791379402997834,1.1405470189315967) circle (3pt);
						\draw [fill=qqqqff] (-4.537177711985486,-2.2344463892843995) circle (3pt);
						\draw [fill=qqqqff] (-2.5,-4.94) circle (3pt);
						\draw [fill=qqqqff] (1.24,-4.48) circle (3pt);
						\draw [fill=qqqqff] (2.92,-2.06) circle (3pt);
					\end{scriptsize}
				\end{tikzpicture}
				\caption{The directed power graph of $D_4$}
				\label{DPG_D_4_labeled}
			\end{figure}
			
			From now on we always refer to $A_{\vec{\mathcal{P}}(G)}$ with $A$, if it is clear which group we are referring to.
			\begin{Oss}
				\label{PG_transitive}Let $G$ be a group. Then $\vec{\mathcal{P}}(G)$ is transitive.
			\end{Oss}
			\begin{proof}
				Let $x,y,z \in G$ such that $(x,y), (y,z)\in A$. Then there exist two integers $m$ and $n$ such that $y=x^m$ and $z=y^n$. Hence we have $z=x^{mn}$, it follows that $(x,z)\in A$ holds.
			\end{proof}
					
			\begin{Oss}
				\label{isoGroup-isoGraph}If $\psi$ is a group isomorphism between two groups $G_1$ and $G_2$, then $\psi$ is also a graph isomorphism between $\mathcal{P}(G_1)$ and $\mathcal{P}(G_2)$ and also a digraph isomorphism between $\vec{\mathcal{P}}(G_1)$ and $\vec{\mathcal{P}}(G_2)$.
			\end{Oss}
			\begin{proof}
				By definition of group isomorphism, $\psi$ is a bijection between $G_1$ and $G_2$, so we only need to prove that the condition on the edges/arcs is satisfied.
				For $x\not=y$:
				\begin{itemize}
					\item $\{x,y\}  \in E_{\mathcal{P}(G_1)}\Leftrightarrow x=y^k \text{ or } y=x^k \text{ for some } k \in \mathbb{N} \Leftrightarrow \psi(x)=\psi(y^k)=\psi(y)^k \text{ or } \psi(y)=\psi(x^k)=\psi(x)^k \Leftrightarrow \{ \psi(x), \psi(y) \} \in E_{\mathcal{P}(G_2)} .$
					
					\item $(x,y)\in A_{\vec{\mathcal{P}}(G_1)} \Leftrightarrow \text{ there exists } k\in \mathbb{N} \text{ such that } y=x^k \Leftrightarrow \psi(y)=\psi(x^k)=\psi(x)^k \Leftrightarrow (\psi (x), \psi(y))\in A_{\vec{\mathcal{P}}(G_2)}.$
				\end{itemize}
			\end{proof}
			
			Note that by Definitions \ref{defPG} and \ref{defDPG} we can associate a graph and a digraph to the same group $G$. 
			$\mathcal{P}(G)$ and $\vec{\mathcal{P}}(G)$ are surely related in some way since they came from the same group $G$. Let's see if we are able to analyse better this relation.
			$\mathcal{P}(G)$ and $\vec{\mathcal{P}}(G)$ have clearly the same vertex set, that is $G$. So what can we say about $E$ and $A$?
			
			\begin{Oss}
				\label{Edge-Arcs}	Let $G$ be a group. Then $\left\lbrace x,y\right\rbrace \in E$ if and only if $(x,y)\in A$ or $(y,x)\in A$.
			\end{Oss}
		
			The Remark above tells us that we are able to present $E$ if we know $A$. Indeed $\vec{\mathcal{P}}(G)$ shows us which elements of $G$ are joined in $\mathcal{P}(G)$.
			Instead, starting from $\mathcal{P}(G)$, we have to decide between two possible arcs for all the edges, and, at first, we do not know which one we have to choose.
			
			We say that we can \emph{reconstruct} $\vec{\mathcal{P}}(G)$ from $\mathcal{P}(G)$, if we are able to show $A$ only thanks to the information given by the power graph $\mathcal{P}(G)$.
			We say that we can \emph{reconstruct} $\mathcal{P}(G)$ from $\vec{\mathcal{P}}(G)$, if we are able to show $E$ only thanks to the information given by the directed power graph $\vec{\mathcal{P}}(G)$. 
			
			With this new notation, Remark \ref{Edge-Arcs}, tells us that we can reconstruct $\mathcal{P}(G)$ from $\vec{\mathcal{P}}(G)$. 
			Let's try to reconstruct $\vec{\mathcal{P}}(G)$ from $\mathcal{P}(G)$. If it is not always possible (we expect that), at some point, we find which groups fails and why.
			The problem is the following:
			
			\begin{Prob}
				\label{edge-arcs?}Given $\left\lbrace x,y\right\rbrace \in E$, can we decide which of $(x,y)$ and $(y,x)$ is an arc of $\vec{\mathcal{P}}(G)$? If the answer is yes: How can we choose between $(x,y)$ and $(y,x)$?
			\end{Prob}	
			
			\section{Three equivalence relations}
			\label{equivalenceRelations}
			
			To answer Question \ref{edge-arcs?} we will study three distinct equivalence relations defined on the elements of $G$. In this section we are going to present these three equivalence relations with the related results that we will use. 
			In particular all the definitions and results below are crucial for all that will follow in Section \ref{3}.
			
			\begin{definition}
				\rm Given two elements $x$ and $y$ of a group $G$ we say that $x\diamond y$ if $x=y$ or $(x,y)$ and $(y,x)$ are both arcs of $\vec{\mathcal{P}}(G)$.
			\end{definition}
		
			The relation $\diamond$ is clearly an equivalence relation and hence its classes form a partition of $G$. The $\diamond$-class of an element $x$ is denoted by $[x]_{\diamond}$. Despite the definition of the $\diamond$ relation is given through the directed power graph, this relation can also be seen directly from the group. This is a consequence of the next result.  
			
			\begin{lemma}
				\label{OssDiamondRelation}Let $x$ and $y$ be two elements of a group $G$. Then $x\diamond y$ if and only if $\langle x\rangle=\langle y\rangle $. Moreover, for all $x\in G$, $|[x]_{\diamond}|=\phi(o(x))$.
			\end{lemma}
			\begin{proof}
				If $x\diamond y$ then there exist $n$ and $m$ in $\mathbb{N}$ such that $x=y^n$ and $y=x^m$. Thus $\langle x\rangle=\langle y^n\rangle\subseteq \langle y\rangle$ and $\langle y\rangle=\langle x^m\rangle\subseteq \langle x\rangle $; then $\langle x\rangle =\langle y\rangle$. The converse is trivial. 
				
				Finally note that 	$$|[x]_{\diamond}|= |\left\lbrace y \in G \, | \, \langle x \rangle = \langle y \rangle  \right\rbrace |= $$ $$=|\left\lbrace y \in G \, | \, y=x^k \; \text{for some}\; 1\leq k < o(x) \; \text{with} \;  \gcd(k,o(x))=1\right\rbrace|=$$ $$=|\left\lbrace k\in \mathbb{N} \, | \,  1\leq k < o(x) \; \text{with} \; \gcd(k,o(x))=1\right\rbrace |= \phi(o(x)) $$
			\end{proof}
			
			\begin{corollary}
			\label{TrivialDiamondClasses}	Let $G$ be a group and $x\in G$. Then $|[x]_{\diamond}|=1$ if and only if $x=1$ or $x$ is an involution.
			\end{corollary}
			\begin{proof}
				By the previous lemma we have $|[x]_{\diamond}|=\phi(o(x))$. Now $\phi(o(x))=1$ if and only if $o(x)=1$ or $o(x)=2$. Therefore $|[x]_{\diamond}|=1$ if and only if $x=1$ or $x$ is an involution.
			\end{proof}
		
			It is useful to highlight that we have $[y]_{\diamond}\subseteq \langle y \rangle$ for all $y \in G$. In particular $[y]_{\diamond}=\langle y \rangle$ if and only if $y=1$.
			
			To know all the $\diamond$-classes is a strong information about the group whenever we want to construct the directed power graph of that same group since holds the following crucial lemma.   
			
			\begin{lemma}
				\label{lemmaClassiDiamond} Let $G$ be a group. In $\vec{\mathcal{P}}(G)$ all the following facts hold:
				\begin{enumerate}
					
					\item[\rm i)] The subdigraph induced by a $\diamond$-class is a complete digraph.
					
					\item[\rm ii)] Given two distinct $\diamond$-classes $X$ and $Y$, if there is at least one arc directed from $X$ to $Y$, then all the vertices of $X$ are joined to all the vertices of $Y$ by arcs directed from $X$ to $Y$ and there are no arcs directed from $Y$ to $X$.
					
					\item[\rm iii)] If $X$ and $Y$ are two distinct joined $\diamond$-classes, with $|X|>|Y|$, then all the arcs between them are directed from $X$ to $Y$.
					
					\item[\rm iv)] If $X$ and $Y$ are two distinct joined $\diamond$-classes, with $1\not=|X|=|Y|$, then exactly one of them, let say $X$, has all its elements joined with an involution with arcs that are directed from $X$ to the involution. In this case, all the arcs between $X$ and $Y$ are directed from $X$ to $Y$.
					
					\item[\rm v)] If $X$ and $Y$ are two distinct joined $\diamond$-classes, with $1=|X|=|Y|$, then one is the class of the identity and the other the class of an involution. So, between them, there is exactly one arc directed from the involution to the identity.
				\end{enumerate}
			\end{lemma}
			\begin{proof}
				\begin{enumerate}
										
					\item[i)] For all $\bar{x},\hat{x} \in [x]_{\diamond}$ with $\bar{x}\not= \hat{x}$, by the definition of the relation $\diamond$, $(\bar{x},\hat{x})$ and $(\hat{x}, \bar{x})$ are both arcs of $\vec{\mathcal{P}}(G)$ and hence also of the subdigraph induced by $[x]_{\diamond}$. This means that $\vec{\mathcal{P}}(G)_{[x]_{\diamond}}$ is a complete digraph.
					
					\item[ii)] Let $X$ and $Y$ be two distinct $\diamond$-classes.
					For $x \in X$ and $y\in Y$ we cannot have  both arcs $(x,y)$ and $(y,x)$, otherwise $x\diamond y$ and then $X=Y$.
					Suppose that there is an arc from $x$ to $y$. It follows that $y=x^m$ for some positive integer $m$.
					Let $\bar{y}$ be an element of $Y$, then we have $\bar{y}=y^n$ for some positive integer $n$.
					Hence $\bar{y}=x^{mn}$ and so there is an arc from $x$ to $\bar{y}$.
					Let $\bar{x}$ be an element of $X$, then $x=\bar{x}^k$ for some positive integer $k$.
					Hence $y=\bar{x}^{km}$ and so there is an arc from $\bar{x}$ to $y$.
					In conclusion we proved that if there is at least an arc from $X$ to $Y$ all the elements of $X$ are joined to all the elements of $Y$ by arcs directed from $X$ to $Y$.
					
					\item[iii)] By ii) it is enough to show that there is an arc directed from $X$ to $Y$ to prove that all the arcs between the two classes do the same. Since the classes are joined there exist $x \in X$ and $y \in Y$ two joined vertices. Suppose, by contradiction, that the arc between this two elements is $(y,x)$. Then $x=y^m$, for some integer $m\not= 1$. Thus $o(x)\mid o(y)$ and by Lemma \ref{phiEulero} and by Lemma \ref{OssDiamondRelation}, we have $|X|=\phi (o(x)) \mid \phi(o(y))=|Y|$ that is absurd since $|X|>|Y|$.

					\item[iv)] Let $X=[x]_{\diamond}$ and $Y=[y]_{\diamond}$ for some $x,y \in G$ with $(x,y) \in A$.
					Then $o(y) \mid o(x)$. Now $|X|=|Y|$, so, by Lemma \ref{OssDiamondRelation}, $\phi(o(x))=\phi(o(y))$.
					Moreover by Lemma \ref{phiEulero}, we have that there are only two possibilities: $o(x)=o(y)$ or $o(x)=2o(y)$ with $o(y)$ odd.
					Assume that $o(x)=o(y)$. Then $X=Y$ against the assumption.
					Hence we have $o(x)=2o(y)$ and, since $|Y|\not=1$, we have $o(y)\not=1$. So $x \not= x^{o(y)}$ since otherwise:
					$$x=x^{o(y)} \Rightarrow x^2=x^{2o(y)}=x^{o(x)}=1 $$
					hence $o(x)=2$, then $o(y)=1$ that is a contradiction.					
					It follows that $(x, x^{o(y)}) \in A$. Therefore, since by Lemma \ref{TrivialDiamondClasses} $\{x^{o(y)}\}$ is a $\diamond$-class, by ii) all the vertices of $X$ are joined to the involution $x^{o(y)}$ by arcs that are directed from $X$ to the involution. 
					
					\item[v)] Since $|X|=|Y|=1$, by Corollary \ref{TrivialDiamondClasses}, we have $X=\left\lbrace x\right\rbrace $ and $Y=\left\lbrace y\right\rbrace $ with $x,y$ given by $1$ or involutions. Assume, by contradiction, that $X=\left\lbrace x\right\rbrace $ and $Y=\left\lbrace y\right\rbrace$ for $x$ and $y$ both involutions with $x\not=y$. Now since $X$ and $Y$ are joined, then up to renaming $X$ and $Y$, we have $(x,y)\in A$ so $y=x^m$ for some $m \in \mathbb{N}$. But $x^m\in \left\lbrace 1,x\right\rbrace $ while $y\notin \left\lbrace 1,x\right\rbrace$.
				\end{enumerate}
			\end{proof}
			
			We emphasize that we extrapolate this result from the proof of \cite[Theorem 2]{Cameron_2} where it is used without a formal proof.
			In Figure \ref{SchemaLemmacClassiDiamond} we propose a schematic representation of some of the points of Lemma \ref{lemmaClassiDiamond}. 
			 
			\begin{figure}
				\centering
				\includegraphics[width=0.7\textwidth]{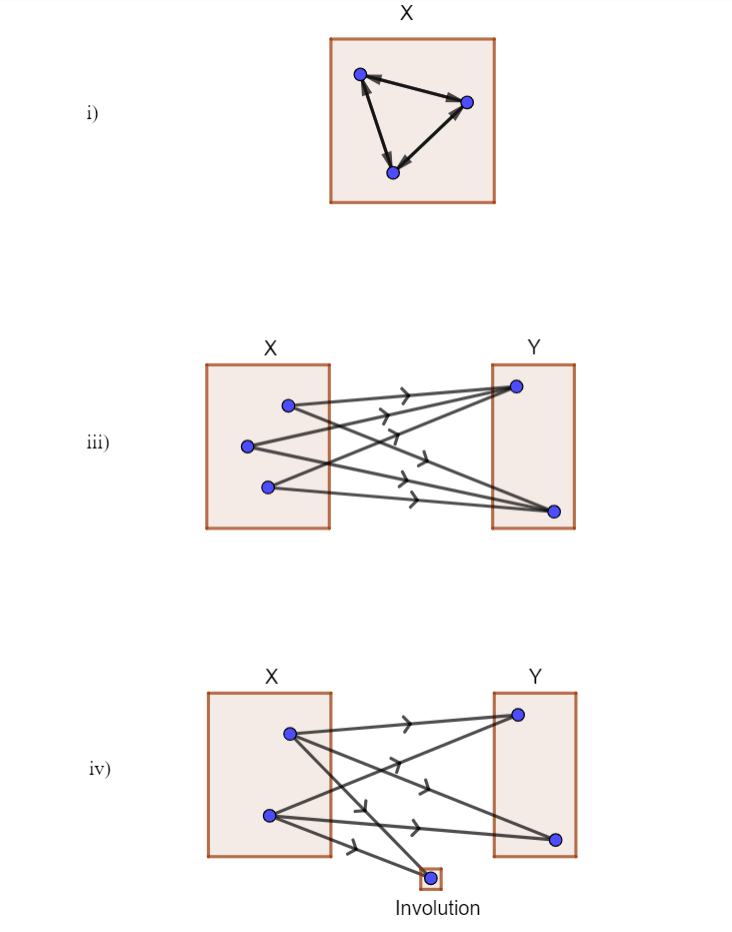}
				\caption{A schematic representation of points i), iii) and iv) of Lemma \ref{lemmaClassiDiamond}.}
				\label{SchemaLemmacClassiDiamond}
			\end{figure}
			
			Now we can begin to understand why Lemma \ref{lemmaClassiDiamond} is said to be crucial.
			Suppose that $\left\lbrace x, y \right\rbrace $ is an edge of $\mathcal{P}(G)$. Suppose also to know all the $\diamond$-classes of $G$. In order to manage Question \ref{edge-arcs?} we distinguish the following cases:
			
			\begin{itemize}
				\item \underline{There exists a $\diamond$-class $X$ such that $x,y \in X$}.
				
				Then both $(x,y)$ and $(y,x)$ are arcs of $\vec{\mathcal{P}}(G)$ by Lemma \ref{lemmaClassiDiamond} i). This is the only case where we take both arcs.
				
				\item \underline{There exist $X$ and $Y$ two distinct $\diamond$-classes such that $x \in X$ and $y\in Y$}.
				
				Then one of them, say $X$, has size greater or equal to the other one, that is $|X|\geq |Y|$.
				We consider two subcases:
				\begin{enumerate}
					\item[a)] \underline{$|X|=1$}. Then we also have $|Y|=1$. So $X=\left\lbrace x\right\rbrace $ and $Y=\left\lbrace y\right\rbrace$ with $x\not=y$. Since $\left\lbrace x,y\right\rbrace \in E$ holds, then $X$ and $Y$ are joined in $\vec{\mathcal{P}}(G)$. By Lemma \ref{lemmaClassiDiamond} v) one of them is the class of the identity and the other one is the class of an involution. Renaming if necessary we can assume $y=1$ and $x$ an involution so that $(x,y)\in A$.
					
					\item[b)] \underline{$|X|>1$}. Then there are two possibilities: $|X|\not=|Y|$ or $|X|=|Y|$.
					In the first case we have $|X|>|Y|$ and by Lemma \ref{lemmaClassiDiamond} iii), we get $(x,y)\in A$. In the second case we have $|X|=|Y|\not=1$ and one of the two vertices, let say $x$, must be joined to an involution. Then $(x,y)\in A$ by Lemma \ref{lemmaClassiDiamond} iv).
				\end{enumerate}
				
			\end{itemize}
			
			In conclusion, in order to reconstruct $\vec{\mathcal{P}}(G)$ from $\mathcal{P}(G)$ \emph{it is sufficient to recognize, in $\mathcal{P}(G)$, all the $\diamond$-classes, the identity and the involutions.}
			Note that we only need a way to determine the identity and the $\diamond$-classes since, by Corollary \ref{TrivialDiamondClasses}, once we have this two information the involutions are the vertices $x\not=1$ such that $|[x]_{\diamond}|=1$.

			\begin{figure}
				\centering
				\begin{tikzpicture}[line cap=round,line join=round,>=triangle 45,x=1.0cm,y=1.0cm]
					\clip(-6.,2.) rectangle (4.,12.);
					\fill[line width=2.pt,color=zzttqq,fill=zzttqq,fill opacity=0.10000000149011612] (-3.00711852963106,9.809911400437345) -- (-1.1933294275569224,10.487257783623239) -- (-1.3240850694054276,11.397934907263139) -- (-3.6127226668446086,10.710096634610261) -- cycle;
					\fill[line width=2.pt,color=zzttqq,fill=zzttqq,fill opacity=0.10000000149011612] (-5.110910354913101,9.090550324695746) -- (-4.265322159020893,8.56129342866397) -- (-3.00711852963106,9.809911400437345) -- (-3.6127226668446086,10.710096634610261) -- cycle;
					\fill[line width=2.pt,color=zzttqq,fill=zzttqq,fill opacity=0.10000000149011612] (-5.110910354913101,9.090550324695746) -- (-5.490950168512438,6.796738592614046) -- (-4.58366337471664,6.826333803122147) -- (-4.265322159020893,8.56129342866397) -- cycle;
					\fill[line width=2.pt,color=zzttqq,fill=zzttqq,fill opacity=0.10000000149011612] (-3.7241420923381203,4.629779414821492) -- (-4.255820774314594,3.7428472333336757) -- (-5.490950168512438,6.796738592614046) -- (-4.58366337471664,6.826333803122147) -- cycle;
					\fill[line width=2.pt,color=zzttqq,fill=zzttqq,fill opacity=0.10000000149011612] (2.604291484014996,7.284127314600211) -- (2.400341814808526,8.181505859108677) -- (3.22633797509473,8.528220296759674) -- (3.389497710459906,7.284127314600211) -- cycle;
					\fill[line width=2.pt,color=zzttqq,fill=zzttqq,fill opacity=0.10000000149011612] (2.400341814808526,8.181505859108677) -- (1.9006651252526743,8.92592215171229) -- (2.5838965170943493,9.5377711593317) -- (3.22633797509473,8.528220296759674) -- cycle;
					\fill[line width=2.pt,color=zzttqq,fill=zzttqq,fill opacity=0.10000000149011612] (1.9006651252526743,8.92592215171229) -- (0.9013117461409711,9.762115795458815) -- (1.5669320840466672,10.502126775207563) -- (2.5838965170943493,9.5377711593317) -- cycle;
					\fill[line width=2.pt,color=zzttqq,fill=zzttqq,fill opacity=0.10000000149011612] (0.9013117461409711,9.762115795458815) -- (0.13650048661670827,10.1904101007924) -- (0.5625411481055639,11.112905047063636) -- (1.5669320840466672,10.502126775207563) -- cycle;
					\fill[line width=2.pt,color=zzttqq,fill=zzttqq,fill opacity=0.10000000149011612] (0.13650048661670827,10.1904101007924) -- (-0.5248128745958504,10.407672479699302) -- (-0.2789755820072525,11.384362056777448) -- (0.5625411481055639,11.112905047063636) -- cycle;
					\fill[line width=2.pt,color=zzttqq,fill=zzttqq,fill opacity=0.10000000149011612] (-0.2789755820072525,11.384362056777448) -- (-1.3240850694054276,11.397934907263139) -- (-1.1933294275569224,10.487257783623239) -- (-0.5248128745958504,10.407672479699302) -- cycle;
					\fill[line width=2.pt,color=zzttqq,fill=zzttqq,fill opacity=0.10000000149011612] (-2.,8.) -- (8.905437690318596E-4,7.999107867261943) -- (0.,6.) -- (-2.,6.) -- cycle;
					\fill[line width=2.pt,color=zzttqq,fill=zzttqq,fill opacity=0.10000000149011612] (-4.255820774314594,3.7428472333336757) -- (-0.20647165890010305,2.544644452014348) -- (-0.3808547962166728,3.589809124493197) -- (-3.7241420923381203,4.629779414821492) -- cycle;
					\fill[line width=2.pt,color=zzttqq,fill=zzttqq,fill opacity=0.10000000149011612] (-0.3808547962166728,3.589809124493197) -- (1.7884928071891157,5.08147088717034) -- (2.655278901316614,4.388042011868344) -- (-0.20647165890010305,2.544644452014348) -- cycle;
					\fill[line width=2.pt,color=zzttqq,fill=zzttqq,fill opacity=0.10000000149011612] (1.7884928071891157,5.08147088717034) -- (2.604291484014996,7.284127314600211) -- (3.389497710459906,7.284127314600211) -- (2.655278901316614,4.388042011868344) -- cycle;
					\fill[line width=2.pt,color=zzttqq,fill=zzttqq,fill opacity=0.10000000149011612] (2.604291484014996,7.284127314600211) -- (1.1171708774262428,6.54424864751115) -- (1.7884928071891157,5.08147088717034) -- cycle;
					\fill[line width=2.pt,color=zzttqq,fill=zzttqq,fill opacity=0.10000000149011612] (1.7884928071891157,5.08147088717034) -- (0.027873487448356438,5.488419264647737) -- (-0.3808547962166728,3.589809124493197) -- cycle;
					\fill[line width=2.pt,color=zzttqq,fill=zzttqq,fill opacity=0.10000000149011612] (-0.3808547962166728,3.589809124493197) -- (-1.580788935019527,5.203823602504077) -- (-3.7241420923381203,4.629779414821492) -- cycle;
					
					\draw [line width=1.pt] (-3.6189179549840818,10.083616523866972)-- (-1.,7.);
					\draw [line width=1.pt] (-4.470198577067798,9.077557606858948)-- (-1.,7.);
					\draw [line width=1.pt] (-4.702366019454266,8.458444427161703)-- (-1.,7.);
					\draw [line width=1.pt] (-4.992575322437351,7.374996362691523)-- (-1.,7.);
					\draw [line width=1.pt] (-4.973228035571812,6.252853724490266)-- (-1.,7.);
					\draw [line width=1.pt] (-4.276725708412408,4.666376201516074)-- (-1.,7.);
					\draw [line width=1.pt] (-3.541528807521925,3.9118320137600566)-- (-1.,7.);
					\draw [line width=1.pt] (-1.7666619570249664,4.628983561782251)-- (-1.,7.);
					\draw [line width=1.pt] (-0.7168249251532304,3.041204104810805)-- (-1.,7.);
					\draw [line width=1.pt] (0.07641383633386878,3.1379405391384996)-- (-1.,7.);
					\draw [line width=1.pt] (0.428937407628609,4.898618571476549)-- (-1.,7.);
					\draw [line width=1.pt] (2.069184383484386,4.414861472264069)-- (-1.,7.);
					\draw [line width=1.pt] (2.533519268257322,5.033974651961314)-- (-1.,7.);
					\draw [line width=1.pt] (1.661554594802546,6.400870768344779)-- (-1.,7.);
					\draw [line width=1.pt] (3.,7.)-- (-1.,7.);
					\draw [line width=1.pt] (2.920465005568102,7.742594813136765)-- (-1.,7.);
					\draw [line width=1.pt] (2.533519268257322,8.69061186954817)-- (-1.,7.);
					\draw [line width=1.pt] (1.759627793635761,9.754712647152811)-- (-1.,7.);
					\draw [line width=1.pt] (0.7535688766277328,10.567298695505446)-- (-1.,7.);
					\draw [line width=1.pt] (-0.02032259799382734,10.85750799848853)-- (-1.,7.);
					\draw [line width=1.pt] (-0.7168249251532315,10.915549859085147)-- (-1.,7.);
					\draw [line width=1.pt] (-1.4520218260437137,10.915549859085147)-- (-1.,7.);
					\draw [line width=1.pt] (-3.1158884964800677,10.373825826850057)-- (-1.,7.);
					
					\draw [line width=1.pt,color=zzttqq] (-3.00711852963106,9.809911400437345)-- (-1.1933294275569224,10.487257783623239);
					\draw [line width=1.pt,color=zzttqq] (-1.1933294275569224,10.487257783623239)-- (-1.3240850694054276,11.397934907263139);
					\draw [line width=1.pt,color=zzttqq] (-1.3240850694054276,11.397934907263139)-- (-3.6127226668446086,10.710096634610261);
					\draw [line width=1.pt,color=zzttqq] (-3.6127226668446086,10.710096634610261)-- (-3.00711852963106,9.809911400437345);
					\draw [line width=1.pt,color=zzttqq] (-5.110910354913101,9.090550324695746)-- (-4.265322159020893,8.56129342866397);
					\draw [line width=1.pt,color=zzttqq] (-4.265322159020893,8.56129342866397)-- (-3.00711852963106,9.809911400437345);
					\draw [line width=1.pt,color=zzttqq] (-3.00711852963106,9.809911400437345)-- (-3.6127226668446086,10.710096634610261);
					\draw [line width=1.pt,color=zzttqq] (-3.6127226668446086,10.710096634610261)-- (-5.110910354913101,9.090550324695746);
					\draw [line width=1.pt,color=zzttqq] (-5.110910354913101,9.090550324695746)-- (-5.490950168512438,6.796738592614046);
					\draw [line width=1.pt,color=zzttqq] (-5.490950168512438,6.796738592614046)-- (-4.58366337471664,6.826333803122147);
					\draw [line width=1.pt,color=zzttqq] (-4.58366337471664,6.826333803122147)-- (-4.265322159020893,8.56129342866397);
					\draw [line width=1.pt,color=zzttqq] (-4.265322159020893,8.56129342866397)-- (-5.110910354913101,9.090550324695746);
					\draw [line width=1.pt,color=zzttqq] (-3.7241420923381203,4.629779414821492)-- (-4.255820774314594,3.7428472333336757);
					\draw [line width=1.pt,color=zzttqq] (-4.255820774314594,3.7428472333336757)-- (-5.490950168512438,6.796738592614046);
					\draw [line width=1.pt,color=zzttqq] (-5.490950168512438,6.796738592614046)-- (-4.58366337471664,6.826333803122147);
					\draw [line width=1.pt,color=zzttqq] (-4.58366337471664,6.826333803122147)-- (-3.7241420923381203,4.629779414821492);
					\draw [line width=1.pt,color=zzttqq] (2.604291484014996,7.284127314600211)-- (2.400341814808526,8.181505859108677);
					\draw [line width=1.pt,color=zzttqq] (2.400341814808526,8.181505859108677)-- (3.22633797509473,8.528220296759674);
					\draw [line width=1.pt,color=zzttqq] (3.22633797509473,8.528220296759674)-- (3.389497710459906,7.284127314600211);
					\draw [line width=1.pt,color=zzttqq] (3.389497710459906,7.284127314600211)-- (2.604291484014996,7.284127314600211);
					\draw [line width=1.pt,color=zzttqq] (2.400341814808526,8.181505859108677)-- (1.9006651252526743,8.92592215171229);
					\draw [line width=1.pt,color=zzttqq] (1.9006651252526743,8.92592215171229)-- (2.5838965170943493,9.5377711593317);
					\draw [line width=1.pt,color=zzttqq] (2.5838965170943493,9.5377711593317)-- (3.22633797509473,8.528220296759674);
					\draw [line width=1.pt,color=zzttqq] (3.22633797509473,8.528220296759674)-- (2.400341814808526,8.181505859108677);
					\draw [line width=1.pt,color=zzttqq] (1.9006651252526743,8.92592215171229)-- (0.9013117461409711,9.762115795458815);
					\draw [line width=1.pt,color=zzttqq] (0.9013117461409711,9.762115795458815)-- (1.5669320840466672,10.502126775207563);
					\draw [line width=1.pt,color=zzttqq] (1.5669320840466672,10.502126775207563)-- (2.5838965170943493,9.5377711593317);
					\draw [line width=1.pt,color=zzttqq] (2.5838965170943493,9.5377711593317)-- (1.9006651252526743,8.92592215171229);
					\draw [line width=1.pt,color=zzttqq] (0.9013117461409711,9.762115795458815)-- (0.13650048661670827,10.1904101007924);
					\draw [line width=1.pt,color=zzttqq] (0.13650048661670827,10.1904101007924)-- (0.5625411481055639,11.112905047063636);
					\draw [line width=1.pt,color=zzttqq] (0.5625411481055639,11.112905047063636)-- (1.5669320840466672,10.502126775207563);
					\draw [line width=1.pt,color=zzttqq] (1.5669320840466672,10.502126775207563)-- (0.9013117461409711,9.762115795458815);
					\draw [line width=1.pt,color=zzttqq] (0.13650048661670827,10.1904101007924)-- (-0.5248128745958504,10.407672479699302);
					\draw [line width=1.pt,color=zzttqq] (-0.5248128745958504,10.407672479699302)-- (-0.2789755820072525,11.384362056777448);
					\draw [line width=1.pt,color=zzttqq] (-0.2789755820072525,11.384362056777448)-- (0.5625411481055639,11.112905047063636);
					\draw [line width=1.pt,color=zzttqq] (0.5625411481055639,11.112905047063636)-- (0.13650048661670827,10.1904101007924);
					\draw [line width=1.pt,color=zzttqq] (-0.2789755820072525,11.384362056777448)-- (-1.3240850694054276,11.397934907263139);
					\draw [line width=1.pt,color=zzttqq] (-1.3240850694054276,11.397934907263139)-- (-1.1933294275569224,10.487257783623239);
					\draw [line width=1.pt,color=zzttqq] (-1.1933294275569224,10.487257783623239)-- (-0.5248128745958504,10.407672479699302);
					\draw [line width=1.pt,color=zzttqq] (-0.5248128745958504,10.407672479699302)-- (-0.2789755820072525,11.384362056777448);
					\draw [line width=1.pt,color=zzttqq] (-2.,8.)-- (8.905437690318596E-4,7.999107867261943);
					\draw [line width=1.pt,color=zzttqq] (8.905437690318596E-4,7.999107867261943)-- (0.,6.);
					\draw [line width=1.pt,color=zzttqq] (0.,6.)-- (-2.,6.);
					\draw [line width=1.pt,color=zzttqq] (-2.,6.)-- (-2.,8.);
					\draw [line width=1.pt,color=zzttqq] (-4.255820774314594,3.7428472333336757)-- (-0.20647165890010305,2.544644452014348);
					\draw [line width=1.pt,color=zzttqq] (-0.20647165890010305,2.544644452014348)-- (-0.3808547962166728,3.589809124493197);
					\draw [line width=1.pt,color=zzttqq] (-0.3808547962166728,3.589809124493197)-- (-3.7241420923381203,4.629779414821492);
					\draw [line width=1.pt,color=zzttqq] (-3.7241420923381203,4.629779414821492)-- (-4.255820774314594,3.7428472333336757);
					\draw [line width=1.pt,color=zzttqq] (-0.3808547962166728,3.589809124493197)-- (1.7884928071891157,5.08147088717034);
					\draw [line width=1.pt,color=zzttqq] (1.7884928071891157,5.08147088717034)-- (2.655278901316614,4.388042011868344);
					\draw [line width=1.pt,color=zzttqq] (2.655278901316614,4.388042011868344)-- (-0.20647165890010305,2.544644452014348);
					\draw [line width=1.pt,color=zzttqq] (-0.20647165890010305,2.544644452014348)-- (-0.3808547962166728,3.589809124493197);
					\draw [line width=1.pt,color=zzttqq] (1.7884928071891157,5.08147088717034)-- (2.604291484014996,7.284127314600211);
					\draw [line width=1.pt,color=zzttqq] (2.604291484014996,7.284127314600211)-- (3.389497710459906,7.284127314600211);
					\draw [line width=1.pt,color=zzttqq] (3.389497710459906,7.284127314600211)-- (2.655278901316614,4.388042011868344);
					\draw [line width=1.pt,color=zzttqq] (2.655278901316614,4.388042011868344)-- (1.7884928071891157,5.08147088717034);
					\draw [line width=1.pt,color=zzttqq] (2.604291484014996,7.284127314600211)-- (1.1171708774262428,6.54424864751115);
					\draw [line width=1.pt,color=zzttqq] (1.1171708774262428,6.54424864751115)-- (1.7884928071891157,5.08147088717034);
					\draw [line width=1.pt,color=zzttqq] (1.7884928071891157,5.08147088717034)-- (2.604291484014996,7.284127314600211);
					\draw [line width=1.pt,color=zzttqq] (1.7884928071891157,5.08147088717034)-- (0.027873487448356438,5.488419264647737);
					\draw [line width=1.pt,color=zzttqq] (0.027873487448356438,5.488419264647737)-- (-0.3808547962166728,3.589809124493197);
					\draw [line width=1.pt,color=zzttqq] (-0.3808547962166728,3.589809124493197)-- (1.7884928071891157,5.08147088717034);
					\draw [line width=1.pt,color=zzttqq] (-0.3808547962166728,3.589809124493197)-- (-1.580788935019527,5.203823602504077);
					\draw [line width=1.pt,color=zzttqq] (-1.580788935019527,5.203823602504077)-- (-3.7241420923381203,4.629779414821492);
					\draw [line width=1.pt,color=zzttqq] (-3.7241420923381203,4.629779414821492)-- (-0.3808547962166728,3.589809124493197);
					
					\draw [line width=1.pt] (-4.470198577067798,9.077557606858948)-- (-3.6189179549840818,10.083616523866972);
					\draw [line width=1.pt] (-3.1158884964800677,10.373825826850057)-- (-1.4520218260437137,10.915549859085147);
					\draw [line width=1.pt] (-4.702366019454266,8.458444427161703)-- (-4.992575322437351,7.374996362691523);
					\draw [line width=1.pt] (-4.973228035571812,6.252853724490266)-- (-4.276725708412408,4.666376201516074);
					\draw [line width=1.pt] (-3.541528807521925,3.9118320137600566)-- (-1.7666619570249664,4.628983561782251);
					\draw [line width=1.pt] (-3.541528807521925,3.9118320137600566)-- (-0.7168249251532304,3.041204104810805);
					\draw [line width=1.pt] (-1.7666619570249664,4.628983561782251)-- (-0.7168249251532304,3.041204104810805);
					\draw [line width=1.pt] (0.428937407628609,4.898618571476549)-- (0.07641383633386878,3.1379405391384996);
					\draw [line width=1.pt] (0.428937407628609,4.898618571476549)-- (2.069184383484386,4.414861472264069);
					\draw [line width=1.pt] (0.07641383633386878,3.1379405391384996)-- (2.069184383484386,4.414861472264069);
					\draw [line width=1.pt] (1.661554594802546,6.400870768344779)-- (2.533519268257322,5.033974651961314);
					\draw [line width=1.pt] (1.661554594802546,6.400870768344779)-- (3.,7.);
					\draw [line width=1.pt] (3.,7.)-- (2.533519268257322,5.033974651961314);
					\begin{scriptsize}
						\draw [fill=qqqqff] (-1.4520218260437137,10.915549859085147) circle (3pt);
						\draw [fill=qqqqff] (-3.1158884964800677,10.373825826850057) circle (3pt);
						\draw [fill=qqqqff] (-4.470198577067798,9.077557606858948) circle (3pt);
						\draw [fill=qqqqff] (-3.6189179549840818,10.083616523866972) circle (3pt);
						\draw [fill=qqqqff] (-4.702366019454266,8.458444427161703) circle (3pt);
						\draw [fill=qqqqff] (-4.992575322437351,7.374996362691523) circle (3pt);
						\draw [fill=qqqqff] (-4.973228035571812,6.252853724490266) circle (3pt);
						\draw [fill=qqqqff] (-4.276725708412408,4.666376201516074) circle (3pt);
						\draw [fill=qqqqff] (2.533519268257322,5.033974651961314) circle (3pt);
						\draw [fill=qqqqff] (1.661554594802546,6.400870768344779) circle (3pt);
						\draw [fill=qqqqff] (2.069184383484386,4.414861472264069) circle (3pt);
						\draw [fill=qqqqff] (3.,7.) circle (3pt);
						\draw [fill=qqqqff] (-0.7168249251532315,10.915549859085147) circle (3pt);
						\draw [fill=qqqqff] (-0.7168249251532304,3.041204104810805) circle (3pt);
						\draw [fill=qqqqff] (-3.541528807521925,3.9118320137600566) circle (3pt);
						\draw [fill=qqqqff] (-0.02032259799382734,10.85750799848853) circle (3pt);
						\draw [fill=qqqqff] (2.920465005568102,7.742594813136765) circle (3pt);
						\draw [fill=qqqqff] (0.07641383633386878,3.1379405391384996) circle (3pt);
						\draw [fill=qqqqff] (-1.7666619570249664,4.628983561782251) circle (3pt);
						\draw [fill=qqqqff] (0.428937407628609,4.898618571476549) circle (3pt);
						\draw [fill=qqqqff] (1.759627793635761,9.754712647152811) circle (3pt);
						\draw [fill=qqqqff] (2.533519268257322,8.69061186954817) circle (3pt);
						\draw [fill=qqqqff] (-1.,7.) circle (3pt);
						\draw [fill=qqqqff] (0.7535688766277328,10.567298695505446) circle (3pt);
					\end{scriptsize}
				\end{tikzpicture}
				\caption{$\mathcal{P}(S_4)$ where all the elements in the same box are also in the same $\diamond$-class.}
				\label{UPG_S_4_diamond}
			\end{figure}
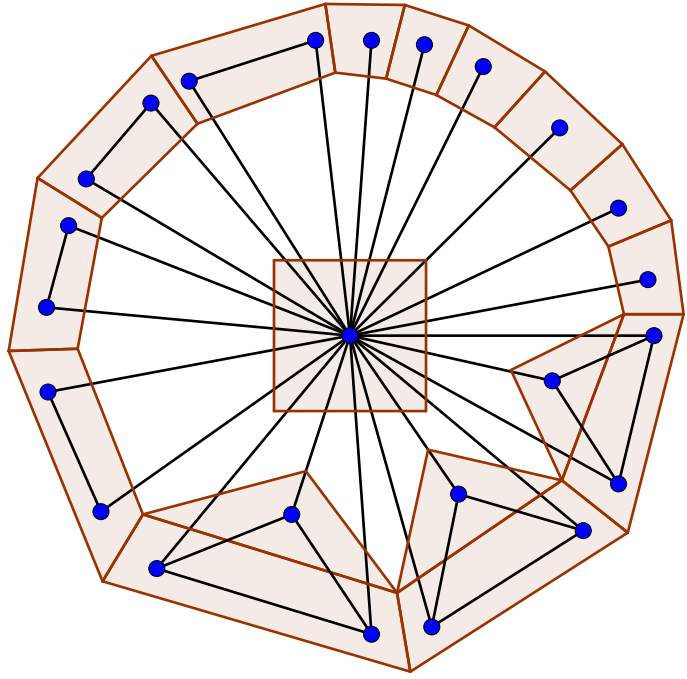
			
			Let's have a look at  Figure \ref{UPG_S_4_diamond}, that represents $\mathcal{P}(S_4)$ with all the $\diamond$-classes highlighted. Note that we clearly recognize the identity as the only star vertex.
			So now we have enough information to reconstruct $\vec{\mathcal{P}}(S_4)$ from $\mathcal{P}(S_4)$.
			
			For all the edges of the form $\{x,1\}$ we know that $(x,1)\in A$ and $(1,x)\notin A$.
			So we can just examine the edges of $\mathcal{P}^*(S_4)$, that appear in Figure \ref{PPG_S_4}.
			Now, except for the edges between the vertices in the triangles, for all the edges we have to choose both possible arcs since the vertices are in the same $\diamond$-class.
			Instead, in the triangles, two vertices are in the same $\diamond$-class and the remaining one is in another $\diamond$-class.
			For the edge between the two vertices in the same $\diamond$-class we have to choose both arcs.
			Finally, for the two edges remaining in the triangles, we have to choose the arc that is directed from the bigger $\diamond$-class to the smaller one.
			In Figure \ref{DPG_S_4} the result of this procedure: $\vec{\mathcal{P}}(S_4)$.

		\begin{figure}
			\centering
			\begin{tikzpicture}[line cap=round,line join=round,>=triangle 45,x=1.0cm,y=1.0cm]
				\clip(-6.,2.) rectangle (4.,12.);
				\draw [-latex,line width=2.pt] (-3.1158884964800677,10.373825826850057) -- (-1.4520218260437137,10.915549859085147);
				\draw [-latex,line width=2.pt] (-1.4520218260437137,10.915549859085147) -- (-3.1158884964800677,10.373825826850057);
				\draw [-latex,line width=2.pt] (-3.6189179549840818,10.083616523866972) -- (-4.470198577067798,9.077557606858948);
				\draw [-latex,line width=2.pt] (-4.470198577067798,9.077557606858948) -- (-3.6189179549840818,10.083616523866972);
				\draw [-latex,line width=2.pt] (-4.702366019454266,8.458444427161703) -- (-4.992575322437351,7.374996362691523);
				\draw [-latex,line width=2.pt] (-4.992575322437351,7.374996362691523) -- (-4.702366019454266,8.458444427161703);
				\draw [-latex,line width=2.pt] (-4.973228035571812,6.252853724490266) -- (-4.276725708412408,4.666376201516074);
				\draw [-latex,line width=2.pt] (-4.276725708412408,4.666376201516074) -- (-4.973228035571812,6.252853724490266);
				\draw [-latex,line width=2.pt] (-3.541528807521925,3.9118320137600566) -- (-1.7666619570249664,4.628983561782251);
				\draw [-latex,line width=2.pt] (-0.7168249251532304,3.041204104810805) -- (-1.7666619570249664,4.628983561782251);
				\draw [-latex,line width=2.pt] (-3.541528807521925,3.9118320137600566) -- (-0.7168249251532304,3.041204104810805);
				\draw [-latex,line width=2.pt] (-0.7168249251532304,3.041204104810805) -- (-3.541528807521925,3.9118320137600566);
				\draw [-latex,line width=2.pt] (0.07641383633386878,3.1379405391384996) -- (0.428937407628609,4.898618571476549);
				\draw [-latex,line width=2.pt] (2.069184383484386,4.414861472264069) -- (0.428937407628609,4.898618571476549);
				\draw [-latex,line width=2.pt] (0.07641383633386878,3.1379405391384996) -- (2.069184383484386,4.414861472264069);
				\draw [-latex,line width=2.pt] (2.069184383484386,4.414861472264069) -- (0.07641383633386889,3.1379405391384996);
				\draw [-latex,line width=2.pt] (2.533519268257322,5.033974651961314) -- (1.661554594802546,6.400870768344779);
				\draw [-latex,line width=2.pt] (2.533519268257322,5.033974651961314) -- (3.,7.);
				\draw [-latex,line width=2.pt] (3.,7.) -- (1.661554594802546,6.400870768344779);
				\draw [-latex,line width=2.pt] (3.,7.) -- (2.533519268257322,5.033974651961314);
				\draw [line width=2.pt] (-3.6189179549840818,10.083616523866972)-- (-1.,7.);
				\draw [line width=1.5pt] (-2.2084453947227316,8.422870915167943) -- (-2.4521837936106934,8.420591962610315);
				\draw [line width=1.5pt] (-2.2084453947227316,8.422870915167943) -- (-2.166734161373389,8.663024561256657);
				\draw [line width=2.pt] (-4.470198577067798,9.077557606858948)-- (-1.,7.);
				\draw [line width=1.5pt] (-2.601214790057432,7.9586240941029205) -- (-2.8312849397257627,7.878117405257711);
				\draw [line width=1.5pt] (-2.601214790057432,7.9586240941029205) -- (-2.6389136373420383,8.199440201601233);
				\draw [line width=2.pt] (-4.702366019454266,8.458444427161703)-- (-1.,7.);
				\draw [line width=1.5pt] (-2.7059971974046944,7.672030288803024) -- (-2.9198133194605265,7.554999238793923);
				\draw [line width=1.5pt] (-2.7059971974046944,7.672030288803024) -- (-2.782552699993743,7.903445188367777);
				\draw [line width=2.pt] (-4.992575322437351,7.374996362691523)-- (-1.,7.);
				\draw [line width=1.5pt] (-2.8409270808149505,7.172906182985797) -- (-3.013798059250633,7.001065484861294);
				\draw [line width=1.5pt] (-2.8409270808149505,7.172906182985797) -- (-2.9787772631867178,7.37393087783023);
				\draw [line width=2.pt] (-4.973228035571812,6.252853724490266)-- (-1.,7.);
				\draw [line width=1.5pt] (-2.8332575394987596,6.655264804732126) -- (-2.952008486801514,6.442399088300557);
				\draw [line width=1.5pt] (-2.8332575394987596,6.655264804732126) -- (-3.0212195487702993,6.810454636189707);
				\draw [line width=2.pt] (-4.276725708412408,4.666376201516074)-- (-1.,7.);
				\draw [line width=1.5pt] (-2.511257981778891,5.923709853750097) -- (-2.5297367506157316,5.68066225384526);
				\draw [line width=1.5pt] (-2.511257981778891,5.923709853750097) -- (-2.7469889577966784,5.985713947670812);
				\draw [line width=2.pt] (-3.541528807521925,3.9118320137600566)-- (-1.,7.);
				\draw [line width=1.5pt] (-2.1716046765011137,5.576403287741003) -- (-2.126179666727793,5.336924334168208);
				\draw [line width=1.5pt] (-2.1716046765011137,5.576403287741003) -- (-2.4153491407941337,5.574907679591847);
				\draw [line width=2.pt] (-1.7666619570249664,4.628983561782251)-- (-1.,7.);
				\draw [line width=1.5pt] (-1.335321831690352,5.962967227281868) -- (-1.2051604428435905,5.7568808047045685);
				\draw [line width=1.5pt] (-1.335321831690352,5.962967227281868) -- (-1.5615015141813748,5.872102757077682);
				\draw [line width=2.pt] (-0.7168249251532304,3.041204104810805)-- (-1.,7.);
				\draw [line width=1.5pt] (-0.8695459624196558,5.176248708459312) -- (-0.6716364753119225,5.03396225221705);
				\draw [line width=1.5pt] (-0.8695459624196558,5.176248708459312) -- (-1.0451884498413087,5.007241852593754);
				\draw [line width=2.pt] (0.07641383633386878,3.1379405391384996)-- (-1.,7.);
				\draw [line width=1.5pt] (-0.503688161070532,5.219285400860831) -- (-0.2814149242831663,5.119244364654209);
				\draw [line width=1.5pt] (-0.503688161070532,5.219285400860831) -- (-0.6421712393829652,5.0186961744842895);
				\draw [line width=2.pt] (0.428937407628609,4.898618571476549)-- (-1.,7.);
				\draw [line width=1.5pt] (-0.37327651045814075,6.078346542884425) -- (-0.13068658761031388,6.054603542865207);
				\draw [line width=1.5pt] (-0.37327651045814075,6.078346542884425) -- (-0.4403760047610772,5.844015028611341);
				\draw [line width=2.pt] (2.069184383484386,4.414861472264069)-- (-1.,7.);
				\draw [line width=1.5pt] (0.41524290018099735,5.807957264786713) -- (0.6552240261278082,5.850649886005468);
				\draw [line width=1.5pt] (0.41524290018099735,5.807957264786713) -- (0.4139603573565776,5.5642115862586);
				\draw [line width=2.pt] (2.533519268257322,5.033974651961314)-- (-1.,7.);
				\draw [line width=1.5pt] (0.6304008434492131,6.092856401129611) -- (0.857802524307408,6.180617874795993);
				\draw [line width=1.5pt] (0.6304008434492131,6.092856401129611) -- (0.6757167439499148,5.853356777165319);
				\draw [line width=2.pt] (1.661554594802546,6.400870768344779)-- (-1.,7.);
				\draw [line width=1.5pt] (0.1785423397538285,6.734704233432324) -- (0.3718999165131968,6.88311733334932);
				\draw [line width=1.5pt] (0.1785423397538285,6.734704233432324) -- (0.2896546782893492,6.517753434995456);
				\draw [line width=2.pt] (3.,7.)-- (-1.,7.);
				\draw [line width=1.5pt] (0.8439556589955292,7.) -- (1.,7.1872532092053625);
				\draw [line width=1.5pt] (0.8439556589955292,7.) -- (1.,6.812746790794636);
				\draw [line width=2.pt] (2.920465005568102,7.742594813136765)-- (-1.,7.);
				\draw [line width=1.5pt] (0.8069143030405113,7.342256642340844) -- (0.925383585711005,7.555279246260629);
				\draw [line width=1.5pt] (0.8069143030405113,7.342256642340844) -- (0.9950814198570981,7.187315566876136);
				\draw [line width=2.pt] (2.533519268257322,8.69061186954817)-- (-1.,7.);
				\draw [line width=1.5pt] (0.6259970513931011,7.7779580939688024) -- (0.6859422251623214,8.014221034056757);
				\draw [line width=1.5pt] (0.6259970513931011,7.7779580939688024) -- (0.8475770430950026,7.676390835491413);
				\draw [line width=2.pt] (1.759627793635761,9.754712647152811)-- (-1.,7.);
				\draw [line width=1.5pt] (0.26937557864717276,8.267114706030467) -- (0.247523955762754,8.509882305381252);
				\draw [line width=1.5pt] (0.26937557864717276,8.267114706030467) -- (0.5121038378730068,8.244830341771557);
				\draw [line width=2.pt] (0.7535688766277328,10.567298695505446)-- (-1.,7.);
				\draw [line width=1.5pt] (-0.19205443219021545,8.643609902355214) -- (-0.2912628961631437,8.866255992357617);
				\draw [line width=1.5pt] (-0.19205443219021545,8.643609902355214) -- (0.044831772790876454,8.701042703147824);
				\draw [line width=2.pt] (-0.02032259799382734,10.85750799848853)-- (-1.,7.);
				\draw [line width=1.5pt] (-0.548571949657045,8.777510955518661) -- (-0.6916529514676375,8.97484678003642);
				\draw [line width=1.5pt] (-0.548571949657045,8.777510955518661) -- (-0.3286696465261894,8.882661218452107);
				\draw [line width=2.pt] (-0.7168249251532315,10.915549859085147)-- (-1.,7.);
				\draw [line width=1.5pt] (-0.8696682921555857,8.802137072129897) -- (-1.0451778914718264,8.971281925037339);
				\draw [line width=1.5pt] (-0.8696682921555857,8.802137072129897) -- (-0.6716470336814035,8.944267934047808);
				\draw [line width=2.pt] (-1.4520218260437137,10.915549859085147)-- (-1.,7.);
				\draw [line width=1.5pt] (-1.2081155770148475,8.802760113125833) -- (-1.4120286927219452,8.936300526334163);
				\draw [line width=1.5pt] (-1.2081155770148475,8.802760113125833) -- (-1.039993133321768,8.979249332750983);
				\draw [line width=2.pt] (-3.1158884964800677,10.373825826850057)-- (-1.,7.);
				\draw [line width=1.5pt] (-1.9750368045792581,8.554715363730812) -- (-2.216581307873093,8.587423981032098);
				\draw [line width=1.5pt] (-1.9750368045792581,8.554715363730812) -- (-1.8993071886069746,8.786401845817958);
				\begin{scriptsize}
					\draw [fill=qqqqff] (-1.4520218260437137,10.915549859085147) circle (3pt);
					\draw [fill=qqqqff] (-3.1158884964800677,10.373825826850057) circle (3pt);
					\draw [fill=qqqqff] (-4.470198577067798,9.077557606858948) circle (3pt);
					\draw [fill=qqqqff] (-3.6189179549840818,10.083616523866972) circle (3pt);
					\draw [fill=qqqqff] (-4.702366019454266,8.458444427161703) circle (3pt);
					\draw [fill=qqqqff] (-4.992575322437351,7.374996362691523) circle (3pt);
					\draw [fill=qqqqff] (-4.973228035571812,6.252853724490266) circle (3pt);
					\draw [fill=qqqqff] (-4.276725708412408,4.666376201516074) circle (3pt);
					\draw [fill=qqqqff] (2.533519268257322,5.033974651961314) circle (3pt);
					\draw [fill=qqqqff] (1.661554594802546,6.400870768344779) circle (3pt);
					\draw [fill=qqqqff] (2.069184383484386,4.414861472264069) circle (3pt);
					\draw [fill=qqqqff] (3.,7.) circle (3pt);
					\draw [fill=qqqqff] (-0.7168249251532315,10.915549859085147) circle (3pt);
					\draw [fill=qqqqff] (-0.7168249251532304,3.041204104810805) circle (3pt);
					\draw [fill=qqqqff] (-3.541528807521925,3.9118320137600566) circle (3pt);
					\draw [fill=qqqqff] (-0.02032259799382734,10.85750799848853) circle (3pt);
					\draw [fill=qqqqff] (2.920465005568102,7.742594813136765) circle (3pt);
					\draw [fill=qqqqff] (0.07641383633386878,3.1379405391384996) circle (3pt);
					\draw [fill=qqqqff] (-1.7666619570249664,4.628983561782251) circle (3pt);
					\draw [fill=qqqqff] (0.428937407628609,4.898618571476549) circle (3pt);
					\draw [fill=qqqqff] (1.759627793635761,9.754712647152811) circle (3pt);
					\draw [fill=qqqqff] (2.533519268257322,8.69061186954817) circle (3pt);
					\draw [fill=qqqqff] (-1.,7.) circle (3pt);
					\draw [fill=qqqqff] (0.7535688766277328,10.567298695505446) circle (3pt);
				\end{scriptsize}
			\end{tikzpicture}
			\caption{The representation of the directed power graph of $S_4$}
			\label{DPG_S_4}
		\end{figure}
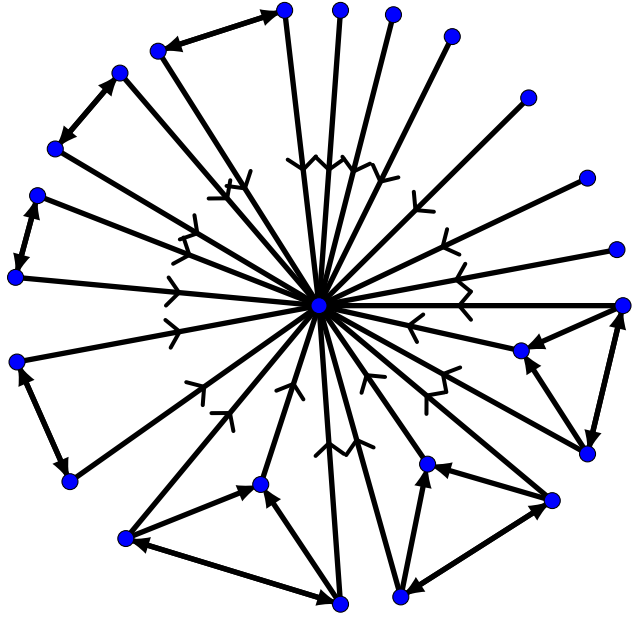
					
			\begin{definition}
				\label{Ndef}\rm Given two elements $x$ and $y$ of a group $G$ we write $x \mathtt{N}y$ if $N_{\mathcal{P}(G)}[x]=N_{\mathcal{P}(G)}[y]$
			\end{definition}
			
			$\mathtt{N}$ is clearly an equivalence relation. The $\mathtt{N}$-class of an element $x$ is denoted by $[x]_{\mathtt{N}}$. It follows immediately by the definition that two elements of the same $\mathtt{N}$-class are joined. Then, in particular, $[y]_{\mathtt{N}}\subseteq N_{\mathcal{P}(G)}[y]$ for all $y \in G$, and we also have that $\mathcal{P}(G)_{[y]_{\mathtt{N}}}$ is a complete graph.

			In Figure \ref{C_6xC_2Ncolored} there is an example of a power graph with the $\mathtt{N}$-classes highlighted by colouring the vertices of the same $\mathtt{N}$-class with the same colour. The figure represents the power graph of $C_6 \times C_2$. In the centre, in blue, there is the identity. \\ In Figure \ref{PPG_S_4}, where is represented $\mathcal{P}^*(S_4)$, the $\mathtt{N}$-classes are not highlighted with different colours, but they are easily found as the connected components of the graph. It is only missing the $\mathtt{N}$-class of $1$ that is composed by exactly the identity. The case of $S_4$ is not a general case because the proper power graph is not always the union of connected components that are the $\mathtt{N}$-classes. Indeed the proper power graph can be also connected, even if there exists more than one $\mathtt{N}$-class. Take as example the power graph of $C_6\times C_2$ in Figure \ref{C_6xC_2Ncolored}. The cyan vertex and the magentas are connected also in the proper power graph; then their $\mathtt{N}$-classes are not of two distinct components as happens for $\mathtt{N}$-classes in $\mathcal{P}^*(S_4)$. We remark again, after this examples, that the subgraph induced by an $\mathtt{N}$-class is always a complete graph. 
			
			\begin{figure}
				\centering
				\begin{tikzpicture}[line cap=round,line join=round,>=triangle 45,x=1.0cm,y=1.0cm]
					\clip(-2.,-2.) rectangle (6.,6.5);
					\draw [line width=1.pt] (-1.207135638670726,3.871916365273084)-- (2.005153099527752,2.4041290017516808);
					\draw [line width=1.pt] (-1.4958806937896902,1.8867941113302027)-- (2.005153099527752,2.4041290017516808);
					\draw [line width=1.pt] (-1.4958806937896902,1.8867941113302027)-- (-0.737924924102409,0.10619960476325441);
					\draw [line width=1.pt] (-1.4958806937896902,1.8867941113302027)-- (1.,-1.);
					\draw [line width=1.pt] (-1.4958806937896902,1.8867941113302027)-- (3.,-1.);
					\draw [line width=1.pt] (-1.4958806937896902,1.8867941113302027)-- (4.664013815414881,0.07010647287338384);
					\draw [line width=1.pt] (3.,-1.)-- (2.005153099527752,2.4041290017516808);
					\draw [line width=1.pt] (3.,-1.)-- (-0.737924924102409,0.10619960476325441);
					\draw [line width=1.pt] (4.664013815414881,0.07010647287338384)-- (2.005153099527752,2.4041290017516808);
					\draw [line width=1.pt] (4.664013815414881,0.07010647287338384)-- (-0.737924924102409,0.10619960476325441);
					\draw [line width=1.pt] (4.664013815414881,0.07010647287338384)-- (1.,-1.);
					\draw [line width=1.pt] (4.664013815414881,0.07010647287338384)-- (3.,-1.);
					\draw [line width=1.pt] (1.,-1.)-- (2.005153099527752,2.4041290017516808);
					\draw [line width=1.pt] (0.08018606540132314,5.363765816721068)-- (2.005153099527752,2.4041290017516808);
					\draw [line width=1.pt] (0.08018606540132314,5.363765816721068)-- (-0.737924924102409,0.10619960476325441);
					\draw [line width=1.pt] (0.08018606540132314,5.363765816721068)-- (3.89402700176431,5.375796860684358);
					\draw [line width=1.pt] (0.08018606540132314,5.363765816721068)-- (3.,-1.);
					\draw [line width=1.pt] (0.08018606540132314,5.363765816721068)-- (5.506186892845194,1.8988251552934927);
					\draw [line width=1.pt] (2.,6.)-- (2.005153099527752,2.4041290017516808);
					\draw [line width=1.pt] (2.,6.)-- (3.,-1.);
					\draw [line width=1.pt] (2.,6.)-- (-1.207135638670726,3.871916365273084);
					\draw [line width=1.pt] (2.,6.)-- (-0.737924924102409,0.10619960476325441);
					\draw [line width=1.pt] (2.,6.)-- (5.22947288168952,3.8959784531996644);
					\draw [line width=1.pt] (3.89402700176431,5.375796860684358)-- (2.005153099527752,2.4041290017516808);
					\draw [line width=1.pt] (5.22947288168952,3.8959784531996644)-- (-0.737924924102409,0.10619960476325441);
					\draw [line width=1.pt] (5.22947288168952,3.8959784531996644)-- (-1.207135638670726,3.871916365273084);
					\draw [line width=1.pt] (5.22947288168952,3.8959784531996644)-- (3.,-1.);
					\draw [line width=1.pt] (5.506186892845194,1.8988251552934927)-- (2.005153099527752,2.4041290017516808);
					\draw [line width=1.pt] (5.506186892845194,1.8988251552934927)-- (3.,-1.);
					\draw [line width=1.pt] (5.506186892845194,1.8988251552934927)-- (3.89402700176431,5.375796860684358);
					\draw [line width=1.pt] (5.506186892845194,1.8988251552934927)-- (-0.737924924102409,0.10619960476325441);
					\draw [line width=1.pt] (2.005153099527752,2.4041290017516808)-- (-0.737924924102409,0.10619960476325441);
					\draw[line width=1.pt]					(2.005153099527752,2.4041290017516808)--
					(5.22947288168952,3.8959784531996644);
					\begin{scriptsize}
						\draw [fill=qqffff] (-1.207135638670726,3.871916365273084) circle (3pt);
						\draw [fill=ffffww] (0.08018606540132314,5.363765816721068) circle (3pt);
						\draw [fill=ffqqff] (2.,6.) circle (3pt);
						\draw [fill=yqqqqq] (3.89402700176431,5.375796860684358) circle (3pt);
						\draw [fill=ffqqff] (5.22947288168952,3.8959784531996644) circle (3pt);
						\draw [fill=qqffqq] (-1.4958806937896902,1.8867941113302027) circle (3pt);
						\draw [fill=ffqqqq] (-0.737924924102409,0.10619960476325441) circle (3pt);
						\draw [fill=ffzztt] (1.,-1.) circle (3pt);
						\draw [fill=ffqqqq] (3.,-1.) circle (3pt);
						\draw [fill=qqffqq] (4.664013815414881,0.07010647287338384) circle (3pt);
						\draw [fill=ffffww] (5.506186892845194,1.8988251552934927) circle (3pt);
						\draw [fill=qqqqff] (2.005153099527752,2.4041290017516808) circle (3pt);
					\end{scriptsize}
				\end{tikzpicture}
				\caption{The power graph of $C_6 \times C_2$ where vertex of the same colour are in the same $\mathtt{N}$-class. }
				\label{C_6xC_2Ncolored}
			\end{figure}
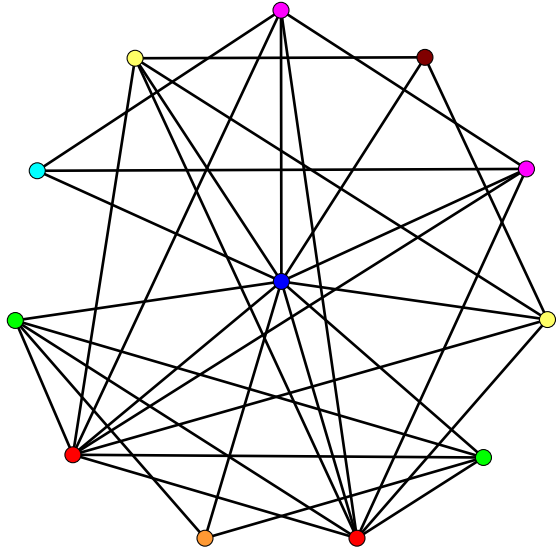

			Note that in the power graph the elements of the same $\mathtt{N}$-class are indistinguishable in the following sense: \emph{all the maps that permutate the vertices of an $\mathtt{N}$-class and fix all the others vertices are graph isomorphism for $\mathcal{P}(G)$.}
			In fact, if  we have $x,y\in G$ such that $x\mathtt{N}y$ holds, then the map $\psi : G \rightarrow G$, defined by $\psi(x)=y$, $\psi(y)=x$ and $\psi(z)=z$ for all $z \in G\setminus \left\lbrace x,y\right\rbrace $, is a graph automorphism for $\mathcal{P}(G)$:
			
			Let $v,w$ be two distinct vertices of $\mathcal{P}(G)$ with $\left\lbrace v,w\right\rbrace \in E$. 
			We distinguish two cases:
			\begin{enumerate}
				\item \underline{Both $v,w$ are neither $x$ nor $y$}. Then $\left\lbrace \psi(v),\psi(w) \right\rbrace=\left\lbrace v,w\right\rbrace \in E$.
				
				\item \underline{At least one of $v,w$ is equal to $x$ or $y$}. Up to renaming, if necessary, we have $\left\lbrace x,w\right\rbrace \in E$, that implies $w\in N_{\mathcal{P}(G)}[x]=N_{\mathcal{P}(G)}[y]$. Now, if $w=y$, then $\psi(w)=\psi(y)=x$, so $\left\lbrace \psi(x), \psi(w) \right\rbrace= \left\lbrace y,x\right\rbrace  $ that is in $E$ since $x\mathtt{N}y$ holds. Otherwise, if $w\not=y$, then $\left\lbrace \psi(x), \psi(w) \right\rbrace=\left\lbrace y,w\right\rbrace $ that is in $E$ since we have $w \in N_{\mathcal{P}(G)}[y]$.
			\end{enumerate}
			
			This proves that $\psi$ is a graph isomorphism since it is clearly a bijection from $G$ to $G$.
			Note that the graph isomorphism explained above is not induced by a group isomorphism. So it is an example of the fact that \emph{not all the graph isomorphisms between two power graphs are induced by a group isomorphism} (see Remark \ref{isoGroup-isoGraph} for the converse).
			
			Speaking of isomorphisms is crucial to underline the following result that will be used later in Section \ref{Case I}.
			\begin{lemma}
				\label{Niso}Let be $G_1$ and $G_2$ groups, $x \in G_1$ and $\psi$ a graph isomorphism. If $\psi(\mathcal{P}(G_1))=\mathcal{P}(G_2)$, then all the following are true:
				\begin{enumerate}
					\item[\rm i)] $N_{\mathcal{P}(G_2)}[\psi(x)]=\psi(N_{\mathcal{P}(G_1)}[x])$, in particular $|N_{\mathcal{P}(G_2)}[\psi(x)]|=|N_{\mathcal{P}(G_1)}[x]|$;
					
					\item[\rm ii)] $[\psi(x)]_{\mathtt{N}}=\psi([x]_{\mathtt{N}})$, and, in particular we have that $\mathcal{P}(G)_{[x]_{\mathtt{N}}}\cong \mathcal{P}(G)_{[\psi(x)]_{\mathtt{N}}}$;
					
					\item[\rm iii)] $G_1$ and $G_2$ has the same amount of $\mathtt{N}$-classes.
				\end{enumerate}  
			\end{lemma}
			
			\begin{proof}
				\begin{enumerate}
					\item[\rm i)] Let $y'\in N_{\mathcal{P}(G_2)}[\psi(x)]$. Since $\psi$ is a graph isomorphism, there exists $y \in G$ such that $\psi(y)=y'$. Now we have that $\{\psi(y), \psi(x)\}\in E_{\mathcal{P}(G_2)}$, this happens if and only if $\{y,x\}\in E_{\mathcal{P}(G_1)}$ holds that is equivalent to say that $y\in N_{\mathcal{P}(G_1)}[x]$. Therefore $N_{\mathcal{P}(G_2)}[\psi(x)]=\psi(N_{\mathcal{P}(G_1)}[x])$.
					
					\item[\rm ii)] We want to prove that $[\psi(x)]_{\mathtt{N}}=\psi([x]_{\mathtt{N}})$. If this holds, then $|[x]_{\mathtt{N}}|=|[\psi(x)]_{\mathtt{N}}|$, and hence $\mathcal{P}(G)_{[x]_{\mathtt{N}}}\cong \mathcal{P}(G)_{[\psi(x)]_{\mathtt{N}}}$ because are complete graphs of the same order.
					So take $y' \in [\psi(x)]_{\mathtt{N}}$. Since $\psi$ is a bijection, there exists $y\in G_1$ such that $y'=\psi(y)$. We want to prove that $y \in [x]_{\mathtt{N}}$. Indeed we have, by i), that  $$\psi(N_{\mathcal{P}(G_1)}[y])=N_{\mathcal{P}(G_2)}[\psi(y)]=N_{\mathcal{P}(G_2)}[\psi(x)]=\psi(N_{\mathcal{P}(G_1)}[x]),$$ and hence, since $\psi$ is injective, we obtain that $N_{\mathcal{P}(G_1)}[y]=N_{\mathcal{P}(G_1)}[x]$. It follows that $y \in [x]_{\mathtt{N}}$. 
					For the converse we simply follow the argument above backwards. 
					
					\item[\rm iii)] Let $k$ be a positive integer and let $x_1, x_2, ..., x_k$ be a set of representatives for the $\mathtt{N}$-classes of $G_1$. By ii), we have that $\psi(x_1), \psi(x_2), ..., \psi(x_k)$ are a set of representatives for the  $\mathtt{N}$-classes of $G_2$. Indeed suppose, by contradiction, that there exists $y'\in G_2$ such that $[y']_{\mathtt{N}}\not= [\psi(x_i)]_{\mathtt{N}}$ for all $i\in [k]$. There exists $y \in G_1$ such that $\psi(y)=y'$, and there exists $i \in [k]$ such that $y \in [x_i]_{\mathtt{N}}$. Therefore $$[y']_{\mathtt{N}}=[\psi(y)]_{\mathtt{N}}=\psi([y]_{\mathtt{N}})=\psi([x_i]_{\mathtt{N}})=[\psi(x_i)]_{\mathtt{N}},$$ a contradiction. Note also that $[\psi(x_i)]_{\mathtt{N}}\not=[\psi(x_j)]_{\mathtt{N}}$ for all $i\not=j$ in $[k]$ otherwise we would have $\psi([x_i]_{\mathtt{N}})=\psi([x_j]_{\mathtt{N}})$ and hence, since $\psi$ is injective, we would have $[x_i]_{\mathtt{N}}=[x_j]_{\mathtt{N}}$, a contradiction.
				\end{enumerate}
			\end{proof} 
			
			Before continuing with the study of $\mathtt{N}$-classes note that it is easy to understand how the closed neighbourhood of an element behaves whenever we move to a subgroup.
			Let $G$ be a group. If $H\leq G$, then, for all $h\in H$, we have that $$N_{\mathcal{P}(H)}[h]=H \cap N_{\mathcal{P}(G)}[h].$$ Indeed we have that if $x\in H$ and $x\in N_{\mathcal{P}(G)}[h]$ hold, then it follows clearly that $x\in N_{\mathcal{P}(H)}[h]$; and if we have $x \in N_{\mathcal{P}(H)}[h]$, then we have $x \in H$ and $\{x,h\}\in E_{\mathcal{P}(H)}\subseteq E_{\mathcal{P}(G)}$, and hence $x\in H \cap N_{\mathcal{P}(G)}[h]$ holds.
			
			From now on we always refer to $N_{\mathcal{P}(G)}[x]$ with $N[x]$ if the group is clear from the context.
				
			Note that the $\mathtt{N}$-class of the identity is composed by all the star vertices, since 1 itself is a star vertex. Then $[1]_{\mathtt{N}}=\mathcal{S}_{\mathcal{P}(G)}$. So we call the class of the identity the \emph{star class} and we simply write $\mathcal{S}$ whenever there is no possible misunderstanding on which group we are referring to.
			
			\begin{Oss}
				\label{Nclassi-centralizzante}Let $y \in G$. Then $N[y] \subseteq C_G(y)$ holds. We also have that $[y]_{\mathtt{N}}\subseteq C_G(y)$.  
			\end{Oss}
		 	\begin{proof}
		 		Let $x\in N[y]$. Then we have that $x \in \langle y \rangle$ or $y \in \langle x \rangle$. So, in both cases, $x\in C_G(y)$. Now remember that we have $[y]_{\mathtt{N}}\subseteq N[y]$, thus we also have $[y]_{\mathtt{N}}\subseteq C_G(y)$.
		 	\end{proof}

			\begin{Oss}
				\label{DiamondRefinementOfN} The relation $\diamond$ is a refinement of the relation $\mathtt{N}$. In particular each $\mathtt{N}$-class is union of $\diamond$-classes.
			\end{Oss}
			\begin{proof}
				We need to show that $x\diamond y$ implies $x\mathtt{N}y$.
				So let $x,y \in G$ such that $x\diamond y$.
				We have $x=y^k$ and $y=x^s$ for some $k,s \in \mathbb{N}$ such that $\gcd(k,o(y))=1$ and $\gcd(s,o(x))=1$. Let $z\in N[x]$. Then there exists a positive integer $n$ such that $x=z^n$ or $z=x^n$. In the first case $y=x^s=z^{ns}$; in the other case $z=x^n=y^{kn}$.
				In any case we get $z\in N[y]$, thus $N[x]\subseteq N[y]$.
				In the same way, starting with $z\in N[y]$ we obtain $z\in N[x]$, and then $N[x]=N[y]$.
			\end{proof}
		
			For example, in Figure \ref{PPG_S_4}, that is $\mathcal{P}^*(S_4)$, the elements of a triangle form an $\mathtt{N}$-class. But we know that these elements are of order $2$ and $4$. So each of these three $\mathtt{N}$-classes is union of two $\diamond$-classes: the class of the only element of order $2$ and the class of the $2$ remaining elements, that have order $4$.
			\emph{So, in general, in an $\mathtt{N}$-class is possible to find elements of different orders.}
			Note that, since we know how many elements of each order are in those specific $\mathtt{N}$-classes, is up to us the division in $\diamond$-classes since, as we already stated, the elements in each $\mathtt{N}$-class are indistinguishable: In each triangle we can choose any couple of vertices as the two elements of order $4$. 
			
			In Figure \ref{S_4NeDiamond} we represent $\vec{\mathcal{P}}(S_4)$ with both relations $\mathtt{N}$ and $\diamond$ highlighted. The choice of the colours is not entirely random. The white vertex is the identity. The ones in shades of green are all the elements of order $3$. The vertices in shades of red/pink are all involutions. And, finally, the vertices coloured in a yellow shade composed, with the identity, the three $4$-cycles of $S_4$.
		
			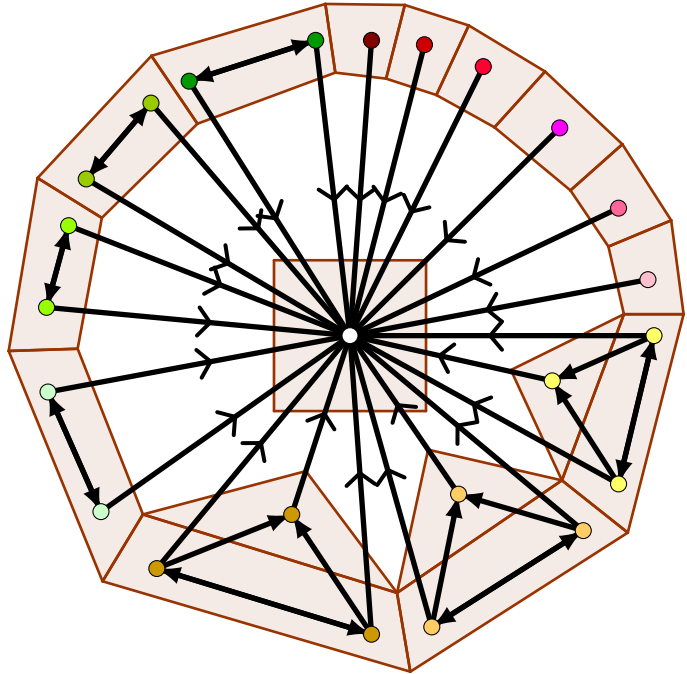
\begin{figure}
				\centering
				\begin{tikzpicture}[line cap=round,line join=round,>=triangle 45,x=1.0cm,y=1.0cm]
					\clip(-6.,2.) rectangle (4.,12.);
					\fill[line width=2.pt,color=zzttqq,fill=zzttqq,fill opacity=0.10000000149011612] (-3.00711852963106,9.809911400437345) -- (-1.1933294275569224,10.487257783623239) -- (-1.3240850694054276,11.397934907263139) -- (-3.6127226668446086,10.710096634610261) -- cycle;
					\fill[line width=2.pt,color=zzttqq,fill=zzttqq,fill opacity=0.10000000149011612] (-5.110910354913101,9.090550324695746) -- (-4.265322159020893,8.56129342866397) -- (-3.00711852963106,9.809911400437345) -- (-3.6127226668446086,10.710096634610261) -- cycle;
					\fill[line width=2.pt,color=zzttqq,fill=zzttqq,fill opacity=0.10000000149011612] (-5.110910354913101,9.090550324695746) -- (-5.490950168512438,6.796738592614046) -- (-4.58366337471664,6.826333803122147) -- (-4.265322159020893,8.56129342866397) -- cycle;
					\fill[line width=2.pt,color=zzttqq,fill=zzttqq,fill opacity=0.10000000149011612] (-3.7241420923381203,4.629779414821492) -- (-4.255820774314594,3.7428472333336757) -- (-5.490950168512438,6.796738592614046) -- (-4.58366337471664,6.826333803122147) -- cycle;
					\fill[line width=2.pt,color=zzttqq,fill=zzttqq,fill opacity=0.10000000149011612] (2.604291484014996,7.284127314600211) -- (2.400341814808526,8.181505859108677) -- (3.22633797509473,8.528220296759674) -- (3.389497710459906,7.284127314600211) -- cycle;
					\fill[line width=2.pt,color=zzttqq,fill=zzttqq,fill opacity=0.10000000149011612] (2.400341814808526,8.181505859108677) -- (1.9006651252526743,8.92592215171229) -- (2.5838965170943493,9.5377711593317) -- (3.22633797509473,8.528220296759674) -- cycle;
					\fill[line width=2.pt,color=zzttqq,fill=zzttqq,fill opacity=0.10000000149011612] (1.9006651252526743,8.92592215171229) -- (0.9013117461409711,9.762115795458815) -- (1.5669320840466672,10.502126775207563) -- (2.5838965170943493,9.5377711593317) -- cycle;
					\fill[line width=2.pt,color=zzttqq,fill=zzttqq,fill opacity=0.10000000149011612] (0.9013117461409711,9.762115795458815) -- (0.13650048661670827,10.1904101007924) -- (0.5625411481055639,11.112905047063636) -- (1.5669320840466672,10.502126775207563) -- cycle;
					\fill[line width=2.pt,color=zzttqq,fill=zzttqq,fill opacity=0.10000000149011612] (0.13650048661670827,10.1904101007924) -- (-0.5248128745958504,10.407672479699302) -- (-0.2789755820072525,11.384362056777448) -- (0.5625411481055639,11.112905047063636) -- cycle;
					\fill[line width=2.pt,color=zzttqq,fill=zzttqq,fill opacity=0.10000000149011612] (-0.2789755820072525,11.384362056777448) -- (-1.3240850694054276,11.397934907263139) -- (-1.1933294275569224,10.487257783623239) -- (-0.5248128745958504,10.407672479699302) -- cycle;
					\fill[line width=2.pt,color=zzttqq,fill=zzttqq,fill opacity=0.10000000149011612] (-2.,8.) -- (8.905437690318596E-4,7.999107867261943) -- (0.,6.) -- (-2.,6.) -- cycle;
					\fill[line width=2.pt,color=zzttqq,fill=zzttqq,fill opacity=0.10000000149011612] (-4.255820774314594,3.7428472333336757) -- (-0.20647165890010305,2.544644452014348) -- (-0.3808547962166728,3.589809124493197) -- (-3.7241420923381203,4.629779414821492) -- cycle;
					\fill[line width=2.pt,color=zzttqq,fill=zzttqq,fill opacity=0.10000000149011612] (-0.3808547962166728,3.589809124493197) -- (1.7884928071891157,5.08147088717034) -- (2.655278901316614,4.388042011868344) -- (-0.20647165890010305,2.544644452014348) -- cycle;
					\fill[line width=2.pt,color=zzttqq,fill=zzttqq,fill opacity=0.10000000149011612] (1.7884928071891157,5.08147088717034) -- (2.604291484014996,7.284127314600211) -- (3.389497710459906,7.284127314600211) -- (2.655278901316614,4.388042011868344) -- cycle;
					\fill[line width=2.pt,color=zzttqq,fill=zzttqq,fill opacity=0.10000000149011612] (2.604291484014996,7.284127314600211) -- (1.1171708774262428,6.54424864751115) -- (1.7884928071891157,5.08147088717034) -- cycle;
					\fill[line width=2.pt,color=zzttqq,fill=zzttqq,fill opacity=0.10000000149011612] (1.7884928071891157,5.08147088717034) -- (0.027873487448356438,5.488419264647737) -- (-0.3808547962166728,3.589809124493197) -- cycle;
					\fill[line width=2.pt,color=zzttqq,fill=zzttqq,fill opacity=0.10000000149011612] (-0.3808547962166728,3.589809124493197) -- (-1.580788935019527,5.203823602504077) -- (-3.7241420923381203,4.629779414821492) -- cycle;
					\draw [line width=1.pt,color=zzttqq] (-3.00711852963106,9.809911400437345)-- (-1.1933294275569224,10.487257783623239);
					\draw [line width=1.pt,color=zzttqq] (-1.1933294275569224,10.487257783623239)-- (-1.3240850694054276,11.397934907263139);
					\draw [line width=1.pt,color=zzttqq] (-1.3240850694054276,11.397934907263139)-- (-3.6127226668446086,10.710096634610261);
					\draw [line width=1.pt,color=zzttqq] (-3.6127226668446086,10.710096634610261)-- (-3.00711852963106,9.809911400437345);
					\draw [line width=1.pt,color=zzttqq] (-5.110910354913101,9.090550324695746)-- (-4.265322159020893,8.56129342866397);
					\draw [line width=1.pt,color=zzttqq] (-4.265322159020893,8.56129342866397)-- (-3.00711852963106,9.809911400437345);
					\draw [line width=1.pt,color=zzttqq] (-3.00711852963106,9.809911400437345)-- (-3.6127226668446086,10.710096634610261);
					\draw [line width=1.pt,color=zzttqq] (-3.6127226668446086,10.710096634610261)-- (-5.110910354913101,9.090550324695746);
					\draw [line width=1.pt,color=zzttqq] (-5.110910354913101,9.090550324695746)-- (-5.490950168512438,6.796738592614046);
					\draw [line width=1.pt,color=zzttqq] (-5.490950168512438,6.796738592614046)-- (-4.58366337471664,6.826333803122147);
					\draw [line width=1.pt,color=zzttqq] (-4.58366337471664,6.826333803122147)-- (-4.265322159020893,8.56129342866397);
					\draw [line width=1.pt,color=zzttqq] (-4.265322159020893,8.56129342866397)-- (-5.110910354913101,9.090550324695746);
					\draw [line width=1.pt,color=zzttqq] (-3.7241420923381203,4.629779414821492)-- (-4.255820774314594,3.7428472333336757);
					\draw [line width=1.pt,color=zzttqq] (-4.255820774314594,3.7428472333336757)-- (-5.490950168512438,6.796738592614046);
					\draw [line width=1.pt,color=zzttqq] (-5.490950168512438,6.796738592614046)-- (-4.58366337471664,6.826333803122147);
					\draw [line width=1.pt,color=zzttqq] (-4.58366337471664,6.826333803122147)-- (-3.7241420923381203,4.629779414821492);
					\draw [line width=1.pt,color=zzttqq] (2.604291484014996,7.284127314600211)-- (2.400341814808526,8.181505859108677);
					\draw [line width=1.pt,color=zzttqq] (2.400341814808526,8.181505859108677)-- (3.22633797509473,8.528220296759674);
					\draw [line width=1.pt,color=zzttqq] (3.22633797509473,8.528220296759674)-- (3.389497710459906,7.284127314600211);
					\draw [line width=1.pt,color=zzttqq] (3.389497710459906,7.284127314600211)-- (2.604291484014996,7.284127314600211);
					\draw [line width=1.pt,color=zzttqq] (2.400341814808526,8.181505859108677)-- (1.9006651252526743,8.92592215171229);
					\draw [line width=1.pt,color=zzttqq] (1.9006651252526743,8.92592215171229)-- (2.5838965170943493,9.5377711593317);
					\draw [line width=1.pt,color=zzttqq] (2.5838965170943493,9.5377711593317)-- (3.22633797509473,8.528220296759674);
					\draw [line width=1.pt,color=zzttqq] (3.22633797509473,8.528220296759674)-- (2.400341814808526,8.181505859108677);
					\draw [line width=1.pt,color=zzttqq] (1.9006651252526743,8.92592215171229)-- (0.9013117461409711,9.762115795458815);
					\draw [line width=1.pt,color=zzttqq] (0.9013117461409711,9.762115795458815)-- (1.5669320840466672,10.502126775207563);
					\draw [line width=1.pt,color=zzttqq] (1.5669320840466672,10.502126775207563)-- (2.5838965170943493,9.5377711593317);
					\draw [line width=1.pt,color=zzttqq] (2.5838965170943493,9.5377711593317)-- (1.9006651252526743,8.92592215171229);
					\draw [line width=1.pt,color=zzttqq] (0.9013117461409711,9.762115795458815)-- (0.13650048661670827,10.1904101007924);
					\draw [line width=1.pt,color=zzttqq] (0.13650048661670827,10.1904101007924)-- (0.5625411481055639,11.112905047063636);
					\draw [line width=1.pt,color=zzttqq] (0.5625411481055639,11.112905047063636)-- (1.5669320840466672,10.502126775207563);
					\draw [line width=1.pt,color=zzttqq] (1.5669320840466672,10.502126775207563)-- (0.9013117461409711,9.762115795458815);
					\draw [line width=1.pt,color=zzttqq] (0.13650048661670827,10.1904101007924)-- (-0.5248128745958504,10.407672479699302);
					\draw [line width=1.pt,color=zzttqq] (-0.5248128745958504,10.407672479699302)-- (-0.2789755820072525,11.384362056777448);
					\draw [line width=1.pt,color=zzttqq] (-0.2789755820072525,11.384362056777448)-- (0.5625411481055639,11.112905047063636);
					\draw [line width=1.pt,color=zzttqq] (0.5625411481055639,11.112905047063636)-- (0.13650048661670827,10.1904101007924);
					\draw [line width=1.pt,color=zzttqq] (-0.2789755820072525,11.384362056777448)-- (-1.3240850694054276,11.397934907263139);
					\draw [line width=1.pt,color=zzttqq] (-1.3240850694054276,11.397934907263139)-- (-1.1933294275569224,10.487257783623239);
					\draw [line width=1.pt,color=zzttqq] (-1.1933294275569224,10.487257783623239)-- (-0.5248128745958504,10.407672479699302);
					\draw [line width=1.pt,color=zzttqq] (-0.5248128745958504,10.407672479699302)-- (-0.2789755820072525,11.384362056777448);
					\draw [line width=1.pt,color=zzttqq] (-2.,8.)-- (8.905437690318596E-4,7.999107867261943);
					\draw [line width=1.pt,color=zzttqq] (8.905437690318596E-4,7.999107867261943)-- (0.,6.);
					\draw [line width=1.pt,color=zzttqq] (0.,6.)-- (-2.,6.);
					\draw [line width=1.pt,color=zzttqq] (-2.,6.)-- (-2.,8.);
					\draw [line width=1.pt,color=zzttqq] (-4.255820774314594,3.7428472333336757)-- (-0.20647165890010305,2.544644452014348);
					\draw [line width=1.pt,color=zzttqq] (-0.20647165890010305,2.544644452014348)-- (-0.3808547962166728,3.589809124493197);
					\draw [line width=1.pt,color=zzttqq] (-0.3808547962166728,3.589809124493197)-- (-3.7241420923381203,4.629779414821492);
					\draw [line width=1.pt,color=zzttqq] (-3.7241420923381203,4.629779414821492)-- (-4.255820774314594,3.7428472333336757);
					\draw [line width=1.pt,color=zzttqq] (-0.3808547962166728,3.589809124493197)-- (1.7884928071891157,5.08147088717034);
					\draw [line width=1.pt,color=zzttqq] (1.7884928071891157,5.08147088717034)-- (2.655278901316614,4.388042011868344);
					\draw [line width=1.pt,color=zzttqq] (2.655278901316614,4.388042011868344)-- (-0.20647165890010305,2.544644452014348);
					\draw [line width=1.pt,color=zzttqq] (-0.20647165890010305,2.544644452014348)-- (-0.3808547962166728,3.589809124493197);
					\draw [line width=1.pt,color=zzttqq] (1.7884928071891157,5.08147088717034)-- (2.604291484014996,7.284127314600211);
					\draw [line width=1.pt,color=zzttqq] (2.604291484014996,7.284127314600211)-- (3.389497710459906,7.284127314600211);
					\draw [line width=1.pt,color=zzttqq] (3.389497710459906,7.284127314600211)-- (2.655278901316614,4.388042011868344);
					\draw [line width=1.pt,color=zzttqq] (2.655278901316614,4.388042011868344)-- (1.7884928071891157,5.08147088717034);
					\draw [line width=1.pt,color=zzttqq] (2.604291484014996,7.284127314600211)-- (1.1171708774262428,6.54424864751115);
					\draw [line width=1.pt,color=zzttqq] (1.1171708774262428,6.54424864751115)-- (1.7884928071891157,5.08147088717034);
					\draw [line width=1.pt,color=zzttqq] (1.7884928071891157,5.08147088717034)-- (2.604291484014996,7.284127314600211);
					\draw [line width=1.pt,color=zzttqq] (1.7884928071891157,5.08147088717034)-- (0.027873487448356438,5.488419264647737);
					\draw [line width=1.pt,color=zzttqq] (0.027873487448356438,5.488419264647737)-- (-0.3808547962166728,3.589809124493197);
					\draw [line width=1.pt,color=zzttqq] (-0.3808547962166728,3.589809124493197)-- (1.7884928071891157,5.08147088717034);
					\draw [line width=1.pt,color=zzttqq] (-0.3808547962166728,3.589809124493197)-- (-1.580788935019527,5.203823602504077);
					\draw [line width=1.pt,color=zzttqq] (-1.580788935019527,5.203823602504077)-- (-3.7241420923381203,4.629779414821492);
					\draw [line width=1.pt,color=zzttqq] (-3.7241420923381203,4.629779414821492)-- (-0.3808547962166728,3.589809124493197);
					\draw [-latex,line width=2.pt] (-3.1158884964800677,10.373825826850057) -- (-1.4520218260437137,10.915549859085147);
					\draw [-latex,line width=2.pt] (-1.4520218260437137,10.915549859085147) -- (-3.1158884964800677,10.373825826850057);
					\draw [-latex,line width=2.pt] (-3.6189179549840818,10.083616523866972) -- (-4.470198577067798,9.077557606858948);
					\draw [-latex,line width=2.pt] (-4.470198577067798,9.077557606858948) -- (-3.6189179549840818,10.083616523866972);
					\draw [-latex,line width=2.pt] (-4.702366019454266,8.458444427161703) -- (-4.992575322437351,7.374996362691523);
					\draw [-latex,line width=2.pt] (-4.992575322437351,7.374996362691523) -- (-4.702366019454266,8.458444427161703);
					\draw [-latex,line width=2.pt] (-4.973228035571812,6.252853724490266) -- (-4.276725708412408,4.666376201516074);
					\draw [-latex,line width=2.pt] (-4.276725708412408,4.666376201516074) -- (-4.973228035571812,6.252853724490266);
					\draw [-latex,line width=2.pt] (-3.541528807521925,3.9118320137600566) -- (-1.7666619570249664,4.628983561782251);
					\draw [-latex,line width=2.pt] (-0.7168249251532304,3.041204104810805) -- (-1.7666619570249664,4.628983561782251);
					\draw [-latex,line width=2.pt] (-3.541528807521925,3.9118320137600566) -- (-0.7168249251532304,3.041204104810805);
					\draw [-latex,line width=2.pt] (-0.7168249251532304,3.041204104810805) -- (-3.541528807521925,3.9118320137600566);
					\draw [-latex,line width=2.pt] (0.07641383633386878,3.1379405391384996) -- (0.428937407628609,4.898618571476549);
					\draw [-latex,line width=2.pt] (2.069184383484386,4.414861472264069) -- (0.428937407628609,4.898618571476549);
					\draw [-latex,line width=2.pt] (0.07641383633386878,3.1379405391384996) -- (2.069184383484386,4.414861472264069);
					\draw [-latex,line width=2.pt] (2.069184383484386,4.414861472264069) -- (0.07641383633386889,3.1379405391384996);
					\draw [-latex,line width=2.pt] (2.533519268257322,5.033974651961314) -- (1.661554594802546,6.400870768344779);
					\draw [-latex,line width=2.pt] (2.533519268257322,5.033974651961314) -- (3.,7.);
					\draw [-latex,line width=2.pt] (3.,7.) -- (1.661554594802546,6.400870768344779);
					\draw [-latex,line width=2.pt] (3.,7.) -- (2.533519268257322,5.033974651961314);
					\draw [line width=2.pt] (-3.6189179549840818,10.083616523866972)-- (-1.,7.);
					\draw [line width=1.5pt] (-2.2084453947227316,8.422870915167943) -- (-2.4521837936106934,8.420591962610315);
					\draw [line width=1.5pt] (-2.2084453947227316,8.422870915167943) -- (-2.166734161373389,8.663024561256657);
					\draw [line width=2.pt] (-4.470198577067798,9.077557606858948)-- (-1.,7.);
					\draw [line width=1.5pt] (-2.601214790057432,7.9586240941029205) -- (-2.8312849397257627,7.878117405257711);
					\draw [line width=1.5pt] (-2.601214790057432,7.9586240941029205) -- (-2.6389136373420383,8.199440201601233);
					\draw [line width=2.pt] (-4.702366019454266,8.458444427161703)-- (-1.,7.);
					\draw [line width=1.5pt] (-2.7059971974046944,7.672030288803024) -- (-2.9198133194605265,7.554999238793923);
					\draw [line width=1.5pt] (-2.7059971974046944,7.672030288803024) -- (-2.782552699993743,7.903445188367777);
					\draw [line width=2.pt] (-4.992575322437351,7.374996362691523)-- (-1.,7.);
					\draw [line width=1.5pt] (-2.8409270808149505,7.172906182985797) -- (-3.013798059250633,7.001065484861294);
					\draw [line width=1.5pt] (-2.8409270808149505,7.172906182985797) -- (-2.9787772631867178,7.37393087783023);
					\draw [line width=2.pt] (-4.973228035571812,6.252853724490266)-- (-1.,7.);
					\draw [line width=1.5pt] (-2.8332575394987596,6.655264804732126) -- (-2.952008486801514,6.442399088300557);
					\draw [line width=1.5pt] (-2.8332575394987596,6.655264804732126) -- (-3.0212195487702993,6.810454636189707);
					\draw [line width=2.pt] (-4.276725708412408,4.666376201516074)-- (-1.,7.);
					\draw [line width=1.5pt] (-2.511257981778891,5.923709853750097) -- (-2.5297367506157316,5.68066225384526);
					\draw [line width=1.5pt] (-2.511257981778891,5.923709853750097) -- (-2.7469889577966784,5.985713947670812);
					\draw [line width=2.pt] (-3.541528807521925,3.9118320137600566)-- (-1.,7.);
					\draw [line width=1.5pt] (-2.1716046765011137,5.576403287741003) -- (-2.126179666727793,5.336924334168208);
					\draw [line width=1.5pt] (-2.1716046765011137,5.576403287741003) -- (-2.4153491407941337,5.574907679591847);
					\draw [line width=2.pt] (-1.7666619570249664,4.628983561782251)-- (-1.,7.);
					\draw [line width=1.5pt] (-1.335321831690352,5.962967227281868) -- (-1.2051604428435905,5.7568808047045685);
					\draw [line width=1.5pt] (-1.335321831690352,5.962967227281868) -- (-1.5615015141813748,5.872102757077682);
					\draw [line width=2.pt] (-0.7168249251532304,3.041204104810805)-- (-1.,7.);
					\draw [line width=1.5pt] (-0.8695459624196558,5.176248708459312) -- (-0.6716364753119225,5.03396225221705);
					\draw [line width=1.5pt] (-0.8695459624196558,5.176248708459312) -- (-1.0451884498413087,5.007241852593754);
					\draw [line width=2.pt] (0.07641383633386878,3.1379405391384996)-- (-1.,7.);
					\draw [line width=1.5pt] (-0.503688161070532,5.219285400860831) -- (-0.2814149242831663,5.119244364654209);
					\draw [line width=1.5pt] (-0.503688161070532,5.219285400860831) -- (-0.6421712393829652,5.0186961744842895);
					\draw [line width=2.pt] (0.428937407628609,4.898618571476549)-- (-1.,7.);
					\draw [line width=1.5pt] (-0.37327651045814075,6.078346542884425) -- (-0.13068658761031388,6.054603542865207);
					\draw [line width=1.5pt] (-0.37327651045814075,6.078346542884425) -- (-0.4403760047610772,5.844015028611341);
					\draw [line width=2.pt] (2.069184383484386,4.414861472264069)-- (-1.,7.);
					\draw [line width=1.5pt] (0.41524290018099735,5.807957264786713) -- (0.6552240261278082,5.850649886005468);
					\draw [line width=1.5pt] (0.41524290018099735,5.807957264786713) -- (0.4139603573565776,5.5642115862586);
					\draw [line width=2.pt] (2.533519268257322,5.033974651961314)-- (-1.,7.);
					\draw [line width=1.5pt] (0.6304008434492131,6.092856401129611) -- (0.857802524307408,6.180617874795993);
					\draw [line width=1.5pt] (0.6304008434492131,6.092856401129611) -- (0.6757167439499148,5.853356777165319);
					\draw [line width=2.pt] (1.661554594802546,6.400870768344779)-- (-1.,7.);
					\draw [line width=1.5pt] (0.1785423397538285,6.734704233432324) -- (0.3718999165131968,6.88311733334932);
					\draw [line width=1.5pt] (0.1785423397538285,6.734704233432324) -- (0.2896546782893492,6.517753434995456);
					\draw [line width=2.pt] (3.,7.)-- (-1.,7.);
					\draw [line width=1.5pt] (0.8439556589955292,7.) -- (1.,7.1872532092053625);
					\draw [line width=1.5pt] (0.8439556589955292,7.) -- (1.,6.812746790794636);
					\draw [line width=2.pt] (2.920465005568102,7.742594813136765)-- (-1.,7.);
					\draw [line width=1.5pt] (0.8069143030405113,7.342256642340844) -- (0.925383585711005,7.555279246260629);
					\draw [line width=1.5pt] (0.8069143030405113,7.342256642340844) -- (0.9950814198570981,7.187315566876136);
					\draw [line width=2.pt] (2.533519268257322,8.69061186954817)-- (-1.,7.);
					\draw [line width=1.5pt] (0.6259970513931011,7.7779580939688024) -- (0.6859422251623214,8.014221034056757);
					\draw [line width=1.5pt] (0.6259970513931011,7.7779580939688024) -- (0.8475770430950026,7.676390835491413);
					\draw [line width=2.pt] (1.759627793635761,9.754712647152811)-- (-1.,7.);
					\draw [line width=1.5pt] (0.26937557864717276,8.267114706030467) -- (0.247523955762754,8.509882305381252);
					\draw [line width=1.5pt] (0.26937557864717276,8.267114706030467) -- (0.5121038378730068,8.244830341771557);
					\draw [line width=2.pt] (0.7535688766277328,10.567298695505446)-- (-1.,7.);
					\draw [line width=1.5pt] (-0.19205443219021545,8.643609902355214) -- (-0.2912628961631437,8.866255992357617);
					\draw [line width=1.5pt] (-0.19205443219021545,8.643609902355214) -- (0.044831772790876454,8.701042703147824);
					\draw [line width=2.pt] (-0.02032259799382734,10.85750799848853)-- (-1.,7.);
					\draw [line width=1.5pt] (-0.548571949657045,8.777510955518661) -- (-0.6916529514676375,8.97484678003642);
					\draw [line width=1.5pt] (-0.548571949657045,8.777510955518661) -- (-0.3286696465261894,8.882661218452107);
					\draw [line width=2.pt] (-0.7168249251532315,10.915549859085147)-- (-1.,7.);
					\draw [line width=1.5pt] (-0.8696682921555857,8.802137072129897) -- (-1.0451778914718264,8.971281925037339);
					\draw [line width=1.5pt] (-0.8696682921555857,8.802137072129897) -- (-0.6716470336814035,8.944267934047808);
					\draw [line width=2.pt] (-1.4520218260437137,10.915549859085147)-- (-1.,7.);
					\draw [line width=1.5pt] (-1.2081155770148475,8.802760113125833) -- (-1.4120286927219452,8.936300526334163);
					\draw [line width=1.5pt] (-1.2081155770148475,8.802760113125833) -- (-1.039993133321768,8.979249332750983);
					\draw [line width=2.pt] (-3.1158884964800677,10.373825826850057)-- (-1.,7.);
					\draw [line width=1.5pt] (-1.9750368045792581,8.554715363730812) -- (-2.216581307873093,8.587423981032098);
					\draw [line width=1.5pt] (-1.9750368045792581,8.554715363730812) -- (-1.8993071886069746,8.786401845817958);
					
					\begin{scriptsize}
						\draw [fill=qqzzqq] (-1.4520218260437137,10.915549859085147) circle (3pt);
						\draw [fill=qqzzqq] (-3.1158884964800677,10.373825826850057) circle (3pt);
						\draw [fill=zzccqq] (-4.470198577067798,9.077557606858948) circle (3pt);
						\draw [fill=zzccqq] (-3.6189179549840818,10.083616523866972) circle (3pt);
						\draw [fill=zzffqq] (-4.702366019454266,8.458444427161703) circle (3pt);
						\draw [fill=zzffqq] (-4.992575322437351,7.374996362691523) circle (3pt);
						\draw [fill=ccffcc] (-4.973228035571812,6.252853724490266) circle (3pt);
						\draw [fill=ccffcc] (-4.276725708412408,4.666376201516074) circle (3pt);
						\draw [fill=ffffww] (2.533519268257322,5.033974651961314) circle (3pt);
						\draw [fill=ffffww] (1.661554594802546,6.400870768344779) circle (3pt);
						\draw [fill=ffccww] (2.069184383484386,4.414861472264069) circle (3pt);
						\draw [fill=ffffww] (3.,7.) circle (3pt);
						\draw [fill=yqqqqq] (-0.7168249251532315,10.915549859085147) circle (3pt);
						\draw [fill=cczzqq] (-0.7168249251532304,3.041204104810805) circle (3pt);
						\draw [fill=cczzqq] (-3.541528807521925,3.9118320137600566) circle (3pt);
						\draw [fill=ccqqqq] (-0.02032259799382734,10.85750799848853) circle (3pt);
						\draw [fill=ffcqcb] (2.920465005568102,7.742594813136765) circle (3pt);
						\draw [fill=ffccww] (0.07641383633386878,3.1379405391384996) circle (3pt);
						\draw [fill=cczzqq] (-1.7666619570249664,4.628983561782251) circle (3pt);
						\draw [fill=ffccww] (0.428937407628609,4.898618571476549) circle (3pt);
						\draw [fill=ffttww] (1.759627793635761,9.754712647152811) circle (3pt);
						\draw [fill=ffwwzz] (2.533519268257322,8.69061186954817) circle (3pt);
						\draw [fill=ffffff] (-1.,7.) circle (3pt);
						\draw [fill=ffqqtt] (0.7535688766277328,10.567298695505446) circle (3pt);
					\end{scriptsize}
				\end{tikzpicture}
				\caption{$\vec{\mathcal{P}}(S_4)$ where all the vertices with the same colour are in the same $\mathtt{N}$-class, and vertices in the same box are in the same $\diamond$-class.}
				\label{S_4NeDiamond}
			\end{figure}
			
			\begin{definition}
			\rm	Given two elements $x$ and $y$ of a group $G$ we write $x\circ y$ if $o(x)=o(y)$.
			\end{definition}
		
			As the relations previously defined, also $\circ$ is an equivalence relation. The $\circ$-class of an element $x$ is denoted by $[x]_{\circ}$.
			
			Let $x,y$ elements of a group $G$.
			If $x\diamond y$, then $x\circ y$ and $x\mathtt{N}y$ both hold. Hence $\diamond$ is a refinement of both $\mathtt{N}$ and $\circ$, thus an $\mathtt{N}$-class and a $\circ$-class are both union of $\diamond$-classes.
			Instead if we have $x\mathtt{N}y$ neither $x\diamond y$ nor $x\circ y$ must necessary hold; and the same is true if we have $x \circ y$.
			But there is a specific connection between these three classes that the next lemma clarifies. This was a result of the article \cite{Cameron_2} that Cameron did not enclose in a lemma. 
			
			\begin{lemma}
				\label{PrimoLegameDueClassi}Let $x$ and $y$ be two distinct elements of $G$ such that $x \circ y$. Then the following conditions are equivalent:
				
				\begin{itemize}
					\item[\rm (a)] $x$ is joined to $y$ in the power graph;
					
					\item[\rm (b)]  $x\diamond y$;
					
					\item[\rm (c)] $x \mathtt{N} y$.
				\end{itemize}
			\end{lemma}
			\begin{proof}
				$\text{(a)}\Rightarrow \text{(b)}$: If $x$ is joined to $y$ in the power graph, then $x \in \langle y \rangle $ or $y \in \langle x \rangle$; in both cases, since $o(x)=o(y)$, we have that $\langle x \rangle= \langle y \rangle$, that is $x\diamond y$.
				
				$\text{(b)}\Rightarrow \text{(c)}$: This implication comes from the Remark \ref{DiamondRefinementOfN}.
				
				$\text{(c)} \Rightarrow \text{(a)}$  If $x\mathtt{N}y$, then $N[x]=N[y]$. Hence $y \in N[y]=N[x]$ gives $\left\lbrace x,y \right\rbrace \in E$ since $x\not= y$.
			\end{proof}
			
			\begin{corollary}
				\label{o-partition}Let $X$ be an $\mathtt{N}$-class. The partition of $X$ obtained by considering all the $\circ$-classes in $X$ is the same partition obtained by considering all the $\diamond$-classes in $X$.
			\end{corollary}			
			
			\begin{proof}
				Let, for a $k \in \mathbb{N}$, $[x_1]_{\circ}, [x_2]_{\circ}, ..., [x_k]_{\circ}$ the $\circ$-classes that partitioned $X$.
						
				For all $i\in [k]$, if $y \in [x_i]_{\circ}$, then $o(x_i)=o(y)$ and also $x_i \mathtt{N} y$, since we are in $X$. By Lemma \ref{PrimoLegameDueClassi}, we have $x_i \diamond y$. Then $y \in [x_i]_{\diamond}$.
				Thus $[x_i]_{\circ} \subseteq [x_i]_{\diamond}$ for all $i\in [k]$.
				Vice versa, for all $i \in [k]$, if $y \in [x_i]_{\diamond}$, then we have $x_i \diamond y$ that implies $o(x_i)=o(y)$. So, by definition of the $\circ$ relation, $y \in [x_i]_{\circ}$. Thus $[x_i]_{\diamond}\subseteq [x_i]_{\circ}$ for all $i \in [k]$. Therefore $[x_i]_{\circ}=[x_i]_{\diamond}$, thus $[x_1]_{\circ},[x_2]_{\circ},...[x_k]_{\circ}$ is the partition of $X$ in $\diamond$-classes.
			\end{proof}
			
			For example,  take $G\cong C_{p^n}$, a cyclic group of prime power order. $\mathcal{P}(G)$ is a complete graph, so all its vertices belong to the same $\mathtt{N}$-class. Then the partition of $G$ in $\circ$-classes is also the partition of $G$ in $\diamond$-classes.
			
			In Figure \ref{S_4Partitioned} is represented $\vec{\mathcal{P}}(S_4)$ with the partitions given by the three  relations highlighted. In this figure we can clearly see the complex interconnections among the three equivalence relations. Pick an $\mathtt{N}$-class, for example the one of the yellow vertices. Since $\diamond$ is a refinement of $\mathtt{N}$, then the entire $\diamond$-class of each vertex is contained in the $\mathtt{N}$-class. The same does not happen for the $\circ$-classes of the vertices in the $\mathtt{N}$-class. However we have proved, in Lemma \ref{PrimoLegameDueClassi}, and here we can see it, that the intersection between an $\mathtt{N}$-class and a $\circ$-class is exactly a $\diamond$-class, even if, at the beginning, it could have been union of $\diamond$-classes.   
			
			\begin{figure}
				\centering
				\begin{tikzpicture}[line cap=round,line join=round,>=triangle 45,x=1.0cm,y=1.0cm]
					\clip(-6.,2.) rectangle (4.,12.);
					\fill[line width=2.pt,color=zzttqq,fill=zzttqq,fill opacity=0.10000000149011612] (-3.00711852963106,9.809911400437345) -- (-1.1933294275569224,10.487257783623239) -- (-1.3240850694054276,11.397934907263139) -- (-3.6127226668446086,10.710096634610261) -- cycle;
					\fill[line width=2.pt,color=zzttqq,fill=zzttqq,fill opacity=0.10000000149011612] (-5.110910354913101,9.090550324695746) -- (-4.265322159020893,8.56129342866397) -- (-3.00711852963106,9.809911400437345) -- (-3.6127226668446086,10.710096634610261) -- cycle;
					\fill[line width=2.pt,color=zzttqq,fill=zzttqq,fill opacity=0.10000000149011612] (-5.110910354913101,9.090550324695746) -- (-5.490950168512438,6.796738592614046) -- (-4.58366337471664,6.826333803122147) -- (-4.265322159020893,8.56129342866397) -- cycle;
					\fill[line width=2.pt,color=zzttqq,fill=zzttqq,fill opacity=0.10000000149011612] (-3.7241420923381203,4.629779414821492) -- (-4.255820774314594,3.7428472333336757) -- (-5.490950168512438,6.796738592614046) -- (-4.58366337471664,6.826333803122147) -- cycle;
					\fill[line width=2.pt,color=zzttqq,fill=zzttqq,fill opacity=0.10000000149011612] (2.604291484014996,7.284127314600211) -- (2.400341814808526,8.181505859108677) -- (3.22633797509473,8.528220296759674) -- (3.389497710459906,7.284127314600211) -- cycle;
					\fill[line width=2.pt,color=zzttqq,fill=zzttqq,fill opacity=0.10000000149011612] (2.400341814808526,8.181505859108677) -- (1.9006651252526743,8.92592215171229) -- (2.5838965170943493,9.5377711593317) -- (3.22633797509473,8.528220296759674) -- cycle;
					\fill[line width=2.pt,color=zzttqq,fill=zzttqq,fill opacity=0.10000000149011612] (1.9006651252526743,8.92592215171229) -- (0.9013117461409711,9.762115795458815) -- (1.5669320840466672,10.502126775207563) -- (2.5838965170943493,9.5377711593317) -- cycle;
					\fill[line width=2.pt,color=zzttqq,fill=zzttqq,fill opacity=0.10000000149011612] (0.9013117461409711,9.762115795458815) -- (0.13650048661670827,10.1904101007924) -- (0.5625411481055639,11.112905047063636) -- (1.5669320840466672,10.502126775207563) -- cycle;
					\fill[line width=2.pt,color=zzttqq,fill=zzttqq,fill opacity=0.10000000149011612] (0.13650048661670827,10.1904101007924) -- (-0.5248128745958504,10.407672479699302) -- (-0.2789755820072525,11.384362056777448) -- (0.5625411481055639,11.112905047063636) -- cycle;
					\fill[line width=2.pt,color=zzttqq,fill=zzttqq,fill opacity=0.10000000149011612] (-0.2789755820072525,11.384362056777448) -- (-1.3240850694054276,11.397934907263139) -- (-1.1933294275569224,10.487257783623239) -- (-0.5248128745958504,10.407672479699302) -- cycle;
					\fill[line width=2.pt,color=zzttqq,fill=zzttqq,fill opacity=0.10000000149011612] (-2.,8.) -- (8.905437690318596E-4,7.999107867261943) -- (0.,6.) -- (-2.,6.) -- cycle;
					\fill[line width=2.pt,color=zzttqq,fill=zzttqq,fill opacity=0.10000000149011612] (-4.255820774314594,3.7428472333336757) -- (-0.20647165890010305,2.544644452014348) -- (-0.3808547962166728,3.589809124493197) -- (-3.7241420923381203,4.629779414821492) -- cycle;
					\fill[line width=2.pt,color=zzttqq,fill=zzttqq,fill opacity=0.10000000149011612] (-0.3808547962166728,3.589809124493197) -- (1.7884928071891157,5.08147088717034) -- (2.655278901316614,4.388042011868344) -- (-0.20647165890010305,2.544644452014348) -- cycle;
					\fill[line width=2.pt,color=zzttqq,fill=zzttqq,fill opacity=0.10000000149011612] (1.7884928071891157,5.08147088717034) -- (2.604291484014996,7.284127314600211) -- (3.389497710459906,7.284127314600211) -- (2.655278901316614,4.388042011868344) -- cycle;
					\fill[line width=2.pt,color=zzttqq,fill=zzttqq,fill opacity=0.10000000149011612] (2.604291484014996,7.284127314600211) -- (1.1171708774262428,6.54424864751115) -- (1.7884928071891157,5.08147088717034) -- cycle;
					\fill[line width=2.pt,color=zzttqq,fill=zzttqq,fill opacity=0.10000000149011612] (1.7884928071891157,5.08147088717034) -- (0.027873487448356438,5.488419264647737) -- (-0.3808547962166728,3.589809124493197) -- cycle;
					\fill[line width=2.pt,color=zzttqq,fill=zzttqq,fill opacity=0.10000000149011612] (-0.3808547962166728,3.589809124493197) -- (-1.580788935019527,5.203823602504077) -- (-3.7241420923381203,4.629779414821492) -- cycle;
					\draw [-latex,line width=2.pt] (-3.1158884964800677,10.373825826850057) -- (-1.4520218260437137,10.915549859085147);
					\draw [-latex,line width=2.pt] (-1.4520218260437137,10.915549859085147) -- (-3.1158884964800677,10.373825826850057);
					\draw [-latex,line width=2.pt] (-3.6189179549840818,10.083616523866972) -- (-4.470198577067798,9.077557606858948);
					\draw [-latex,line width=2.pt] (-4.470198577067798,9.077557606858948) -- (-3.6189179549840818,10.083616523866972);
					\draw [-latex,line width=2.pt] (-4.702366019454266,8.458444427161703) -- (-4.992575322437351,7.374996362691523);
					\draw [-latex,line width=2.pt] (-4.992575322437351,7.374996362691523) -- (-4.702366019454266,8.458444427161703);
					\draw [-latex,line width=2.pt] (-4.973228035571812,6.252853724490266) -- (-4.276725708412408,4.666376201516074);
					\draw [-latex,line width=2.pt] (-4.276725708412408,4.666376201516074) -- (-4.973228035571812,6.252853724490266);
					\draw [-latex,line width=2.pt] (-3.541528807521925,3.9118320137600566) -- (-1.7666619570249664,4.628983561782251);
					\draw [-latex,line width=2.pt] (-0.7168249251532304,3.041204104810805) -- (-1.7666619570249664,4.628983561782251);
					\draw [-latex,line width=2.pt] (-3.541528807521925,3.9118320137600566) -- (-0.7168249251532304,3.041204104810805);
					\draw [-latex,line width=2.pt] (-0.7168249251532304,3.041204104810805) -- (-3.541528807521925,3.9118320137600566);
					\draw [-latex,line width=2.pt] (0.07641383633386878,3.1379405391384996) -- (0.428937407628609,4.898618571476549);
					\draw [-latex,line width=2.pt] (2.069184383484386,4.414861472264069) -- (0.428937407628609,4.898618571476549);
					\draw [-latex,line width=2.pt] (0.07641383633386878,3.1379405391384996) -- (2.069184383484386,4.414861472264069);
					\draw [-latex,line width=2.pt] (2.069184383484386,4.414861472264069) -- (0.07641383633386889,3.1379405391384996);
					\draw [-latex,line width=2.pt] (2.533519268257322,5.033974651961314) -- (1.661554594802546,6.400870768344779);
					\draw [-latex,line width=2.pt] (2.533519268257322,5.033974651961314) -- (3.,7.);
					\draw [-latex,line width=2.pt] (3.,7.) -- (1.661554594802546,6.400870768344779);
					\draw [-latex,line width=2.pt] (3.,7.) -- (2.533519268257322,5.033974651961314);
					\draw [line width=2.pt] (-3.6189179549840818,10.083616523866972)-- (-1.,7.);
					\draw [line width=1.5pt] (-2.247967907618133,8.469406269052197) -- (-2.396341368949588,8.468018978084796);
					\draw [line width=1.5pt] (-2.247967907618133,8.469406269052197) -- (-2.2225765860344926,8.615597545782174);
					\draw [line width=2.pt] (-4.470198577067798,9.077557606858948)-- (-1.,7.);
					\draw [line width=1.5pt] (-2.653598357420002,7.989985376874375) -- (-2.7936514004000177,7.940977686092797);
					\draw [line width=1.5pt] (-2.653598357420002,7.989985376874375) -- (-2.676547176667781,8.136579920766149);
					\draw [line width=2.pt] (-4.702366019454266,8.458444427161703)-- (-1.,7.);
					\draw [line width=1.5pt] (-2.7628025099421927,7.694407166472106) -- (-2.8929610662576266,7.623165613838923);
					\draw [line width=1.5pt] (-2.7628025099421927,7.694407166472106) -- (-2.8094049531966405,7.835278813322779);
					\draw [line width=2.pt] (-4.992575322437351,7.374996362691523)-- (-1.,7.);
					\draw [line width=1.5pt] (-2.9017133669173476,7.178615439380241) -- (-3.006946951577299,7.07400902818417);
					\draw [line width=1.5pt] (-2.9017133669173476,7.178615439380241) -- (-2.98562837086005,7.300987334507353);
					\draw [line width=2.pt] (-4.973228035571812,6.252853724490266)-- (-1.,7.);
					\draw [line width=1.5pt] (-2.893259701827158,6.643981688911731) -- (-2.965548225785986,6.514401683094636);
					\draw [line width=1.5pt] (-2.893259701827158,6.643981688911731) -- (-3.0076798097858255,6.738452041395628);
					\draw [line width=2.pt] (-4.276725708412408,4.666376201516074)-- (-1.,7.);
					\draw [line width=1.5pt] (-2.5609889567810296,5.888292367785723) -- (-2.57223773377298,5.740339423847828);
					\draw [line width=1.5pt] (-2.5609889567810296,5.888292367785723) -- (-2.704487974639428,5.926036777668245);
					\draw [line width=2.pt] (-3.541528807521925,3.9118320137600566)-- (-1.,7.);
					\draw [line width=1.5pt] (-2.2104018510185446,5.529261507586197) -- (-2.182749802913557,5.383480943589127);
					\draw [line width=1.5pt] (-2.2104018510185446,5.529261507586197) -- (-2.3587790046083668,5.528351070170927);
					\draw [line width=2.pt] (-1.7666619570249664,4.628983561782251)-- (-1.,7.);
					\draw [line width=1.5pt] (-1.354105861247133,5.904874814521594) -- (-1.2748713381559207,5.779421640172703);
					\draw [line width=1.5pt] (-1.354105861247133,5.904874814521594) -- (-1.4917906188690468,5.849561921609546);
					\draw [line width=2.pt] (-0.7168249251532304,3.041204104810805)-- (-1.,7.);
					\draw [line width=1.5pt] (-0.8651898760493459,5.11535049257219) -- (-0.7447143343764686,5.028734948572678);
					\draw [line width=1.5pt] (-0.8651898760493459,5.11535049257219) -- (-0.9721105907767633,5.012469156238125);
					\draw [line width=2.pt] (0.07641383633386878,3.1379405391384996)-- (-1.,7.);
					\draw [line width=1.5pt] (-0.487296317874571,5.160473194106289) -- (-0.351989572388617,5.099574152819057);
					\draw [line width=1.5pt] (-0.487296317874571,5.160473194106289) -- (-0.5715965912775124,5.038366386319442);
					\draw [line width=2.pt] (0.428937407628609,4.898618571476549)-- (-1.,7.);
					\draw [line width=1.5pt] (-0.33894537134279434,6.027859504213665) -- (-0.19127103401522508,6.013406175926794);
					\draw [line width=1.5pt] (-0.33894537134279434,6.027859504213665) -- (-0.37979155835616357,5.885212395549755);
					\draw [line width=2.pt] (2.069184383484386,4.414861472264069)-- (-1.,7.);
					\draw [line width=1.5pt] (0.4619394313351985,5.768625316387673) -- (0.60802568804896,5.794614048620427);
					\draw [line width=1.5pt] (0.4619394313351985,5.768625316387673) -- (0.4611586954354257,5.620247423643641);
					\draw [line width=2.pt] (2.533519268257322,5.033974651961314)-- (-1.,7.);
					\draw [line width=1.5pt] (0.6837525008727287,6.063171912977312) -- (0.8221811385246505,6.116595885887776);
					\draw [line width=1.5pt] (0.6837525008727287,6.063171912977312) -- (0.7113381297326743,5.917378766073536);
					\draw [line width=2.pt] (1.661554594802546,6.400870768344779)-- (-1.,7.);
					\draw [line width=1.5pt] (0.23810569659954114,6.721296224334385) -- (0.35581030559566845,6.811641305134467);
					\draw [line width=1.5pt] (0.23810569659954114,6.721296224334385) -- (0.3057442892068792,6.5892294632103106);
					\draw [line width=2.pt] (3.,7.)-- (-1.,7.);
					\draw [line width=1.5pt] (0.9050094728542887,7.) -- (1.,7.113988632574852);
					\draw [line width=1.5pt] (0.9050094728542887,7.) -- (1.,6.886011367425147);
					\draw [line width=2.pt] (2.920465005568102,7.742594813136765)-- (-1.,7.);
					\draw [line width=1.5pt] (0.8669014885292443,7.353619114071962) -- (0.9390185517883466,7.48329462367415);
					\draw [line width=1.5pt] (0.8669014885292443,7.353619114071962) -- (0.9814464537797566,7.259300189462613);
					\draw [line width=2.pt] (2.533519268257322,8.69061186954817)-- (-1.,7.);
					\draw [line width=1.5pt] (0.6810717338153811,7.804308569159679) -- (0.7175627953913748,7.948131415150022);
					\draw [line width=1.5pt] (0.6810717338153811,7.804308569159679) -- (0.8159564728659486,7.74248045439815);
					\draw [line width=2.pt] (1.759627793635761,9.754712647152811)-- (-1.,7.);
					\draw [line width=1.5pt] (0.3125856077777757,8.310247774194377) -- (0.2992836375594462,8.45803027042453);
					\draw [line width=1.5pt] (0.3125856077777757,8.310247774194377) -- (0.46034415607631485,8.29668237672828);
					\draw [line width=2.pt] (0.7535688766277328,10.567298695505446)-- (-1.,7.);
					\draw [line width=1.5pt] (-0.16512057762715113,8.698401650502868) -- (-0.22551279838596006,8.833935366881944);
					\draw [line width=1.5pt] (-0.16512057762715113,8.698401650502868) -- (-0.020918324986306584,8.7333633286235);
					\draw [line width=2.pt] (-0.02032259799382734,10.85750799848853)-- (-1.,7.);
					\draw [line width=1.2pt] (-0.5335434218648701,8.836686216727209) -- (-0.6206426380173825,8.956812546685809);
					\draw [line width=1.2pt] (-0.5335434218648701,8.836686216727209) -- (-0.39967995997644795,8.900695451802719);
					\draw [line width=2.pt] (-0.7168249251532315,10.915549859085147)-- (-1.,7.);
					\draw [line width=1.5pt] (-0.8652643431272725,8.863031845466006) -- (-0.9721041634684966,8.965997186203362);
					\draw [line width=1.5pt] (-0.8652643431272725,8.863031845466006) -- (-0.744720761684734,8.949552672881785);
					\draw [line width=2.pt] (-1.4520218260437137,10.915549859085147)-- (-1.,7.);
					\draw [line width=1.5pt] (-1.2151172952859413,8.86341111582119) -- (-1.3392474894875146,8.944702588259476);
					\draw [line width=1.5pt] (-1.2151172952859413,8.86341111582119) -- (-1.1127743365561966,8.97084727082567);
					\draw [line width=2.pt] (-3.1158884964800677,10.373825826850057)-- (-1.,7.);
					\draw [line width=1.5pt] (-2.007475120510404,8.606438896539885) -- (-2.154513068502206,8.626349960149474);
					\draw [line width=1.5pt] (-2.007475120510404,8.606438896539885) -- (-1.9613754279778606,8.747475866700583);
					
					\draw [line width=1.pt,color=zzttqq] (-3.00711852963106,9.809911400437345)-- (-1.1933294275569224,10.487257783623239);
					\draw [line width=1.pt,color=zzttqq] (-1.1933294275569224,10.487257783623239)-- (-1.3240850694054276,11.397934907263139);
					\draw [line width=1.pt,color=zzttqq] (-1.3240850694054276,11.397934907263139)-- (-3.6127226668446086,10.710096634610261);
					\draw [line width=1.pt,color=zzttqq] (-3.6127226668446086,10.710096634610261)-- (-3.00711852963106,9.809911400437345);
					\draw [line width=1.pt,color=zzttqq] (-5.110910354913101,9.090550324695746)-- (-4.265322159020893,8.56129342866397);
					\draw [line width=1.pt,color=zzttqq] (-4.265322159020893,8.56129342866397)-- (-3.00711852963106,9.809911400437345);
					\draw [line width=1.pt,color=zzttqq] (-3.00711852963106,9.809911400437345)-- (-3.6127226668446086,10.710096634610261);
					\draw [line width=1.pt,color=zzttqq] (-3.6127226668446086,10.710096634610261)-- (-5.110910354913101,9.090550324695746);
					\draw [line width=1.pt,color=zzttqq] (-5.110910354913101,9.090550324695746)-- (-5.490950168512438,6.796738592614046);
					\draw [line width=1.pt,color=zzttqq] (-5.490950168512438,6.796738592614046)-- (-4.58366337471664,6.826333803122147);
					\draw [line width=1.pt,color=zzttqq] (-4.58366337471664,6.826333803122147)-- (-4.265322159020893,8.56129342866397);
					\draw [line width=1.pt,color=zzttqq] (-4.265322159020893,8.56129342866397)-- (-5.110910354913101,9.090550324695746);
					\draw [line width=1.pt,color=zzttqq] (-3.7241420923381203,4.629779414821492)-- (-4.255820774314594,3.7428472333336757);
					\draw [line width=1.pt,color=zzttqq] (-4.255820774314594,3.7428472333336757)-- (-5.490950168512438,6.796738592614046);
					\draw [line width=1.pt,color=zzttqq] (-5.490950168512438,6.796738592614046)-- (-4.58366337471664,6.826333803122147);
					\draw [line width=1.pt,color=zzttqq] (-4.58366337471664,6.826333803122147)-- (-3.7241420923381203,4.629779414821492);
					\draw [line width=1.pt,color=zzttqq] (2.604291484014996,7.284127314600211)-- (2.400341814808526,8.181505859108677);
					\draw [line width=1.pt,color=zzttqq] (2.400341814808526,8.181505859108677)-- (3.22633797509473,8.528220296759674);
					\draw [line width=1.pt,color=zzttqq] (3.22633797509473,8.528220296759674)-- (3.389497710459906,7.284127314600211);
					\draw [line width=1.pt,color=zzttqq] (3.389497710459906,7.284127314600211)-- (2.604291484014996,7.284127314600211);
					\draw [line width=1.pt,color=zzttqq] (2.400341814808526,8.181505859108677)-- (1.9006651252526743,8.92592215171229);
					\draw [line width=1.pt,color=zzttqq] (1.9006651252526743,8.92592215171229)-- (2.5838965170943493,9.5377711593317);
					\draw [line width=1.pt,color=zzttqq] (2.5838965170943493,9.5377711593317)-- (3.22633797509473,8.528220296759674);
					\draw [line width=1.pt,color=zzttqq] (3.22633797509473,8.528220296759674)-- (2.400341814808526,8.181505859108677);
					\draw [line width=1.pt,color=zzttqq] (1.9006651252526743,8.92592215171229)-- (0.9013117461409711,9.762115795458815);
					\draw [line width=1.pt,color=zzttqq] (0.9013117461409711,9.762115795458815)-- (1.5669320840466672,10.502126775207563);
					\draw [line width=1.pt,color=zzttqq] (1.5669320840466672,10.502126775207563)-- (2.5838965170943493,9.5377711593317);
					\draw [line width=1.pt,color=zzttqq] (2.5838965170943493,9.5377711593317)-- (1.9006651252526743,8.92592215171229);
					\draw [line width=1.pt,color=zzttqq] (0.9013117461409711,9.762115795458815)-- (0.13650048661670827,10.1904101007924);
					\draw [line width=1.pt,color=zzttqq] (0.13650048661670827,10.1904101007924)-- (0.5625411481055639,11.112905047063636);
					\draw [line width=1.pt,color=zzttqq] (0.5625411481055639,11.112905047063636)-- (1.5669320840466672,10.502126775207563);
					\draw [line width=1.pt,color=zzttqq] (1.5669320840466672,10.502126775207563)-- (0.9013117461409711,9.762115795458815);
					\draw [line width=1.pt,color=zzttqq] (0.13650048661670827,10.1904101007924)-- (-0.5248128745958504,10.407672479699302);
					\draw [line width=1.pt,color=zzttqq] (-0.5248128745958504,10.407672479699302)-- (-0.2789755820072525,11.384362056777448);
					\draw [line width=1.pt,color=zzttqq] (-0.2789755820072525,11.384362056777448)-- (0.5625411481055639,11.112905047063636);
					\draw [line width=1.pt,color=zzttqq] (0.5625411481055639,11.112905047063636)-- (0.13650048661670827,10.1904101007924);
					\draw [line width=1.pt,color=zzttqq] (-0.2789755820072525,11.384362056777448)-- (-1.3240850694054276,11.397934907263139);
					\draw [line width=1.pt,color=zzttqq] (-1.3240850694054276,11.397934907263139)-- (-1.1933294275569224,10.487257783623239);
					\draw [line width=1.pt,color=zzttqq] (-1.1933294275569224,10.487257783623239)-- (-0.5248128745958504,10.407672479699302);
					\draw [line width=1.pt,color=zzttqq] (-0.5248128745958504,10.407672479699302)-- (-0.2789755820072525,11.384362056777448);
					\draw [line width=1.pt,color=zzttqq] (-2.,8.)-- (8.905437690318596E-4,7.999107867261943);
					\draw [line width=1.pt,color=zzttqq] (8.905437690318596E-4,7.999107867261943)-- (0.,6.);
					\draw [line width=1.pt,color=zzttqq] (0.,6.)-- (-2.,6.);
					\draw [line width=1.pt,color=zzttqq] (-2.,6.)-- (-2.,8.);
					\draw [line width=1.pt,color=zzttqq] (-4.255820774314594,3.7428472333336757)-- (-0.20647165890010305,2.544644452014348);
					\draw [line width=1.pt,color=zzttqq] (-0.20647165890010305,2.544644452014348)-- (-0.3808547962166728,3.589809124493197);
					\draw [line width=1.pt,color=zzttqq] (-0.3808547962166728,3.589809124493197)-- (-3.7241420923381203,4.629779414821492);
					\draw [line width=1.pt,color=zzttqq] (-3.7241420923381203,4.629779414821492)-- (-4.255820774314594,3.7428472333336757);
					\draw [line width=1.pt,color=zzttqq] (-0.3808547962166728,3.589809124493197)-- (1.7884928071891157,5.08147088717034);
					\draw [line width=1.pt,color=zzttqq] (1.7884928071891157,5.08147088717034)-- (2.655278901316614,4.388042011868344);
					\draw [line width=1.pt,color=zzttqq] (2.655278901316614,4.388042011868344)-- (-0.20647165890010305,2.544644452014348);
					\draw [line width=1.pt,color=zzttqq] (-0.20647165890010305,2.544644452014348)-- (-0.3808547962166728,3.589809124493197);
					\draw [line width=1.pt,color=zzttqq] (1.7884928071891157,5.08147088717034)-- (2.604291484014996,7.284127314600211);
					\draw [line width=1.pt,color=zzttqq] (2.604291484014996,7.284127314600211)-- (3.389497710459906,7.284127314600211);
					\draw [line width=1.pt,color=zzttqq] (3.389497710459906,7.284127314600211)-- (2.655278901316614,4.388042011868344);
					\draw [line width=1.pt,color=zzttqq] (2.655278901316614,4.388042011868344)-- (1.7884928071891157,5.08147088717034);
					\draw [line width=1.pt,color=zzttqq] (2.604291484014996,7.284127314600211)-- (1.1171708774262428,6.54424864751115);
					\draw [line width=1.pt,color=zzttqq] (1.1171708774262428,6.54424864751115)-- (1.7884928071891157,5.08147088717034);
					\draw [line width=1.pt,color=zzttqq] (1.7884928071891157,5.08147088717034)-- (2.604291484014996,7.284127314600211);
					\draw [line width=1.pt,color=zzttqq] (1.7884928071891157,5.08147088717034)-- (0.027873487448356438,5.488419264647737);
					\draw [line width=1.pt,color=zzttqq] (0.027873487448356438,5.488419264647737)-- (-0.3808547962166728,3.589809124493197);
					\draw [line width=1.pt,color=zzttqq] (-0.3808547962166728,3.589809124493197)-- (1.7884928071891157,5.08147088717034);
					\draw [line width=1.pt,color=zzttqq] (-0.3808547962166728,3.589809124493197)-- (-1.580788935019527,5.203823602504077);
					\draw [line width=1.pt,color=zzttqq] (-1.580788935019527,5.203823602504077)-- (-3.7241420923381203,4.629779414821492);
					\draw [line width=1.pt,color=zzttqq] (-3.7241420923381203,4.629779414821492)-- (-0.3808547962166728,3.589809124493197);
					\draw [line width=1.pt,dash pattern=on 3pt off 3pt,color=ffwwzz] (2.0189432790032513,7.3505036400956865)-- (3.6575822816348698,7.540751976067993);
					\draw [line width=1.pt,dash pattern=on 3pt off 3pt,color=ffwwzz] (-1.0147315312303216,9.312138053157463)-- (-1.0298739977485987,11.68877307555815);
					\draw [line width=1.pt,dash pattern=on 3pt off 3pt,color=ffwwzz] (-2.6291030481784734,5.359200605410045)-- (-3.141353792134679,4.843272094011473);
					\draw [shift={(-1.,7.)},line width=1.pt,dash pattern=on 3pt off 3pt,color=ffwwzz]  plot[domain=-2.352617534386248:0.11558394689465694,variable=\t]({1.*3.039222157651399*cos(\t r)+0.*3.039222157651399*sin(\t r)},{0.*3.039222157651399*cos(\t r)+1.*3.039222157651399*sin(\t r)});
					\draw [shift={(-1.,7.)},line width=1.pt,dash pattern=on 3pt off 3pt,color=ffwwzz]  plot[domain=0.11558394689465691:1.5771676296999535,variable=\t]({1.*4.688868244024406*cos(\t r)+0.*4.688868244024406*sin(\t r)},{0.*4.688868244024406*cos(\t r)+1.*4.688868244024406*sin(\t r)});
					\draw [shift={(-1.,7.)},line width=1.pt,dash pattern=on 3pt off 3pt,color=ffwwzz]  plot[domain=-2.352617534386248:1.5771676296999535,variable=\t]({1.*2.3121849828400745*cos(\t r)+0.*2.3121849828400745*sin(\t r)},{0.*2.3121849828400745*cos(\t r)+1.*2.3121849828400745*sin(\t r)});
					\draw [line width=1.pt,dash pattern=on 3pt off 3pt,color=qqffqq] (-3.2454172414909825,4.738461508276163)-- (-4.303649840102669,3.6726311981071107);
					\draw [line width=1.pt,dash pattern=on 3pt off 3pt,color=qqffqq] (3.683874522932582,7.216344316411591)-- (2.183525228099331,7.147044415020107);
					\draw [shift={(-1.,7.)},line width=1.pt,dash pattern=on 3pt off 3pt,color=qqffqq]  plot[domain=-2.352617534386248:0.04615637319914843,variable=\t]({1.*3.1869193491416596*cos(\t r)+0.*3.1869193491416596*sin(\t r)},{0.*3.1869193491416596*cos(\t r)+1.*3.1869193491416596*sin(\t r)});
					\draw [shift={(-1.,7.)},line width=1.pt,dash pattern=on 3pt off 3pt,color=qqffqq]  plot[domain=-2.352617534386248:0.04615637319914822,variable=\t]({1.*4.688868244024405*cos(\t r)+0.*4.688868244024405*sin(\t r)},{0.*4.688868244024405*cos(\t r)+1.*4.688868244024405*sin(\t r)});
					\draw [line width=1.pt,dash pattern=on 3pt off 3pt,color=qqffff] (-3.463966576347653,4.978781642529653)-- (-4.625198308606618,4.026210728196055);
					\draw [line width=1.pt,dash pattern=on 3pt off 3pt,color=qqffff] (-1.256149807868846,10.176608602560634)-- (-1.376869499428239,11.67369819203393);
					\draw [shift={(-1.,7.)},line width=1.pt,dash pattern=on 3pt off 3pt,color=qqffff]  plot[domain=1.6512584851084127:3.8285961070177326,variable=\t]({1.*3.186919349141655*cos(\t r)+0.*3.186919349141655*sin(\t r)},{0.*3.186919349141655*cos(\t r)+1.*3.186919349141655*sin(\t r)});
					\draw [shift={(-1.,7.)},line width=1.pt,dash pattern=on 3pt off 3pt,color=qqffff]  plot[domain=1.6512584851084127:3.8285961070177326,variable=\t]({1.*4.688868244024406*cos(\t r)+0.*4.688868244024406*sin(\t r)},{0.*4.688868244024406*cos(\t r)+1.*4.688868244024406*sin(\t r)});
					\draw [line width=1.pt,dash pattern=on 3pt off 3pt,color=ffqqqq] (-1.,7.) circle (0.7817718698171685cm);
					\begin{scriptsize}
						\draw [fill=qqzzqq] (-1.4520218260437137,10.915549859085147) circle (2.5pt);
						\draw [fill=qqzzqq] (-3.1158884964800677,10.373825826850057) circle (2.5pt);
						\draw [fill=zzccqq] (-4.470198577067798,9.077557606858948) circle (2.5pt);
						\draw [fill=zzccqq] (-3.6189179549840818,10.083616523866972) circle (2.5pt);
						\draw [fill=zzffqq] (-4.702366019454266,8.458444427161703) circle (2.5pt);
						\draw [fill=zzffqq] (-4.992575322437351,7.374996362691523) circle (2.5pt);
						\draw [fill=ccffcc] (-4.973228035571812,6.252853724490266) circle (2.5pt);
						\draw [fill=ccffcc] (-4.276725708412408,4.666376201516074) circle (2.5pt);
						\draw [fill=ffffww] (2.533519268257322,5.033974651961314) circle (2.5pt);
						\draw [fill=ffffww] (1.661554594802546,6.400870768344779) circle (2.5pt);
						\draw [fill=ffccww] (2.069184383484386,4.414861472264069) circle (2.5pt);
						\draw [fill=ffffww] (3.,7.) circle (2.5pt);
						\draw [fill=yqqqqq] (-0.7168249251532315,10.915549859085147) circle (2.5pt);
						\draw [fill=cczzqq] (-0.7168249251532304,3.041204104810805) circle (2.5pt);
						\draw [fill=cczzqq] (-3.541528807521925,3.9118320137600566) circle (2.5pt);
						\draw [fill=ccqqqq] (-0.02032259799382734,10.85750799848853) circle (2.5pt);
						\draw [fill=ffcqcb] (2.920465005568102,7.742594813136765) circle (2.5pt);
						\draw [fill=ffccww] (0.07641383633386878,3.1379405391384996) circle (2.5pt);
						\draw [fill=cczzqq] (-1.7666619570249664,4.628983561782251) circle (2.5pt);
						\draw [fill=ffccww] (0.428937407628609,4.898618571476549) circle (2.5pt);
						\draw [fill=ffttww] (1.759627793635761,9.754712647152811) circle (2.5pt);
						\draw [fill=ffwwzz] (2.533519268257322,8.69061186954817) circle (2.5pt);
						\draw [fill=ffffff] (-1.,7.) circle (2.5pt);
						\draw [fill=ffqqtt] (0.7535688766277328,10.567298695505446) circle (2.5pt);
					\end{scriptsize}
				\end{tikzpicture}
				\caption{$\vec{\mathcal{P}}(S_4)$ where all the vertices with the same colour are in the same $\mathtt{N}$-class, all the vertices in the same box are in the same $\diamond$-class and all the vertices in the same dashed sides figure are in the same $\circ$-class.}
				\label{S_4Partitioned}
			\end{figure}
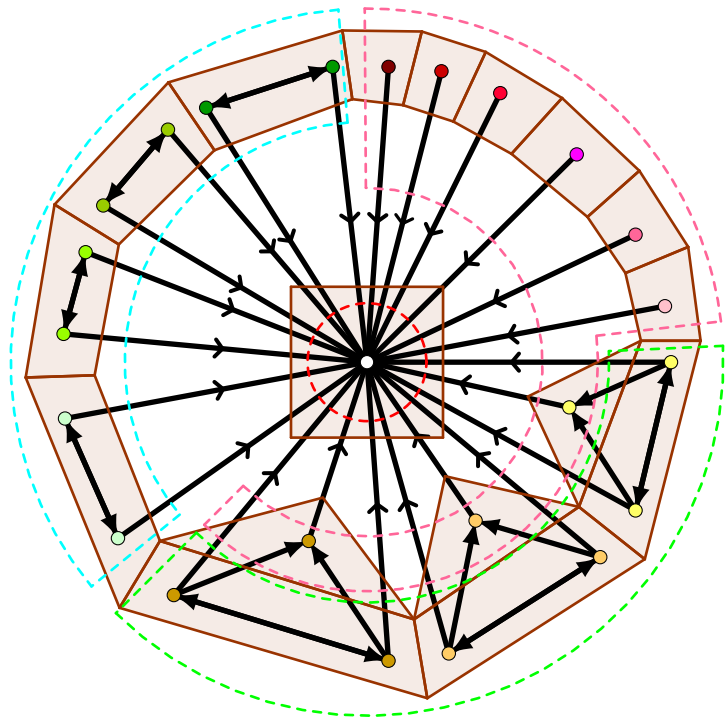
		
			We were interested to recognize the $\diamond$-classes. Now we have that, once we get the $\mathtt{N}$-classes, partitioning this classes by elements with the same order, we obtain exactly all the $\diamond$-classes.
			In conclusion: from the power graph of a group we get all the $\mathtt{N}$-classes; so we only need to \emph{find a way to break the $\mathtt{N}$-classes in $\diamond$-classes.}
			To do that we will need the following lemma.
			
			\begin{lemma}
				\label{NdipotenzeDip}Let $x$ and $y$ be two elements whose orders are powers of a prime $p$, with $o(x)\leq o(y)$. Then $N[x]\supseteq N[y]$ if and only if $x$ is power of $y$.
			\end{lemma}
			\begin{proof}
				The forward implication is clear since $N[x]\supseteq N[y]$ implies that $x$ and $y$ are joined in the power graph, and hence, as $o(x)\leq o(y)$, it must be $x$ a power of $y$. So suppose that $x$ is a power of $y$, and take $z \in N[y]$. If $y$ is a power of $z$, then so is $x$; if $z$ is a power of $y$, then $x$ and $z$ both lie in the cyclic $p$-group generated by $y$, and, by Proposition \ref{PGcompleto}, the power graph of a cyclic $p$-group is a complete graph. So $z\in N[x]$ in either cases.
			\end{proof}
			We emphasize that Lemma \ref{NdipotenzeDip} is extrapolated from the proof of \cite[Proposition 5]{Cameron_2}. 
		
	\chapter{Reconstruction of $\vec{\mathcal{P}}(G)$ from $\mathcal{P}(G)$}
			\label{SectionReconstruction}
		\section{Case I: Abelian groups}
		\label{Case I}
			This section contains the proof that, for an abelian group, power graph determines the group up to group isomorphism. From this result follows that, for $G$ abelian group, we can reconstruct $\vec{\mathcal{P}}(G)$ from $\mathcal{P}(G)$.
			
			
			\begin{theorem}{\rm \cite[Theorem 1]{Cameron_1}}
				\label{abelUPG}Let $G_1$ and $G_2$ be abelian groups with $\mathcal{P}(G_1)\cong \mathcal{P}(G_2)$. Then $G_1\cong G_2$. In particular $\vec{\mathcal{P}}(G_1)\cong \vec{\mathcal{P}}(G_2)$.
			\end{theorem}
			
			To ensue the proof of Cameron et al. in \cite{Cameron_1} it is necessary the following Proposition and its Corollary.

			\begin{prop}{\rm \cite[Proposition 2]{Cameron_1}}
				\label{abelianxNy} Let $x$ and $y$ be distinct elements of an abelian group $G$. Then $x\mathtt{N} y$ if and only if one of the following holds:
				\begin{itemize}
					\item[\rm (a)] $x\diamond y$;
					
					\item[\rm (b)] $G$ is cyclic, and one between $x$ and $y$ is a generator of $G$ and the other the identity;
					
					\item[\rm (c)] $G$ is cyclic of prime power order. 
				\end{itemize}
			\end{prop}
			
			\begin{proof}
				If (a) holds, then, since $\diamond$ is a refinement of $\mathtt{N}$, we have that $x\mathtt{N}y$.
				
				If (b) holds we have that $N[x]=G=N[y]$, then clearly $x \mathtt{N}y$.
				
				If (c) holds we have that $N[x]=G=N[y]$ since, by Proposition \ref{PGcompleto}, $\mathcal{P}(G)$ is a complete graph; then we obtain $x\mathtt{N}y$.
				
				Suppose next that $x\mathtt{N}y$ holds and assume that (a) does not hold. Note that $x$ and $y$ are not both $1$ as they are supposed to be different.
				We have that $x$ and $y$ are joined and so, after renaming, if necessary, we have $x\in \langle y \rangle$. Then $y\not=1$ otherwise also $x=1$. In particular the order of $x$ properly divides the order of $y$.
				Now consider the subgraph of $\mathcal{P}(G)$ induced by $\langle y \rangle$. By Remark \ref{InducedPG}, this subgraph is $\mathcal{P}(\langle y \rangle)$.
				As $x\mathtt{N}y$ and $x,y \in \langle y \rangle$ hold, we have that $N_{\mathcal{P}(\langle y \rangle)}[x]= N_{\mathcal{P}(\langle y \rangle)}[y]$. And since $N_{\mathcal{P}(\langle y \rangle)}[y]=\langle y \rangle$, then $y$ is a star vertex in $\mathcal{P}(\langle y \rangle)$. Therefore also $x$ is a star vertex in $\mathcal{P}(\langle y \rangle)$. 
				
				Assume first that $ y $ has not prime power order. By Proposition \ref{S>1}, we have $x=1$ since $x$ is a star vertex in $\mathcal{P}(\langle y \rangle)$ and it cannot be a generator, otherwise we would have $x \diamond y$. Then $x=1\in \mathcal{S}_{\mathcal{P}(G)}$ and, since $x\mathtt{N}y$, also $y \in \mathcal{S}_{\mathcal{P}(G)}$. So we have $|\mathcal{S}_{\mathcal{P}(G)}|>1$, and, by Proposition \ref{S>1}, it follows that $G$ must be cyclic because $G$ is abelian. Hence we have (b).
				
				Assume next that $ y $ has prime power order. In this case we have that there exist a prime $p$, and two integers $k \in \mathbb{N}_0$ and $l \in \mathbb{N}$ such that $o(x)=p^k$ and $o(y)=p^{k+l}$.
				Suppose, by contradiction, that $G$ is not a cyclic $p$-group. Consider the $p$-Sylow subgroup $P$ of $G$. We want to prove that having $P\not=G$ brings us to a contradiction, so we get that (c) holds.\\
				If $P$ is cyclic, since $G$ is supposed not to be a cyclic $p$-group, then there exists $z\in G$ of prime order $q\not=p$. Let $\bar{q}$ be a positive integer such that $q\bar{q}\equiv 1 (\text{mod} \, p^k)$. Such a $\bar{q}$ exists as $\gcd(q,p^k)=1$. Now $(xz)^{q\bar{q}}=x$ implies that $x$ is joined to $xz$. But $y$ is not joined to $xz$ as $o(xz)=p^kq\nmid p^{k+l}=o(y)$ and $o(y)=p^{k+l} \nmid p^kq=o(xz)$. A contradiction.\\
				If $P$ is not cyclic, then, by \cite[Theorems 8.5, 8.6]{Burnside}, there are at least two subgroups $B=\langle b \rangle$ and $C=\langle c \rangle \not= B$ of order $p$. 
				Suppose $y$ joined to both $b,c$. Since $o(y)=p^{k+l}$ we have that $o(y)\geq o(b)=o(c)=p$. It follows that $b,c \in \langle y \rangle$. Then we reach the contradiction $B=C$ because in $\langle y \rangle $ there is a unique subgroup of order $p$. 
				Therefore the vertex $y$ is not joined to both $b$ and $c$. Hence there exists an element $z\in G$ of order $p$ which is not a power of $y$. 
				Suppose, by contradiction, that $y$ and $yz$ are joined. Since $z\notin \langle y \rangle$ we must have that there exists a positive integer $t$ such that $(yz)^t=y$. If $z^t=1$, then $p\mid t$. So $o(y^t)$ properly divides $o(y)$ and then $(yz)^t=y^t\not=y$, a contradiction. Hence we have $z^t$ a generator of $\langle z \rangle$. Since $(yz)^t=y$ we obtain that $z^t=y^{1-t}\in \langle y \rangle$. Then $\langle z^t \rangle = \langle z \rangle \subseteq \langle y \rangle$, so we reach the contradiction $ z \in \langle y \rangle$.
				Now we have $y^{p^l}=(yz)^{p^l}$ and we also have that $y^{p^l}$ is joined to $x$ because both lay in $\langle y \rangle$ that has complete power graph. In particular, since $o(y^{p^l})=o(x)=p^k$, we have that there exists a positive integer $s$ such that $(y^{p^l})^s=x$. So $x$ is joined to $yz$, since $x= (yz)^{p^ls}$.
				Therefore we have $yz\in N[x]$, but $yz \notin N[y]$ against the hypothesis that $x\mathtt{N}y$.				
				
			\end{proof}
		
			A crucial result follows immediately from the previous proposition. 
		
			\begin{corollary}
				\label{abeldiamond_N_1} Let $G$ be an abelian group and $x,y \in G$. Assume $[x]_{\mathtt{N}}\not=\mathcal{S}$. Then $x\mathtt{N}y$ if and only if $x\diamond y$.
			\end{corollary}
			\begin{proof}
				Assume first that $x\mathtt{N}y$ holds. Since $[x]_\mathtt{N}\not=\mathcal{S}$ we have neither that $x$ is the identity or a generator of $G$, nor that $G$ is a cyclic group of prime power order. Then, by Proposition \ref{abelianxNy}, we get $x\diamond y$.
				
				Assume next $x\diamond y$, it follows immediately $x\mathtt{N}y$.
			\end{proof}
		
			If we focus on not cyclic abelian groups, the above corollary could be written in the following form. 
		
			\begin{corollary}
				\label{abeldiamond_N} Let $G$ be a not cyclic abelian group and let $x,y \in G$. Then $x\mathtt{N} y$ if and only if $x\diamond y$.
			\end{corollary}

			We are now able to give a proof of the main theorem of this section. Note that this proof ensue Cameron's proof in article \cite{Cameron_1}, but we had to add missing details. In particular Cameron's proof left the case of $p=2$ that we instead cover precisely.
			
			\begin{proof}[ Proof of Theorem \ref{abelUPG}]
				By $\mathcal{P}(G_1)\cong \mathcal{P}(G_2)$, it follows $|G_1|=|G_2|$.
				
				Moreover, by Remark \ref{CardinalityS}, we have $|\mathcal{S}_1|=|\mathcal{S}_2|$. If now $G_1$ is cyclic, then $|\mathcal{S}_1|>1$ so that also $|\mathcal{S}_2|>1$ and, by Corollary \ref{abelUPG_ciclico}, $G_2$ is also cyclic. Thus $G_1$ and $G_2$ are both cyclic of the same order and $G_1\cong G_2$ holds. 
				
				Suppose next that $G_1$ and $G_2$ are both not cyclic.
				
				Let's focus on $G:=G_1$.
				By the power graph $\mathcal{P}(G)$ we are able to determine the $\mathtt{N}$-classes of $G$ and, by Corollary \ref{abeldiamond_N}, they are exactly its $\diamond$-classes. By Corollary \ref{abelUPG_ciclico}, since we are supposing $G$ not cyclic, we recognize the identity as the only star vertex of the power graph.

					
				Now, let $p$ be the smallest prime divisor of $|G|$.
				For the following we have to distinguish the case of $p\ne2$ and of $p=2$.
				\setcounter{claim}{0}
				\begin{claim}
					\label{claim1}
					\rm For $p\ne 2$, let $x\in G$ and $i\in \mathbb{N}$. Then $o(x)=p^i$ if and only if $|[x]_{\diamond}|=p^{i-1}(p-1)$ and $x$ has $p^{i-1}$ neighbours among the elements $z$ such that $o(z)\mid p^{i-1}$.
				\end{claim} 
				If $i=1$, then the claim become: $o(x)=p$ if and only if $|[x]_{\diamond}|=p-1$. Indeed, for all $x\in G$, it is true that $x$ has $p^0=1$ neighbour among the elements $z$ such that $o(z)\mid 1$, because $1$ is the only element $z$ with $o(z)\mid 1$.\\
				Assume first that $o(x)=p$. Then, by Lemma \ref{OssDiamondRelation}, $|[x]_{\diamond}|=\phi(p)=p-1$.\\
				Assume next that $|[x]_{\diamond}|=p-1$. Then, by Lemma \ref{OssDiamondRelation}, we have that $\phi(o(x))=p-1$.
				Suppose, by contradiction, that $o(x)\not= p$. We distinguish three cases:
				\begin{itemize}
					\item $o(x)=q$ for some prime $q\ne p$. Then $\phi(o(x))=q-1$, thus $q-1=p-1$, so $q=p$, a contradiction.
					
					\item $o(x)=q^l$, with $l\geq2$ and $q$ prime. Then we have $\phi(o(x))=q^{l-1}(q-1)$. Hence $p=q^{l-1}(q-1)+1>q$ holds. Against the fact that $p$ is the smallest prime divisor of $|G|$.
					
					\item  $o(x)=q^lm$, with $m\in \mathbb{N}\setminus 1$ such that $q\nmid m$, $l \in \mathbb{N}$ and $q$ a prime. Remember that $o(x)\mid |G|$, hence, since $m\mid o(x)$, we have that $m\ne 2$, otherwise $p=2$. Thus $m\geq 3$. We have that $\phi(o(x))=q^{l-1}(q-1)\phi(m)$ and thus $q^{l-1}(q-1)\phi(m)=p-1$. Note that $\phi(m)>1$ since $m\geq3$. In particular we have $p-1> q^{l-1}(q-1)$. Now if $l=1$ we have $p> q$. If instead $l\geq 2$ we have $p\geq q^{l-1}(q-1)+1>q$; both cases contradict the minimality of $p$.
				\end{itemize}
				Hence the claim for $i=1$ has been proved.
				
				Now we prove the claim for $i\geq 2$.\\
				Assume first $x\in G$ such that $o(x)=p^i$.  Then, by Lemma \ref{OssDiamondRelation}, $|[x]_{\diamond}|=\phi(p^i)=p^{i-1}(p-1)$. We also have that there exists a unique subgroup $X$ of $\langle x \rangle$ of size $p^{i-1}$ since $\langle x \rangle $ is cyclic of order $p^i$. So clearly $x$ has $p^{i-1}$ neighbours among the elements $z$ such that $o(z)\mid p^{i-1}$.\\
				Assume next that $|[x]_{\diamond}|=p^{i-1}(p-1)$ and $x$ has $p^{i-1}$ neighbours among the elements $z$ such that $o(z)\mid p^{i-1}$.
				Suppose, by contradiction, that $o(x)\not=p^{i-1}$.
				Let be $$X:=\{z\in G \, | \, \{z,x\}\in E \text{ and }o(z)\mid p^{i-1} \}.$$ Now $X\subset \langle x \rangle$. Indeed $y\in X$ implies $y\in \langle x \rangle $ or $x\in \langle y \rangle$ because $\{y,x\}\in E$ holds. If $x \in \langle y \rangle$ holds, we also have that $o(x)\mid o(y)\mid p^{i-1}$. Then $o(x)=p^k$ with $k\leq i-1$. Thus $\phi(o(x))=p^{k-1}(p-1)$, but also $\phi(o(x))=p^{i-1}(p-1)$. Therefore $p^{k-1}=p^{i-1}$ that implies $k=i$, against $k\leq i-1$. In particular we have reached a contradiction since $o(x)\mid p^{i-1}$, then we also have $x\notin X$, otherwise we would have $o(x)\mid p^{i-1}$. \\ We now want to prove that $X$ is a subgroup of $G$. Remember that $X\leq G$ if and only if $X=\langle X \rangle$. Suppose, by contradiction, that there exists $y\in \langle X \rangle\setminus X$. We have that $\langle X \rangle \leq \langle x \rangle$ since $X\subset \langle x \rangle$. Hence $y\in \langle x \rangle$ holds. We also have that, since $X$ has only elements of order a power of $p$ at most $p^{i-1}$ and $G$ is abelian, $o(y)$ is a power of $p$ at most $p^{i-1}$.  Therefore $y\in X$, a contradiction. Then $\langle x \rangle$ contains a subgroup of order $p^{i-1}$ because $|X|=p^{i-1}$. In particular $p^{i-1}$ properly divides $o(x)$ since $ X  \ne \langle x \rangle$. Note that, we did not use the hypothesis of $p\ne 2$ yet in this part of the proof with $i\geq2$.
				Now if $p^{i}\nmid o(x)$, then $o(x)=p^{i-1}m$ with $m\in \mathbb{N}\setminus \{1\}$ coprime with $p$. Note that $m\ne 2$ because $m\mid |G|$ and $2\nmid |G|$ otherwise $p=2$. Since $\phi(o(x))=p^{i-1}(p-1)$ but also $\phi(o(x))=p^{i-2}(p-1)\phi(m)$, we have that $\phi(m)=p$.
				Note that $\phi(m)$ is even because $m\ne 1$ and $m\ne 2$ hold. Therefore, by $\phi(m)=p$, we have $p=2$; a contradiction. So $p^i\mid o(x)$ and then, by Lemma \ref{phiEulero}, since $\phi(o(x))=\phi(p^i)$, we have that either $o(x)=p^i$, that is what we want, or $o(x)=2p^i$. We exclude the second case because, otherwise, $2\mid |G|$ and then $p=2$.
				
				
				\begin{claim}
					\label{claim2}
					\rm Let $x\in G$ and $i\in \mathbb{N}$. Then $o(x)=2^i$ if and only if $[x]_{\diamond}\ne [1]_{\diamond}$, $|[x]_{\diamond}|=2^{i-1}$, $x$ has $2^{i-1}$ neighbours among the elements $z$ such that $o(z)\mid 2^{i-1}$ and $x$ has no neighbours among the elements of order $3$. 
				\end{claim}
				If $i=1$ the claim become: $o(x)=2$ if and only if $[x]_{\diamond}\ne [1]_{\diamond}$, $|[x]_{\diamond}|=1$ and $x$ has no neighbours among the elements of order $3$. Indeed, for all $x\in G$, it is true that $x$ has $1$ neighbour among the elements $z$ such that $o(z)\mid 1$, because $1$ is the only element $z$ with $o(z)\mid 1$.\\
				Assume first that $o(x)=2$. Then clearly $[x]_{\diamond}\ne [1]_{\diamond}$, $|[x]_{\diamond}|=1$ and $x$ has no neighbours among the elements of order $3$.\\
				Assume next that $[x]_{\diamond}\ne [1]_{\diamond}$, $|[x]_{\diamond}|=1$ and $x$ has no neighbours among the elements of order $3$. Then, by Lemma \ref{OssDiamondRelation}, $\phi(o(x))=1$. Thus $o(x)=2$ or $o(x)=1$. If $o(x)=1$, then $x=1$ against the assumption that $[x]_{\diamond}\ne[1]_{\diamond}$. 
				
				Now we prove the claim for $i\geq2$.\\
				Assume first that $o(x)=2^i$. Then clearly $[x]_{\diamond}\ne [1]_{\diamond}$, $|[x]_{\diamond}|=2^{i-1}$, $x$ has $2^{i-1}$ neighbours among the elements $z$ such that $o(z)\mid 2^{i-1}$ and $x$ has no neighbours among the elements of order $3$.\\
				Assume next that $[x]_{\diamond}\ne [1]_{\diamond}$, $|[x]_{\diamond}|=2^{i-1}$, $x$ has $2^{i-1}$ neighbours among the elements $z$ such that $o(z)\mid 2^{i-1}$ and $x$ has no neighbours among the elements of order $3$. 
				Then, as for the previous claim, $\langle x \rangle $ contains a proper subgroup of order $2^{i-1}$, that is the subgroup $$X:=\{z\in G \, | \, \{z,x\}\in E \text{ and }o(z)\mid 2^{i-1} \}.$$ In particular $2^{i-1}$ properly divides $o(x)$. If $2^i\nmid o(x)$, then $o(x)=2^{i-1}m$ with $m\in \mathbb{N}\setminus \{1\}$ odd. So we have $\phi(o(x))=2^{i-1}$ but also $\phi(o(x))=2^{i-2}\phi(m)$; thus $\phi(m)=2$ holds. Note that $\phi(n)=2$ if and only if $n\in \{3,4,6\}$.
				Therefore, since $m$ is odd, we have that $m=3$ and then $x$ must have to be joined to an element of order $3$, against the assumptions.
				Hence $2^i\mid o(x)$ holds. By Lemma \ref{phiEulero}, since $\phi(o(x))=\phi(2^i)$, we have that $o(x)=2^i$ as we wanted.
				
				\begin{claim}
					\rm We recognize all the vertices that correspond to elements of order $p^i$ for all $i \in \mathbb{N}$.
				\end{claim}
				We distinguish the cases $p\ne 2$ and $p=2$ and, for both, we proceed by induction on $i$.\\
				Assume $p\ne2$. 
				By Claim \ref{claim1}, we recognize the elements of order $p$ as those elements lying in $\diamond$-classes of size $p-1$.
				So, suppose the claim true for all $j\leq i-1 \geq 1$. By inductive hypothesis we recognize all the elements of order a divisor of $p^{i-1}$. Then, by Claim \ref{claim1}, we recognize the elements of order $p^i$ as those vertices lying in $\diamond$-classes of size $p^{i-1}(p-1)$ and having $p^{i-1}$ among the elements of order a divisor of $p^{i-1}$.\\
				Assume next $p=2$.
				By Claim \ref{claim2}, we recognize the elements of order $2$ as those vertices lying in $\diamond$-classes of size $1$ distinct from the class $[1]_{\diamond}$.
				Note that we can also recognize the elements of order $3$ since they are the elements in $\diamond$-classes of order $2$ and such that they not have neighbours among the elements of order $2$. Indeed $\phi(n)=2$ if and only if $n\in \{3,4,6\}$.
				Then, suppose the claim true for all $j\leq i-1\geq 1$. By inductive hypothesis we recognize all the elements of order a divisor of $2^{i-1}$. Then, by Claim \ref{claim2}, we recognize the elements of order $2^i$ as those vertices lying in $\diamond$-classes of size $2^{i-1}$, having $2^{i-1}$ neighbours among the elements of order a divisor of $2^{i-1}$ and not having neighbours among the elements of order $3$.
				
				Therefore we recognize the Sylow $p$-subgroup $P$ of $G$, but, currently, not its theoretical structure. By Corollary \ref{structureTheoremOfAbelianpGroups}, we have that each Sylow subgroup is direct product of cyclic groups. Then since we know the number of elements of each order in $P$ we know how $P$ decomposes into cyclic $p$-groups. By Theorem \ref{structureTheoremOfNilpotentGroup}, $G$ is direct product of its Sylow subgroups, in particular $G\cong P\times H$ for $H$ the direct product of all Sylows distinct from $P$. Note that $H$ is also the unique complement of $P$. Now $H$ consists of those elements which have no neighbours in $P$ apart from the identity. By Remark \ref{InducedPG}, we have that the subgraph induced by $H$ is $\mathcal{P}(H)$. By induction on the order of $G$, we can determine $H$ up to group isomorphism.
				
				Note that all we did above on $G_1$ is invariant by graph isomorphism.
				Indeed we only use the structure of the $\mathtt{N}$-classes of $\mathcal{P}(G_1)$ that, by Lemma \ref{Niso}, is maintained by graphs isomorphism.
				So, since $\mathcal{P}(G_1)\cong \mathcal{P}(G_2)$, we have $G_1\cong G_2$.
				In conclusion, by Remark \ref{isoGroup-isoGraph}, since $G_1\cong G_2$, we also have $\vec{\mathcal{P}}(G_1)\cong \vec{\mathcal{P}}(G_2)$.
			\end{proof}
			
			So if $G$ is an abelian group we not only can reconstruct $\vec{\mathcal{P}}(G)$ from $\mathcal{P}(G)$ but we can also determine $G$, up to group isomorphism from $\mathcal{P}(G)$. 
			
		\section{Case II: Generic groups}
	\label{3}
			
			If in Theorem \ref{abelUPG} we delete the hypothesis that the two groups are abelian the theorem fails. Indeed take a group of exponent $3$, which means that $x^3=1$ for all $x \in G$. Then $\mathcal{P}(G)$ consists of $(|G|-1)/2$ triangles sharing a common vertex (the identity). Then two non isomorphic of such groups have the same power graph. For example let $G_1=C_3\times C_3 \times C_3$ and $G_2$ be the unique non abelian group of order $27$ of exponent $3$.
			
			So, in general, $\mathcal{P}(G)$ does not uniquely determine the group $G$ up to group isomorphism. But our starting question remains:
			
			\emph{Are we able to reconstruct $\vec{\mathcal{P}}(G)$ from $\mathcal{P}(G)$?}
			
			We recall that we can reconstruct $\vec{\mathcal{P}}(G)$ from $\mathcal{P}(G)$ if we recognize the identity and all the $\diamond$-classes in $\mathcal{P}(G)$.
			 
			Let's focus on the research of the $\diamond$-classes.
			To find them, we also recall that we only need a way to break $\mathtt{N}$-classes since they are union of $\diamond$-classes.
			The proposition that follows takes a step in that direction.

				\begin{prop}{\rm \cite[Proposition 5]{Cameron_2}}
					\label{propC_y}Let $C$ be an $\mathtt{N}$-class different from the star class. Then exactly one of the following holds:
					
					\begin{enumerate}
						\item $C$ is a $\diamond$-class;
						
						\item If $y$ is an element in $C$ of maximum order, then $o(y)=p^r$ for some prime $p$ and some integer $r \geq 2$ and there exists $s\in [r-2]_{0} $ such that 
						
						$$C= \left\lbrace z \in \langle y\rangle \, |\, p^{s+1}\leq o(z)\leq p^r \right\rbrace. $$ 
						
						In particular the number of $\diamond$-classes into which $C$ splits is $r-s \geq 2$. The orders of the elements in those $\diamond$-classes are $p^{s+1},p^{s+2}, \dots, p^r$.
					\end{enumerate} 	
				\end{prop}
				
				\begin{proof}
					If $C$ has all elements with the same order then, by Lemma \ref{PrimoLegameDueClassi}, $C$ is a $\diamond$-class.
					
					Assume next that there exist $x,y\in C$ such that $o(x)<o(y)$. Then $x\neq y$. Note also that both $x$ and $y$ are different from $1$ as $C \not= \mathcal{S}$.
					Since $x\mathtt{N}y$, we have $N[x]=N[y]$. In particular, we deduce that $x$ and $y$ are joined and $o(x)$ is a proper divisor of $o(y)$. Let then $k:=o(x)$ and $kl:=o(y)$ for some $k,l\geq 2$. We show that both $k$ and $l$ are power of the same prime. It suffices to show that if $p$ is a prime dividing $k$, then the only prime dividing $l$ is $p$ itself. Assume, by contradiction,  that  we have a  prime $p\mid k$ and a prime $q\neq p$ such that $q\mid l.$ Then $(kq)/p=(k/p)q\mid kl=o(y)$ and thus there exists an element $z\in \langle y\rangle$ with $o(z)=(k/p)q$. Now $z\in N[y]=N[x]$. But we do not have $o(z)\mid o(x)$ nor $o(x)\mid o(z)$. Indeed assume that $o(z)\mid o(x)$. Then $(k/p)q\mid k=p(k/p)$ which implies the contradiction $q\mid p$. Assume next $o(x)\mid o(z)$. Then $k\mid (k/p)q$ and thus $p(k/p)\mid q(k/p)$ so that we reach the contradiction $p\mid q.$
					Then all the elements in $C$ have order a power of the same prime $p.$
					
					Choose now $x\in C$ such that $o(x)$ is minimum in $C$ and $y\in C$ such that $o(y)$ is maximum in $C$. Since those orders are distinct and different from $1$, we have that $o(x)=p^{s+1}$ for some $s\geq 0$ and $o(y)=p^{r}$ for some $r\geq s+2$. We are ready to show that $C$ has the form in the statement with respect to that $y$. 
					Pick $z\in C$. Since $z, y\in C$ we have $z\mathtt{N}y$ and since $o(z)\leq o(y)$, we have that $z$ is a power of $y$. Moreover $o(z)\geq p^{s+1}$. This shows that $$C\subseteq \{z\in \langle y\rangle: p^{s+1}\leq o(z)\leq p^r\}.$$ 
					Let next $z\in \langle y\rangle$ with $ o(z)\geq p^{s+1}$. Then $z\in N[y]=N[x]$ so that $z=x\in C$ or $z$ and $x$ are joined. In this last case we have that $x$ and $z$ are of prime power order and we also have that $o(z)\geq p^{s+1}=o(x),$ so we deduce that $x$ is power of $z$. Hence, by Lemma \ref{NdipotenzeDip}, $N[x]\supseteq N[z]$. Moreover, since $z$ is a power of $y$, again by Lemma \ref{NdipotenzeDip}, we also have $N[z]\supseteq N[y]=N(x)$ and thus $N[x]= N[z]$, so that $z\in C.$ Thus $C= \{z\in \langle y\rangle: o(z)\geq p^{s+1}\}.$
				\end{proof}
				
				The previous proposition highlights that there are two distinct types of $\mathtt{N}$-classes different from the star class. Let's distinguish them formally with new definitions. 
				
				\begin{definition}
					\label{NfirstType}\rm Let $C$ an $\mathtt{N}$-class different from the star class. If $C$ is also a $\diamond$-class, then $C$ is said to be a \emph{ $\mathtt{N}$-class of the first type}. In particular, for all $y \in C$, $C=[y]_{\mathtt{N}}=[y]_{\diamond}$.
				\end{definition}
			
				\begin{definition}
					\label{NsecondType}\rm Let $C$ an $\mathtt{N}$-class. If there exist a prime $p$, an integer $r\geq 2$ and $y\in C$ with $o(y)=p^r$ and also an integer $s \in [r-2]_0$ such that $$C=\{z \in \langle y \rangle \, | \, p^{s+1} \leq o(z) \leq p^r\},$$ then we denote $C$ with $C_y$, and we call $C_y$ a \emph{$\mathtt{N}$-class of the second type}, we call the triad $(p,r,s)$ the \emph{parameters of} $C$.
				\end{definition}
				
				If $C$ is either an $\mathtt{N}$-class of the first or of the second type, then we have that $C\subset \langle y \rangle$ for some $y \in C$.
				
				Note that with $C_y$ we refer exclusively to $\mathtt{N}$-classes of the second type. In particular such a class $C_y$ is not a $\diamond$-class, since contains elements of distinct orders. It is not the star class either because it does not contain $1$. 
				Then, by Proposition \ref{propC_y} and by Definitions \ref{NfirstType} and \ref{NsecondType}, taken $C$, an $\mathtt{N}$-class, we know that holds exactly one of these three cases:
				
				\begin{itemize}
					
					\item \emph{$C$ is of the first type},
					
					\item \emph{$C$ is of the second type} or
					
					\item \emph{$C$ is the star class}.
				\end{itemize}

				For an example of each type of $\mathtt{N}$-classes look at Figure \ref{UPG_S_4_NeDiamond} where $\mathcal{P}(S_4)$ is represented with all the $\mathtt{N}$ and $\diamond$-classes highlighted. The $\mathtt{N}$-class of the identity, the white vertex, is clearly the star class. The classes of the elements coloured with a shade of green, red or pink are examples of $\mathtt{N}$-classes of the first type. The remaining vertices, the ones in a shade of yellow/brown, form three $\mathtt{N}$-classes of the second type. If we examine in details the $\mathtt{N}$-classes of the second type of this example we obtain that the parameters $(p,r,s)$ are exactly: $p=2, r=2$ and $s=0$. 
				
				Remember that, by Corollary \ref{abeldiamond_N_1}, if $G$ is an abelian group, then an $\mathtt{N}$-class of $G$ can only be the star class or a class of the first type.
			
				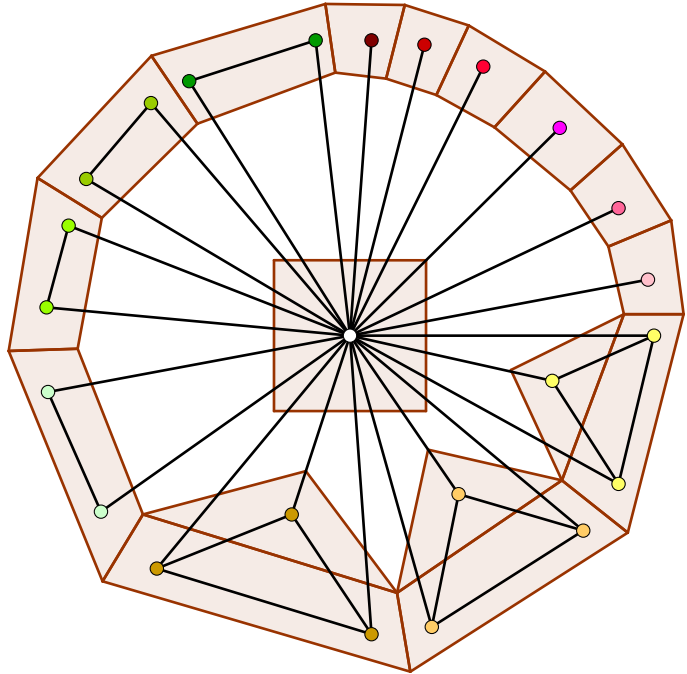
\begin{figure}
					\centering
					\begin{tikzpicture}[line cap=round,line join=round,>=triangle 45,x=1.0cm,y=1.0cm]
						\clip(-6.,2.) rectangle (4.,12.);
						\fill[line width=2.pt,color=zzttqq,fill=zzttqq,fill opacity=0.10000000149011612] (-3.00711852963106,9.809911400437345) -- (-1.1933294275569224,10.487257783623239) -- (-1.3240850694054276,11.397934907263139) -- (-3.6127226668446086,10.710096634610261) -- cycle;
						\fill[line width=2.pt,color=zzttqq,fill=zzttqq,fill opacity=0.10000000149011612] (-5.110910354913101,9.090550324695746) -- (-4.265322159020893,8.56129342866397) -- (-3.00711852963106,9.809911400437345) -- (-3.6127226668446086,10.710096634610261) -- cycle;
						\fill[line width=2.pt,color=zzttqq,fill=zzttqq,fill opacity=0.10000000149011612] (-5.110910354913101,9.090550324695746) -- (-5.490950168512438,6.796738592614046) -- (-4.58366337471664,6.826333803122147) -- (-4.265322159020893,8.56129342866397) -- cycle;
						\fill[line width=2.pt,color=zzttqq,fill=zzttqq,fill opacity=0.10000000149011612] (-3.7241420923381203,4.629779414821492) -- (-4.255820774314594,3.7428472333336757) -- (-5.490950168512438,6.796738592614046) -- (-4.58366337471664,6.826333803122147) -- cycle;
						\fill[line width=2.pt,color=zzttqq,fill=zzttqq,fill opacity=0.10000000149011612] (2.604291484014996,7.284127314600211) -- (2.400341814808526,8.181505859108677) -- (3.22633797509473,8.528220296759674) -- (3.389497710459906,7.284127314600211) -- cycle;
						\fill[line width=2.pt,color=zzttqq,fill=zzttqq,fill opacity=0.10000000149011612] (2.400341814808526,8.181505859108677) -- (1.9006651252526743,8.92592215171229) -- (2.5838965170943493,9.5377711593317) -- (3.22633797509473,8.528220296759674) -- cycle;
						\fill[line width=2.pt,color=zzttqq,fill=zzttqq,fill opacity=0.10000000149011612] (1.9006651252526743,8.92592215171229) -- (0.9013117461409711,9.762115795458815) -- (1.5669320840466672,10.502126775207563) -- (2.5838965170943493,9.5377711593317) -- cycle;
						\fill[line width=2.pt,color=zzttqq,fill=zzttqq,fill opacity=0.10000000149011612] (0.9013117461409711,9.762115795458815) -- (0.13650048661670827,10.1904101007924) -- (0.5625411481055639,11.112905047063636) -- (1.5669320840466672,10.502126775207563) -- cycle;
						\fill[line width=2.pt,color=zzttqq,fill=zzttqq,fill opacity=0.10000000149011612] (0.13650048661670827,10.1904101007924) -- (-0.5248128745958504,10.407672479699302) -- (-0.2789755820072525,11.384362056777448) -- (0.5625411481055639,11.112905047063636) -- cycle;
						\fill[line width=2.pt,color=zzttqq,fill=zzttqq,fill opacity=0.10000000149011612] (-0.2789755820072525,11.384362056777448) -- (-1.3240850694054276,11.397934907263139) -- (-1.1933294275569224,10.487257783623239) -- (-0.5248128745958504,10.407672479699302) -- cycle;
						\fill[line width=2.pt,color=zzttqq,fill=zzttqq,fill opacity=0.10000000149011612] (-2.,8.) -- (8.905437690318596E-4,7.999107867261943) -- (0.,6.) -- (-2.,6.) -- cycle;
						\fill[line width=2.pt,color=zzttqq,fill=zzttqq,fill opacity=0.10000000149011612] (-4.255820774314594,3.7428472333336757) -- (-0.20647165890010305,2.544644452014348) -- (-0.3808547962166728,3.589809124493197) -- (-3.7241420923381203,4.629779414821492) -- cycle;
						\fill[line width=2.pt,color=zzttqq,fill=zzttqq,fill opacity=0.10000000149011612] (-0.3808547962166728,3.589809124493197) -- (1.7884928071891157,5.08147088717034) -- (2.655278901316614,4.388042011868344) -- (-0.20647165890010305,2.544644452014348) -- cycle;
						\fill[line width=2.pt,color=zzttqq,fill=zzttqq,fill opacity=0.10000000149011612] (1.7884928071891157,5.08147088717034) -- (2.604291484014996,7.284127314600211) -- (3.389497710459906,7.284127314600211) -- (2.655278901316614,4.388042011868344) -- cycle;
						\fill[line width=2.pt,color=zzttqq,fill=zzttqq,fill opacity=0.10000000149011612] (2.604291484014996,7.284127314600211) -- (1.1171708774262428,6.54424864751115) -- (1.7884928071891157,5.08147088717034) -- cycle;
						\fill[line width=2.pt,color=zzttqq,fill=zzttqq,fill opacity=0.10000000149011612] (1.7884928071891157,5.08147088717034) -- (0.027873487448356438,5.488419264647737) -- (-0.3808547962166728,3.589809124493197) -- cycle;
						\fill[line width=2.pt,color=zzttqq,fill=zzttqq,fill opacity=0.10000000149011612] (-0.3808547962166728,3.589809124493197) -- (-1.580788935019527,5.203823602504077) -- (-3.7241420923381203,4.629779414821492) -- cycle;
						\draw [line width=1.pt,color=zzttqq] (-3.00711852963106,9.809911400437345)-- (-1.1933294275569224,10.487257783623239);
						\draw [line width=1.pt,color=zzttqq] (-1.1933294275569224,10.487257783623239)-- (-1.3240850694054276,11.397934907263139);
						\draw [line width=1.pt,color=zzttqq] (-1.3240850694054276,11.397934907263139)-- (-3.6127226668446086,10.710096634610261);
						\draw [line width=1.pt,color=zzttqq] (-3.6127226668446086,10.710096634610261)-- (-3.00711852963106,9.809911400437345);
						\draw [line width=1.pt,color=zzttqq] (-5.110910354913101,9.090550324695746)-- (-4.265322159020893,8.56129342866397);
						\draw [line width=1.pt,color=zzttqq] (-4.265322159020893,8.56129342866397)-- (-3.00711852963106,9.809911400437345);
						\draw [line width=1.pt,color=zzttqq] (-3.00711852963106,9.809911400437345)-- (-3.6127226668446086,10.710096634610261);
						\draw [line width=1.pt,color=zzttqq] (-3.6127226668446086,10.710096634610261)-- (-5.110910354913101,9.090550324695746);
						\draw [line width=1.pt,color=zzttqq] (-5.110910354913101,9.090550324695746)-- (-5.490950168512438,6.796738592614046);
						\draw [line width=1.pt,color=zzttqq] (-5.490950168512438,6.796738592614046)-- (-4.58366337471664,6.826333803122147);
						\draw [line width=1.pt,color=zzttqq] (-4.58366337471664,6.826333803122147)-- (-4.265322159020893,8.56129342866397);
						\draw [line width=1.pt,color=zzttqq] (-4.265322159020893,8.56129342866397)-- (-5.110910354913101,9.090550324695746);
						\draw [line width=1.pt,color=zzttqq] (-3.7241420923381203,4.629779414821492)-- (-4.255820774314594,3.7428472333336757);
						\draw [line width=1.pt,color=zzttqq] (-4.255820774314594,3.7428472333336757)-- (-5.490950168512438,6.796738592614046);
						\draw [line width=1.pt,color=zzttqq] (-5.490950168512438,6.796738592614046)-- (-4.58366337471664,6.826333803122147);
						\draw [line width=1.pt,color=zzttqq] (-4.58366337471664,6.826333803122147)-- (-3.7241420923381203,4.629779414821492);
						\draw [line width=1.pt,color=zzttqq] (2.604291484014996,7.284127314600211)-- (2.400341814808526,8.181505859108677);
						\draw [line width=1.pt,color=zzttqq] (2.400341814808526,8.181505859108677)-- (3.22633797509473,8.528220296759674);
						\draw [line width=1.pt,color=zzttqq] (3.22633797509473,8.528220296759674)-- (3.389497710459906,7.284127314600211);
						\draw [line width=1.pt,color=zzttqq] (3.389497710459906,7.284127314600211)-- (2.604291484014996,7.284127314600211);
						\draw [line width=1.pt,color=zzttqq] (2.400341814808526,8.181505859108677)-- (1.9006651252526743,8.92592215171229);
						\draw [line width=1.pt,color=zzttqq] (1.9006651252526743,8.92592215171229)-- (2.5838965170943493,9.5377711593317);
						\draw [line width=1.pt,color=zzttqq] (2.5838965170943493,9.5377711593317)-- (3.22633797509473,8.528220296759674);
						\draw [line width=1.pt,color=zzttqq] (3.22633797509473,8.528220296759674)-- (2.400341814808526,8.181505859108677);
						\draw [line width=1.pt,color=zzttqq] (1.9006651252526743,8.92592215171229)-- (0.9013117461409711,9.762115795458815);
						\draw [line width=1.pt,color=zzttqq] (0.9013117461409711,9.762115795458815)-- (1.5669320840466672,10.502126775207563);
						\draw [line width=1.pt,color=zzttqq] (1.5669320840466672,10.502126775207563)-- (2.5838965170943493,9.5377711593317);
						\draw [line width=1.pt,color=zzttqq] (2.5838965170943493,9.5377711593317)-- (1.9006651252526743,8.92592215171229);
						\draw [line width=1.pt,color=zzttqq] (0.9013117461409711,9.762115795458815)-- (0.13650048661670827,10.1904101007924);
						\draw [line width=1.pt,color=zzttqq] (0.13650048661670827,10.1904101007924)-- (0.5625411481055639,11.112905047063636);
						\draw [line width=1.pt,color=zzttqq] (0.5625411481055639,11.112905047063636)-- (1.5669320840466672,10.502126775207563);
						\draw [line width=1.pt,color=zzttqq] (1.5669320840466672,10.502126775207563)-- (0.9013117461409711,9.762115795458815);
						\draw [line width=1.pt,color=zzttqq] (0.13650048661670827,10.1904101007924)-- (-0.5248128745958504,10.407672479699302);
						\draw [line width=1.pt,color=zzttqq] (-0.5248128745958504,10.407672479699302)-- (-0.2789755820072525,11.384362056777448);
						\draw [line width=1.pt,color=zzttqq] (-0.2789755820072525,11.384362056777448)-- (0.5625411481055639,11.112905047063636);
						\draw [line width=1.pt,color=zzttqq] (0.5625411481055639,11.112905047063636)-- (0.13650048661670827,10.1904101007924);
						\draw [line width=1.pt,color=zzttqq] (-0.2789755820072525,11.384362056777448)-- (-1.3240850694054276,11.397934907263139);
						\draw [line width=1.pt,color=zzttqq] (-1.3240850694054276,11.397934907263139)-- (-1.1933294275569224,10.487257783623239);
						\draw [line width=1.pt,color=zzttqq] (-1.1933294275569224,10.487257783623239)-- (-0.5248128745958504,10.407672479699302);
						\draw [line width=1.pt,color=zzttqq] (-0.5248128745958504,10.407672479699302)-- (-0.2789755820072525,11.384362056777448);
						\draw [line width=1.pt,color=zzttqq] (-2.,8.)-- (8.905437690318596E-4,7.999107867261943);
						\draw [line width=1.pt,color=zzttqq] (8.905437690318596E-4,7.999107867261943)-- (0.,6.);
						\draw [line width=1.pt,color=zzttqq] (0.,6.)-- (-2.,6.);
						\draw [line width=1.pt,color=zzttqq] (-2.,6.)-- (-2.,8.);
						\draw [line width=1.pt,color=zzttqq] (-4.255820774314594,3.7428472333336757)-- (-0.20647165890010305,2.544644452014348);
						\draw [line width=1.pt,color=zzttqq] (-0.20647165890010305,2.544644452014348)-- (-0.3808547962166728,3.589809124493197);
						\draw [line width=1.pt,color=zzttqq] (-0.3808547962166728,3.589809124493197)-- (-3.7241420923381203,4.629779414821492);
						\draw [line width=1.pt,color=zzttqq] (-3.7241420923381203,4.629779414821492)-- (-4.255820774314594,3.7428472333336757);
						\draw [line width=1.pt,color=zzttqq] (-0.3808547962166728,3.589809124493197)-- (1.7884928071891157,5.08147088717034);
						\draw [line width=1.pt,color=zzttqq] (1.7884928071891157,5.08147088717034)-- (2.655278901316614,4.388042011868344);
						\draw [line width=1.pt,color=zzttqq] (2.655278901316614,4.388042011868344)-- (-0.20647165890010305,2.544644452014348);
						\draw [line width=1.pt,color=zzttqq] (-0.20647165890010305,2.544644452014348)-- (-0.3808547962166728,3.589809124493197);
						\draw [line width=1.pt,color=zzttqq] (1.7884928071891157,5.08147088717034)-- (2.604291484014996,7.284127314600211);
						\draw [line width=1.pt,color=zzttqq] (2.604291484014996,7.284127314600211)-- (3.389497710459906,7.284127314600211);
						\draw [line width=1.pt,color=zzttqq] (3.389497710459906,7.284127314600211)-- (2.655278901316614,4.388042011868344);
						\draw [line width=1.pt,color=zzttqq] (2.655278901316614,4.388042011868344)-- (1.7884928071891157,5.08147088717034);
						\draw [line width=1.pt,color=zzttqq] (2.604291484014996,7.284127314600211)-- (1.1171708774262428,6.54424864751115);
						\draw [line width=1.pt,color=zzttqq] (1.1171708774262428,6.54424864751115)-- (1.7884928071891157,5.08147088717034);
						\draw [line width=1.pt,color=zzttqq] (1.7884928071891157,5.08147088717034)-- (2.604291484014996,7.284127314600211);
						\draw [line width=1.pt,color=zzttqq] (1.7884928071891157,5.08147088717034)-- (0.027873487448356438,5.488419264647737);
						\draw [line width=1.pt,color=zzttqq] (0.027873487448356438,5.488419264647737)-- (-0.3808547962166728,3.589809124493197);
						\draw [line width=1.pt,color=zzttqq] (-0.3808547962166728,3.589809124493197)-- (1.7884928071891157,5.08147088717034);
						\draw [line width=1.pt,color=zzttqq] (-0.3808547962166728,3.589809124493197)-- (-1.580788935019527,5.203823602504077);
						\draw [line width=1.pt,color=zzttqq] (-1.580788935019527,5.203823602504077)-- (-3.7241420923381203,4.629779414821492);
						\draw [line width=1.pt,color=zzttqq] (-3.7241420923381203,4.629779414821492)-- (-0.3808547962166728,3.589809124493197);
						\draw [line width=1.pt] (-3.6189179549840818,10.083616523866972)-- (-1.,7.);
						\draw [line width=1.pt] (-4.470198577067798,9.077557606858948)-- (-1.,7.);
						\draw [line width=1.pt] (-4.702366019454266,8.458444427161703)-- (-1.,7.);
						\draw [line width=1.pt] (-4.992575322437351,7.374996362691523)-- (-1.,7.);
						\draw [line width=1.pt] (-4.973228035571812,6.252853724490266)-- (-1.,7.);
						\draw [line width=1.pt] (-4.276725708412408,4.666376201516074)-- (-1.,7.);
						\draw [line width=1.pt] (-3.541528807521925,3.9118320137600566)-- (-1.,7.);
						\draw [line width=1.pt] (-1.7666619570249664,4.628983561782251)-- (-1.,7.);
						\draw [line width=1.pt] (-0.7168249251532304,3.041204104810805)-- (-1.,7.);
						\draw [line width=1.pt] (0.07641383633386878,3.1379405391384996)-- (-1.,7.);
						\draw [line width=1.pt] (0.428937407628609,4.898618571476549)-- (-1.,7.);
						\draw [line width=1.pt] (2.069184383484386,4.414861472264069)-- (-1.,7.);
						\draw [line width=1.pt] (2.533519268257322,5.033974651961314)-- (-1.,7.);
						\draw [line width=1.pt] (1.661554594802546,6.400870768344779)-- (-1.,7.);
						\draw [line width=1.pt] (3.,7.)-- (-1.,7.);
						\draw [line width=1.pt] (2.920465005568102,7.742594813136765)-- (-1.,7.);
						\draw [line width=1.pt] (2.533519268257322,8.69061186954817)-- (-1.,7.);
						\draw [line width=1.pt] (1.759627793635761,9.754712647152811)-- (-1.,7.);
						\draw [line width=1.pt] (0.7535688766277328,10.567298695505446)-- (-1.,7.);
						\draw [line width=1.pt] (-0.02032259799382734,10.85750799848853)-- (-1.,7.);
						\draw [line width=1.pt] (-0.7168249251532315,10.915549859085147)-- (-1.,7.);
						\draw [line width=1.pt] (-1.4520218260437137,10.915549859085147)-- (-1.,7.);
						\draw [line width=1.pt] (-3.1158884964800677,10.373825826850057)-- (-1.,7.);
						
						\draw [line width=1.pt] (-4.470198577067798,9.077557606858948)-- (-3.6189179549840818,10.083616523866972);
						\draw [line width=1.pt] (-3.1158884964800677,10.373825826850057)-- (-1.4520218260437137,10.915549859085147);
						\draw [line width=1.pt] (-4.702366019454266,8.458444427161703)-- (-4.992575322437351,7.374996362691523);
						\draw [line width=1.pt] (-4.973228035571812,6.252853724490266)-- (-4.276725708412408,4.666376201516074);
						\draw [line width=1.pt] (-3.541528807521925,3.9118320137600566)-- (-1.7666619570249664,4.628983561782251);
						\draw [line width=1.pt] (-3.541528807521925,3.9118320137600566)-- (-0.7168249251532304,3.041204104810805);
						\draw [line width=1.pt] (-1.7666619570249664,4.628983561782251)-- (-0.7168249251532304,3.041204104810805);
						\draw [line width=1.pt] (0.428937407628609,4.898618571476549)-- (0.07641383633386878,3.1379405391384996);
						\draw [line width=1.pt] (0.428937407628609,4.898618571476549)-- (2.069184383484386,4.414861472264069);
						\draw [line width=1.pt] (0.07641383633386878,3.1379405391384996)-- (2.069184383484386,4.414861472264069);
						\draw [line width=1.pt] (1.661554594802546,6.400870768344779)-- (2.533519268257322,5.033974651961314);
						\draw [line width=1.pt] (1.661554594802546,6.400870768344779)-- (3.,7.);
						\draw [line width=1.pt] (3.,7.)-- (2.533519268257322,5.033974651961314);
						\begin{scriptsize}
							\draw [fill=ffffww] (2.533519268257322,5.033974651961314) circle (2.5pt);
							\draw [fill=ffffww] (1.661554594802546,6.400870768344779) circle (2.5pt);
							\draw [fill=ffffww] (3.,7.) circle (2.5pt);
							
							\draw [fill=cczzqq] (-0.7168249251532304,3.041204104810805) circle (2.5pt);
							\draw [fill=cczzqq] (-3.541528807521925,3.9118320137600566) circle (2.5pt);
							\draw [fill=cczzqq] (-1.7666619570249664,4.628983561782251) circle (2.5pt);
							
							\draw [fill=ffccww] (0.07641383633386878,3.1379405391384996) circle (2.5pt);
							\draw [fill=ffccww] (0.428937407628609,4.898618571476549) circle (2.5pt);
							\draw [fill=ffccww] (2.069184383484386,4.414861472264069) circle (2.5pt);
							
							\draw [fill=qqzzqq] (-1.4520218260437137,10.915549859085147) circle (2.5pt);
							\draw [fill=qqzzqq] (-3.1158884964800677,10.373825826850057) circle (2.5pt);
							
							\draw [fill=zzccqq] (-4.470198577067798,9.077557606858948) circle (2.5pt);
							\draw [fill=zzccqq] (-3.6189179549840818,10.083616523866972) circle (2.5pt);
							
							\draw [fill=zzffqq] (-4.702366019454266,8.458444427161703) circle (2.5pt);
							\draw [fill=zzffqq] (-4.992575322437351,7.374996362691523) circle (2.5pt);
							
							\draw [fill=ccffcc] (-4.973228035571812,6.252853724490266) circle (2.5pt);
							\draw [fill=ccffcc] (-4.276725708412408,4.666376201516074) circle (2.5pt);
							
							\draw [fill=ffttww] (1.759627793635761,9.754712647152811) circle (2.5pt);
							
							\draw [fill=yqqqqq] (-0.7168249251532315,10.915549859085147) circle (2.5pt);
								
							\draw [fill=ccqqqq] (-0.02032259799382734,10.85750799848853) circle (2.5pt);
						
							\draw [fill=ffcqcb] (2.920465005568102,7.742594813136765) circle (2.5pt);
							
							\draw [fill=ffwwzz] (2.533519268257322,8.69061186954817) circle (2.5pt);
							
							\draw [fill=ffqqtt] (0.7535688766277328,10.567298695505446) circle (2.5pt);
							
							\draw [fill=ffffff] (-1.,7.) circle (2.5pt);
						\end{scriptsize}
					\end{tikzpicture}
					\caption{$\mathcal{P}(S_4)$ where all the vertices with the same colour are in the same $\mathtt{N}$-class, and vertices in the same box are in the same $\diamond$-class.}
					\label{UPG_S_4_NeDiamond}
				\end{figure}
							
				The distinction of the types of $\mathtt{N}$-classes is going to allow us to break $\mathtt{N}$-classes in $\diamond$-classes. In general, at the moment, we are not able to recognize which of the above three types is a given $\mathtt{N}$-class.
				We have to introduce a new player that will help us in the distinction of the type of all $\mathtt{N}$-classes.
				
				\begin{definition}
					\rm For any subset $X$ of a group $G$, let $N[X]:=\bigcap_{x\in X}  N[x]$, and let $\hat{X}=N[N[X]]$.
				\end{definition}
				
				Before using this new definition we expose some immediate results related to it, all contained in the following lemma. Note that, in the article \cite{Cameron_2}, Cameron never strictly refers to some of the results contained in the lemma below, but they seem to be quite necessary and useful to make next proofs formal and clear. 
				
				\begin{lemma}
					\label{operatoreChiusura} Let $G$ be a group and $C$ be an $\mathtt{N}$-class. Then the followings hold:
					\begin{enumerate}
						\item[\rm i)] If $A \subseteq B\subseteq G$ then $N[A] \supseteq N[B]$;
						
						\item[\rm ii)] For all $X\subseteq G$ we have that $X\subseteq \hat{X}$ and $\mathcal{S} \subseteq \hat{X}$;
						
						\item[\rm iii)] If $y\in C$, then we have that $N[C]=N[y]$, and that $\hat{C}=\bigcap_{z \in N[y]} N[z] \subseteq N[y] $.
						
						\item[\rm iv)] $\hat{C}$ is union of $\mathtt{N}$-classes.
					\end{enumerate}
				\end{lemma}
				\begin{proof}
					\begin{enumerate}
						\item[i)] $N[B]=\bigcap_{x\in B}  N[x] = \bigcap_{x \in A} N[v] \cap \bigcap_{x \in B\setminus A} N[x] \subseteq N[A] $.
						
						\item[ii)] 
						Let $y\in X$, let's see that $y \in \hat{X}=\bigcap_{x \in N[X]} N[x] $. Let $x \in N[X]=\bigcap_{z\in X}N[z]$, then $x \in N[y]$, hence $y \in N[x]$. Now let $s\in \mathcal{S}$. $s \in N[x]$ for all $x \in N[X]$, then $s\in \hat{X}$.

						\item[iii)] $N[C]=\bigcap_{z \in C} N[z]= N[y] $ because the neighbourhoods of the elements in $C$ are all equal. Hence, $\hat{C}=N[N[C]]=\bigcap_{z\in N[y]} N[z]$. Now, by $y\in N[y]$, it follows that $\hat{C}\subseteq N[y]$.
						
						\item[iv)] Let $x\in \hat{C}$. We pick $z\in G$ such that $z\mathtt{N} x$, and we show that $z \in \hat{C}$. 
						From $z\mathtt{N}x$ we have $N[z]=N[x]$. Now let $y \in N[C]$, we have that $x\in N[y]$, hence $y \in N[x]=N[z]$, thus $z\in N[y]$ holds. Therefore we have that $z \in \hat{C}$.
					\end{enumerate}
				\end{proof}
				
				Let's see some example exclusively related to $\mathtt{N}$-classes.
				Look at Figure \ref{UPG_C_12_Nclassi}, it represents $\mathcal{P}(C_{12})$. In that figure the set of white vertices is the star class. Consider the $\mathtt{N}$-class composed by the two green vertices, let call it $C$. Since $C_{12}$ is abelian, $C$ is a $\diamond$-class.
				For $y \in C$, by lemma \ref{operatoreChiusura} iii), $N[C]=N[y]$. Then $N[C]$ is composed by all the vertices coloured in green, purple, pink and white. Instead $\hat{C}$ is composed by only the green, purple and white vertices. Indeed the pink one is not joined to the purple vertices so it is not in $\hat{C}$. Note that if we start with $C'$ as the $\mathtt{N}$-class of the purple vertices, then $N[C']\not=N[C]$, but $\hat{C}'=\hat{C}$.
				
				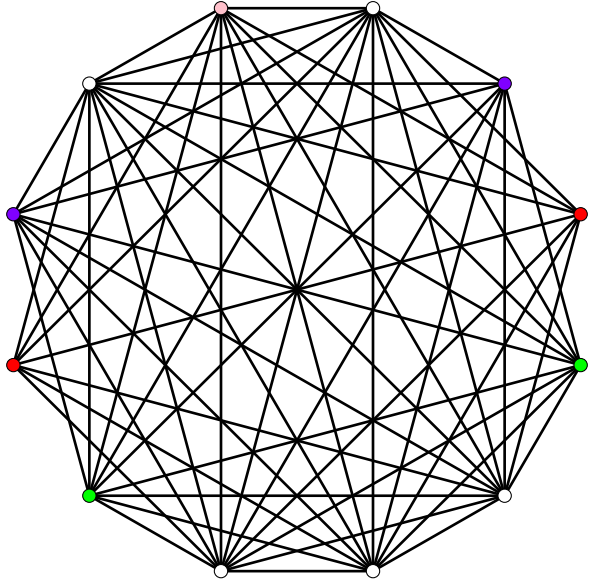
\begin{figure}
					\centering
					\begin{tikzpicture}[line cap=round,line join=round,>=triangle 45,x=1.0cm,y=1.0cm]
						\clip(1.,-4.5) rectangle (9.,4.);
						\draw [line width=1.pt] (4.,-4.)-- (2.2679491924311215,-3.);
						\draw [line width=1.pt] (4.,-4.)-- (1.2679491924311215,-1.2679491924311197);
						\draw [line width=1.pt] (4.,-4.)-- (1.2679491924311224,0.7320508075688783);
						\draw [line width=1.pt] (4.,-4.)-- (2.2679491924311233,2.4641016151377566);
						\draw [line width=1.pt] (4.,-4.)-- (4.,3.4641016151377553);
						\draw [line width=1.pt] (4.,-4.)-- (6.,3.4641016151377553);
						\draw [line width=1.pt] (4.,-4.)-- (7.7320508075688785,2.4641016151377544);
						\draw [line width=1.pt] (4.,-4.)-- (8.732050807568879,0.7320508075688769);
						\draw [line width=1.pt] (4.,-4.)-- (8.732050807568879,-1.2679491924311224);
						\draw [line width=1.pt] (4.,-4.)-- (7.732050807568877,-3.);
						\draw [line width=1.pt] (4.,-4.)-- (6.,-4.);
						\draw [line width=1.pt] (2.2679491924311215,-3.)-- (1.2679491924311224,0.7320508075688783);
						\draw [line width=1.pt] (2.2679491924311215,-3.)-- (4.,3.4641016151377553);
						\draw [line width=1.pt] (2.2679491924311215,-3.)-- (7.7320508075688785,2.4641016151377544);
						\draw [line width=1.pt] (2.2679491924311215,-3.)-- (8.732050807568879,-1.2679491924311224);
						\draw [line width=1.pt] (2.2679491924311215,-3.)-- (6.,-4.);
						\draw [line width=1.pt] (1.2679491924311215,-1.2679491924311197)-- (4.,3.4641016151377553);
						\draw [line width=1.pt] (1.2679491924311215,-1.2679491924311197)-- (8.732050807568879,0.7320508075688769);
						\draw [line width=1.pt] (1.2679491924311215,-1.2679491924311197)-- (6.,-4.);
						\draw [line width=1.pt] (1.2679491924311224,0.7320508075688783)-- (7.7320508075688785,2.4641016151377544);
						\draw [line width=1.pt] (1.2679491924311224,0.7320508075688783)-- (6.,-4.);
						\draw [line width=1.pt] (2.2679491924311233,2.4641016151377566)-- (8.732050807568879,-1.2679491924311224);
						\draw [line width=1.pt] (2.2679491924311233,2.4641016151377566)-- (1.2679491924311215,-1.2679491924311197);
						\draw [line width=1.pt] (2.2679491924311233,2.4641016151377566)-- (7.7320508075688785,2.4641016151377544);
						\draw [line width=1.pt] (2.2679491924311233,2.4641016151377566)-- (4.,3.4641016151377553);
						\draw [line width=1.pt] (2.2679491924311233,2.4641016151377566)-- (6.,3.4641016151377553);
						\draw [line width=1.pt] (2.2679491924311233,2.4641016151377566)-- (8.732050807568879,0.7320508075688769);
						\draw [line width=1.pt] (2.2679491924311233,2.4641016151377566)-- (1.2679491924311224,0.7320508075688783);
						\draw [line width=1.pt] (2.2679491924311233,2.4641016151377566)-- (7.732050807568877,-3.);
						\draw [line width=1.pt] (2.2679491924311233,2.4641016151377566)-- (6.,-4.);
						\draw [line width=1.pt] (2.2679491924311233,2.4641016151377566)-- (2.2679491924311215,-3.);
						\draw [line width=1.pt] (4.,3.4641016151377553)-- (6.,-4.);
						\draw [line width=1.pt] (6.,3.4641016151377553)-- (4.,3.4641016151377553);
						\draw [line width=1.pt] (6.,3.4641016151377553)-- (1.2679491924311224,0.7320508075688783);
						\draw [line width=1.pt] (6.,3.4641016151377553)-- (1.2679491924311215,-1.2679491924311197);
						\draw [line width=1.pt] (6.,3.4641016151377553)-- (2.2679491924311215,-3.);
						\draw [line width=1.pt] (6.,3.4641016151377553)-- (6.,-4.);
						\draw [line width=1.pt] (6.,3.4641016151377553)-- (7.732050807568877,-3.);
						\draw [line width=1.pt] (6.,3.4641016151377553)-- (8.732050807568879,-1.2679491924311224);
						\draw [line width=1.pt] (6.,3.4641016151377553)-- (8.732050807568879,0.7320508075688769);
						\draw [line width=1.pt] (6.,3.4641016151377553)-- (7.7320508075688785,2.4641016151377544);
						\draw [line width=1.pt] (7.7320508075688785,2.4641016151377544)-- (6.,-4.);
						\draw [line width=1.pt] (8.732050807568879,0.7320508075688769)-- (6.,-4.);
						\draw [line width=1.pt] (8.732050807568879,-1.2679491924311224)-- (7.7320508075688785,2.4641016151377544);
						\draw [line width=1.pt] (8.732050807568879,-1.2679491924311224)-- (4.,3.4641016151377553);
						\draw [line width=1.pt] (8.732050807568879,-1.2679491924311224)-- (1.2679491924311224,0.7320508075688783);
						\draw [line width=1.pt] (7.732050807568877,-3.)-- (6.,-4.);
						\draw [line width=1.pt] (7.732050807568877,-3.)-- (2.2679491924311215,-3.);
						\draw [line width=1.pt] (7.732050807568877,-3.)-- (1.2679491924311215,-1.2679491924311197);
						\draw [line width=1.pt] (7.732050807568877,-3.)-- (1.2679491924311224,0.7320508075688783);
						\draw [line width=1.pt] (7.732050807568877,-3.)-- (4.,3.4641016151377553);
						\draw [line width=1.pt] (7.732050807568877,-3.)-- (7.7320508075688785,2.4641016151377544);
						\draw [line width=1.pt] (7.732050807568877,-3.)-- (8.732050807568879,0.7320508075688769);
						\draw [line width=1.pt] (7.732050807568877,-3.)-- (8.732050807568879,-1.2679491924311224);
						\draw [line width=1.pt] (4.,3.4641016151377553)-- (8.732050807568879,0.7320508075688769);
						\draw [line width=1.pt] (8.732050807568879,-1.2679491924311224)-- (6.,-4.);
						\begin{scriptsize}
							\draw [fill=ffffff] (4.,-4.) circle (2.5pt);
							\draw [fill=ffffff] (6.,-4.) circle (2.5pt);
							\draw [fill=ffffff] (7.732050807568877,-3.) circle (2.5pt);
							\draw [fill=qqffqq] (8.732050807568879,-1.2679491924311224) circle (2.5pt);
							\draw [fill=ffqqqq] (8.732050807568879,0.7320508075688769) circle (2.5pt);
							\draw [fill=xfqqff] (7.7320508075688785,2.4641016151377544) circle (2.5pt);
							\draw [fill=ffffff] (6.,3.4641016151377553) circle (2.5pt);
							\draw [fill=ffcqcb] (4.,3.4641016151377553) circle (2.5pt);
							\draw [fill=ffffff] (2.2679491924311233,2.4641016151377566) circle (2.5pt);
							\draw [fill=xfqqff] (1.2679491924311224,0.7320508075688783) circle (2.5pt);
							\draw [fill=ffqqqq] (1.2679491924311215,-1.2679491924311197) circle (2.5pt);
							\draw [fill=qqffqq] (2.2679491924311215,-3.) circle (2.5pt);
						\end{scriptsize}
					\end{tikzpicture}
					\caption{$\mathcal{P}(C_{12})$ where all the vertices with the same colour are in the same $\mathtt{N}$-class.}
					\label{UPG_C_12_Nclassi}
				\end{figure}
				
				We have already seen in Figure \ref{UPG_S_4_NeDiamond} that in $\mathcal{P}(S_4)$ there are three $\mathtt{N}$-classes of the second type. Let $C$ be one of them. We have that $N[C]=C\cup \{1\}=\hat{C}$.  
				We can construct $\mathtt{N}$-classes $C$ of the second type having $\hat{C}=C\cup \{1\}$ taking $D_n$, with $n=p^k$ for $p$ a prime and $k\in \mathbb{N}$. Indeed, look at Figure \ref{UPG_D_p^k_labeled} where $\mathcal{P}(D_{p^k})$ is represented schematically. Recall that $$ D_{p^k}:=\langle a,b \, |\, a^{p^k}=1=b^2, b^{-1}ab=a^{-1} \rangle.$$ $\mathcal{P}(D_{p^k})$ is composed by the complete graph $K_{p^k}$ and $p^k$ more vertices joined only with the identity. Therefore we have that $[a]_{\mathtt{N}}=\langle a \rangle \setminus \{1\}$. Note that $[a]_{\mathtt{N}}$ cannot be neither of the first type, since it contains elements of different orders, nor the star class, since it does not contain the identity. In conclusion, $[a]_{\mathtt{N}}$ is of the second type with $\hat{[a]}_{\mathtt{N}}=\langle a \rangle= [a]_{\mathtt{N}} \cup \{1\}$. 
			
				\begin{figure}
					\centering
					\begin{tikzpicture}[line cap=round,line join=round,>=triangle 45,x=1.0cm,y=1.0cm]
						\clip(3.5,-6.) rectangle (9.5,0.);
						\draw [line width=1.5pt] (6.0149005420162185,-1.6193168603621786) circle (1.554274322171015cm);
						\draw (6.0865139017686305,-2.439893402030971) node[anchor=north west] {$1$};
						\draw (3.7968124992139227,-4.219454077591621) node[anchor=north west] {$b$};
						\draw (4.757775264016676,-5.1092344153719464) node[anchor=north west] {$ab$};
						\draw (7.344070112498159,-5.0855069396978045) node[anchor=north west] {$a^{p^k-2}b$};
						\draw (8.487,-4.343181553265763) node[anchor=north ] {$a^{p^k-1}b$};
						\draw (5.338341974740563,-1.8) node[anchor=north] {$a^{p^k-1}$};
						\draw (5.1,-0.6) node[anchor=north] {$a^{p^k-2}$};
						\draw (6.728082970496916,-2.375108245456738) node[anchor=west] {$a$};
						\draw (6.774610696318749,-1.) node[anchor=north west] {$a^2$};
						\draw (5.7,-1.3721569966945808) node[anchor=north west] {$K_{p^k}$};
						\draw [line width=1.pt] (3.7912078968897096,-4.3621630915358125)-- (5.996709459753548,-2.777371549358208);
						\draw [line width=1.pt] (4.755289418381088,-5.247005035918309)-- (5.996709459753548,-2.777371549358208);
						\draw [line width=1.pt] (7.343782270604515,-5.233798439733495)-- (5.996709459753548,-2.777371549358208);
						\draw [line width=1.pt] (8.2154176188022,-4.388576283905439)-- (5.996709459753548,-2.777371549358208);
						\begin{scriptsize}
							\draw [fill=qqqqff] (5.996709459753548,-2.777371549358208) circle (2.5pt);
							\draw [fill=qqqqff] (3.7912078968897096,-4.3621630915358125) circle (2.5pt);
							\draw [fill=qqqqff] (4.755289418381088,-5.247005035918309) circle (2.5pt);
							\draw [fill=qqqqff] (7.343782270604515,-5.233798439733495) circle (2.5pt);
							\draw [fill=qqqqff] (8.2154176188022,-4.388576283905439) circle (2.5pt);
							\draw [fill=black] (5.48453197899291,-5.528482828939864) circle (1.0pt);
							\draw [fill=black] (5.98453197899291,-5.528482828939864) circle (1.0pt);
							\draw [fill=black] (6.48453197899291,-5.548482828939863) circle (1.0pt);
							\draw [fill=qqqqff] (5.338341974740563,-2.3165047094911086) circle (2.5pt);
							\draw [fill=qqqqff] (6.728082970496916,-2.375108245456738) circle (2.5pt);
							\draw [fill=qqqqff] (5.1,-1.161177857597276) circle (2.5pt);
							\draw [fill=qqqqff] (6.954125180650059,-1.161177857597276) circle (2.5pt);
							\draw [fill=black] (5.522524516346826,-0.6421179676159887) circle (1.0pt);
							\draw [fill=black] (6.027157890160115,-0.5161555274114648) circle (1.0pt);
							\draw [fill=black] (6.5606442963094045,-0.6756057024534912) circle (1.0pt);
						\end{scriptsize}
					\end{tikzpicture}
					\caption{$\mathcal{P}(D_{p^k})$.}
					\label{UPG_D_p^k_labeled}
				\end{figure}
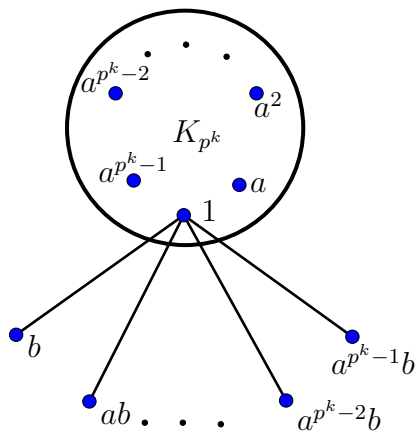
				
				The next results about $\hat{C}$, for $C$ an $\mathtt{N}$-class, clarify how $\hat{C}$ is useful to determine the type of $C$.
				
				\begin{lemma}
					\label{CarachetisationN-classes_1}Let $C_y$ be an $\mathtt{N}$-class of the second type for some $y \in G$ with parameters $(p,r,s)$. Then we have $|C_y|=p^r-p^s$ and $\hat{C_y}=\langle y \rangle $. In particular, $|\hat{C_y}|=p^r$.
				\end{lemma}
				\begin{proof}
					Since $C_y$ is an $\mathtt{N}$-class of the second type we have $C_y \subseteq \langle y\rangle=:Y$. In particular $|C_y|=|Y|-|\left\lbrace z\in Y: o(z)\leq p^s \right\rbrace |=p^r-p^s$. We also have that $\langle C_y\rangle=Y$
					since $y\in C_y$.
					
					\setcounter{claim}{0}
					\begin{claim}
					\rm	$Y \subseteq \hat{C_y}$
					\end{claim}
					For all $a \in Y$, $a$ is power of $y$, so, by Lemma \ref{NdipotenzeDip}, we have $N[a]\supseteq N[y]$. Since $N[y]=N[c]$ for all $c\in C_y$, we also have $N[a]\supseteq N[c]$ for all $a \in Y$ and $c \in C_y$.
					Hence $$N[C_y]=\bigcap_{c\in C_y} N[c] \subseteq \bigcap_{a \in Y} N[a]=N[Y].$$
					Therefore, by Lemma \ref{operatoreChiusura} i) and ii), $Y\subseteq \hat{Y}=N[N[Y]] \subseteq N[N[C_y]]=\hat{C_y}$.
					
					\begin{claim}
					\rm	$Y=\hat{C_y}$.
					\end{claim}
					Suppose, by contradiction, that there exists $u\in \hat{C_y}\setminus Y$.
					We consider two cases:
					\begin{itemize}
						\item \underline{The order of $u$ is not a power of $p$}.
						As, by Lemma \ref{operatoreChiusura} iii), $u \in \hat{C_y}\subseteq N[y]$, $u$ and $y$ must be joined. Then $y$ is a power of $u$, otherwise $u \in Y$. Hence $\langle u\rangle>\langle y \rangle$. Then there exist $t\in \mathbb{N}_0$ and $m \in \mathbb{N}$ with $m\geq 2$ and $\gcd(m,p)=1$, such that $o(u)=p^{r+t}m$.
						By $o(u^{p^tm})=o(y)=p^r$, we get that there exists $k \in \mathbb{N}$, with $\gcd(k,p^r)=1$, such that $u^{p^tmk}=y$. Thus $y^p=(u^{p^tmk})^p=u^{p^{t+1}mk}$. Note that $y\mathtt{N} y^p$ as $C_y$ is an $\mathtt{N}$-class of the second type and surely $y^p\in C_y$ since $o(y^p)=p^{r-1}$. Then we have $N[y]=N[y_p]$. Let's see that this brings us to a contradiction. We set $w:=u^{p^{t+1}}$. $w$ is adjacent to $y^p$ since $y^p=w^{mk}$ and $o(w)=p^{r-1}m\not= p^{r-1}=o(y^p)$ that implies that $y^p\not=w$. But $w$ is not joined to $y$ since we have that $o(y)=p^r \nmid o(w)=p^{r-1}m$ and $o(w) \nmid o(y)$. Therefore $N[y]\not=N[y^p]$, which is a contradiction.
						
						\item \underline{The order of $u$ is a power of $p$}. By Lemma \ref{operatoreChiusura} iii) we have that $u\in \hat{C_y}\setminus Y \subseteq N[y]\setminus Y$. Thus necessarily $y$ is a power of $u$ and then $o(u)>o(y)$. Then, by Lemma \ref{NdipotenzeDip}, $N[u]\subseteq N[y]$ and the containment is proper since $u \notin C_y\subseteq Y$. So there is a $w \in N[y]$ such that $w \notin N[u]$. Thus $u\notin N[w]$, that implies $u \notin \hat{C_y}=\bigcap_{z\in N[y]} N[z]$, a contradiction.
					\end{itemize}
				
				\end{proof}
				
				Lemma \ref{CarachetisationN-classes_1} was proved by Cameron in \cite{Cameron_2}. It gives us an idea to recognize the type of an $\mathtt{N}$-class. An $\mathtt{N}$-class $C$ of the second type has algebraic restrictions on $|C|$ and $|\hat{C}|$. In the following lemma is shown that these algebraic restrictions are satisfied by an $\mathtt{N}$-class of the first type only in a specific case. So every $\mathtt{N}$-class, different from $\mathcal{S}$, that does not satisfy the algebraic restrictions on the orders is necessarily of the first type. We follow Cameron's idea, but we reach a different result. Indeed Cameron version of this lemma\footnote{Cameron never enunciates formally this result, but he gives its proof in \cite[pg. 5-6]{Cameron_2} } get a wrong conclusion. In \cite{Cameron_2} the same proof is mistakenly used to prove that there is no $\mathtt{N}$-class $C$ of the first type having $|C|=p^r-p^s$ and $|\hat{C}|=p^r$ with $p,r$ and $s$ as in Lemma \ref{CarachetisationN-classes_1}.
				
				\begin{lemma}
					\label{CarachetisationN-classes_2}	Let $G$ be a group and $C$ be an $\mathtt{N}$-class of the first type. Assume that there exist $p$ be a prime, an integer $r\geq 2$ and an integer $s\in [r-2]_0$ such that $|\hat{C}|=p^r$ and $|C|=p^r-p^s$. Then $\hat{C}=C\cup \{1\}$. Moreover $s=0$ and $C=[y]_{\diamond}$ for some $y\in G$ with $o(y)$ not a prime power and such that $\phi(o(y))=p^r-1$.
				\end{lemma}
				\begin{proof}
					By the conditions $|\hat{C}|=p^r$ and $|C|=p^r-p^s$ it follows that $|\hat{C}|< 2|C|$. Indeed we have $p^r< 2(p^r-p^s) $ because it is equivalent to
					$ 2< p^{r-s} $ and we surely have that $$p^{r-s}\geq p^2> 2. $$
					By Lemma \ref{operatoreChiusura} iv) $\hat{C}$ is union of $\diamond$-classes and $C \subseteq \hat{C}$.
					So any $\diamond$-class included in $\hat{C}$ and distinct from $C$ must have size smaller than $|C|$. Pick $y\in C$.
					Since $C$ is of the first type we have $C=[y]_{\mathtt{N}}=[y]_{\diamond}$.
					
					\setcounter{claim}{0}
					\begin{claim}
						\label{claim1_1}
						\rm $\hat{C}\setminus [y]_{\diamond}$ cannot contain elements of order greater or equal than $o(y)$.
					\end{claim}
					Assume, by contradiction, that there exists $x\in \hat{C}\setminus [y]_{\diamond}$ such that $o(x)\geq o(y)$. By Proposition \ref{operatoreChiusura} iii) we have $\hat{C}\subseteq N[y]$ and thus $x \in N[y]\setminus [y]_{\diamond}$. Hence, since $o(x)\geq o(y)$, we necessarily have $y\in \langle x \rangle$. Then $o(y)\mid o(x)$, which implies, by Lemma \ref{phiEulero}, $\phi(o(y))\mid \phi(o(x))$. In particular we have that $\phi(o(y))\leq \phi(o(x))$.
					Since $\hat{C}$ is union of $\diamond$-classes, we have that $[x]_{\diamond}\subseteq \hat{C}$. It follows that $|[x]_{\diamond}|=\phi(o(x))\geq \phi(o(y))=|C|$, a contradiction. 
					
					We next show that $\hat{C}\subseteq \langle y \rangle$. Suppose, by contradiction, that there exists $x\in \hat{C}\setminus \langle y \rangle .$ Then $x\in \hat{C} \setminus [y]_{\diamond}$ and by Claim \ref{claim1_1} we deduce $o(x)< o(y)$. Now, since $\hat{C}\subseteq N[y]$, we have that $x \in N[y]$ and thus necessarily $x \in \langle y \rangle$, a contradiction. 
					
					By Lemma \ref{operatoreChiusura} ii) we also have $1\in \hat{C}$ and hence following holds:
					\begin{equation}\label{inequali}
						C\cup \{1\}\subseteq \hat{C}\subseteq \langle y \rangle.
					\end{equation}

					Now we distinguish two cases:
					\begin{itemize}
						\item \underline{There exist a prime $q$ and an integer $t\geq1$ such that $C$ consists of elements of} \\ \underline{order $q^t$}.
						\begin{claim}
							\rm	$\hat{C}=\langle y \rangle$.
						\end{claim}
						
						Let $z\in N[y]$, so $z\in \langle y\rangle$ or $y\in \langle z\rangle$. If $o(z)\leq o(y)$, then $z \in \langle y \rangle$. Moreover we have $\langle y \rangle \subseteq N[z]$ since $o(y)=q^t$ and, by Proposition \ref{PGcompleto}, $\mathcal{P}(\langle y \rangle)$ is a complete graph. Now, if $o(z)>o(y)$, then $\langle y \rangle \leq \langle z \rangle \subseteq N[z]$. Thus we have  $N[z]\supseteq \langle y \rangle$ for all $z\in N[y]$. Now, by Lemma \ref{operatoreChiusura} iii) we have $\hat{C}= \bigcap_{z\in N[y]} N[z] $. Then $\langle y \rangle \subseteq \hat{C}$ and, by \eqref{inequali}, $\hat{C}=\langle y \rangle$ follows.
						
						It follows that $p^r=|\hat{C}|=|\langle y \rangle|=q^t$. But then $p=q$, $r=t$, and hence $p^r-p^s=|C|=\phi(p^r)=p^r -p^{r-1}$ which implies $s=r-1$, a contradiction.
						
						\item \underline{There exists a positive integer $m$, not a prime power, such that $C$ consist of elements} \underline{of order $m$}.
						
						\begin{claim}
							$\hat{C}=C\cup \{1\}$.
						\end{claim}
					
						By \eqref{inequali}, we just need to show that $\hat{C}\subseteq C\cup \{1\}$. Let $x\in \hat{C}$. By \eqref{inequali}, we have that $x\in \langle y \rangle$ and thus $x=y^k$, for some $k \in \mathbb{N}$. Assume, by contradiction, that $y^k\notin C\cup \{1\}$.  Then $y^k$  is neither a generator nor the identity of the cyclic group $\langle y \rangle\cong C_m$ of order $m$. Since, by Lemma \ref{nonGeneratoriInPG}, the power graph of $C_m$ is not complete, there exists $w\in \langle y \rangle $ such that $\{y^k,w\}\notin E_{\mathcal{P}(C_m)}$. Since $\mathcal{P}(C_m)$ is an induced subgraph of $\mathcal{P}(G)$, we also have that $\{y^k,w\}\notin E_{\mathcal{P}(G)}$. It follows that $w\in N[y]$ while  $y^k\notin N[w]$. Since by Lemma \ref{operatoreChiusura} iii)  we have $\hat{C}=\bigcap_{z \in N[y]} N[z]$, we deduce that $x=y^k \notin \hat{C}$, a contradiction.

						As a consequence, we also have  $|\hat{C}|=|C|+1$ and thus $p^r-p^s=\phi(m)=|C|=|\hat{C}|-1=p^r-1$ which implies $s=0$.
						
					\end{itemize}
				\end{proof}
				
				Let's see, in details, how we use the previous lemmas to recognize the type of an $\mathtt{N}$-class.
				Let $C$ be an $\mathtt{N}$-class, distinct from the star class.
				Assume first that $\hat{C}\not=C \cup \{1\}$, then, by Lemma \ref{CarachetisationN-classes_1} and Lemma \ref{CarachetisationN-classes_2}, $C$ is of the second type if there exist a prime $p$ and an integer $r\geq 2$ and $s\in [r-2]_0$ for which $|C|=p^r-p^s$ and $|\hat{C}|=p^r$. Otherwise $C$ is of the first type.
				Assume next that $\hat{C}=C\cup \{1\}$, but $|\hat{C}|\ne p^r$ for a prime $p$ and an $r\in \mathbb{N}$ with $r\geq 2$, then, by Lemma \ref{CarachetisationN-classes_1} and \ref{CarachetisationN-classes_2}, $C$ is of the first type.
				It remains the case of $\hat{C}=C\cup\{1\}$ and $|\hat{C}|=p^r$ with $p$ a prime and $r$ an integer such that $r\geq2$. For this case we need some new results.
						
				We emphasize that in Cameron's paper \cite{Cameron_2} we read that the results so far are enough to distinguish the type of all $\mathtt{N}$-classes. But this is not true. Indeed the arguments on the orders of $C$ and $\hat{C}$ do not distinguish completely the type of the classes having $\hat{C}=C\cup \{1\}$ and $|\hat{C}|=p^r$. Then we introduce the following crucial definition.  
				
				\begin{definition}
				 \rm	A \emph{critical class}, is an $\mathtt{N}$-class $C$ with $\hat{C}=C \cup \{1\}$ and for which there exist a prime number $p$ and an integer $r\geq 2$, such that $|\hat{C}|=p^r$. We call $(p,r)$, with $p$ and $r$ defined as above, the \emph{parameters of the critical class}.
				\end{definition}
				
				We give an example that underlines the need of some new arguments to distinguish the type of critical classes. 
				\begin{Oss}
					\label{numeriCritici} Let $m$ be a positive integer. If there exist  a prime  $p$ and an integer $r\geq 2$ for which $\phi(m)=p^r-1$, then $p\geq3$. Moreover $m=15$ is the first integer for which $\phi(m)=p^r-1$. In this particular case $\phi(15)=8$, so that  $p=3$ and $r=2$. 
				\end{Oss}
				\begin{proof}
					For $p=2$ and $r\geq 2$ there cannot be a positive integer $m$ such that $\phi(m)=p^r-1$.  Indeed $2^r-1$ is odd and is greater or equal than $3$. But $\phi(m)$ is odd if and only if $\phi(m)=1$. 
					So pick $p=3$ and $r=2$. Now, the first positive integer $m$ to have $\phi(m)=3^2-1=8$ is $15$.
					Note that the values of the Euler's function before $15$ are in the set $\{1,2,4,6, 10, 12\}$. None of them is $q^k-1$ for some $q$ prime and $k\geq2$. This concludes the proof.
				\end{proof}
			
				\begin{es}
					\rm Consider $D_{15}
					$ and refer to its representation \eqref{Dn_presentation}. Then $[a]_{\mathtt{N}}$ is a critical class of the first type with parameters $(3,2)$. In particular $|[a]_{\mathtt{N}}|=8$ and $|\hat{[a]}_{\mathtt{N}}|=9$.
				\end{es}
				\begin{proof}
					$[a]_{\mathtt{N}}$ is surely of the first type since $o(a)=15$ is not a prime power. Thus $[a]_{\mathtt{N}}=[a]_{\diamond}$.
					We want to prove that $\hat{[a]}_{\mathtt{N}}=[a]_{\mathtt{N}} \cup \{1\}$.
					By Lemma \ref{operatoreChiusura} ii), we have $[a]_{\mathtt{N}} \cup \{1\}\subseteq \hat{[a]}_{\mathtt{N}}$.
					Suppose, by contradiction, that there exists $x\in \hat{[a]}_{\mathtt{N}}\setminus ([a]_{\mathtt{N}}\cup \{1\})$. Since we have $x\in \hat{[a]}_{\mathtt{N}}$, by Lemma \ref{operatoreChiusura} iii), we also have $x\in N[a]=\langle a \rangle$. But $x\notin [a]_{\mathtt{N}} \cup \{1\}=[a]_{\diamond} \cup \{1\}$, hence, by Lemma \ref{nonGeneratoriInPG}, there exists $y \in \langle a \rangle=N[a]$ such that $x\notin N[y]$. Therefore, since, by Lemma \ref{operatoreChiusura} iii), $\hat{[a]}_{\mathtt{N}}= \bigcap_{z\in N[a]}N[z]$, it follows that $x\notin \hat{[a]}_{\mathtt{N}}$, a contradiction. 
					In conclusion, we have that $[a]_{\mathtt{N}}$ is a critical class of the first type, and, since $$|[a]_{\mathtt{N}}|=|[a]_{\diamond}|=\phi(15)=8=3^2-1,$$ it has parameters $(3,2)$.   
				\end{proof}
			
				\begin{es}
					\rm Consider $D_9$ and refer to its representation \eqref{Dn_presentation}. Then $[a]_{\mathtt{N}}$ is a critical class of the second type with parameters $(3,2)$. In particular $|[a]_{\mathtt{N}}|=8$ and $|\hat{[a]}_{\mathtt{N}}|=9$.
				\end{es}
				\begin{proof}
					We have already seen, in the discussion after Lemma \ref{operatoreChiusura}, that $[a]_{\mathtt{N}}$ is of the second type having $\hat{[a]}_{\mathtt{N}}=[a]_{\mathtt{N}}\cup \{1\}$. Therefore $[a]_{\mathtt{N}}$ is a critical class of the second type. Since $\hat{[a]}_{\mathtt{N}}=\langle a \rangle$, we then have that $|[a]_{\mathtt{N}}|=8=3^2-1$, which means that $[a]_{\mathtt{N}}$ has parameters $(3,2)$.
				\end{proof}
				
				Note that we cannot distinguish the type of the two classes of the examples above using Lemmas \ref{CarachetisationN-classes_1} and \ref{CarachetisationN-classes_2}.
				
				The critical classes are not rare dealing with power graphs. Indeed it is not difficult to find plenty of them.
				
				\begin{Oss}
					\label{FirstTypeMethod_1}Let $G$ be a group and let $m\in \mathbb{N}$ be a not prime power integer such that $\phi(m)=p^r-1$ for a prime $p$ and an integer $r\geq 2$. Let $y \in G$ be an element with $o(y)=m$. Consider $C:=[y]_{\mathtt{N}}$. If $\hat{C}\subseteq \langle y \rangle$, then $C$ is a critical class of the first type.
				\end{Oss}
				\begin{proof}
					$C$ is of the first type, because $o(y)$ is not a prime power.
					Since, by Lemma \ref{operatoreChiusura} ii), we have that $C\cup \{1\} \subseteq \hat{C}$, then it holds $$C\cup \{1\} \subseteq \hat{C} \subseteq \langle y \rangle.$$
					Now we ensue the proof of the second point of Lemma \ref{CarachetisationN-classes_2} obtaining that $C$ is a critical class.
				\end{proof}
			
				We propose two distinct examples for which the hypothesis of the previous lemma hold.\\
				
				\begin{Oss}
					Let $G$ be a group. Let $m$ and $y$ be as in {\rm Remark \ref{FirstTypeMethod_1}}. Consider $C:=[y]_{\mathtt{N}}$. If $N[y]=\langle y \rangle$, then $\hat{C}\subseteq \langle y \rangle$. Hence $C$ is a critical class of the first type.
				\end{Oss}
				\begin{proof}
					Remember that, by Lemma \ref{operatoreChiusura} iii), $\hat{C}\subseteq N[y]$. So $\hat{C}\subseteq \langle y \rangle$. Hence, by Remark \ref{FirstTypeMethod_1}, $C$ is a critical class of the first type.
				\end{proof}
				
				\begin{es}
					\rm Let $G$ be a group and let $m$ and $y$ be as in Remark \ref{FirstTypeMethod_1}. If there is no elements in $G$ having order $l$ such that $m$ properly divides $l$, then $\hat{C}=\langle y \rangle$.\\
					Therefore the class of a permutation with type $[3,5]$ in $S_8$ is a critical class of the first type with parameters $(3,2)$. 
				\end{es}
				\begin{proof}
					Surely $\langle y \rangle \subseteq N[y]$. Suppose, by contradiction, that there exists $x\in N[y]\setminus \langle y \rangle$. Then we have $\{x,y\}\in E$. Hence we have that $y\in \langle x \rangle \setminus [x]_{\diamond}$, because $x \notin \langle y \rangle$. It follows that $o(y)=m$ properly divides $o(x)$, a contradiction.
					
					Let $y\in S_8$ be an element having type $[3,5]$. Then $o(y)=15$. Now, in $S_8$  there is no elements with order properly divided by $15$. Thus $[y]_{\mathtt{N}}$ is a critical class of the first type of parameters $(3,2)$. Indeed $\phi(o(y))=8=3^2-1$.
				\end{proof}
				
				\begin{Oss}
					Let $G$ be a group and let $m$ and $y$ be as in {\rm Remark \ref{FirstTypeMethod_1}}. Consider $C:=[y]_{\mathtt{N}}$. If $C_G(y)=\langle y \rangle$, then $\hat{C}\subseteq \langle y \rangle$. Hence $C$ is a critical class of the first type.
				\end{Oss}
				\begin{proof}
					By Remark \ref{Nclassi-centralizzante}, we have that $\hat{C}\subseteq C_G(y)$. Then $\hat{C}\subseteq \langle y \rangle$ holds, and hence, by Remark \ref{FirstTypeMethod_1}, $C$ is a critical class of the first type.
				\end{proof}
			
				\begin{es}
					\rm In $GL(2,5)$ the $\mathtt{N}$-class of the Singer's cycle is a critical class of the first type with parameters $(3,2)$.
				\end{es}
				\begin{proof}
						Let's denote $s$ as the Singer's cycle of $G:=GL(2,5)$. We have that $C_G(s)=\langle s \rangle$. Since $o(s)=24$ and $\phi(24)=8=3^2-1$ hold, then $[s]_{\mathtt{N}}$ is a critical class of the first type.
				\end{proof}
			
				Now that we have clear in mind the existence of critical classes we need a way to distinguish their type.
				
				\begin{lemma}
					\label{latiNclassi}Let $X$ and $Y$ two $\mathtt{N}$-classes. If $\{x,y\}\in E$ for $x\in X$ and $y\in Y$, then, for all $\bar{x}\in X$ and $\bar{y}\in Y$, we have that $\{\bar{x},\bar{y}\}\in E$.
				\end{lemma}
				\begin{proof}
					Let $\bar{x} \in X$ and $\bar{y}\in Y$. Since $N[x]=N[\bar{x}]$ and $N[y]=N[\bar{y}]$, then $\{\bar{x},\bar{y}\}\in E$. Indeed, since $\{x,y\}\in E$, we have that also $\{\bar{x},y\}\in E$ because $x$ and $\bar{x}$ have the same closed neighbourhood. Next, since $\{\bar{x},y\}\in E$, then $\{\bar{x},\bar{y}\}\in E$, because $N[y]=N[\bar{y}]$. 
				\end{proof}
			
				The next proposition is a transcription of \cite[Lemma 3.5 ]{Ma et al} using our notation. This proposition is the foundations of the arguments that distinguish the type of critical classes.
				
				\begin{prop}{\rm \cite[Lemma 3.5]{Ma et al}}
					\label{propLatiNclassi}Let $G$ be a group and $x,y \in G$. Let $[x]_{\mathtt{N}}$ and $[y]_{\mathtt{N}}$ be two distinct $\mathtt{N}$-classes of first or second type. If $\langle x \rangle < \langle y \rangle $, then $|[x]_{\mathtt{N}}|\leq |[y]_{\mathtt{N}}|$, with equality if and only if the following two conditions hold.
					\begin{enumerate}
						\item[\rm 1)] Both $[x]_{\mathtt{N}}$ and $[y]_{\mathtt{N}}$ are of the first type.
						
						\item[\rm 2)] $o(y)=2o(x)$ with $o(x)$ odd and at least $3$.
					\end{enumerate} 
				\end{prop}
			
				\begin{proof}
					We divide the proof in three cases:
					\begin{enumerate}
						\item[\rm i)] \underline{$[x]_{\mathtt{N}}$ is of second type with parameters $(p,r,s)$ and $[y]_{\mathtt{N}}$ of first or second type}.
						Pick elements $x_1$ and $x_r$ in $[x]_{\mathtt{N}}$ of minimum and maximum order, $p^{s+1}$ and $p^r$, respectively. Then $\langle x_1\rangle \leq \langle x\rangle \leq \langle x_r\rangle$. Since $\langle x \rangle <\langle y \rangle$, we have $y\in N[x]=N[x_r]$. Note that any element $z$ satisfying $\langle x \rangle \leq \langle z \rangle \leq \langle x_r \rangle$ belongs to $[x]_{\mathtt{N}}$. Then $\langle x_r \rangle <\langle y \rangle$, otherwise $y \in [x]_{\mathtt{N}}$, and we have $[x]_{\mathtt{N}}\not=[y]_{\mathtt{N}}$. So $o(x_r)=p^r$ properly divides $o(y)$. Hence there exists $p'$ a prime divisor of $\frac{o(y)}{p^r}$. Pick an element $z_0\in \langle y\rangle$ of order $p^{s+1}p'$. Then $z_0\in N[x_1]=N[x_r]$, because $\langle x_1 \rangle \leq \langle y \rangle$ and $o(x_1)=p^{s+1}$, thus $\langle x_1 \rangle \leq \langle z_0 \rangle$. Hence we have $\{z_0,x_r\}\in E$, so one of $p^{s+1}p'$ and $p^r$ is divisible by the other. If $\frac{p^r}{p^{s+1}p'}\in \mathbb{N}$, then $p'\mid p$, and hence $p'=p$. If $\frac{p^{s+1}p'}{p^r}\in \mathbb{N}$, then $p^{s+1-r}p'\in \mathbb{N}$. In view of $r-s\geq 2 $, we get $p'=p$, also in this case. It follows that $p^{r+1}$ divides $o(y)$, and so $|[y]_{\diamond}|\geq p^{r}(p-1)\geq p^{r}$. By Lemma \ref{CarachetisationN-classes_1}, we have that $|[x]_{\mathtt{N}}|=p^r-p^s<p^r$. Since $[y]_{\diamond}\subseteq [y]_{\mathtt{N}}$, we have that $|[x]_{\mathtt{N}}|<|[y]_{\mathtt{N}}|$.
						
						\item[\rm ii)] \underline{$[y]_{\mathtt{N}}$ is of the second type, with parameters $(q,t,j)$, and $[x]_{\mathtt{N}}$ of the first type }. Pick an element $y_1$ in $[y]_{\mathtt{N}}$ of order $p^{j+1}$. Since any element $z$ satisfying $\langle y_1 \rangle \leq \langle z \rangle \leq \langle y \rangle$ belongs to $[y]_{\mathtt{N}}$, we obtain that $\langle x\rangle < \langle y_1 \rangle$. Hence $o(x)$ properly divides $o(y_1)=q^{j+1}$. So pick $y_0 \in \langle y_1 \rangle $ of order $q^j$. Then $\langle x \rangle \leq \langle y_0 \rangle$. Therefore we have that $[x]_{\mathtt{N}}\subseteq \langle y_0 \rangle \setminus \{1\}$. Indeed $[x]_{\mathtt{N}}$ is of the first type  so it does not contains $1$ and we also  have that $[x]_{\mathtt{N}}=[x]_{\diamond}\subset \langle x \rangle$. Hence $|[x]_{\mathtt{N}}|\leq |\langle y_0\rangle \setminus \{1\} | = q^j-1<q^t-q^j$. Since $[y]_{\mathtt{N}}$ is of the second type, and then, by Lemma \ref{CarachetisationN-classes_1}, $|[y]_{\mathtt{N}}|=q^t-q^j$, we have $|[x]_{\mathtt{N}}|<|[y]_{\mathtt{N}}|$.
						
						\item[\rm iii)] \underline{Both $[x]_{\mathtt{N}}$ and $[y]_{\mathtt{N}}$ are of the first type}. Then $|[x]_{\mathtt{N}}|=|[x]_{\diamond}|=\phi(o(x))$ and $|[y]_{\mathtt{N}}|=\phi(o(y))$. Since $o(x)\mid o(y)$, it follows that $|[x]_{\mathtt{N}}|\mid|[y]_{\mathtt{N}}|$, and so, by Lemma \ref{phiEulero}, $|[x]_{\mathtt{N}}|\leq |[y]_{\mathtt{N}}|$. Note that $o(x)\not=o(y)$ and $x\not=1$. Then, again by Lemma \ref{phiEulero}, $|[y]_{\mathtt{N}}|=|[x]_{\mathtt{N}}|$ if and only if $o(y)=2o(x)$ and $o(x)$ odd and at least $3$.
						
						Combining all these cases, we get the desired result.

					\end{enumerate}
				\end{proof}
				
				Now we can focus only on the critical classes. Note that whenever $|\mathcal{S}|>1$, for the reconstruction of the directed power graph from the power graph, we will use an argument that does not consider $\mathtt{N}$ or $\diamond$-classes. So we study the critical classes only in the case of $\mathcal{S}=\{1\}$.
				
				\begin{lemma}
					\label{LemmaCriticalClassMio}Let $G$ be a group with $\mathcal{S}=\{1\}$. Let $C=[y]_{\mathtt{N}}$ be a critical class. There exists $x\in G\setminus \hat{C}$ such that $|[x]_{\mathtt{N}}|\leq |C|$ and such that $\{x,y\}\in E$ if and only if $C$ is of the first type.
				\end{lemma}
				\begin{proof}
					Assume first that there exists  $x\in G\setminus \hat{C}$ such that $|[x]_{\mathtt{N}}|\leq |C|=[y]_{\mathtt{N}}$ and such that $\{x,y\}\in E$.  
					Since $\{x,y\}\in E$, then there are two possibilities, $\langle y \rangle < \langle x \rangle$ or $\langle x \rangle < \langle y \rangle$. Indeed $\langle x \rangle \not= \langle y \rangle$, otherwise we would have $x\in C\subseteq \hat{C}$. \\
					If $\langle y \rangle < \langle x \rangle$, by Proposition \ref{propLatiNclassi}, we have that $|[y]_{\mathtt{N}}|\leq |[x]_{\mathtt{N}}|$, and hence $|[x]_{\mathtt{N}}|=|[y]_{\mathtt{N}}|$. So, again by Proposition \ref{propLatiNclassi}, $[y]_{\mathtt{N}}$ is of the first type.\\
					If $\langle x \rangle < \langle y \rangle$, in particular $x\in \langle y \rangle $.
					Suppose, by contradiction, that $C$ is of the second type. Then, by Lemma \ref{CarachetisationN-classes_1}, $\hat{C}=\langle y \rangle$ and hence $x\in \hat{C}$, a contradiction.
					
					Assume next that $C$ is of the first type. Then $o(y)=m$, with $m$ a not a prime power integer. So, by Lemma \ref{nonGeneratoriInPG}, there exists $x \in \langle y \rangle\setminus \{1\}$ such that $x \notin [y]_{\mathtt{N}}$. Therefore we have that $\langle x \rangle < \langle y \rangle$, and hence, by Proposition \ref{propLatiNclassi}, $|[x]_{\mathtt{N}}|\leq |C|$.
				\end{proof}
				
				In the end, whenever $\mathcal{S}=\{1\}$, for each $\mathtt{N}$-class we know how to recognize of which type it is. If $C$ is an $\mathtt{N}$-class not critical, then, by Lemma \ref{CarachetisationN-classes_1} and Lemma \ref{CarachetisationN-classes_2}, we are able to distinguish the type of $C$. If $C$ is a critical class, then, by Lemma \ref{LemmaCriticalClassMio}, we are able to distinguish the type of $C$. 
				
				So now we have all the tools to prove the main result of this section.
				
				\begin{theorem}{\rm \cite[Theorem 2]{Cameron_2}}
				\label{UPG-DPG}	If $G_1$ and $G_2$ are groups with $\mathcal{P}(G_1)\cong \mathcal{P}(G_2)$, then $\vec{\mathcal{P}}(G_1)\cong \vec{\mathcal{P}}(G_2)$.
				\end{theorem}
				\begin{proof}
					Due to the graph isomorphism we have $|G_1|=|G_2|=n$.
					Let $\mathcal{S}_i=\mathcal{S}_{\mathcal{P}(G_i)}$. We have $|\mathcal{S}_1|=|\mathcal{
					S}_2|$ by Remark \ref{CardinalityS}.
					By Proposition \ref{S>1} we have $4$ possibilities:
					\begin{itemize}
						\item \underline{$|\mathcal{S}_i|=n$ for $i \in \left\lbrace 1,2\right\rbrace $}. Then $\mathcal{P}(G_i)$ is a complete graph. By the characterisation in the Proposition \ref{S>1} we get that $G_i$ is cyclic of prime-power order and therefore $G_1 \cong G_2$. It follows that $\vec{\mathcal{P}}(G_1)\cong \vec{\mathcal{P}}(G_2)$.
						
						\item \underline{$|\mathcal{S}_i|=1+\phi(n)$ for $i \in \left\lbrace 1,2\right\rbrace $}. By the characterisation in the Proposition \ref{S>1} we have that $G_i$ is cyclic of order $n$ for $i \in \left\lbrace 1,2\right\rbrace  $. Therefore $G_1 \cong G_2$, then we have $\vec{\mathcal{P}}(G_1)\cong \vec{\mathcal{P}}(G_2)$.
						
						\item \underline{$|\mathcal{S}_i|=2\not=n$ for $i \in \left\lbrace 1,2\right\rbrace $}. Then we are not in one of the previous cases since $|\mathcal{S}_i|=1+\phi(n)=2$ if and only if $\phi(n)=1$, then we must have $n=2$, that we already exclude, or $n=1$ that is a contradiction since $|\mathcal{S}_i|\leq n$. Hence, by Proposition \ref{S>1}, $G_i$ is generalised quaternion. Thus $G_1 \cong G_2$, and hence $\vec{\mathcal{P}}(G_1)\cong \vec{\mathcal{P}}(G_2)$.
						
						\item \underline{$|\mathcal{S}_i|=1$ for $i \in \left\lbrace 1,2\right\rbrace $}. From the power graph we recognize all the $\mathtt{N}$-classes. We are able to identify the type of every $\mathtt{N}$-classes.
						The star class is $[1]_{\diamond}$, and all the $\mathtt{N}$-classes of the first type are already splitted in $\diamond$-classes since they are themself $\diamond$-classes. We only have to analyse $\mathtt{N}$-classes of the second type.
						Let $C_y$ an $\mathtt{N}$-class of the second type with parameters $(p,r,s)$. Remember that the elements of an $\mathtt{N}$-class are indistinguishable (see discussion after Definition \ref{Ndef}). So we partition it arbitrarily into subsets of sizes $p^i-p^{i-1}=p^{i-1}(p-1)=\phi(p^i)$ for $i=s+1,...,r$. Now, up to graph isomorphism, we have partitioned the vertices of $C_y$ into equivalence classes of the relation $\circ$ and hence, by Corollary \ref{o-partition}, of the relation $\diamond$. So we can partition $G_i$ in $\diamond$-classes for $i \in \{1,2\}$.\\
						Now for all $\left\lbrace x,y\right\rbrace \in E_{\mathcal{P}(G_i)} $, retracing the discussion that follows Lemma \ref{lemmaClassiDiamond}, we are able to identify which of $(x,y)$ and $(y,x)$ is in $A_{\vec{\mathcal{P}}(G_i)}$ for $i \in \{1,2\}$. Hence we can reconstruct $\vec{\mathcal{P}}(G_i)$ from $\mathcal{P}(G_i)$ for $i\in \{1,2\}$.
						Therefore, since $\mathcal{P}(G_1)\cong \mathcal{P}(G_2)$ we obtain that $\vec{\mathcal{P}}(G_1)\cong \vec{\mathcal{P}}(G_2)$.
					\end{itemize}
				\end{proof}
			
			\section{Two examples of reconstruction}
			
			In this Section we show concretely the reconstruction of $\vec{\mathcal{P}}(G)$ from $\mathcal{P}(G)$ in a couple of significant examples.
			Let $G$ be a group with $\mathcal{P}(G)$ the graph in Figure \ref{PGD15}.
			We immediately recognize the identity as the only star vertex, thus, for the reconstruction of $\vec{\mathcal{P}}(G)$, we have to follow the last point of the proof of Theorem \ref{UPG-DPG}.
			We can focus on $\mathcal{P}^*(G)$, represented in Figure \ref{PPGD15}. We already know the directions of all the arcs that join the identity with all the other vertices. Then we easily manage the blue vertices. They are clearly involutions so they are joined to the identity by arcs that are directed from the involutions to the identity.
			
			\begin{figure}
				\centering
			\begin{tikzpicture}[line cap=round,line join=round,>=triangle 45,x=1.0cm,y=1.0cm]
				\clip(-6.,-8.) rectangle (8.,10.);
				\draw [line width=2.pt] (2.993657089646132,-0.5747290551732837)-- (4.641980191492829,0.6246686768274756);
				\draw [line width=2.pt] (2.993657089646132,-0.5747290551732837)-- (5.659959266685525,2.390806431953024);
				\draw [line width=2.pt] (2.993657089646132,-0.5747290551732837)-- (5.871576485073833,4.418302947484353);
				\draw [line width=2.pt] (2.993657089646132,-0.5747290551732837)-- (5.240241307116339,6.356585656658499);
				\draw [line width=2.pt] (2.993657089646132,-0.5747290551732837)-- (3.8751173205819267,7.870507870313641);
				\draw [line width=2.pt] (2.993657089646132,-0.5747290551732837)-- (2.0122468645044806,8.698298684157265);
				\draw [line width=2.pt] (2.993657089646132,-0.5747290551732837)-- (-0.026262835613412694,8.696825546232356);
				\draw [line width=2.pt] (2.993657089646132,-0.5747290551732837)-- (-1.8879349333421296,7.866343175469165);
				\draw [line width=2.pt] (2.993657089646132,-0.5747290551732837)-- (-3.2508694102243165,6.350449518468134);
				\draw [line width=2.pt] (2.993657089646132,-0.5747290551732837)-- (-3.8794025133360366,4.411256359986279);
				\draw [line width=2.pt] (2.993657089646132,-0.5747290551732837)-- (-3.6648551591052874,2.3840678141418907);
				\draw [line width=2.pt] (2.993657089646132,-0.5747290551732837)-- (-2.6443245341800923,0.6194031969412541);
				\draw [line width=2.pt] (2.993657089646132,-0.5747290551732837)-- (-0.9942696548396859,-0.5776109478268854);
				\draw [line width=2.pt] (-0.9942696548396859,-0.5776109478268854)-- (-2.6443245341800923,0.6194031969412541);
				\draw [line width=2.pt] (-0.9942696548396859,-0.5776109478268854)-- (-3.6648551591052874,2.3840678141418907);
				\draw [line width=2.pt] (-0.9942696548396859,-0.5776109478268854)-- (-3.8794025133360366,4.411256359986279);
				\draw [line width=2.pt] (-0.9942696548396859,-0.5776109478268854)-- (-3.2508694102243165,6.350449518468134);
				\draw [line width=2.pt] (-0.9942696548396859,-0.5776109478268854)-- (-1.8879349333421296,7.866343175469165);
				\draw [line width=2.pt] (-0.9942696548396859,-0.5776109478268854)-- (-0.026262835613412694,8.696825546232356);
				\draw [line width=2.pt] (-0.9942696548396859,-0.5776109478268854)-- (2.0122468645044806,8.698298684157265);
				\draw [line width=2.pt] (-0.9942696548396859,-0.5776109478268854)-- (3.8751173205819267,7.870507870313641);
				\draw [line width=2.pt] (-0.9942696548396859,-0.5776109478268854)-- (5.240241307116339,6.356585656658499);
				\draw [line width=2.pt] (-0.9942696548396859,-0.5776109478268854)-- (5.871576485073833,4.418302947484353);
				\draw [line width=2.pt] (-0.9942696548396859,-0.5776109478268854)-- (5.659959266685525,2.390806431953024);
				\draw [line width=2.pt] (-0.9942696548396859,-0.5776109478268854)-- (4.641980191492829,0.6246686768274756);
				\draw [line width=2.pt] (4.641980191492829,0.6246686768274756)-- (5.659959266685525,2.390806431953024);
				\draw [line width=2.pt] (4.641980191492829,0.6246686768274756)-- (5.871576485073833,4.418302947484353);
				\draw [line width=2.pt] (4.641980191492829,0.6246686768274756)-- (5.240241307116339,6.356585656658499);
				\draw [line width=2.pt] (4.641980191492829,0.6246686768274756)-- (3.8751173205819267,7.870507870313641);
				\draw [line width=2.pt] (4.641980191492829,0.6246686768274756)-- (2.0122468645044806,8.698298684157265);
				\draw [line width=2.pt] (4.641980191492829,0.6246686768274756)-- (-0.026262835613412694,8.696825546232356);
				\draw [line width=2.pt] (4.641980191492829,0.6246686768274756)-- (-1.8879349333421296,7.866343175469165);
				\draw [line width=2.pt] (4.641980191492829,0.6246686768274756)-- (-3.2508694102243165,6.350449518468134);
				\draw [line width=2.pt] (4.641980191492829,0.6246686768274756)-- (-3.8794025133360366,4.411256359986279);
				\draw [line width=2.pt] (4.641980191492829,0.6246686768274756)-- (-3.6648551591052874,2.3840678141418907);
				\draw [line width=2.pt] (4.641980191492829,0.6246686768274756)-- (-2.6443245341800923,0.6194031969412541);
				\draw [line width=2.pt] (5.871576485073833,4.418302947484353)-- (5.659959266685525,2.390806431953024);
				\draw [line width=2.pt] (5.871576485073833,4.418302947484353)-- (5.240241307116339,6.356585656658499);
				\draw [line width=2.pt] (5.871576485073833,4.418302947484353)-- (3.8751173205819267,7.870507870313641);
				\draw [line width=2.pt] (5.871576485073833,4.418302947484353)-- (2.0122468645044806,8.698298684157265);
				\draw [line width=2.pt] (5.871576485073833,4.418302947484353)-- (-0.026262835613412694,8.696825546232356);
				\draw [line width=2.pt] (5.871576485073833,4.418302947484353)-- (-1.8879349333421296,7.866343175469165);
				\draw [line width=2.pt] (5.871576485073833,4.418302947484353)-- (-3.2508694102243165,6.350449518468134);
				\draw [line width=2.pt] (5.871576485073833,4.418302947484353)-- (-3.8794025133360366,4.411256359986279);
				\draw [line width=2.pt] (5.871576485073833,4.418302947484353)-- (-3.6648551591052874,2.3840678141418907);
				\draw [line width=2.pt] (5.871576485073833,4.418302947484353)-- (-2.6443245341800923,0.6194031969412541);
				\draw [line width=2.pt] (2.0122468645044806,8.698298684157265)-- (5.240241307116339,6.356585656658499);
				\draw [line width=2.pt] (2.0122468645044806,8.698298684157265)-- (3.8751173205819267,7.870507870313641);
				\draw [line width=2.pt] (2.0122468645044806,8.698298684157265)-- (-0.026262835613412694,8.696825546232356);
				\draw [line width=2.pt] (2.0122468645044806,8.698298684157265)-- (-1.8879349333421296,7.866343175469165);
				\draw [line width=2.pt] (2.0122468645044806,8.698298684157265)-- (-3.2508694102243165,6.350449518468134);
				\draw [line width=2.pt] (2.0122468645044806,8.698298684157265)-- (-3.8794025133360366,4.411256359986279);
				\draw [line width=2.pt] (2.0122468645044806,8.698298684157265)-- (-3.6648551591052874,2.3840678141418907);
				\draw [line width=2.pt] (2.0122468645044806,8.698298684157265)-- (-2.6443245341800923,0.6194031969412541);
				\draw [line width=2.pt] (-0.026262835613412694,8.696825546232356)-- (3.8751173205819267,7.870507870313641);
				\draw [line width=2.pt] (-0.026262835613412694,8.696825546232356)-- (5.240241307116339,6.356585656658499);
				\draw [line width=2.pt] (-0.026262835613412694,8.696825546232356)-- (5.659959266685525,2.390806431953024);
				\draw [line width=2.pt] (2.0122468645044806,8.698298684157265)-- (5.659959266685525,2.390806431953024);
				\draw [line width=2.pt] (-0.026262835613412694,8.696825546232356)-- (-2.6443245341800923,0.6194031969412541);
				\draw [line width=2.pt] (-0.026262835613412694,8.696825546232356)-- (-1.8879349333421296,7.866343175469165);
				\draw [line width=2.pt] (-0.026262835613412694,8.696825546232356)-- (-3.2508694102243165,6.350449518468134);
				\draw [line width=2.pt] (-0.026262835613412694,8.696825546232356)-- (-3.8794025133360366,4.411256359986279);
				\draw [line width=2.pt] (-0.026262835613412694,8.696825546232356)-- (-3.6648551591052874,2.3840678141418907);
				\draw [line width=2.pt] (-3.8794025133360366,4.411256359986279)-- (-3.2508694102243165,6.350449518468134);
				\draw [line width=2.pt] (-3.8794025133360366,4.411256359986279)-- (-1.8879349333421296,7.866343175469165);
				\draw [line width=2.pt] (-3.8794025133360366,4.411256359986279)-- (3.8751173205819267,7.870507870313641);
				\draw [line width=2.pt] (-3.8794025133360366,4.411256359986279)-- (5.240241307116339,6.356585656658499);
				\draw [line width=2.pt] (-3.8794025133360366,4.411256359986279)-- (5.659959266685525,2.390806431953024);
				\draw [line width=2.pt] (-3.8794025133360366,4.411256359986279)-- (-2.6443245341800923,0.6194031969412541);
				\draw [line width=2.pt] (-3.8794025133360366,4.411256359986279)-- (-3.6648551591052874,2.3840678141418907);
				\draw [line width=2.pt] (-2.6443245341800923,0.6194031969412541)-- (5.659959266685525,2.390806431953024);
				\draw [line width=2.pt] (-2.6443245341800923,0.6194031969412541)-- (5.240241307116339,6.356585656658499);
				\draw [line width=2.pt] (-2.6443245341800923,0.6194031969412541)-- (3.8751173205819267,7.870507870313641);
				\draw [line width=2.pt] (-2.6443245341800923,0.6194031969412541)-- (-1.8879349333421296,7.866343175469165);
				\draw [line width=2.pt] (-2.6443245341800923,0.6194031969412541)-- (-3.2508694102243165,6.350449518468134);
				\draw [line width=2.pt] (-2.6443245341800923,0.6194031969412541)-- (-3.6648551591052874,2.3840678141418907);
				\draw [line width=2.pt] (5.659959266685525,2.390806431953024)-- (3.8751173205819267,7.870507870313641);
				\draw [line width=2.pt] (-1.8879349333421296,7.866343175469165)-- (5.659959266685525,2.390806431953024);
				\draw [line width=2.pt] (-3.6648551591052874,2.3840678141418907)-- (5.659959266685525,2.390806431953024);
				\draw [line width=2.pt] (5.240241307116339,6.356585656658499)-- (-3.2508694102243165,6.350449518468134);
				\draw [line width=2.pt] (3.8751173205819267,7.870507870313641)-- (-3.6648551591052874,2.3840678141418907);
				\draw [line width=2.pt] (3.8751173205819267,7.870507870313641)-- (-1.8879349333421296,7.866343175469165);
				\draw [line width=2.pt] (-1.8879349333421296,7.866343175469165)-- (-3.6648551591052874,2.3840678141418907);
				\draw [line width=2.pt] (-0.9942696548396859,-0.5776109478268854)-- (1.,-1.);
				\draw [line width=2.pt] (-2.6443245341800923,0.6194031969412541)-- (1.,-1.);
				\draw [line width=2.pt] (-3.6648551591052874,2.3840678141418907)-- (1.,-1.);
				\draw [line width=2.pt] (-3.8794025133360366,4.411256359986279)-- (1.,-1.);
				\draw [line width=2.pt] (-3.2508694102243165,6.350449518468134)-- (1.,-1.);
				\draw [line width=2.pt] (-1.8879349333421296,7.866343175469165)-- (1.,-1.);
				\draw [line width=2.pt] (-0.026262835613412694,8.696825546232356)-- (1.,-1.);
				\draw [line width=2.pt] (2.0122468645044806,8.698298684157265)-- (1.,-1.);
				\draw [line width=2.pt] (3.8751173205819267,7.870507870313641)-- (1.,-1.);
				\draw [line width=2.pt] (5.240241307116339,6.356585656658499)-- (1.,-1.);
				\draw [line width=2.pt] (5.871576485073833,4.418302947484353)-- (1.,-1.);
				\draw [line width=2.pt] (5.659959266685525,2.390806431953024)-- (1.,-1.);
				\draw [line width=2.pt] (4.641980191492829,0.6246686768274756)-- (1.,-1.);
				\draw [line width=2.pt] (2.993657089646132,-0.5747290551732837)-- (1.,-1.);
				\draw [line width=2.pt] (-4.76303418111067,-0.9849782327152496)-- (1.,-1.);
				\draw [line width=2.pt] (-4.632934599631879,-2.2177177104782277)-- (1.,-1.);
				\draw [line width=2.pt] (-4.206057079291719,-3.471792530391428)-- (1.,-1.);
				\draw [line width=2.pt] (-3.393564195655358,-4.729528426597333)-- (1.,-1.);
				\draw [line width=2.pt] (-2.7330563266424126,-5.390567057711037)-- (1.,-1.);
				\draw [line width=2.pt] (-1.9550595614531452,-5.947768346881919)-- (1.,-1.);
				\draw [line width=2.pt] (6.762864846375749,-0.9533375038290286)-- (1.,-1.);
				\draw [line width=2.pt] (6.626501119770371,-2.247106160423023)-- (1.,-1.);
				\draw [line width=2.pt] (6.191394438175647,-3.502441290704189)-- (1.,-1.);
				\draw [line width=2.pt] (5.404330695240475,-4.716807736903384)-- (1.,-1.);
				\draw [line width=2.pt] (4.72012682257888,-5.401527581428005)-- (1.,-1.);
				\draw [line width=2.pt] (3.891677844983721,-5.985076515658771)-- (1.,-1.);
				\draw [line width=2.pt] (0.9853608434239401,-6.7630351657123375)-- (1.,-1.);
				\draw [line width=2.pt] (-0.28027978423590794,-6.619045497254774)-- (1.,-1.);
				\draw [line width=2.pt] (2.296625560040442,-6.615296143854935)-- (1.,-1.);
				\begin{scriptsize}
					\draw [fill=ffffff] (1.,-1.) circle (2.5pt);
					\draw [fill=ffqqqq] (2.993657089646132,-0.5747290551732837) circle (2.5pt);
					\draw [fill=ffqqqq] (4.641980191492829,0.6246686768274756) circle (2.5pt);
					\draw [fill=qqffqq] (5.659959266685525,2.390806431953024) circle (2.5pt);
					\draw [fill=ffqqqq] (5.871576485073833,4.418302947484353) circle (2.5pt);
					\draw [fill=ffqqff] (5.240241307116339,6.356585656658499) circle (2.5pt);
					\draw [fill=qqffqq] (3.8751173205819267,7.870507870313641) circle (2.5pt);
					\draw [fill=ffqqqq] (2.0122468645044806,8.698298684157265) circle (2.5pt);
					\draw [fill=ffqqqq] (-0.026262835613412694,8.696825546232356) circle (2.5pt);
					\draw [fill=qqffqq] (-1.8879349333421296,7.866343175469165) circle (2.5pt);
					\draw [fill=ffqqff] (-3.2508694102243165,6.350449518468134) circle (2.5pt);
					\draw [fill=ffqqqq] (-3.8794025133360366,4.411256359986279) circle (2.5pt);
					\draw [fill=qqffqq] (-3.6648551591052874,2.3840678141418907) circle (2.5pt);
					\draw [fill=ffqqqq] (-2.6443245341800923,0.6194031969412541) circle (2.5pt);
					\draw [fill=ffqqqq] (-0.9942696548396859,-0.5776109478268854) circle (2.5pt);
					\draw [fill=qqqqff] (-4.76303418111067,-0.9849782327152496) circle (2.5pt);
					\draw [fill=qqqqff] (-4.632934599631879,-2.2177177104782277) circle (2.5pt);
					\draw [fill=qqqqff] (-4.206057079291719,-3.471792530391428) circle (2.5pt);
					\draw [fill=qqqqff] (-3.393564195655358,-4.729528426597333) circle (2.5pt);
					\draw [fill=qqqqff] (-1.9550595614531452,-5.947768346881919) circle (2.5pt);
					\draw [fill=qqqqff] (3.891677844983721,-5.985076515658771) circle (2.5pt);
					\draw [fill=qqqqff] (5.404330695240475,-4.716807736903384) circle (2.5pt);
					\draw [fill=qqqqff] (6.191394438175647,-3.502441290704189) circle (2.5pt);
					\draw [fill=qqqqff] (6.626501119770371,-2.247106160423023) circle (2.5pt);
					\draw [fill=qqqqff] (6.762864846375749,-0.9533375038290286) circle (2.5pt);
					\draw [fill=qqqqff] (-0.28027978423590794,-6.619045497254774) circle (2.5pt);
					\draw [fill=qqqqff] (2.296625560040442,-6.615296143854935) circle (2.5pt);
					\draw [fill=qqqqff] (0.9853608434239401,-6.7630351657123375) circle (2.5pt);
					\draw [fill=qqqqff] (-2.7330563266424126,-5.390567057711037) circle (2.5pt);
					\draw [fill=qqqqff] (4.72012682257888,-5.401527581428005) circle (2.5pt);
				\end{scriptsize}
			\end{tikzpicture}
			
			\caption{$\mathcal{P}(D_{15})$}
			\label{PGD15}
			\end{figure}
		
			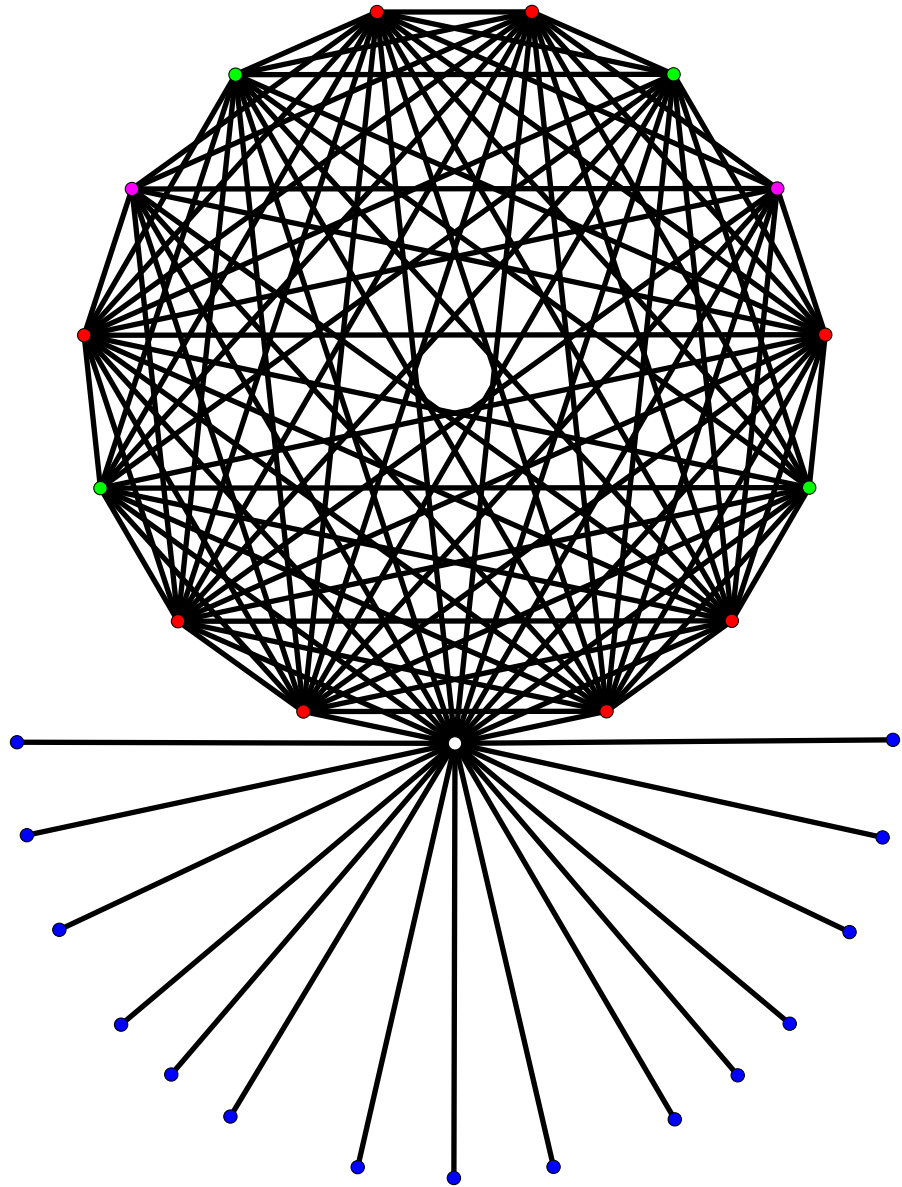
\begin{figure}
				\centering
				\begin{tikzpicture}[line cap=round,line join=round,>=triangle 45,x=1.0cm,y=1.0cm]
					\clip(-6.,-8.) rectangle (8.,10.);
					\draw [line width=2.pt] (2.993657089646132,-0.5747290551732837)-- (4.641980191492829,0.6246686768274756);
					\draw [line width=2.pt] (2.993657089646132,-0.5747290551732837)-- (5.659959266685525,2.390806431953024);
					\draw [line width=2.pt] (2.993657089646132,-0.5747290551732837)-- (5.871576485073833,4.418302947484353);
					\draw [line width=2.pt] (2.993657089646132,-0.5747290551732837)-- (5.240241307116339,6.356585656658499);
					\draw [line width=2.pt] (2.993657089646132,-0.5747290551732837)-- (3.8751173205819267,7.870507870313641);
					\draw [line width=2.pt] (2.993657089646132,-0.5747290551732837)-- (2.0122468645044806,8.698298684157265);
					\draw [line width=2.pt] (2.993657089646132,-0.5747290551732837)-- (-0.026262835613412694,8.696825546232356);
					\draw [line width=2.pt] (2.993657089646132,-0.5747290551732837)-- (-1.8879349333421296,7.866343175469165);
					\draw [line width=2.pt] (2.993657089646132,-0.5747290551732837)-- (-3.2508694102243165,6.350449518468134);
					\draw [line width=2.pt] (2.993657089646132,-0.5747290551732837)-- (-3.8794025133360366,4.411256359986279);
					\draw [line width=2.pt] (2.993657089646132,-0.5747290551732837)-- (-3.6648551591052874,2.3840678141418907);
					\draw [line width=2.pt] (2.993657089646132,-0.5747290551732837)-- (-2.6443245341800923,0.6194031969412541);
					\draw [line width=2.pt] (2.993657089646132,-0.5747290551732837)-- (-0.9942696548396859,-0.5776109478268854);
					\draw [line width=2.pt] (-0.9942696548396859,-0.5776109478268854)-- (-2.6443245341800923,0.6194031969412541);
					\draw [line width=2.pt] (-0.9942696548396859,-0.5776109478268854)-- (-3.6648551591052874,2.3840678141418907);
					\draw [line width=2.pt] (-0.9942696548396859,-0.5776109478268854)-- (-3.8794025133360366,4.411256359986279);
					\draw [line width=2.pt] (-0.9942696548396859,-0.5776109478268854)-- (-3.2508694102243165,6.350449518468134);
					\draw [line width=2.pt] (-0.9942696548396859,-0.5776109478268854)-- (-1.8879349333421296,7.866343175469165);
					\draw [line width=2.pt] (-0.9942696548396859,-0.5776109478268854)-- (-0.026262835613412694,8.696825546232356);
					\draw [line width=2.pt] (-0.9942696548396859,-0.5776109478268854)-- (2.0122468645044806,8.698298684157265);
					\draw [line width=2.pt] (-0.9942696548396859,-0.5776109478268854)-- (3.8751173205819267,7.870507870313641);
					\draw [line width=2.pt] (-0.9942696548396859,-0.5776109478268854)-- (5.240241307116339,6.356585656658499);
					\draw [line width=2.pt] (-0.9942696548396859,-0.5776109478268854)-- (5.871576485073833,4.418302947484353);
					\draw [line width=2.pt] (-0.9942696548396859,-0.5776109478268854)-- (5.659959266685525,2.390806431953024);
					\draw [line width=2.pt] (-0.9942696548396859,-0.5776109478268854)-- (4.641980191492829,0.6246686768274756);
					\draw [line width=2.pt] (4.641980191492829,0.6246686768274756)-- (5.659959266685525,2.390806431953024);
					\draw [line width=2.pt] (4.641980191492829,0.6246686768274756)-- (5.871576485073833,4.418302947484353);
					\draw [line width=2.pt] (4.641980191492829,0.6246686768274756)-- (5.240241307116339,6.356585656658499);
					\draw [line width=2.pt] (4.641980191492829,0.6246686768274756)-- (3.8751173205819267,7.870507870313641);
					\draw [line width=2.pt] (4.641980191492829,0.6246686768274756)-- (2.0122468645044806,8.698298684157265);
					\draw [line width=2.pt] (4.641980191492829,0.6246686768274756)-- (-0.026262835613412694,8.696825546232356);
					\draw [line width=2.pt] (4.641980191492829,0.6246686768274756)-- (-1.8879349333421296,7.866343175469165);
					\draw [line width=2.pt] (4.641980191492829,0.6246686768274756)-- (-3.2508694102243165,6.350449518468134);
					\draw [line width=2.pt] (4.641980191492829,0.6246686768274756)-- (-3.8794025133360366,4.411256359986279);
					\draw [line width=2.pt] (4.641980191492829,0.6246686768274756)-- (-3.6648551591052874,2.3840678141418907);
					\draw [line width=2.pt] (4.641980191492829,0.6246686768274756)-- (-2.6443245341800923,0.6194031969412541);
					\draw [line width=2.pt] (5.871576485073833,4.418302947484353)-- (5.659959266685525,2.390806431953024);
					\draw [line width=2.pt] (5.871576485073833,4.418302947484353)-- (5.240241307116339,6.356585656658499);
					\draw [line width=2.pt] (5.871576485073833,4.418302947484353)-- (3.8751173205819267,7.870507870313641);
					\draw [line width=2.pt] (5.871576485073833,4.418302947484353)-- (2.0122468645044806,8.698298684157265);
					\draw [line width=2.pt] (5.871576485073833,4.418302947484353)-- (-0.026262835613412694,8.696825546232356);
					\draw [line width=2.pt] (5.871576485073833,4.418302947484353)-- (-1.8879349333421296,7.866343175469165);
					\draw [line width=2.pt] (5.871576485073833,4.418302947484353)-- (-3.2508694102243165,6.350449518468134);
					\draw [line width=2.pt] (5.871576485073833,4.418302947484353)-- (-3.8794025133360366,4.411256359986279);
					\draw [line width=2.pt] (5.871576485073833,4.418302947484353)-- (-3.6648551591052874,2.3840678141418907);
					\draw [line width=2.pt] (5.871576485073833,4.418302947484353)-- (-2.6443245341800923,0.6194031969412541);
					\draw [line width=2.pt] (2.0122468645044806,8.698298684157265)-- (5.240241307116339,6.356585656658499);
					\draw [line width=2.pt] (2.0122468645044806,8.698298684157265)-- (3.8751173205819267,7.870507870313641);
					\draw [line width=2.pt] (2.0122468645044806,8.698298684157265)-- (-0.026262835613412694,8.696825546232356);
					\draw [line width=2.pt] (2.0122468645044806,8.698298684157265)-- (-1.8879349333421296,7.866343175469165);
					\draw [line width=2.pt] (2.0122468645044806,8.698298684157265)-- (-3.2508694102243165,6.350449518468134);
					\draw [line width=2.pt] (2.0122468645044806,8.698298684157265)-- (-3.8794025133360366,4.411256359986279);
					\draw [line width=2.pt] (2.0122468645044806,8.698298684157265)-- (-3.6648551591052874,2.3840678141418907);
					\draw [line width=2.pt] (2.0122468645044806,8.698298684157265)-- (-2.6443245341800923,0.6194031969412541);
					\draw [line width=2.pt] (-0.026262835613412694,8.696825546232356)-- (3.8751173205819267,7.870507870313641);
					\draw [line width=2.pt] (-0.026262835613412694,8.696825546232356)-- (5.240241307116339,6.356585656658499);
					\draw [line width=2.pt] (-0.026262835613412694,8.696825546232356)-- (5.659959266685525,2.390806431953024);
					\draw [line width=2.pt] (2.0122468645044806,8.698298684157265)-- (5.659959266685525,2.390806431953024);
					\draw [line width=2.pt] (-0.026262835613412694,8.696825546232356)-- (-2.6443245341800923,0.6194031969412541);
					\draw [line width=2.pt] (-0.026262835613412694,8.696825546232356)-- (-1.8879349333421296,7.866343175469165);
					\draw [line width=2.pt] (-0.026262835613412694,8.696825546232356)-- (-3.2508694102243165,6.350449518468134);
					\draw [line width=2.pt] (-0.026262835613412694,8.696825546232356)-- (-3.8794025133360366,4.411256359986279);
					\draw [line width=2.pt] (-0.026262835613412694,8.696825546232356)-- (-3.6648551591052874,2.3840678141418907);
					\draw [line width=2.pt] (-3.8794025133360366,4.411256359986279)-- (-3.2508694102243165,6.350449518468134);
					\draw [line width=2.pt] (-3.8794025133360366,4.411256359986279)-- (-1.8879349333421296,7.866343175469165);
					\draw [line width=2.pt] (-3.8794025133360366,4.411256359986279)-- (3.8751173205819267,7.870507870313641);
					\draw [line width=2.pt] (-3.8794025133360366,4.411256359986279)-- (5.240241307116339,6.356585656658499);
					\draw [line width=2.pt] (-3.8794025133360366,4.411256359986279)-- (5.659959266685525,2.390806431953024);
					\draw [line width=2.pt] (-3.8794025133360366,4.411256359986279)-- (-2.6443245341800923,0.6194031969412541);
					\draw [line width=2.pt] (-3.8794025133360366,4.411256359986279)-- (-3.6648551591052874,2.3840678141418907);
					\draw [line width=2.pt] (-2.6443245341800923,0.6194031969412541)-- (5.659959266685525,2.390806431953024);
					\draw [line width=2.pt] (-2.6443245341800923,0.6194031969412541)-- (5.240241307116339,6.356585656658499);
					\draw [line width=2.pt] (-2.6443245341800923,0.6194031969412541)-- (3.8751173205819267,7.870507870313641);
					\draw [line width=2.pt] (-2.6443245341800923,0.6194031969412541)-- (-1.8879349333421296,7.866343175469165);
					\draw [line width=2.pt] (-2.6443245341800923,0.6194031969412541)-- (-3.2508694102243165,6.350449518468134);
					\draw [line width=2.pt] (-2.6443245341800923,0.6194031969412541)-- (-3.6648551591052874,2.3840678141418907);
					\draw [line width=2.pt] (5.659959266685525,2.390806431953024)-- (3.8751173205819267,7.870507870313641);
					\draw [line width=2.pt] (-1.8879349333421296,7.866343175469165)-- (5.659959266685525,2.390806431953024);
					\draw [line width=2.pt] (-3.6648551591052874,2.3840678141418907)-- (5.659959266685525,2.390806431953024);
					\draw [line width=2.pt] (5.240241307116339,6.356585656658499)-- (-3.2508694102243165,6.350449518468134);
					\draw [line width=2.pt] (3.8751173205819267,7.870507870313641)-- (-3.6648551591052874,2.3840678141418907);
					\draw [line width=2.pt] (3.8751173205819267,7.870507870313641)-- (-1.8879349333421296,7.866343175469165);
					\draw [line width=2.pt] (-1.8879349333421296,7.866343175469165)-- (-3.6648551591052874,2.3840678141418907);
					\begin{scriptsize}
						\draw [fill=ffqqqq] (2.993657089646132,-0.5747290551732837) circle (2.5pt);
						\draw [fill=ffqqqq] (4.641980191492829,0.6246686768274756) circle (2.5pt);
						\draw [fill=qqffqq] (5.659959266685525,2.390806431953024) circle (2.5pt);
						\draw [fill=ffqqqq] (5.871576485073833,4.418302947484353) circle (2.5pt);
						\draw [fill=ffqqff] (5.240241307116339,6.356585656658499) circle (2.5pt);
						\draw [fill=qqffqq] (3.8751173205819267,7.870507870313641) circle (2.5pt);
						\draw [fill=ffqqqq] (2.0122468645044806,8.698298684157265) circle (2.5pt);
						\draw [fill=ffqqqq] (-0.026262835613412694,8.696825546232356) circle (2.5pt);
						\draw [fill=qqffqq] (-1.8879349333421296,7.866343175469165) circle (2.5pt);
						\draw [fill=ffqqff] (-3.2508694102243165,6.350449518468134) circle (2.5pt);
						\draw [fill=ffqqqq] (-3.8794025133360366,4.411256359986279) circle (2.5pt);
						\draw [fill=qqffqq] (-3.6648551591052874,2.3840678141418907) circle (2.5pt);
						\draw [fill=ffqqqq] (-2.6443245341800923,0.6194031969412541) circle (2.5pt);
						\draw [fill=ffqqqq] (-0.9942696548396859,-0.5776109478268854) circle (2.5pt);
						\draw [fill=qqqqff] (-4.76303418111067,-0.9849782327152496) circle (2.5pt);
						\draw [fill=qqqqff] (-4.632934599631879,-2.2177177104782277) circle (2.5pt);
						\draw [fill=qqqqff] (-4.206057079291719,-3.471792530391428) circle (2.5pt);
						\draw [fill=qqqqff] (-3.393564195655358,-4.729528426597333) circle (2.5pt);
						\draw [fill=qqqqff] (-1.9550595614531452,-5.947768346881919) circle (2.5pt);
						\draw [fill=qqqqff] (3.891677844983721,-5.985076515658771) circle (2.5pt);
						\draw [fill=qqqqff] (5.404330695240475,-4.716807736903384) circle (2.5pt);
						\draw [fill=qqqqff] (6.191394438175647,-3.502441290704189) circle (2.5pt);
						\draw [fill=qqqqff] (6.626501119770371,-2.247106160423023) circle (2.5pt);
						\draw [fill=qqqqff] (6.762864846375749,-0.9533375038290286) circle (2.5pt);
						\draw [fill=qqqqff] (-0.28027978423590794,-6.619045497254774) circle (2.5pt);
						\draw [fill=qqqqff] (2.296625560040442,-6.615296143854935) circle (2.5pt);
						\draw [fill=qqqqff] (0.9853608434239401,-6.7630351657123375) circle (2.5pt);
						\draw [fill=qqqqff] (-2.7330563266424126,-5.390567057711037) circle (2.5pt);
						\draw [fill=qqqqff] (4.72012682257888,-5.401527581428005) circle (2.5pt);
					\end{scriptsize}
				\end{tikzpicture}
				
				\caption{$\mathcal{P}^*(D_{15})$}
				\label{PPGD15}
			\end{figure}

			We now have to analyse the remaining three $\mathtt{N}$-classes. We use $[y]_{\mathtt{N}}$ to denote the class of red vertices and we use $[x]_{\mathtt{N}}$  and $[z]_{\mathtt{N}}$ to denote respectively the classes of green and of pink vertices.
			On the one hand, we have that $\hat{[y]}_{\mathtt{N}}=[y]_{\mathtt{N}} \cup \{1\}$ and that $|\hat{[y]}_{\mathtt{N}}|=9$, which means that $[y]_{\mathtt{N}}$ is a critical class. On the other hand $\hat{[x]}_{\mathtt{N}}=[x]_{\mathtt{N}}\cup [y]_{\mathtt{N}}\cup \{1\}$ and $\hat{[z]}_{\mathtt{N}}=[z]_{\mathtt{N}}\cup [y]_{\mathtt{N}}\cup \{1\}$, hence they are not critical classes.
			Let's focus on the non-critical classes. We have that $|[x]_{\mathtt{N}}|=4$, $|\hat{[x]}_{\mathtt{N}}|=12$, $|[z]_{\mathtt{N}}|=2$ and $|\hat{[z]}_{\mathtt{N}}|=10$. Therefore, neither $[x]_{\mathtt{N}}$ nor $[z]_{\mathtt{N}}$ are of the second type and thus we have $[x]_{\mathtt{N}}=[x]_{\diamond}$ and $[z]_{\mathtt{N}}=[z]_{\diamond}$. \\
			To recognize the type of $[y]_{\mathtt{N}}$ we use Lemma \ref{LemmaCriticalClassMio}. We consider $x$, the representative element of the class of green vertices. We have $x\in G\setminus \hat{[y]}_{\mathtt{N}}$, $4=|[x]_{\mathtt{N}}|\leq|[y]_{\mathtt{N}}|=8$ and $\{x,y\}\in E$. Therefore $[y]_{\mathtt{N}}=[y]_{\diamond}$ because it is an $\mathtt{N}$-class of the first type.
			We can now use Lemma \ref{lemmaClassiDiamond} to reconstruct the directed power graph since we know the identity and all the $\diamond$-classes.
			Pick two elements in the same $\diamond$-class. They are joined with an arc in both directions. Otherwise, if we pick a red vertex and one green or pink, the arc between the two vertices is directed from the red one, that belong to a class having greater size, to the other one.
			Finally there are no edges between pink and green vertices, thus there are also no arcs between them.
			Note that we have reconstructed the directed power graph without knowing the group, that was $D_{15}$ for the previous example.
			
			Now consider the power graph in Figure \ref{UPGD8_1}. As before we are in the case of $\mathcal{S}=\{1\}$. The white vertex is the only star vertex. But now the vertices that are neither the identity nor an involution are all in the same $\mathtt{N}$-class. Let's denote this $\mathtt{N}$-class $C$. Since
			$\hat{C}=C \cup \{1\}$ and $|\hat{C}|=8$, we have that $C$ is a critical class. We also have that $|C|=7=2^3-1$ and $|\hat{C}|=8=2^3$, so it has parameters $(p,r)=(2,3)$.
			There are two ways to prove that $C$ is of the second type.
			\begin{itemize}
				\item By Lemma \ref{LemmaCriticalClassMio}, $C$ is of the second type because there is no $x\in G\setminus \hat{C}$ joined with a vertex in $C$. 
				
				\item We can also give an arithmetic explanation: There is no integer $n$ such that $\phi(n)=7$, thus $C$ cannot be of the first type.  
			\end{itemize} 
			Therefore there exists $y \in C$, with $o(y)=2^3$, such that: $$ C=\{z\in \langle y \rangle \, | \, 2\leq o(z) \leq 2^3\}.$$
			Now we can partition the vertices of $C$ in three sets, one formed by $4$ elements, one by $2$ elements and the last one by $1$ element. Up to graph isomorphism we have partitioned the vertices of $C$ in three sets where all the elements have the same order, precisely order $8, 4$  and $2$. Look at Figure \ref{UPGD8} for an example of such a partition.
			By Corollary \ref{o-partition}, we now have the partition of $G$ in $\diamond$-classes. Then, by Lemma \ref{lemmaClassiDiamond}, we are able to reconstruct the directed power graph. 
			\begin{figure}
				\centering
				\begin{tikzpicture}[line cap=round,line join=round,>=triangle 45,x=1.0cm,y=1.0cm]
					\clip(-1.,-3.8) rectangle (5.,1.9);
					\draw [line width=1.pt] (-0.6476093605195885,-1.518936914502213)-- (2.,-1.);
					\draw [line width=1.pt] (-0.2504106699601461,-2.485508342942641)-- (2.,-1.);
					\draw [line width=1.pt] (0.4893710737630752,-3.233625632830833)-- (2.,-1.);
					\draw [line width=1.pt] (2.,-1.)-- (1.4591323776205516,-3.6416936679873135);
					\draw [line width=1.pt] (2.,-1.)-- (2.511235874008093,-3.647587789205458);
					\draw [line width=1.pt] (2.,-1.)-- (3.485508342942641,-3.2504106699601456);
					\draw [line width=1.pt] (2.,-1.)-- (4.233625632830833,-2.5106289262369246);
					\draw [line width=1.pt] (2.,-1.)-- (4.641715239301444,-1.5331665818853286);
					\draw [line width=1.pt] (2.,-1.)-- (2.9742724689345477,-0.602822880754688);
					\draw [line width=1.pt] (2.,-1.)-- (3.3823405040910277,0.3669384231027888);
					\draw [line width=1.pt] (2.,-1.)-- (2.985163384845716,1.3412108920373362);
					\draw [line width=1.pt] (2.,-1.)-- (2.015402080988239,1.7492789271938167);
					\draw [line width=1.pt] (2.,-1.)-- (1.0411296120536917,1.3521018079485048);
					\draw [line width=1.pt] (2.,-1.)-- (0.6330615768972112,0.3823405040910282);
					\draw [line width=1.pt] (2.,-1.)-- (1.030238696142523,-0.5919319648435194);
					\draw [line width=1.pt] (1.030238696142523,-0.5919319648435194)-- (2.9742724689345477,-0.602822880754688);
					\draw [line width=1.pt] (1.030238696142523,-0.5919319648435194)-- (3.3823405040910277,0.3669384231027888);
					\draw [line width=1.pt] (1.030238696142523,-0.5919319648435194)-- (2.985163384845716,1.3412108920373362);
					\draw [line width=1.pt] (1.030238696142523,-0.5919319648435194)-- (2.015402080988239,1.7492789271938167);
					\draw [line width=1.pt] (1.030238696142523,-0.5919319648435194)-- (1.0411296120536917,1.3521018079485048);
					\draw [line width=1.pt] (1.030238696142523,-0.5919319648435194)-- (0.6330615768972112,0.3823405040910282);
					\draw [line width=1.pt] (1.0411296120536917,1.3521018079485048)-- (0.6330615768972112,0.3823405040910282);
					\draw [line width=1.pt] (1.0411296120536917,1.3521018079485048)-- (2.9742724689345477,-0.602822880754688);
					\draw [line width=1.pt] (1.0411296120536917,1.3521018079485048)-- (3.3823405040910277,0.3669384231027888);
					\draw [line width=1.pt] (1.0411296120536917,1.3521018079485048)-- (2.985163384845716,1.3412108920373362);
					\draw [line width=1.pt] (1.0411296120536917,1.3521018079485048)-- (2.015402080988239,1.7492789271938167);
					\draw [line width=1.pt] (2.015402080988239,1.7492789271938167)-- (2.985163384845716,1.3412108920373362);
					\draw [line width=1.pt] (2.015402080988239,1.7492789271938167)-- (3.3823405040910277,0.3669384231027888);
					\draw [line width=1.pt] (2.015402080988239,1.7492789271938167)-- (2.9742724689345477,-0.602822880754688);
					\draw [line width=1.pt] (2.985163384845716,1.3412108920373362)-- (3.3823405040910277,0.3669384231027888);
					\draw [line width=1.pt] (2.985163384845716,1.3412108920373362)-- (2.9742724689345477,-0.602822880754688);
					\draw [line width=1.pt] (3.3823405040910277,0.3669384231027888)-- (0.6330615768972112,0.3823405040910282);
					\draw [line width=1.pt] (3.3823405040910277,0.3669384231027888)-- (2.9742724689345477,-0.602822880754688);
					\draw [line width=1.pt] (0.6330615768972112,0.3823405040910282)-- (2.015402080988239,1.7492789271938167);
					\draw [line width=1.pt] (0.6330615768972112,0.3823405040910282)-- (2.985163384845716,1.3412108920373362);
					\draw [line width=1.pt] 
					(2.9742724689345477,-0.602822880754688)--
					(0.6330615768972112,0.3823405040910282);
					\begin{scriptsize}
						
						\draw [fill=ffffff] (2.,-1.) circle (2.5pt);

						\draw [fill=qqqqff] (2.9742724689345477,-0.602822880754688) circle (2.5pt);
						\draw [fill=qqqqff] (3.3823405040910277,0.3669384231027888) circle (2.5pt);
						\draw [fill=qqqqff] (2.985163384845716,1.3412108920373362) circle (2.5pt);
						\draw [fill=qqqqff] (2.015402080988239,1.7492789271938167) circle (2.5pt);
						\draw [fill=qqqqff] (1.0411296120536917,1.3521018079485048) circle (2.5pt);
						\draw [fill=qqqqff] (0.6330615768972112,0.3823405040910282) circle (2.5pt);
						\draw [fill=qqqqff] (1.030238696142523,-0.5919319648435194) circle (2.5pt);
						
						\draw [fill=qqqqff] (-0.6476093605195885,-1.518936914502213) circle (2.0pt);
						\draw [fill=qqqqff] (-0.2504106699601461,-2.485508342942641) circle (2.0pt);
						\draw [fill=qqqqff] (0.4893710737630752,-3.233625632830833) circle (2.0pt);
						\draw [fill=qqqqff] (1.4591323776205516,-3.6416936679873135) circle (2.0pt);
						\draw [fill=qqqqff] (2.511235874008093,-3.647587789205458) circle (2.0pt);
						\draw [fill=qqqqff] (3.485508342942641,-3.2504106699601456) circle (2.0pt);
						\draw [fill=qqqqff] (4.233625632830833,-2.5106289262369246) circle (2.0pt);
						\draw [fill=qqqqff] (4.641715239301444,-1.5331665818853286) circle (2.0pt);
					\end{scriptsize}
				\end{tikzpicture}
				\caption{$\mathcal{P}(D_8)$}
				\label{UPGD8_1}
			\end{figure}
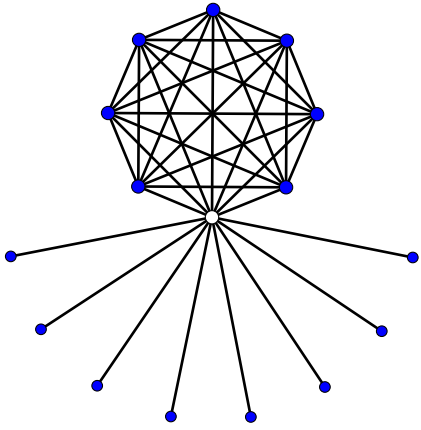
		
			\begin{figure}
				\centering
				\begin{tikzpicture}[line cap=round,line join=round,>=triangle 45,x=1.0cm,y=1.0cm]
					\clip(-1.,-3.8) rectangle (5.,1.9);
					\draw [line width=1.pt] (-0.6476093605195885,-1.518936914502213)-- (2.,-1.);
					\draw [line width=1.pt] (-0.2504106699601461,-2.485508342942641)-- (2.,-1.);
					\draw [line width=1.pt] (0.4893710737630752,-3.233625632830833)-- (2.,-1.);
					\draw [line width=1.pt] (2.,-1.)-- (1.4591323776205516,-3.6416936679873135);
					\draw [line width=1.pt] (2.,-1.)-- (2.511235874008093,-3.647587789205458);
					\draw [line width=1.pt] (2.,-1.)-- (3.485508342942641,-3.2504106699601456);
					\draw [line width=1.pt] (2.,-1.)-- (4.233625632830833,-2.5106289262369246);
					\draw [line width=1.pt] (2.,-1.)-- (4.641715239301444,-1.5331665818853286);
					\draw [line width=1.pt] (2.,-1.)-- (2.9742724689345477,-0.602822880754688);
					\draw [line width=1.pt] (2.,-1.)-- (3.3823405040910277,0.3669384231027888);
					\draw [line width=1.pt] (2.,-1.)-- (2.985163384845716,1.3412108920373362);
					\draw [line width=1.pt] (2.,-1.)-- (2.015402080988239,1.7492789271938167);
					\draw [line width=1.pt] (2.,-1.)-- (1.0411296120536917,1.3521018079485048);
					\draw [line width=1.pt] (2.,-1.)-- (0.6330615768972112,0.3823405040910282);
					\draw [line width=1.pt] (2.,-1.)-- (1.030238696142523,-0.5919319648435194);
					\draw [line width=1.pt] (1.030238696142523,-0.5919319648435194)-- (2.9742724689345477,-0.602822880754688);
					\draw [line width=1.pt] (1.030238696142523,-0.5919319648435194)-- (3.3823405040910277,0.3669384231027888);
					\draw [line width=1.pt] (1.030238696142523,-0.5919319648435194)-- (2.985163384845716,1.3412108920373362);
					\draw [line width=1.pt] (1.030238696142523,-0.5919319648435194)-- (2.015402080988239,1.7492789271938167);
					\draw [line width=1.pt] (1.030238696142523,-0.5919319648435194)-- (1.0411296120536917,1.3521018079485048);
					\draw [line width=1.pt] (1.030238696142523,-0.5919319648435194)-- (0.6330615768972112,0.3823405040910282);
					\draw [line width=1.pt] (1.0411296120536917,1.3521018079485048)-- (0.6330615768972112,0.3823405040910282);
					\draw [line width=1.pt] (1.0411296120536917,1.3521018079485048)-- (2.9742724689345477,-0.602822880754688);
					\draw [line width=1.pt] (1.0411296120536917,1.3521018079485048)-- (3.3823405040910277,0.3669384231027888);
					\draw [line width=1.pt] (1.0411296120536917,1.3521018079485048)-- (2.985163384845716,1.3412108920373362);
					\draw [line width=1.pt] (1.0411296120536917,1.3521018079485048)-- (2.015402080988239,1.7492789271938167);
					\draw [line width=1.pt] (2.015402080988239,1.7492789271938167)-- (2.985163384845716,1.3412108920373362);
					\draw [line width=1.pt] (2.015402080988239,1.7492789271938167)-- (3.3823405040910277,0.3669384231027888);
					\draw [line width=1.pt] (2.015402080988239,1.7492789271938167)-- (2.9742724689345477,-0.602822880754688);
					\draw [line width=1.pt] (2.985163384845716,1.3412108920373362)-- (3.3823405040910277,0.3669384231027888);
					\draw [line width=1.pt] (2.985163384845716,1.3412108920373362)-- (2.9742724689345477,-0.602822880754688);
					\draw [line width=1.pt] (3.3823405040910277,0.3669384231027888)-- (0.6330615768972112,0.3823405040910282);
					\draw [line width=1.pt] (3.3823405040910277,0.3669384231027888)-- (2.9742724689345477,-0.602822880754688);
					\draw [line width=1.pt] (0.6330615768972112,0.3823405040910282)-- (2.015402080988239,1.7492789271938167);
					\draw [line width=1.pt] (0.6330615768972112,0.3823405040910282)-- (2.985163384845716,1.3412108920373362);
					\draw [line width=1.pt] 
					(2.9742724689345477,-0.602822880754688)--
					(0.6330615768972112,0.3823405040910282);
					\begin{scriptsize}
						
					\draw [fill=ffffff] (2.,-1.) circle (2.5pt);

						\draw [fill=qqffqq] (2.9742724689345477,-0.602822880754688) circle (2.5pt);
							\draw [fill=ffqqqq] (3.3823405040910277,0.3669384231027888) circle (2.5pt);
							\draw [fill=qqffqq] (2.985163384845716,1.3412108920373362) circle (2.5pt);
							\draw [fill=cyan] (2.015402080988239,1.7492789271938167) circle (2.5pt);
							\draw [fill=qqffqq] (1.0411296120536917,1.3521018079485048) circle (2.5pt);
							\draw [fill=ffqqqq] (0.6330615768972112,0.3823405040910282) circle (2.5pt);
							\draw [fill=qqffqq] (1.030238696142523,-0.5919319648435194) circle (2.5pt);
						
						\draw [fill=qqqqff] (-0.6476093605195885,-1.518936914502213) circle (2.0pt);
						\draw [fill=qqqqff] (-0.2504106699601461,-2.485508342942641) circle (2.0pt);
						\draw [fill=qqqqff] (0.4893710737630752,-3.233625632830833) circle (2.0pt);
						\draw [fill=qqqqff] (1.4591323776205516,-3.6416936679873135) circle (2.0pt);
						\draw [fill=qqqqff] (2.511235874008093,-3.647587789205458) circle (2.0pt);
						\draw [fill=qqqqff] (3.485508342942641,-3.2504106699601456) circle (2.0pt);
						\draw [fill=qqqqff] (4.233625632830833,-2.5106289262369246) circle (2.0pt);
						\draw [fill=qqqqff] (4.641715239301444,-1.5331665818853286) circle (2.0pt);
					\end{scriptsize}
				\end{tikzpicture}
			\caption{$\mathcal{P}(D_8)$ with elements of the same order highlighted.}
			\label{UPGD8}
			\end{figure}
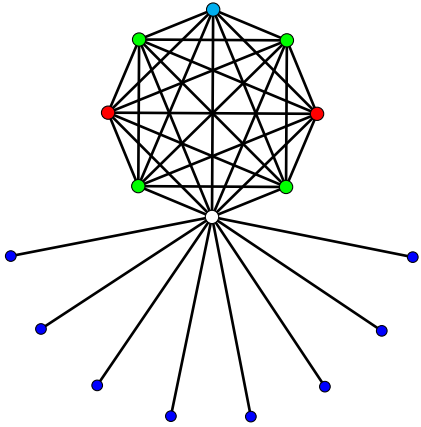
			   
			\section{Some open problems about critical classes}
			\label{OpenProblems}
			We conclude this chapter with some open questions.
			
			\begin{Prob}
				Do there exist groups having a critical class of the first type and a critical class of the second type with the same parameters? 
			\end{Prob}
			
			But before try to answer at the question above it is reasonable to study the following problem.
			\begin{Prob}
				Does there exist a group having a critical class of the first type and a critical class of the second type?
			\end{Prob}
		
			During the research of interesting examples another question has arisen:
			\begin{Prob}
				Does there exist $\mathtt{N}$-classes of the second type that are not critical classes? In other words, does there exist a group having an $\mathtt{N}$-class of the second type with parameters $(p,r,s)$ with $s\ne 0$?
			\end{Prob}
		
			At the moment we have no example of this kind.
			
	\chapter{Order of elements and power graph}
			\label{RelatedResults}
			\begin{prop}{\rm \cite[Proposition 1]{Cameron_2}}
				\label{DPG-elementOrder}Let $G_1$ and $G_2$ be groups with $\vec{\mathcal{P}}(G_1)\cong \vec{\mathcal{P}}(G_2)$. Then $G_1$ and $G_2$ have the same number of elements of each order.
			\end{prop}
			
			\begin{proof}
				In the directed power graph of a group $G$ the maximal complete digraphs are, by Lemma \ref{lemmaClassiDiamond}, the equivalence classes of the relation $\diamond$. Let's see that the $\diamond$-classes are partially ordered by the relation $[y]_{\diamond}\leq [x]_{\diamond}$ if $(x,y)\in A$ or $[x]_{\diamond}=[y]_{\diamond}$, which means that both $(x,y)$ and $(y,x)$ are in $A$. First of all let's see that the partial order is well given:
				
				Let $x,x_1, y, y_1 \in G$ such that $[x]_{\diamond}=[x_1]_{\diamond}$ and $[y]_{\diamond}=[y_1]_{\diamond}$, with $[y]_{\diamond}\leq[x]_{\diamond}$. Thus, if $[x]_{\diamond}=[y]_{\diamond}$, then, clearly, also $[x_1]_{\diamond}=[y_1]_{\diamond}$; else if $(x,y)\in A$ and $[x]_{\diamond}\not=[y]_{\diamond}$, then, by Lemma \ref{lemmaClassiDiamond}, all the vertices of $[x]_{\diamond}$ are joined with all the vertices of $[y]_{\diamond}$ with arcs that are directed from $[x]_{\diamond}$ to $[y]_{\diamond}$, and hence we have also $(x_1,y_1) \in A$. In both cases $[y_1]_{\diamond}\leq[x_1]_{\diamond}$, and so the relation is well defined.
				
				Now we have to prove that $\leq$ is a relation of order:
				
				\begin{itemize}
					\item \underline{Reflexivity} follows immediately by definition.
					
					\item \underline{Antisymmetry} holds since, if we have $[x]_{\diamond}\leq [y]_{\diamond}$ and $[y]_{\diamond}\leq [x]_{\diamond}$, then both $(y,x)$ and $(x,y)$ are in $A$. Thus $[x]_{\diamond}=[y]_{\diamond}$.
					
					\item \underline{Transitivity} also holds. Let be $[x]_{\diamond}\leq [y]_{\diamond}$ and $[y]_{\diamond}\leq [z]_{\diamond}$ for some $x,y,z \in G$. We surely have $(y,x)$ and $(z,y)$ in $A$. Now, since $\vec{\mathcal{P}}(G)$ is transitive, then $(z,x)\in A$, and hence $[x]_{\diamond}\leq [z]_{\diamond}$.
				\end{itemize}
				
				Now consider $\Diamond ^*$ the set of all $\diamond$-classes less the class of the identity. In $\Diamond ^*$, a minimal $\diamond$-class, under the relation $\leq$, consists of elements of prime order $p$, and has cardinality $p-1$. For any equivalence class $[x]_{\diamond}$, we can now determine the set of primes dividing the order of $x$ from the set of minimal classes below $[x]_{\diamond}$ in the partial order. If $x$ has order $n=p_1^{a_1}p_2^{a_2}...p_k^{a_k}$, where $p_i$ are distinct primes and $a_i \in \mathbb{N}$ for all $i \in [k]$, then we now know those primes and hence also the value of $\phi(n)=|[x]_{\diamond}|$. The exponent $a_i$ is one more than the exponent of the power of $p_i$ dividing $$\frac{|[x]_{\diamond}|}{\prod_{i=1}^{k}(p_i-1)}=\prod_{i=1}^{k}p^{a_i-1}.$$ So, for all $x\in G$, the order of $x$ is determined by the directed power graph.
				Therefore, since $\vec{\mathcal{P}}(G_1)\cong \vec{\mathcal{P}}(G_2)$, $G_1$ and $G_2$ have the same number of elements of each order.
			\end{proof}
			
			Let's see some examples of the relation $\leq$ defined in the proposition above.
			
			With $\Diamond$ we refer to the set of all the $\diamond$-classes.
		
			\begin{figure}
				\centering
				\begin{tikzpicture}[line cap=round,line join=round,>=triangle 45,x=1.0cm,y=1.0cm]
					\clip(2.,-6.5) rectangle (6.5,-1.5);
					\fill[line width=2.pt,color=zzttqq,fill=zzttqq,fill opacity=0.10000000149011612] (2.3533590643557503,-1.57835036299306) -- (3.3533590643557503,-1.57835036299306) -- (3.3533590643557503,-2.07835036299306) -- (2.3533590643557503,-2.07835036299306) -- cycle;
					\fill[line width=2.pt,color=zzttqq,fill=zzttqq,fill opacity=0.10000000149011612] (3.8533590643557503,-1.57835036299306) -- (3.8533590643557503,-2.07835036299306) -- (4.853359064355752,-2.07835036299306) -- (4.853359064355752,-1.57835036299306) -- cycle;
					\fill[line width=2.pt,color=zzttqq,fill=zzttqq,fill opacity=0.10000000149011612] (5.353359064355752,-1.57835036299306) -- (5.353359064355752,-2.07835036299306) -- (6.353359064355752,-2.07835036299306) -- (6.353359064355752,-1.57835036299306) -- cycle;
					\fill[line width=2.pt,color=zzttqq,fill=zzttqq,fill opacity=0.10000000149011612] (2.3533590643557503,-4.07835036299306) -- (2.3533590643557503,-4.57835036299306) -- (3.3533590643557503,-4.57835036299306) -- (3.3533590643557503,-4.07835036299306) -- cycle;
					\fill[line width=2.pt,color=zzttqq,fill=zzttqq,fill opacity=0.10000000149011612] (3.8533590643557503,-4.07835036299306) -- (3.8533590643557503,-4.57835036299306) -- (4.35335906435575,-4.57835036299306) -- (4.35335906435575,-4.07835036299306) -- cycle;
					\fill[line width=2.pt,color=zzttqq,fill=zzttqq,fill opacity=0.10000000149011612] (4.853359064355752,-4.07835036299306) -- (4.853359064355752,-4.57835036299306) -- (5.353359064355752,-4.57835036299306) -- (5.353359064355752,-4.07835036299306) -- cycle;
					\fill[line width=2.pt,color=zzttqq,fill=zzttqq,fill opacity=0.10000000149011612] (5.853359064355752,-4.07835036299306) -- (5.853359064355752,-4.57835036299306) -- (6.353359064355752,-4.57835036299306) -- (6.353359064355752,-4.07835036299306) -- cycle;
					\fill[line width=2.pt,color=zzttqq,fill=zzttqq,fill opacity=0.10000000149011612] (3.8533590643557503,-5.57835036299306) -- (3.8533590643557503,-6.07835036299306) -- (4.35335906435575,-6.07835036299306) -- (4.35335906435575,-5.57835036299306) -- cycle;
					\draw (-0.4609052993849324,5.559523070650456) node[anchor=north west] {$C_6xC_2=<a,b>$};
					\draw [line width=1.pt,color=zzttqq] (2.3533590643557503,-1.57835036299306)-- (3.3533590643557503,-1.57835036299306);
					\draw [line width=1.pt,color=zzttqq] (3.3533590643557503,-1.57835036299306)-- (3.3533590643557503,-2.07835036299306);
					\draw [line width=1.pt,color=zzttqq] (3.3533590643557503,-2.07835036299306)-- (2.3533590643557503,-2.07835036299306);
					\draw [line width=1.pt,color=zzttqq] (2.3533590643557503,-2.07835036299306)-- (2.3533590643557503,-1.57835036299306);
					\draw [line width=1.pt,color=zzttqq] (3.8533590643557503,-1.57835036299306)-- (3.8533590643557503,-2.07835036299306);
					\draw [line width=1.pt,color=zzttqq] (3.8533590643557503,-2.07835036299306)-- (4.853359064355752,-2.07835036299306);
					\draw [line width=1.pt,color=zzttqq] (4.853359064355752,-2.07835036299306)-- (4.853359064355752,-1.57835036299306);
					\draw [line width=1.pt,color=zzttqq] (4.853359064355752,-1.57835036299306)-- (3.8533590643557503,-1.57835036299306);
					\draw [line width=1.pt,color=zzttqq] (5.353359064355752,-1.57835036299306)-- (5.353359064355752,-2.07835036299306);
					\draw [line width=1.pt,color=zzttqq] (5.353359064355752,-2.07835036299306)-- (6.353359064355752,-2.07835036299306);
					\draw [line width=1.pt,color=zzttqq] (6.353359064355752,-2.07835036299306)-- (6.353359064355752,-1.57835036299306);
					\draw [line width=1.pt,color=zzttqq] (6.353359064355752,-1.57835036299306)-- (5.353359064355752,-1.57835036299306);
					\draw [line width=1.pt,color=zzttqq] (2.3533590643557503,-4.07835036299306)-- (2.3533590643557503,-4.57835036299306);
					\draw [line width=1.pt,color=zzttqq] (2.3533590643557503,-4.57835036299306)-- (3.3533590643557503,-4.57835036299306);
					\draw [line width=1.pt,color=zzttqq] (3.3533590643557503,-4.57835036299306)-- (3.3533590643557503,-4.07835036299306);
					\draw [line width=1.pt,color=zzttqq] (3.3533590643557503,-4.07835036299306)-- (2.3533590643557503,-4.07835036299306);
					\draw [line width=1.pt,color=zzttqq] (3.8533590643557503,-4.07835036299306)-- (3.8533590643557503,-4.57835036299306);
					\draw [line width=1.pt,color=zzttqq] (3.8533590643557503,-4.57835036299306)-- (4.35335906435575,-4.57835036299306);
					\draw [line width=1.pt,color=zzttqq] (4.35335906435575,-4.57835036299306)-- (4.35335906435575,-4.07835036299306);
					\draw [line width=1.pt,color=zzttqq] (4.35335906435575,-4.07835036299306)-- (3.8533590643557503,-4.07835036299306);
					\draw [line width=1.pt,color=zzttqq] (4.853359064355752,-4.07835036299306)-- (4.853359064355752,-4.57835036299306);
					\draw [line width=1.pt,color=zzttqq] (4.853359064355752,-4.57835036299306)-- (5.353359064355752,-4.57835036299306);
					\draw [line width=1.pt,color=zzttqq] (5.353359064355752,-4.57835036299306)-- (5.353359064355752,-4.07835036299306);
					\draw [line width=1.pt,color=zzttqq] (5.353359064355752,-4.07835036299306)-- (4.853359064355752,-4.07835036299306);
					\draw [line width=1.pt,color=zzttqq] (5.853359064355752,-4.07835036299306)-- (5.853359064355752,-4.57835036299306);
					\draw [line width=1.pt,color=zzttqq] (5.853359064355752,-4.57835036299306)-- (6.353359064355752,-4.57835036299306);
					\draw [line width=1.pt,color=zzttqq] (6.353359064355752,-4.57835036299306)-- (6.353359064355752,-4.07835036299306);
					\draw [line width=1.pt,color=zzttqq] (6.353359064355752,-4.07835036299306)-- (5.853359064355752,-4.07835036299306);
					\draw [line width=1.pt,color=zzttqq] (3.8533590643557503,-5.57835036299306)-- (3.8533590643557503,-6.07835036299306);
					\draw [line width=1.pt,color=zzttqq] (3.8533590643557503,-6.07835036299306)-- (4.35335906435575,-6.07835036299306);
					\draw [line width=1.pt,color=zzttqq] (4.35335906435575,-6.07835036299306)-- (4.35335906435575,-5.57835036299306);
					\draw [line width=1.pt,color=zzttqq] (4.35335906435575,-5.57835036299306)-- (3.8533590643557503,-5.57835036299306);
					\draw [line width=1.pt] (2.8533590643557503,-4.57835036299306)-- (3.8533590643557503,-5.57835036299306);
					\draw [line width=1.pt] (4.10335906435575,-4.57835036299306)-- (4.10335906435575,-5.57835036299306);
					\draw [line width=1.pt] (5.103359064355752,-4.57835036299306)-- (4.10335906435575,-5.57835036299306);
					\draw [line width=1.pt] (5.853359064355752,-4.57835036299306)-- (4.35335906435575,-5.57835036299306);
					\draw [line width=1.pt] (4.353359064355751,-2.07835036299306)-- (5.103359064355752,-4.07835036299306);
					\draw [line width=1.pt] (5.853359064355752,-2.07835036299306)-- (6.103359064355752,-4.07835036299306);
					\draw [line width=1.pt] (2.8533590643557503,-2.07835036299306)-- (2.8533590643557503,-4.07835036299306);
					\draw [line width=1.pt] (4.353359064355751,-2.07835036299306)-- (2.8533590643557503,-4.07835036299306);
					\draw [line width=1.pt] (5.853359064355752,-2.07835036299306)-- (2.8533590643557503,-4.07835036299306);
					\draw [line width=1.pt] (2.8533590643557503,-2.07835036299306)-- (4.10335906435575,-4.07835036299306);
					\begin{scriptsize}
						\draw [fill=qqffff] (4.10335906435575,-4.32835036299306) circle (2.5pt);
						\draw [fill=ffffww] (6.103359064355752,-1.82835036299306) circle (2.5pt);
						\draw [fill=ffqqff] (2.6033590643557503,-1.82835036299306) circle (2.5pt);
						\draw [fill=yqqqqq] (6.103359064355752,-4.32835036299306) circle (2.5pt);
						\draw [fill=ffqqff] (3.1033590643557503,-1.82835036299306) circle (2.5pt);
						\draw [fill=qqffqq] (4.10335906435575,-1.82835036299306) circle (2.5pt);
						\draw [fill=ffqqqq] (2.6033590643557503,-4.32835036299306) circle (2.5pt);
						\draw [fill=ffzztt] (5.103359064355752,-4.32835036299306) circle (2.5pt);
						\draw [fill=ffqqqq] (3.1033590643557503,-4.32835036299306) circle (2.5pt);
						\draw [fill=qqffqq] (4.603359064355752,-1.82835036299306) circle (2.5pt);
						\draw [fill=ffffww] (5.603359064355752,-1.82835036299306) circle (2.5pt);
						\draw [fill=qqqqff] (4.10335906435575,-5.82835036299306) circle (2.5pt);
					\end{scriptsize}
				\end{tikzpicture}
				\caption{The representation of $(\Diamond,\leq)$ for $C_6 \times C_2$.}
				\label{C_6xC_2leqRelation}
			\end{figure}
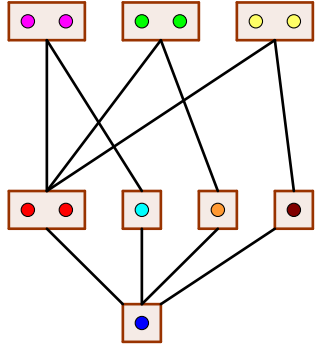
		
			\begin{figure}
				\centering
				\begin{tikzpicture}[line cap=round,line join=round,>=triangle 45,x=1.0cm,y=1.0cm]
					\clip(-7.5,-0.5) rectangle (7.5,3.5);
					\fill[line width=2.pt,color=zzttqq,fill=zzttqq,fill opacity=0.10000000149011612] (-7.2,2.2) -- (-6.4,2.2) -- (-6.4,1.8) -- (-7.2,1.8) -- cycle;
					\fill[line width=2.pt,color=zzttqq,fill=zzttqq,fill opacity=0.10000000149011612] (-1.4,3.2) -- (-0.6,3.2) -- (-0.6,2.8) -- (-1.4,2.8) -- cycle;
					\fill[line width=2.pt,color=zzttqq,fill=zzttqq,fill opacity=0.10000000149011612] (-0.4,3.2) -- (-0.4,2.8) -- (0.4,2.8) -- (0.4,3.2) -- cycle;
					\fill[line width=2.pt,color=zzttqq,fill=zzttqq,fill opacity=0.10000000149011612] (0.6,3.2) -- (0.6,2.8) -- (1.4,2.8) -- (1.4,3.2) -- cycle;
					\fill[line width=2.pt,color=zzttqq,fill=zzttqq,fill opacity=0.10000000149011612] (-5.6,2.2) -- (-5.6,1.8) -- (-4.8,1.8) -- (-4.8,2.2) -- cycle;
					\fill[line width=2.pt,color=zzttqq,fill=zzttqq,fill opacity=0.10000000149011612] (-4.2,2.2) -- (-4.2,1.8) -- (-3.4,1.8) -- (-3.4,2.2) -- cycle;
					\fill[line width=2.pt,color=zzttqq,fill=zzttqq,fill opacity=0.10000000149011612] (-2.6,2.2) -- (-2.6,1.8) -- (-1.8,1.8) -- (-1.8,2.2) -- cycle;
					\fill[line width=2.pt,color=zzttqq,fill=zzttqq,fill opacity=0.10000000149011612] (-1.2,2.2) -- (-1.2,1.8) -- (-0.8,1.8) -- (-0.8,2.2) -- cycle;
					\fill[line width=2.pt,color=zzttqq,fill=zzttqq,fill opacity=0.10000000149011612] (-0.2,2.2) -- (-0.2,1.8) -- (0.2,1.8) -- (0.2,2.2) -- cycle;
					\fill[line width=2.pt,color=zzttqq,fill=zzttqq,fill opacity=0.10000000149011612] (0.8,2.2) -- (0.8,1.8) -- (1.2,1.8) -- (1.2,2.2) -- cycle;
					\fill[line width=2.pt,color=zzttqq,fill=zzttqq,fill opacity=0.10000000149011612] (1.8,2.2) -- (1.8,1.8) -- (2.2,1.8) -- (2.2,2.2) -- cycle;
					\fill[line width=2.pt,color=zzttqq,fill=zzttqq,fill opacity=0.10000000149011612] (2.8,2.2) -- (2.8,1.8) -- (3.2,1.8) -- (3.2,2.2) -- cycle;
					\fill[line width=2.pt,color=zzttqq,fill=zzttqq,fill opacity=0.10000000149011612] (3.8,2.2) -- (4.2,2.2) -- (4.2,1.8) -- (3.8,1.8) -- cycle;
					\fill[line width=2.pt,color=zzttqq,fill=zzttqq,fill opacity=0.10000000149011612] (4.8,2.2) -- (4.8,1.8) -- (5.2,1.8) -- (5.2,2.2) -- cycle;
					\fill[line width=2.pt,color=zzttqq,fill=zzttqq,fill opacity=0.10000000149011612] (5.8,2.2) -- (5.8,1.8) -- (6.2,1.8) -- (6.2,2.2) -- cycle;
					\fill[line width=2.pt,color=zzttqq,fill=zzttqq,fill opacity=0.10000000149011612] (6.8,2.2) -- (6.8,1.8) -- (7.2,1.8) -- (7.2,2.2) -- cycle;
					\fill[line width=2.pt,color=zzttqq,fill=zzttqq,fill opacity=0.10000000149011612] (-0.2,0.2) -- (0.2,0.2) -- (0.2,-0.2) -- (-0.2,-0.2) -- cycle;
					\draw [line width=1.pt,color=zzttqq] (-7.2,2.2)-- (-6.4,2.2);
					\draw [line width=1.pt,color=zzttqq] (-6.4,2.2)-- (-6.4,1.8);
					\draw [line width=1.pt,color=zzttqq] (-6.4,1.8)-- (-7.2,1.8);
					\draw [line width=1.pt,color=zzttqq] (-7.2,1.8)-- (-7.2,2.2);
					\draw [line width=1.pt,color=zzttqq] (-1.4,3.2)-- (-0.6,3.2);
					\draw [line width=1.pt,color=zzttqq] (-0.6,3.2)-- (-0.6,2.8);
					\draw [line width=1.pt,color=zzttqq] (-0.6,2.8)-- (-1.4,2.8);
					\draw [line width=1.pt,color=zzttqq] (-1.4,2.8)-- (-1.4,3.2);
					\draw [line width=1.pt,color=zzttqq] (-0.4,3.2)-- (-0.4,2.8);
					\draw [line width=1.pt,color=zzttqq] (-0.4,2.8)-- (0.4,2.8);
					\draw [line width=1.pt,color=zzttqq] (0.4,2.8)-- (0.4,3.2);
					\draw [line width=1.pt,color=zzttqq] (0.4,3.2)-- (-0.4,3.2);
					\draw [line width=1.pt,color=zzttqq] (0.6,3.2)-- (0.6,2.8);
					\draw [line width=1.pt,color=zzttqq] (0.6,2.8)-- (1.4,2.8);
					\draw [line width=1.pt,color=zzttqq] (1.4,2.8)-- (1.4,3.2);
					\draw [line width=1.pt,color=zzttqq] (1.4,3.2)-- (0.6,3.2);
					\draw [line width=1.pt,color=zzttqq] (-5.6,2.2)-- (-5.6,1.8);
					\draw [line width=1.pt,color=zzttqq] (-5.6,1.8)-- (-4.8,1.8);
					\draw [line width=1.pt,color=zzttqq] (-4.8,1.8)-- (-4.8,2.2);
					\draw [line width=1.pt,color=zzttqq] (-4.8,2.2)-- (-5.6,2.2);
					\draw [line width=1.pt,color=zzttqq] (-4.2,2.2)-- (-4.2,1.8);
					\draw [line width=1.pt,color=zzttqq] (-4.2,1.8)-- (-3.4,1.8);
					\draw [line width=1.pt,color=zzttqq] (-3.4,1.8)-- (-3.4,2.2);
					\draw [line width=1.pt,color=zzttqq] (-3.4,2.2)-- (-4.2,2.2);
					\draw [line width=1.pt,color=zzttqq] (-2.6,2.2)-- (-2.6,1.8);
					\draw [line width=1.pt,color=zzttqq] (-2.6,1.8)-- (-1.8,1.8);
					\draw [line width=1.pt,color=zzttqq] (-1.8,1.8)-- (-1.8,2.2);
					\draw [line width=1.pt,color=zzttqq] (-1.8,2.2)-- (-2.6,2.2);
					\draw [line width=1.pt,color=zzttqq] (-1.2,2.2)-- (-1.2,1.8);
					\draw [line width=1.pt,color=zzttqq] (-1.2,1.8)-- (-0.8,1.8);
					\draw [line width=1.pt,color=zzttqq] (-0.8,1.8)-- (-0.8,2.2);
					\draw [line width=1.pt,color=zzttqq] (-0.8,2.2)-- (-1.2,2.2);
					\draw [line width=1.pt,color=zzttqq] (-0.2,2.2)-- (-0.2,1.8);
					\draw [line width=1.pt,color=zzttqq] (-0.2,1.8)-- (0.2,1.8);
					\draw [line width=1.pt,color=zzttqq] (0.2,1.8)-- (0.2,2.2);
					\draw [line width=1.pt,color=zzttqq] (0.2,2.2)-- (-0.2,2.2);
					\draw [line width=1.pt,color=zzttqq] (0.8,2.2)-- (0.8,1.8);
					\draw [line width=1.pt,color=zzttqq] (0.8,1.8)-- (1.2,1.8);
					\draw [line width=1.pt,color=zzttqq] (1.2,1.8)-- (1.2,2.2);
					\draw [line width=1.pt,color=zzttqq] (1.2,2.2)-- (0.8,2.2);
					\draw [line width=1.pt,color=zzttqq] (1.8,2.2)-- (1.8,1.8);
					\draw [line width=1.pt,color=zzttqq] (1.8,1.8)-- (2.2,1.8);
					\draw [line width=1.pt,color=zzttqq] (2.2,1.8)-- (2.2,2.2);
					\draw [line width=1.pt,color=zzttqq] (2.2,2.2)-- (1.8,2.2);
					\draw [line width=1.pt,color=zzttqq] (2.8,2.2)-- (2.8,1.8);
					\draw [line width=1.pt,color=zzttqq] (2.8,1.8)-- (3.2,1.8);
					\draw [line width=1.pt,color=zzttqq] (3.2,1.8)-- (3.2,2.2);
					\draw [line width=1.pt,color=zzttqq] (3.2,2.2)-- (2.8,2.2);
					\draw [line width=1.pt,color=zzttqq] (3.8,2.2)-- (4.2,2.2);
					\draw [line width=1.pt,color=zzttqq] (4.2,2.2)-- (4.2,1.8);
					\draw [line width=1.pt,color=zzttqq] (4.2,1.8)-- (3.8,1.8);
					\draw [line width=1.pt,color=zzttqq] (3.8,1.8)-- (3.8,2.2);
					\draw [line width=1.pt,color=zzttqq] (4.8,2.2)-- (4.8,1.8);
					\draw [line width=1.pt,color=zzttqq] (4.8,1.8)-- (5.2,1.8);
					\draw [line width=1.pt,color=zzttqq] (5.2,1.8)-- (5.2,2.2);
					\draw [line width=1.pt,color=zzttqq] (5.2,2.2)-- (4.8,2.2);
					\draw [line width=1.pt,color=zzttqq] (5.8,2.2)-- (5.8,1.8);
					\draw [line width=1.pt,color=zzttqq] (5.8,1.8)-- (6.2,1.8);
					\draw [line width=1.pt,color=zzttqq] (6.2,1.8)-- (6.2,2.2);
					\draw [line width=1.pt,color=zzttqq] (6.2,2.2)-- (5.8,2.2);
					\draw [line width=1.pt,color=zzttqq] (6.8,2.2)-- (6.8,1.8);
					\draw [line width=1.pt,color=zzttqq] (6.8,1.8)-- (7.2,1.8);
					\draw [line width=1.pt,color=zzttqq] (7.2,1.8)-- (7.2,2.2);
					\draw [line width=1.pt,color=zzttqq] (7.2,2.2)-- (6.8,2.2);
					\draw [line width=1.pt,color=zzttqq] (-0.2,0.2)-- (0.2,0.2);
					\draw [line width=1.pt,color=zzttqq] (0.2,0.2)-- (0.2,-0.2);
					\draw [line width=1.pt,color=zzttqq] (0.2,-0.2)-- (-0.2,-0.2);
					\draw [line width=1.pt,color=zzttqq] (-0.2,-0.2)-- (-0.2,0.2);
					\draw [line width=1.pt] (-1.,1.8)-- (0.,0.2);
					\draw [line width=1.pt] (0.,1.8)-- (0.,0.2);
					\draw [line width=1.pt] (1.,1.8)-- (0.,0.2);
					\draw [line width=1.pt] (2.,1.8)-- (0.,0.2);
					\draw [line width=1.pt] (-1.,2.8)-- (-1.,2.2);
					\draw [line width=1.pt] (0.,2.8)-- (0.,2.2);
					\draw [line width=1.pt] (1.,2.8)-- (1.,2.2);
					\draw [line width=1.pt] (-2.2,1.8)-- (0.,0.2);
					\draw [line width=1.pt] (3.,1.8)-- (0.,0.2);
					\draw [line width=1.pt] (0.,0.2)-- (4.,1.8);
					\draw [line width=1.pt] (-6.8,1.8)-- (0.,0.2);
					\draw [line width=1.pt] (-3.8,1.8)-- (0.,0.2);
					\draw [line width=1.pt] (-5.2,1.8)-- (0.,0.2);
					\draw [line width=1.pt] (0.,0.2)-- (5.,1.8);
					\draw [line width=1.pt] (6.,1.8)-- (0.,0.2);
					\draw [line width=1.pt] (0.,0.2)-- (7.,1.8);
					\begin{scriptsize}
						\draw [fill=qqzzqq] (-2.,2.) circle (2.5pt);
						\draw [fill=qqzzqq] (-2.4,2.) circle (2.5pt);
						\draw [fill=zzccqq] (-6.6,2.) circle (2.5pt);
						\draw [fill=zzccqq] (-7.,2.) circle (2.5pt);
						\draw [fill=zzffqq] (-5.,2.) circle (2.5pt);
						\draw [fill=zzffqq] (-5.4,2.) circle (2.5pt);
						\draw [fill=ccffcc] (-3.6,2.) circle (2.5pt);
						\draw [fill=ccffcc] (-4.,2.) circle (2.5pt);
						\draw [fill=ffffww] (0.8,3.) circle (2.5pt);
						\draw [fill=ffffww] (1.,2.) circle (2.5pt);
						\draw [fill=ffccww] (0.2,3.) circle (2.5pt);
						\draw [fill=ffffww] (1.2,3.) circle (2.5pt);
						\draw [fill=yqqqqq] (7.,2.) circle (2.5pt);
						\draw [fill=cczzqq] (-0.8,3.) circle (2.5pt);
						\draw [fill=cczzqq] (-1.2,3.) circle (2.5pt);
						\draw [fill=ccqqqq] (6.,2.) circle (2.5pt);
						\draw [fill=ffcqcb] (2.,2.) circle (2.5pt);
						\draw [fill=ffccww] (-0.2,3.) circle (2.5pt);
						\draw [fill=cczzqq] (-1.,2.) circle (2.5pt);
						\draw [fill=ffccww] (0.,2.) circle (2.5pt);
						\draw [fill=ffttww] (4.,2.) circle (2.5pt);
						\draw [fill=ffwwzz] (3.,2.) circle (2.5pt);
						\draw [fill=ffqqtt] (5.,2.) circle (2.5pt);
						\draw [fill=ffffff] (0.,0.) circle (2.5pt);
					\end{scriptsize}
				\end{tikzpicture}
				\caption{The representation of $(\Diamond,\leq)$ for $S_4$.}
				\label{S_4leqRelation}
			\end{figure}
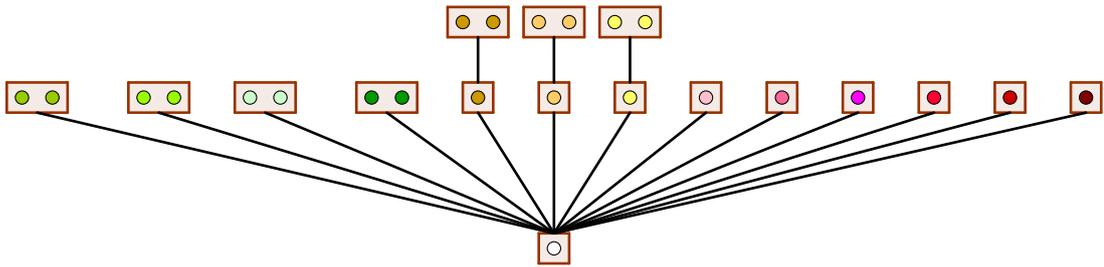
			
			In Figure \ref{C_6xC_2leqRelation} we maintain the colours of the vertices of Figure \ref{C_6xC_2Ncolored} to represent $(\Diamond,\leq)$ of $C_6\times C_2$. These colours represent the $\mathtt{N}$-classes, but, since $C_6 \times C_2$ is abelian, by Corollary \ref{abeldiamond_N}, in this case they represent also the $\diamond$-classes.
			The lowest vertex is the identity, the row above it contains the vertices of prime order and over this row the vertices not having prime order.
			Note that, in general, not all the classes of vertex of prime order are connected with classes of a row above.
			For example see $(\Diamond, \leq)$ of $S_4$ in Figure \ref{S_4leqRelation}. Also in this figure the colours represent the $\mathtt{N}$-classes and the boxes the $\diamond$-classes. Note that the colours are taken from Figure \ref{S_4NeDiamond}.
			
			In the row of the vertex of prime order, for both the examples above, the boxes with $2$ vertices contain elements of order $3$ and the ones with only one vertex contain elements of order $2$.
			In $C_6\times C_2$ the $\diamond$-classes in the third row are, for the relation $\leq$ that we are analysing, above the $\diamond$-class of elements of order $3$ and above one of the classes with an involution. So their elements have order divided by $6=2\cdot 3$, in this case the order is $6$ itself. Instead, for $S_4$, in the third row, all the $\diamond$-classes are above only the class of an involution. Then the three $\diamond$-classes in the third row contain elements of order a power of $2$, that is $4$ in this case.

			Now thanks to the reconstruction of the directed power graph from the power graph, so by Theorem \ref{UPG-DPG}, we immediately  obtain the following corollary of Proposition \ref{DPG-elementOrder}.
	
			\begin{corollary}{\rm \cite[Corollary 3]{Cameron_2}}
				Two groups whose power graph are isomorphic have the same number of elements of each order.
			\end{corollary}
		
		\chapter{Quotient power graphs}
		\label{QuotientGraphs}
		
			Our main references for quotient graphs are the articles \cite{Bubboloni_1, Bubboloni_2}.\\ 
			Let $\Gamma=(V,E)$ be a graph. Consider $E_{\ell}=E \cup \{\{x\} \, |\,  x\in V \}$ and define the \emph{loop graph} associated with $\Gamma$ by $\Gamma_{\ell}=(V,E_{\ell})$. We simply added, for each vertex of $\Gamma$, \emph{loops}, that are the edges $\{x\}$. With abuse of notation we write $\Gamma$ instead of $\Gamma_{\ell}$ and we continue to refer to $\mathcal{P}(G)$ as the power graph of $G$ even if we are considering $\mathcal{P}(G)_{\ell}$.
			
			Note that the addition of loops does not change the sets $[x]_{\mathtt{N}}$. Indeed $N[x]$ does not change if we consider the graph with or without loops.
			
			\begin{definition}
				\rm \label{defQuotient}Let $\Gamma=(V,E)$ be a graph and $\sim$ an equivalence relation on $V$. The quotient graph of $\Gamma$ with respect to $\sim$, denoted by $\Gamma/{\sim}$, is the graph with vertex set $[V]_{\sim}:=V/{\sim}$ and edge
				set $[E]_{\sim}$ defined as follows: for every $[x_1]_{\sim}\in [V]_{\sim}$ and $[x_2]_{\sim}\in [V]_{\sim}$, $\{[x_1]_{\sim}, [x_2]_{\sim}\}\in [E]_{\sim}$ if
				there exist $\tilde{x}_1\in [x_1]_{\sim}$ and $ \tilde{x}_2\in [x_2]_{\sim}$ such that $\{\tilde{x}_1, \tilde{x}_2\}\in E$.
			\end{definition}
		
			Let's see two examples of the definition above. In Figure \ref{S_4/diamond} we represent $\mathcal{P}(S_4)/\diamond$. Note that the colours are the same as in Figure \ref{UPG_S_4_NeDiamond}. Finally note that we do not represent the loops on each of the vertices of the power graph, even if, as we will see in a while, it is useful to consider their presence.
			In Figure \ref{S_4/N} we represent $S_4/\mathtt{N}$. Again, the colours are the same as in Figure \ref{UPG_S_4_NeDiamond} and no loops are drawn.

			The graph $\mathcal{P}(G)/\diamond$ it was called \emph{quotient power graph} by Bubboloni et al. in \cite{Bubboloni_1,Bubboloni_2} and, with the same name, also in \cite{GroupsWithQuotientGraphSpecified}. There is no standard name for $\mathcal{P}(G)/\mathtt{N}$, yet. We decide to follow the notation in \cite[Section 3]{ForbiddenGraphs} and in \cite[Section 7]{GraphsOnGroups} and call $\mathcal{P}(G)/\mathtt{N}$ the \emph{closed twin reduction} of $\mathcal{P}(G)$.\\ Let's define now some new players to clarify the choice of the name used for $\mathcal{P}(G)/\mathtt{N}$.
			
			\section{Introduction of twin reduction}
			
			\begin{definition}
				\rm Consider a graph $\Gamma$, we say that $x,y\in V_{\Gamma}$ are \emph{closed twins} if $x\mathtt{N}y$.\\
				Otherwise, we say that $x,y$ are \emph{open twins} if $N_{\Gamma}(x)=N_{\Gamma}(y)$ holds.\\ 
				We finally say that $x$ and $y$ are \emph{twins} if they are open or closed twins. 
			\end{definition}
		
			Note that, if $x$ and $y$ are distinct open twins, then $\{x,y\}\notin E$. Indeed assume $\{x,y\}\in E$. Then $x\in N_{\Gamma}(y)$ holds. Hence, since $x\notin N_{\Gamma}(x)$, we have $N_{\Gamma}(x)\ne N_{\Gamma}(y)$.
			Recall that, instead, if $x$ and $y$ are closed twins, then $\{x,y\}\in E$.
			Hence, given $x$ and $y$ two vertices of a certain graph $\Gamma$, $x$ and $y$ cannot be both closed and open twins.
			
			\begin{definition}
			\rm \label{RelationOeT}Let $\Gamma$ be a graph and let $x,y\in V_{\Gamma}$. Then we write $x\mathtt{O}y$ if $x$ and $y$ are open twins.\\
			We also write $x\mathtt{T}y$ if $x$ and $y$ are twins.
			\end{definition} 
			
			Note that, for $x,y\in V_{\Gamma}$, we have that $x\mathtt{N}y$ or $ x\mathtt{O}y  $ if and only if $x \mathtt{T}y$.
			
			\begin{Oss}
				Let $\Gamma$ be a graph. The relation $\mathtt{O}$ is an equivalence relation on the set $V_{\Gamma}$.\\
				The relation $\mathtt{T}$ is an equivalence relation on the set $V_{\Gamma}$.\\ 
			\end{Oss}
			\begin{proof}
				We omit the proof for the relation $\mathtt{O}$ since it is trivial.\\
				Let's focus on $\mathtt{T}$. Let $x,y,z\in V_{\Gamma}$.
				Surely we have $x\mathtt{T}x$ because $x \mathtt{O}x$ holds.
				Also the symmetry is trivial because both $\mathtt{O}$ and $\mathtt{N}$ are symmetric.
				Assume then that $x\mathtt{T}y$ and $y\mathtt{T}z$ hold.
				If at least two of $x, y $ and $z$ are equal, we immediately have $x\mathtt{T}z$. Assume next $x,y,z$ to be distinct.
				If we have $x\mathtt{N}y$ and $y\mathtt{N}z$, then surely $x\mathtt{N}z$ holds, and then also $x\mathtt{T}z$ holds.
				In the same way, if we have $x\mathtt{O}y$ and $y\mathtt{O}z$, then $x\mathtt{T}z$ holds.
				We prove that no other case is possible. In other words, it cannot happen neither $x\mathtt{N}y$ and $y\mathtt{O}z$ nor $z\mathtt{N}y$ and $y\mathtt{O}x$.
				Consider the case $x\mathtt{N}y$ and $y\mathtt{O}z$.
				Then we have that $N_{\Gamma}[x]=N_{\Gamma}[y]$ and $N_{\Gamma}(y)=N_{\Gamma}(z)$. By $x\in N_{\Gamma}[x]=N_{\Gamma}[y]$ and $x\ne y$ we get $x\in N_{\Gamma}(y)=N_{\Gamma}(z)$ so that $x\in N_{\Gamma}[z]$. It follows that we have $z\in N_{\Gamma}[x]=N_{\Gamma}[y]$. Therefore $z\in N_{\Gamma}(y)=N_{\Gamma}(z)$, a contradiction.
				Note that the remaining case is obtained by swapping the vertices $x$ and $z$.
				
			\end{proof}

			\begin{definition}
				\rm Let $\Gamma$ be a graph. $\Gamma/\mathtt{O}$ is called the \emph{open twin reduction} of $\Gamma$, we simply write {\scriptsize OTR} of $\Gamma$.\\
				$\Gamma/\mathtt{N}$ is called the \emph{closed twin reduction} of $\Gamma$, we simply write {\scriptsize CTR} of $\Gamma$.\\
				$\Gamma/\mathtt{T}$ is called the \emph{twin reduction} of $\Gamma$, we simply write {\scriptsize TR} of $\Gamma$.
			\end{definition}
		
			Note that the definition of the twin reduction of a graph is quite different from the one given in \cite[Section 3]{ForbiddenGraphs} and in \cite[Section 7]{GraphsOnGroups}. There, with twin reduction, the authors refer to a subsequent application of the twin reduction defined above until, for the graph obtained, the relation $\mathtt{T}$ is equals to the relation of identity.
			
			We now focus on the twins relations in the case of power graphs.  
			We found an easy result that characterizes the open twins of a power graph.
			
			\begin{prop}
				\label{open_twins=involution}	Let $G$ be a group. Then $x,y \in G $ are open twins for $\mathcal{P}(G)$ if and only if $x$ and $y$ are involutions and $N(x)=N(y)=\{1\}$. 
			\end{prop}
		
			\begin{proof}
				Assume that $N(x)=\{1\}=N(y)$. Therefore, by definition, $x$ and $y$ are open twins.
				
				Assume next that $x,y$ are open twins.
				Then we have $N(x)=N(y)$. Suppose, by contradiction, that $[x]_{\mathtt{N}}\ne \{x\}$. Then there exists $z\in G\setminus\{x\}$ such that $x\mathtt{N} z$. Then $z \in N(x)$ holds.
				Hence we have $z \in N(y)$. It follows that $y \in N(z) \subset N[z]=N[x]$ holds. Then $y\in N[x]$ holds and this implies $\{x,y\}\in E$, a contradiction because open twins are never connected.
				Hence we have that $[x]_{\mathtt{N}}=\{x\}$. With analogous argument we get $[y]_{\mathtt{N}}=\{y\}$. Then $|[x]_{\mathtt{N}}|=1=|[y]_{\mathtt{N}}|$. Thus the $\mathtt{N}$-class of $x$ and of $y$ are equal to the $\diamond$-class, respectively, of $x$ and of $y$. Therefore we have $|[x]_{\diamond}|=1=|[y]_{\diamond}|$ and, by Corollary \ref{TrivialDiamondClasses}, it follows that $x$ and $y$ can only be the identity or an involution.
				Now $x$ and $y$ are necessarily both involutions, otherwise, since they are different, exactly one of then must be the identity which would imply $\{x,y\}\in E$, a contradiction. It remains to prove that $N(x)=\{1\}=N(y)$. Assume, by contradiction, that there exists $z\in N(x) \setminus \{1\}=N(y)\setminus \{1\}$. Now, since $x$ and $y$ are involutions, we get that $\langle x\rangle \leq \langle z \rangle$ and that $\langle y \rangle \leq \langle z \rangle$. Thus, since $\langle x \rangle \ne \langle y \rangle$, there would be two distinct subgroup of order $2$ in $\langle z \rangle$, a contradiction.

			\end{proof}
		
			In Figure \ref{S_4OTR} we represent how $\mathcal{P}(S_4)/\mathtt{O}$ appears. Note that the blue vertex represents the $\mathtt{O}$-class of all the transpositions. 
			
			The reader can easily note that $\mathcal{P}(S_4)/\mathtt{T}$ is a single vertex. 
			Whenever a graph become a single vertex after subsequent twin reductions it is a \emph{cograph}. In \cite[Section 3]{ForbiddenGraphs}, Cameron et al. 
			have found a necessary and sufficient condition for a power graph of a nilpotent group to be a cograph. This will not be subject of this thesis.  
			
			\begin{figure}
				\centering
					\begin{tikzpicture}[line cap=round,line join=round,>=triangle 45,x=1.0cm,y=1.0cm]
					\clip(-6.,2.) rectangle (3.5,11.5);
					\draw [line width=1.pt] (-3.9844300659473078,9.572523933685508)-- (-1.,7.);
					\draw [line width=1.pt] (-4.871008319307276,7.7613712161073)-- (-1.,7.);
					\draw [line width=1.pt] (-4.491046210724432,5.392940739274258)-- (-1.,7.);
					\draw [line width=1.pt] (-3.452483113931327,3.923753919420607)-- (-1.,7.);
					\draw [line width=1.pt] (-1.7666619570249664,4.628983561782251)-- (-1.,7.);
					\draw [line width=1.pt] (0.07641383633386878,3.1379405391384996)-- (-1.,7.);
					\draw [line width=1.pt] (0.428937407628609,4.898618571476549)-- (-1.,7.);
					\draw [line width=1.pt] (1.661554594802546,6.400870768344779)-- (-1.,7.);
					\draw [line width=1.pt] (3.,7.)-- (-1.,7.);
					\draw [line width=1.pt] (2.920465005568102,7.742594813136765)-- (-1.,7.);
					\draw [line width=1.pt] (2.533519268257322,8.69061186954817)-- (-1.,7.);
					\draw [line width=1.pt] (1.759627793635761,9.754712647152811)-- (-1.,7.);
					\draw [line width=1.pt] (0.7535688766277328,10.567298695505446)-- (-1.,7.);
					\draw [line width=1.pt] (-0.02032259799382734,10.85750799848853)-- (-1.,7.);
					\draw [line width=1.pt] (-0.7168249251532315,10.915549859085147)-- (-1.,7.);
					\draw [line width=1.pt] (-2.249269770085656,10.649083241336891)-- (-1.,7.);
					
					\draw [line width=1.pt] (3.,7.)-- (1.661554594802546,6.400870768344779);
					\draw [line width=1.pt] (0.428937407628609,4.898618571476549)-- (0.07641383633386878,3.1379405391384996);
					\draw [line width=1.pt] (-1.7666619570249664,4.628983561782251)-- (-3.452483113931327,3.923753919420607);
					\begin{scriptsize}
						\draw [fill=qqzzqq] (-2.249269770085656,10.649083241336891) circle (3.pt);
						\draw [fill=zzccqq] (-3.9844300659473078,9.572523933685508) circle (3.pt);
						\draw [fill=zzffqq] (-4.871008319307276,7.7613712161073) circle (3.pt);
						\draw [fill=ccffcc] (-4.491046210724432,5.392940739274258) circle (3.pt);
						\draw [fill=ffffww] (1.661554594802546,6.400870768344779) circle (3.pt);
						\draw [fill=ffffww] (3.,7.) circle (3.pt);
						\draw [fill=yqqqqq] (-0.7168249251532315,10.915549859085147) circle (3.pt);
						\draw [fill=cczzqq] (-3.452483113931327,3.923753919420607) circle (3.pt);
						\draw [fill=ccqqqq] (-0.02032259799382734,10.85750799848853) circle (3.pt);
						\draw [fill=ffcqcb] (2.920465005568102,7.742594813136765) circle (3.pt);
						\draw [fill=ffccww] (0.07641383633386878,3.1379405391384996) circle (3.pt);
						\draw [fill=cczzqq] (-1.7666619570249664,4.628983561782251) circle (3.pt);
						\draw [fill=ffccww] (0.428937407628609,4.898618571476549) circle (3.pt);
						\draw [fill=ffttww] (1.759627793635761,9.754712647152811) circle (3.pt);
						\draw [fill=ffwwzz] (2.533519268257322,8.69061186954817) circle (3.pt);
						\draw [fill=ffffff] (-1.,7.) circle (3.pt);
						\draw [fill=ffqqtt] (0.7535688766277328,10.567298695505446) circle (3.pt);
					\end{scriptsize}
				\end{tikzpicture}
				\caption{$\mathcal{P}(S_4)/\diamond$}
				\label{S_4/diamond}
			\end{figure}
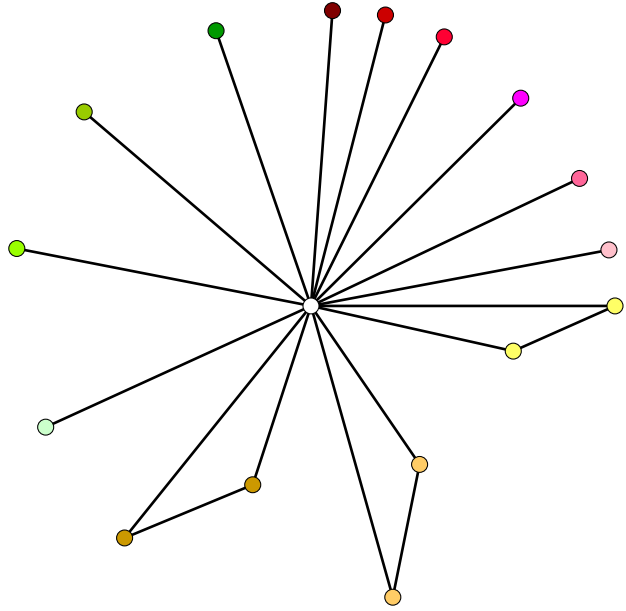
		
			\begin{figure}
				\centering
				\begin{tikzpicture}[line cap=round,line join=round,>=triangle 45,x=1.0cm,y=1.0cm]
					\clip(-7.885374380731168,1.6439812679235615) rectangle (7.313109962582573,12.029612235854547);
					\draw [line width=1.pt] (-3.9844300659473078,9.572523933685508)-- (-1.,7.);
					\draw [line width=1.pt] (-4.871008319307276,7.7613712161073)-- (-1.,7.);
					\draw [line width=1.pt] (-4.491046210724432,5.392940739274258)-- (-1.,7.);
					\draw [line width=1.pt] (-1.9706308904582372,3.607118828934906)-- (-1.,7.);
					\draw [line width=1.pt] (0.9044157311519452,3.923753919420607)-- (-1.,7.);
					\draw [line width=1.pt] (2.6269106233941693,6.013545516626232)-- (-1.,7.);
					\draw [line width=1.pt] (2.920465005568102,7.742594813136765)-- (-1.,7.);
					\draw [line width=1.pt] (2.533519268257322,8.69061186954817)-- (-1.,7.);
					\draw [line width=1.pt] (1.759627793635761,9.754712647152811)-- (-1.,7.);
					\draw [line width=1.pt] (0.7535688766277328,10.567298695505446)-- (-1.,7.);
					\draw [line width=1.pt] (-0.02032259799382734,10.85750799848853)-- (-1.,7.);
					\draw [line width=1.pt] (-0.7168249251532315,10.915549859085147)-- (-1.,7.);
					\draw [line width=1.pt] (-2.249269770085656,10.649083241336891)-- (-1.,7.);
					\begin{scriptsize}
						\draw [fill=qqzzqq] (-2.249269770085656,10.649083241336891) circle (3.pt);
						\draw [fill=zzccqq] (-3.9844300659473078,9.572523933685508) circle (3.pt);
						\draw [fill=zzffqq] (-4.871008319307276,7.7613712161073) circle (3.pt);
						\draw [fill=ccffcc] (-4.491046210724432,5.392940739274258) circle (3.pt);
						\draw [fill=ffffww] (2.6269106233941693,6.013545516626232) circle (3.pt);
						\draw [fill=yqqqqq] (-0.7168249251532315,10.915549859085147) circle (3.pt);
						\draw [fill=cczzqq] (-1.9706308904582372,3.607118828934906) circle (3.pt);
						\draw [fill=ccqqqq] (-0.02032259799382734,10.85750799848853) circle (3.pt);
						\draw [fill=ffcqcb] (2.920465005568102,7.742594813136765) circle (3.pt);
						\draw [fill=ffccww] (0.9044157311519452,3.923753919420607) circle (3.pt);
						\draw [fill=ffttww] (1.759627793635761,9.754712647152811) circle (3.pt);
						\draw [fill=ffwwzz] (2.533519268257322,8.69061186954817) circle (3.pt);
						\draw [fill=ffffff] (-1.,7.) circle (3.pt);
						\draw [fill=ffqqtt] (0.7535688766277328,10.567298695505446) circle (3.pt);
					\end{scriptsize}
				\end{tikzpicture}
				
				\caption{$\mathcal{P}(S_4)/\mathtt{N}$}
				\label{S_4/N}
			\end{figure}
			
			\begin{figure}
				\centering
				\begin{tikzpicture}[line cap=round,line join=round,>=triangle 45,x=1.0cm,y=1.0cm]
					\clip(-6.,2.) rectangle (4.,12.);
					
					\draw [line width=1.pt] (-3.6189179549840818,10.083616523866972)-- (-1.,7.);
					\draw [line width=1.pt] (-4.470198577067798,9.077557606858948)-- (-1.,7.);
					\draw [line width=1.pt] (-4.702366019454266,8.458444427161703)-- (-1.,7.);
					\draw [line width=1.pt] (-4.992575322437351,7.374996362691523)-- (-1.,7.);
					\draw [line width=1.pt] (-4.973228035571812,6.252853724490266)-- (-1.,7.);
					\draw [line width=1.pt] (-4.276725708412408,4.666376201516074)-- (-1.,7.);
					\draw [line width=1.pt] (-3.541528807521925,3.9118320137600566)-- (-1.,7.);
					\draw [line width=1.pt] (-1.7666619570249664,4.628983561782251)-- (-1.,7.);
					\draw [line width=1.pt] (-0.7168249251532304,3.041204104810805)-- (-1.,7.);
					\draw [line width=1.pt] (0.07641383633386878,3.1379405391384996)-- (-1.,7.);
					\draw [line width=1.pt] (0.428937407628609,4.898618571476549)-- (-1.,7.);
					\draw [line width=1.pt] (2.069184383484386,4.414861472264069)-- (-1.,7.);
					\draw [line width=1.pt] (2.533519268257322,5.033974651961314)-- (-1.,7.);
					\draw [line width=1.pt] (1.661554594802546,6.400870768344779)-- (-1.,7.);
					\draw [line width=1.pt] (3.,7.)-- (-1.,7.);
					\draw [line width=1.pt] (1.759627793635761,9.754712647152811)-- (-1.,7.);
					\draw [line width=1.pt] (-1.4520218260437137,10.915549859085147)-- (-1.,7.);
					\draw [line width=1.pt] (-3.1158884964800677,10.373825826850057)-- (-1.,7.);
					
					\draw [line width=1.pt] (-4.470198577067798,9.077557606858948)-- (-3.6189179549840818,10.083616523866972);
					\draw [line width=1.pt] (-3.1158884964800677,10.373825826850057)-- (-1.4520218260437137,10.915549859085147);
					\draw [line width=1.pt] (-4.702366019454266,8.458444427161703)-- (-4.992575322437351,7.374996362691523);
					\draw [line width=1.pt] (-4.973228035571812,6.252853724490266)-- (-4.276725708412408,4.666376201516074);
					\draw [line width=1.pt] (-3.541528807521925,3.9118320137600566)-- (-1.7666619570249664,4.628983561782251);
					\draw [line width=1.pt] (-3.541528807521925,3.9118320137600566)-- (-0.7168249251532304,3.041204104810805);
					\draw [line width=1.pt] (-1.7666619570249664,4.628983561782251)-- (-0.7168249251532304,3.041204104810805);
					\draw [line width=1.pt] (0.428937407628609,4.898618571476549)-- (0.07641383633386878,3.1379405391384996);
					\draw [line width=1.pt] (0.428937407628609,4.898618571476549)-- (2.069184383484386,4.414861472264069);
					\draw [line width=1.pt] (0.07641383633386878,3.1379405391384996)-- (2.069184383484386,4.414861472264069);
					\draw [line width=1.pt] (1.661554594802546,6.400870768344779)-- (2.533519268257322,5.033974651961314);
					\draw [line width=1.pt] (1.661554594802546,6.400870768344779)-- (3.,7.);
					\draw [line width=1.pt] (3.,7.)-- (2.533519268257322,5.033974651961314);
					\begin{scriptsize}
						\draw [fill=ffffww] (2.533519268257322,5.033974651961314) circle (2.5pt);
						\draw [fill=ffffww] (1.661554594802546,6.400870768344779) circle (2.5pt);
						\draw [fill=ffffww] (3.,7.) circle (2.5pt);
						
						\draw [fill=cczzqq] (-0.7168249251532304,3.041204104810805) circle (2.5pt);
						\draw [fill=cczzqq] (-3.541528807521925,3.9118320137600566) circle (2.5pt);
						\draw [fill=cczzqq] (-1.7666619570249664,4.628983561782251) circle (2.5pt);
						
						\draw [fill=ffccww] (0.07641383633386878,3.1379405391384996) circle (2.5pt);
						\draw [fill=ffccww] (0.428937407628609,4.898618571476549) circle (2.5pt);
						\draw [fill=ffccww] (2.069184383484386,4.414861472264069) circle (2.5pt);
						
						\draw [fill=qqzzqq] (-1.4520218260437137,10.915549859085147) circle (2.5pt);
						\draw [fill=qqzzqq] (-3.1158884964800677,10.373825826850057) circle (2.5pt);
						
						\draw [fill=zzccqq] (-4.470198577067798,9.077557606858948) circle (2.5pt);
						\draw [fill=zzccqq] (-3.6189179549840818,10.083616523866972) circle (2.5pt);
						
						\draw [fill=zzffqq] (-4.702366019454266,8.458444427161703) circle (2.5pt);
						\draw [fill=zzffqq] (-4.992575322437351,7.374996362691523) circle (2.5pt);
						
						\draw [fill=ccffcc] (-4.973228035571812,6.252853724490266) circle (2.5pt);
						\draw [fill=ccffcc] (-4.276725708412408,4.666376201516074) circle (2.5pt);
						
						\draw [fill=qqqqff] (1.759627793635761,9.754712647152811) circle (2.5pt);
						
						
						
						
						
						
						\draw [fill=ffffff] (-1.,7.) circle (2.5pt);
					\end{scriptsize}
				\end{tikzpicture}
				\caption{$\mathcal{P}(S_4)/\mathtt{O}$}
				\label{S_4OTR}
			\end{figure}
			
			\section{Quotients and homomorphisms}
			Let $\psi$ be a graph homomorphism between two graphs $\Gamma_1$ and $\Gamma_2$. We have that $\psi(\Gamma_1)$ is the subgraph of $\Gamma_2$ defined by $\psi(\Gamma_1)=(\psi(V_{\Gamma_1}), \psi(E_{\Gamma_1}))$.
			
			\begin{definition}
				\label{Properties-graph-homomorphism}\rm Let $\Gamma_1$ and $\Gamma_2$ be two graphs, let $\psi$ be a graph homomorphism between $\Gamma_1$ and $\Gamma_2$:
				\begin{itemize}
					\item[a)] $\psi$ is \emph{complete} if $\psi(\Gamma_1)=\Gamma_2$.
					
					\item[b)] $\psi$ is \emph{tame} if, for all $x,y \in V_{\Gamma_1}$, $\psi(x)=\psi(y)$ implies that $x$ and $y$ are joined by a path. 
					
					\item[c)] $\psi$ is \emph{locally surjective (injective, bijective)} if, for every $x \in V_{\Gamma_1}$, the restriction of $\psi$ to $N_{\Gamma}[x]$ is surjective (injective, bijective).
					
					\item[d)] $\psi$ is \emph{locally strong} if, for every $x_1,x_2 \in V_{\Gamma_1}$, $\{\psi(x_1), \psi_(x_2)\}\in E_{\Gamma_2}$ implies that, for every $\tilde{x}_1\in \psi^{-1}(\psi(x_1))$, there exists $\tilde{x}_2\in \psi^{-1}(\psi(x_2))$ such that $\{\tilde{x}_1, \tilde{x}_2\}\in E_{\Gamma_1}$.
					
					\item[e)] $\psi$ is \emph{pseudo-covering} if $\psi$ is locally strong and surjective.
					
				\end{itemize}
			\end{definition}
		
		Let $\Gamma=(V,E)$ a graph and $\sim$ an equivalence relation on $V$.
		The \emph{projection} on the quotient graph is the map $\pi_{\sim}$ defined as follows:
		\begin{align*}
			\pi_{\sim}:V\longrightarrow [V]_{\sim}& \\
			x\longmapsto [x]_{\sim}.
		\end{align*}
		
		Thanks to the addition of loops the projection $\pi_{\sim}$ is a graph homomorphism.
		
		\begin{Oss}
		 \label{Projection-homomorphism} Let $\Gamma=(V,E)$ be a graph and $\sim$ an equivalence relation on $V$. Then the projection	$\pi_{\sim}$ is a surjective graph homomorphism between $\Gamma$ and $\Gamma/\sim$.
		\end{Oss}
		\begin{proof}
			$\pi_{\sim}$ is clearly surjective.
			
			Let's see that $\pi_{\sim}$ is a graph homomorphism.
			Let $x$ and $y$ be elements of $V$. We have to prove that if $\{x,y\}\in E$, then $\{\pi_{\sim}(x), \pi_{\sim}(y)\}\in [E]_{\sim}$.\\
			Assume $\{x,y\}\in E$.
			We distinguish two cases:
			\begin{enumerate}
				\item[i)] \underline{$y\in [x]_{\sim}$}. Then $\pi_{\sim}(x)=\pi_{\sim}(y)=[x]_{\sim}$ and hence $\{\pi_{\sim}(x), \pi_{\sim}(y)\}\in [E]_{\sim}$; indeed it is a loop.
				
				\item[ii)] \underline{$y\notin [x]_{\sim}$}. Then, by definition \ref{defQuotient}, we have that $\{\pi_{\sim}(x), \pi_{\sim}(y)\}\in [E]_{\sim}$.
			\end{enumerate}
		\end{proof}
	
		So, for $\sim$ an equivalence relation on a vertex set, now we can see which properties, defined in \ref{Properties-graph-homomorphism}, $\pi_{\sim}$ satisfies.
		We say that a quotient graph is \emph{tame} (\emph{pseudo-covering}) if its projection is tame (pseudo-covering).
	
		We focus on the equivalence relation $\mathtt{N}$. From now on we use $\pi$ instead of $\pi_{\mathtt{N}}$.
		With the following remark and the next lemma we study some of the properties of $\pi$.
		
		\begin{Oss}
			\label{pi_psudo-covering}Let $x$ and $y$ be elements of $G$. Then we have that $\{x,y\}\in E$ if and only if $\{\pi(x), \pi(y)\}\in [E]_{\mathtt{N}}$. In particular $\pi$ is locally strong and hence, since $\pi$ is surjective, is also pseudo-covering.
		\end{Oss}
		\begin{proof}
			If $\{x,y\}\in E$, then $\{\pi(x), \pi(y)\}\in [E]_{\mathtt{N}}$ because, by Remark \ref{Projection-homomorphism}, $\pi$ is a graph homomorphism.
			Instead, if $\{\pi(x), \pi(y)\}\in [E]_{\mathtt{N}}$ we distinguish two cases:
			\begin{enumerate}
				\item[i)] \underline{$\pi(x)=\pi(y)$}. Then $x\mathtt{N}y$, and hence $\{x,y\} \in E$.
				
				\item[ii)] \underline{$\pi(x)\not=\pi(y)$}. Then there exists $\tilde{x}\in [x]_{\mathtt{N}}$ and $\tilde{y}\in [y]_{\mathtt{N}}$ such that $\{\tilde{x}, \tilde{y}\}\in E$. Therefore, remembering that $N[x]=N[\tilde{x}]$ and that $N[y]=N[\tilde{y}]$, we have $\{x,y\}\in E$.
			\end{enumerate} 
			Let's see that the property just proved implies that $\pi$ is locally strong; from this will follows immediately that $\pi$ is pseudo-covering since $\pi$ is surjective.
			Let $x_1, x_2\in G$ such that $\{\pi(x_1), \pi(x_2)\} \in [E]_{\mathtt{N}}$. Then, for every $\tilde{x}_1\in \pi^{-1}(\pi(x_1))=[x_1]_{\mathtt{N}}$, we have that $\{\tilde{x}_1, x_2\}\in E$. Indeed $x_2\in N[x_1]=N[\tilde{x}_1]$. Therefore $\pi$ is locally strong.
		\end{proof}
		
		\begin{lemma}
			\label{piNproperties}$\pi$ is a complete, tame and pseudo-covering graph homomorphism.
		\end{lemma}
		\begin{proof}
			Let's prove that $\pi$ is complete. We already said that $\pi$ is surjective, so we need to check that $\pi(E_{\mathcal{P}(G)})=[E]_{\mathtt{N}}$.
			Surely we have $\pi(E_{\mathcal{P}(G)})\subseteq[E]_{\mathtt{N}}$. So pick $\{[x_1]_{\mathtt{N}}, [x_2]_{\mathtt{N}}\}\in [E]_{\mathtt{N}}$. By definition of $[E]_{\mathtt{N}}$, there exist $\tilde{x}_1\in [x_1]_{\mathtt{N}}$ and $\tilde{x}_2\in [x_2]_{\mathtt{N}}$ such that $\{\tilde{x}_1,\tilde{x}_2\}\in E$. Hence $\pi(\{\tilde{x}_1,\tilde{x}_2\})=\{[x_1]_{\mathtt{N}}, [x_2]_{\mathtt{N}}\}$ and thus $\pi$ is complete.
			
			$\pi$ is tame. Indeed for all $x,y \in G$, such that $\pi(x)=\pi(y)$, we have that $\{x,y\}\in E$, and hence $x$ and $y$ are in the same connected component.
			
			By Remark \ref{pi_psudo-covering}, we have that $\pi$ is pseudo-covering.
		\end{proof}
	
		The tame quotient graphs are of particular interest because Proposition \ref{tame_quotient} holds.
		To state the result cited above we need to set some notation.
		For $\Gamma$ a graph and $G$ a group, we use $c(\Gamma)$, $c^*(G)$, $c_{\diamond}^*(G)$ and $c_{\mathtt{N}}^*(G)$ to refer to the number of connected components of, respectively, $\Gamma$, $\mathcal{P}^*(G)$, $\mathcal{P}^*(G)/\diamond$ and $\mathcal{P}^*(G)/\mathtt{N}$.
		
		\begin{prop}{\rm \cite[Proposition 3.2]{Bubboloni_1}}
			\label{tame_quotient}Let $\Gamma=(V,E)$ be a graph and $\sim$ be an equivalence relation on $V$. Then
			\begin{enumerate}
				\item[i)] $c(\Gamma)\leq c(\Gamma/{\sim})$
				
				\item[ii)] $c(\Gamma)=c(\Gamma/{\sim})$ if and only if $\sim$ is tame.
				
				\item[iii)] $\Gamma$ is connected if and only if $\Gamma/{\sim}$ is connected and tame. 
			\end{enumerate}
		\end{prop}
	
		In \cite{Bubboloni_2}, Bubboloni et al. studied the relation $\diamond$ proving the following result. We state this result within our notation.
		
		\begin{lemma}{\rm \cite[Lemma 3.6]{Bubboloni_2}}
			\label{LemmaBubboloni}Let $G$ be a group. The graph $\mathcal{P}^*(G)/\diamond$ is a tame and pseudo-covering quotient of $\mathcal{P}^*(G)$. In particular $c^*(G)=c_{\diamond}^*(G)$.
		\end{lemma}
		
		We now consider the relations $\mathtt{N}$, $\mathtt{O}$ and $\mathtt{T}$.
		\begin{lemma}
			Let $G$ be a group. The graph $\mathcal{P}^*(G)/\mathtt{N}$ is a tame and pseudo-covering quotient of $\mathcal{P}^*(G)$. In particular $c^*(G)=c_{\mathtt{N}}^*(G)$.
		\end{lemma}
		\begin{proof}
			Let $x, y\in G\setminus\{1\}$ such that $x\mathtt{N} y$. Then $\{x,y\}\in E$. This shows that the quotient graph $\mathcal{P}^*(G)/\mathtt{N}$ is tame. Then, by Proposition \ref{tame_quotient}, we have $c^*(G)=c_{\mathtt{N}}^*(G)$. The fact that the
			quotient is pseudo-covered is an immediate consequence of Remark \ref{pi_psudo-covering}. Indeed the map 
			\begin{align*}
				\pi^*:G\setminus \{1\}\longrightarrow [G\setminus &\{1\}]_{\mathtt{N}} \\
				x\longmapsto \pi(x)
			\end{align*} 
		is clearly surjective and, since $\pi$ is locally strong, $\pi^*$ is locally strong.
		\end{proof}
	
	
		\begin{lemma}
			\label{piO_tame}Let $G$ be a group. The graph $\mathcal{P}^*(G)/\mathtt{O}$ is a tame quotient of $\mathcal{P}^*(G)$ if and only if, for $x,y \in G\setminus\{1\}$, $x\mathtt{O}y$ implies $x=y$.
		\end{lemma}
		\begin{proof}
			Assume $\mathcal{P}^*(G)/\mathtt{O}$ to be a tame quotient of $\mathcal{P}^*(G)$.
			Then the graph homomorphism $$\pi^*_{\mathtt{O}}: G\setminus \{1\} \longrightarrow [G\setminus \{1\}]_{\mathtt{O}}$$ is tame.
			Hence, for all $x,y\in G\setminus \{1\}$, $\pi^*_{\mathtt{O}}(x)=\pi^*_{\mathtt{O}}(y)$ implies $x$ and $y$ joined by a path of $\mathcal{P}^*(G)$.
			But two distinct open twins $x,y$ of $\mathcal{P}(G)$, by Proposition \ref{open_twins=involution}, have $N(x)=N(y)=\{1\}$. So we have that $N_{\mathcal{P}^*(G)}(x)=N_{\mathcal{P}^*(G)}(y)=\emptyset$, and this implies that there is no path joining $x$ and $y$ in $\mathcal{P}^*(G)$. A contradiction.
			
			Assume next that, for $x,y \in G\setminus\{1\}$, $x\mathtt{O}y$ implies $x=y$.
			Then clearly, for all $x,y\in G\setminus \{1\}$, $\pi^*_{\mathtt{O}}(x)=\pi^*_{\mathtt{O}}(y)$ implies $x$ and $y$ joined by a path since $x=y$ and the path is the trivial one of length $0$.
		\end{proof}
	
		Then we have that the graph $\mathcal{P}^*(G)/\mathtt{O}$ is a tame quotient of $\mathcal{P}^*(G)$ if and only if $\mathcal{P}^*(G)/\mathtt{O}=\mathcal{P}^*(G)$.
	
		\begin{corollary}
				Let $G$ be a group. The graph $\mathcal{P}^*(G)/\mathtt{T}$ is a tame quotient of $\mathcal{P}^*(G)$ if and only if, for all $x,y \in G\setminus\{1\}$, $x\mathtt{T}y$ implies $x\mathtt{N}y$.
		\end{corollary}
		\begin{proof}
			Assume $\mathcal{P}^*(G)/\mathtt{T}$ to be a tame quotient of $\mathcal{P}^*(G)$.
			Then the graph homomorphism $$\pi^*_{\mathtt{T}}: G\setminus \{1\} \longrightarrow [G\setminus \{1\}]_{\mathtt{T}}$$ is tame.
			Hence, for all $x,y\in G\setminus \{1\}$, $\pi^*_{\mathtt{T}}(x)=\pi^*_{\mathtt{T}}(y)$ implies $x$ and $y$ joined by a path of $\mathcal{P}^*(G)$.
			Remember that $x\mathtt{T}y$ holds if and only we have $x \mathtt{N}y$ or $\mathtt{O}y$.
			For all $x,y\in G\setminus \{1\}$ such that $x\mathtt{N}y$, surely $x$ and $y$ are joined by a path, they are indeed joined.
			Then consider $x,y \in G\setminus \{1\}$ such that $x\mathtt{O}y$. By Lemma \ref{piO_tame}, $x$ and $y$ are joined by a path of $\mathcal{P}^*(G)$ if and only if $x=y$. It follows that $x\mathtt{N}y$ holds. 
			Thus $x\mathtt{T}y$ implies $x\mathtt{N}y$ for all $x,y\in G\setminus \{1\}$.
			
			Assume next that, for all $x,y \in G\setminus\{1\}$, $x\mathtt{T}y$ implies $x\mathtt{N}y$.
			Then the relation $\mathtt{T}$ is the same as $\mathtt{N}$. Hence $\pi^*_{\mathtt{T}}=\pi^*$ that, by Lemma \ref{piNproperties}, is tame.
		\end{proof}
		
		It follows that the graph $\mathcal{P}^*(G)/\mathtt{T}$ is a tame quotient of $\mathcal{P}^*(G)$ if and only $\mathcal{P}^*(G)/\mathtt{T}=\mathcal{P}^*(G)/\mathtt{N}$.
		
	
		\chapter{Maximal paths and cycles in power graphs}
		\label{MaximalCycles}
		
		In this chapter we are going to study the paths and the cycles in the power graphs of groups. We refer to the definition of path and cycles given in Definition \ref{DefPath} and Definition \ref{DefCycle}.
				
		Let $\mathcal{W}$ and $\mathcal{C}$ be, respectively, a path and a cycle. Let $\Gamma$ be a graph. $\mathcal{W}$ is called a path of $\Gamma$ if it is a subgraph of $\Gamma$. $\mathcal{C}$ is a cycle of $\Gamma$ if it is a subgraph of $\Gamma$. 
		We say that $\mathcal{C}$ is a \emph{maximal cycle} of $\Gamma$ if there is no other cycle $\mathcal{C}'$ in $\Gamma$ such that $V_{\mathcal{C}}\subset V_{\mathcal{C}'}$. Similarly $\mathcal{W}$ is a \emph{maximal path} of $\Gamma$ if there is no other path $\mathcal{W}'$ in $\Gamma$ such that $V_{\mathcal{W}}\subset V_{\mathcal{W}'}$.
		
		We want to study paths and cycles in the power graph. In Figure \ref{EsempioCycle/pathInUPG_C_6} we represent $\mathcal{P}(C_6)$ and we highlight, in red, a cycle and, in green, a path. Note that neither the path nor the cycle are maximal. Indeed consider $\mathcal{C}=1,a^3,a^5,a^4,a,a^2,1$. $V_{\mathcal{C}}=C_6$ so $\mathcal{C}$ is a maximal cycle, and $\mathcal{W}=1,a^3,a^5,a^4,a,a^2$ is a maximal path. Another example of maximal path is $\mathcal{W}'=a^3,a,1,a^5,a^4,a^2$, in this case $a^3,a,1,a^5,a^4,a^2,a^3$ is not a cycle since $\{a^3,a^2\}\notin E$.
		
		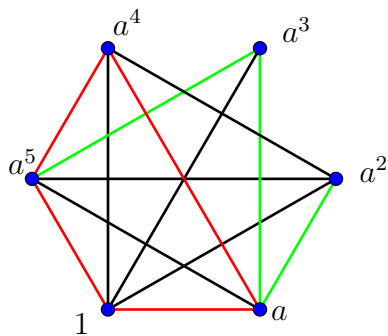
\begin{figure}
			\centering
			\begin{tikzpicture}[line cap=round,line join=round,>=triangle 45,x=1.0cm,y=1.0cm]
				\clip(-1.,-1.5) rectangle (5.,3.);
				\draw [line width=1.pt, color=red] (1.,-1.)-- (0.,0.732050807568879);
				\draw [line width=1.pt] (1.,-1.)-- (1.,2.4641016151377557);
				\draw [line width=1.pt] (1.,-1.)-- (3.,2.4641016151377553);
				\draw [line width=1.pt] (1.,-1.)-- (4.,0.7320508075688774);
				\draw [line width=1.pt, color=red] (1.,-1.)-- (3.,-1.);
				\draw [line width=1.pt, color=red] (0.,0.732050807568879)-- (1.,2.4641016151377557);
				\draw [line width=1.pt] (0.,0.732050807568879)-- (3.,-1.);
				\draw [line width=1.pt] (0.,0.732050807568879)-- (4.,0.7320508075688774);
				\draw [line width=1.pt, color=green] (0.,0.732050807568879)-- (3.,2.4641016151377553);
				\draw [line width=1.pt, color=red] (3.,-1.)-- (1.,2.4641016151377557);
				\draw [line width=1.pt, color=green] (3.,-1.)-- (3.,2.4641016151377553);
				\draw [line width=1.pt, color=green] (3.,-1.)-- (4.,0.7320508075688774);
				\draw [line width=1.pt] (1.,2.4641016151377557)-- (4.,0.7320508075688774);
				\draw (0.9,-0.9086580230955188) node[anchor=north east] {$1$};
				\draw (3.0041230403309376,-0.8146375163736965) node[anchor=north west] {$a$};
				\draw (4.16370928990008,1.1754632092382078) node[anchor=north west] {$a^2$};
				\draw (3.145153800413671,3.0245331747673787) node[anchor=north west] {$a^3$};
				\draw (0.9200018079972094,3.118553681489201) node[anchor=north west] {$a^4$};
				\draw (-0.4589656239228511,1.26948371596003) node[anchor=north west] {$a^5$};
				\begin{scriptsize}
					\draw [fill=qqqqff] (1.,-1.) circle (2.5pt);
					\draw [fill=qqqqff] (3.,-1.) circle (2.5pt);
					\draw [fill=qqqqff] (4.,0.7320508075688774) circle (2.5pt);
					\draw [fill=qqqqff] (3.,2.4641016151377553) circle (2.5pt);
					\draw [fill=qqqqff] (1.,2.4641016151377557) circle (2.5pt);
					\draw [fill=qqqqff] (0.,0.732050807568879) circle (2.5pt);
				\end{scriptsize}
			\end{tikzpicture}
			\caption{In red a $4$-cycle and in green a $3$-path of $\mathcal{P}(C_6)$.}
			\label{EsempioCycle/pathInUPG_C_6}
		\end{figure}
		
		Let $\vec{\mathcal{W}}$ and $\vec{\mathcal{C}}$ be, respectively, a directed path and a directed cycle. Let $\vec{\Gamma}$ be a digraph. $\vec{\mathcal{W}}$ is a directed path of $\vec{\Gamma}$ if it is a subdigraph of $\vec{\Gamma}$. $\vec{\mathcal{C}}$ is a directed cycle of $\vec{\Gamma}$ if it is a subdigraph of $\vec{\Gamma}$. 
		We say that $\vec{\mathcal{C}}$ is a \emph{maximal directed cycle} of $\vec{\Gamma}$ if there is no other directed cycles $\vec{\mathcal{C}}'$ in $\vec{\Gamma}$ such that $V_{\vec{\mathcal{C}}}\subset V_{\vec{\mathcal{C}}'}$. Again, $\vec{\mathcal{W}}$ is a \emph{maximal directed path} if there is no other directed paths $\vec{\mathcal{W}}'$ in $\vec{\Gamma}$ such that $V_{\vec{\mathcal{W}}}\subset V_{\vec{\mathcal{W}}'}$.
		
		In Figure \ref{EsempioDirectedcycleInDPGC_5} we represent a directed cycle and a directed path in $\vec{\mathcal{P}}(C_5)$. The example allows us to highlight some simple considerations on directed paths and cycles. Note that the identity is never a vertex of a directed cycle, indeed $1$ does not dominates any other vertex. Instead, dealing with directed paths, we have that the identity can only be the end of a directed paths. Look at $\vec{\mathcal{C}}=a,a^2,a^3,a^4,a$; it has the maximum length possible for a directed cycle in $\vec{\mathcal{P}}(C_5)$, so $\vec{\mathcal{C}}$ is maximal. Note that $V_{\mathcal{C}}=[a]_{\diamond}$ and that $l(\mathcal{C})=\phi(o(a))$. 
		 
		\begin{figure}
			\centering
			\begin{tikzpicture}[line cap=round,line join=round,>=triangle 45,x=1.0cm,y=1.0cm]
				\clip(-0.5,2.5) rectangle (4.5,7);
				\draw [latex-, line width=2.pt] (0.3819660112501053,4.902113032590307)-- (1.,3.);
				
				\draw [-latex, line width=2.pt, color=red] (1.,3.)-- (2.3,3.);
				\draw [ line width=2.pt, color=red] (1.,3.)-- (3.,3.);
				
				
				\draw [line width=2.pt, color=red] (3.,3.)-- (3.618033988749895,4.902113032590306);
				\draw [-latex, line width=2.pt, color=red] (3.,3.)-- (3.4,4.224);
				
				
				\draw [line width=2.pt, color=red] (3.618033988749895,4.902113032590306)-- (2.,6.077683537175253);
				\draw [-latex, line width=2.pt, color=red] (3.618033988749895,4.902113032590306)-- (2.5,5.715);
				
				\draw [-latex, line width=2.pt] (2.,6.077683537175253)-- (0.3819660112501053,4.902113032590307);
				
				\draw [line width=2.pt, color=red] (1.,3.)-- (2.,6.077683537175253);
				\draw [-latex, line width=2.pt, color=red] (2.,6.077683537175253)-- (1.4,4.232);
				
				\draw[-latex, line width=2.pt, color=green] (2.,6.077683537175253) arc (90:18:1.7);
				
				\draw [line width=2.pt] (0.3819660112501053,4.902113032590307)-- (3.,3.);
				\draw [latex-, line width=2.pt] (0.6,4.752)-- (3.,3.);
				
				\draw [line width=2.pt] (0.3819660112501053,4.902113032590307)-- (3.618033988749895,4.902113032590306);
				\draw [latex-, line width=2.pt] (0.6,4.902113032590307)-- (3.618033988749895,4.902113032590306);
				
				\draw [latex-latex, line width=2.pt] (1.,3.)-- (3.618033988749895,4.902113032590306);
				\draw [latex-latex, line width=2.pt] (3.,3.)-- (2.,6.077683537175253);
				
				\draw[-latex, line width=2.pt, color=green] (3.,3.) arc (306:234:1.7);
				
				\draw [latex-, shift={(2.,4.3763819204711725)},line width=2.pt]  plot[domain=-0.9424777960769379:0.31415926535897953,variable=\t]({1.*1.7013016167040786*cos(\t r)+0.*1.7013016167040786*sin(\t r)},{0.*1.7013016167040786*cos(\t r)+1.*1.7013016167040786*sin(\t r)});

				\draw [latex-, shift={(1.4633131033616402,4.55076206389909)},line width=2.pt, color=green]  plot[domain=1.232800986062001:4.4220657903996266,variable=\t]({1.*1.618493747462486*cos(\t r)+0.*1.618493747462486*sin(\t r)},{0.*1.618493747462486*cos(\t r)+1.*1.618493747462486*sin(\t r)});
				
				\draw (1.,3.) node[anchor=north east] {$a$};
				\draw (3.,3.) node[anchor=north west] {$a^2$};
				\draw (3.6,4.85) node[anchor= west] {$a^3$};
				\draw (2.,6.077683537175253) node[anchor=south] {$a^4$};
				\draw (0.3819660112501053,4.85)
				node[anchor= east] {$1$};
				
				\begin{scriptsize}
					\draw [fill=qqqqff] (1.,3.) circle (3pt);
					\draw [fill=qqqqff] (3.,3.) circle (3pt);
					\draw [fill=qqqqff] (3.618033988749895,4.902113032590306) circle (2.5pt);
					\draw [fill=qqqqff] (2.,6.077683537175253) circle (3pt);
					\draw [fill=qqqqff] (0.3819660112501053,4.902113032590307) circle (3pt);
					\draw [fill=qqqqff] (3.,2.) circle (3pt);
					\draw [fill=qqqqff] (1.,2.) circle (3pt);
					\draw [fill=qqqqff] (0.3819660112501051,0.09788696740969349) circle (3pt);
					\draw [fill=qqqqff] (2.,-1.0776835371752531) circle (3pt);
					\draw [fill=qqqqff] (3.618033988749895,0.09788696740969272) circle (3pt);
				\end{scriptsize}
			\end{tikzpicture}
			\caption{In red a directed $4$-cycle and in green a directed $3$-path of $\mathcal{P}(C_5)$.}
			\label{EsempioDirectedcycleInDPGC_5}
		\end{figure}
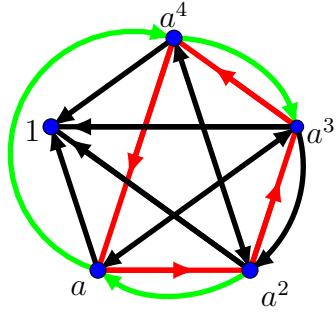
	
		\begin{definition}
			\rm Let $\Gamma$ be a graph. A cycle $\mathcal{C}$ of $\Gamma$ having $V_{\mathcal{C}}=V_{\Gamma}$ is called \emph{hamiltonian cycle}. A path $\mathcal{W}$ of $\Gamma$ with $V_{\mathcal{W}}=V_{\Gamma}$ is an \emph{hamiltonian path}. $\Gamma$ is an \emph{hamiltonian graph} if it has an hamiltonian cycle.
		\end{definition}
	
		Note that hamiltonian cycles/paths are maximal cycles/paths.
		
		\begin{definition}
			\rm Let $\vec{\Gamma}$ be a digraph. A directed cycle $\vec{\mathcal{C}}$ of $\vec{\Gamma}$ having $V_{\vec{\mathcal{C}}}=V_{\vec{\Gamma}}$ is called \emph{hamiltonian directed cycle}. A directed path $\vec{\mathcal{W}}$ of $\vec{\Gamma}$ with $V_{\vec{\mathcal{W}}}=V_{\vec{\Gamma}}$ is an \emph{hamiltonian directed path}. $\vec{\Gamma}$ is an \emph{hamiltonian digraph} if it has an hamiltonian directed cycle.
		\end{definition}
	
		Rédei's Theorem, that we are going to state, allows us to get an easy result about hamiltonian paths and hamiltonian directed paths.
		
		\begin{theorem}{\rm \cite[Theorem 2.3]{GraphTheory}}
			\label{Redei}Let $\vec{\Gamma}$ an orientation of a complete graph. Then $\vec{\Gamma}$ has an hamiltonian directed path. 
		\end{theorem}
	
		\begin{corollary}
			Let $p$ be a prime and $k$ a positive integer. Then $\vec{\mathcal{P}}(C_{p^k})$ has an hamiltonian directed path.
		\end{corollary}
		\begin{proof}
			By Proposition \ref{PGcompleto_1}, $\mathcal{P}(C_{p^k})$ is the complete graph $K_{p^k}$. Then $\vec{\mathcal{P}}(C_{p^k})$ contains, as subdigraph, an orientation $\vec{\Gamma}$ of $K_{p^k}$.
			By Theorem \ref{Redei} $\vec{\Gamma}$ has an hamiltonian directed path and thus also $\vec{\mathcal{P}}(C_{p^k})$ has an hamiltonian directed path. 
			
		\end{proof}
	
		Dealing with cycles, instead of paths, in \cite[Corollary 5.13]{PowerGraphsofPowerGraoups} there is a relevant result about cycles in the power graph of cyclic groups. We state and prove this result.
		
		\begin{prop}
			\label{C_nHamilton}Let $G\cong C_n$ with $n\geq 3$. Then $\mathcal{P}(G)$ is hamiltonian.
		\end{prop}
		\begin{proof}
			Note that if $n\leq2$ there is no cycle in $\mathcal{P}(G)$. \\
			We can assume $[C_n]_{\diamond}=\{d_1,\dots,d_k\}$ where the indexing is such that $\mathcal{C}_{\diamond}:=d_1,\dots,d_k$ is an hamiltonian cycle of $C_n$.
			This is possible by \cite[Theorem 11]{CyclicSubgroupGraphofFiniteGroup}, that states that $\mathcal{P}(C_n)/\diamond$ is hamiltonian.  
			We also use the following notation: $$d_i:=\{x_1^i, \dots, x_{n_i}^i\}=[x_1^i]_{\diamond}$$ for $i\in [k]$, where $n_i:=|d_i|$.\\
			Note that $$|d_i|=|[x_1^i]_{\diamond}|=\phi(o(x_1^i)).$$
			Note also that $n=\sum_{i=1}^k|d_i|$.\\
			Since, by Lemma \ref{lemmaClassiDiamond} i), the subgraph of $\mathcal{P}(G)$ induced by a $\diamond$-class is complete, then $$\mathcal{C}=x_1^1, \dots, x_{n_1}^1, \dots, x_1^k, \dots, x_{n_k}^k, x_1^1$$ is an hamiltonian cycle for $\mathcal{P}(C_n)$.
		\end{proof}
		
		The figures of this chapter illustrate the last few proven results.
		 
		Let's now focus on the research of maximal cycles and maximal directed cycles in $\mathcal{P}(G)$ and $\vec{\mathcal{P}}(G)$ for $G$ a generic group. 
		The directed case is quite easy.
		Indeed we state and prove the next result that, more or less, concludes our study in this environment.
		
		\begin{prop}
			\label{directedCycle}Let $G$ be a group and let $\vec{\mathcal{C}}$ be a directed cycle of $\vec{\mathcal{P}}(G)$. $\vec{\mathcal{C}}$ is a maximal directed cycle if and only if $V_{\vec{\mathcal{C}}}=[x]_{\diamond}$ for some $x\in G$ with $|[x]_{\diamond}|\geq 3$.
		\end{prop}
		\begin{proof}
			Let $\vec{\mathcal{C}}=x_1, \dots, x_k, x_1$. 
			Assume that $\vec{\mathcal{C}}$ is a maximal directed cycle. Suppose, by contradiction, that $V_{\vec{\mathcal{C}}}\ne [x_1]_{\diamond}$.
			If there exists $y \in [x_1]_{\diamond}\setminus V_{\vec{\mathcal{C}}}$, then $\vec{\mathcal{C}}'=x_1, \dots, x_k, y, x_1$ is a directed cycle that contradicts the maximality of $\vec{\mathcal{C}}$. Hence $[x_1]_{\diamond}\subseteq V_{\vec{\mathcal{C}}}$. 
			If there exists an integer $i\in[k]\setminus \{1\}$ such that $x_i\notin [x_1]_{\diamond}$, then we reach a contradiction. Indeed we have that $$x_1, x_2, \dots, x_i $$ and $$x_i, x_{i+1}, \dots, x_k, x_1$$ are directed paths that has $x_1$ and $x_i$ as end points. In particular one directed path starts in $x_1 $ and ends in $x_i$ and the other one starts in $x_i$ and ends in $x_1$. So, since, by Remark \ref{PG_transitive}, $\vec{\mathcal{P}}(G)$ is transitive, $(x_i,x_1)\in A$ and also $(x_1,x_i)\in A$, that means that $x_i\diamond x_1$ holds, a contradiction. Therefore $V_{\vec{\mathcal{C}}}\subseteq [x_1]_{\diamond}$ and hence $V_{\vec{\mathcal{C}}}=[x_1]_{\diamond}$.
			 
			Assume next that $V_{\vec{\mathcal{C}}}=[x_1]_{\diamond}$.
			Suppose, by contradiction, that $\vec{\mathcal{C}}$ is not maximal. 
			Then there exists $\vec{\mathcal{C}}'$ a directed cycle such that $V_{\vec{\mathcal{C}}}\subset V_{\vec{\mathcal{C}}'}$.  
			In particular there exists $y\in V_{\vec{\mathcal{C}}'}$ such that $y\notin [x_1]_{\diamond}$.
			For this reason there exists at least two directed paths having $x_1$ and $y$ as end points; in particular one directed path starts in $x_1 $ and ends in $y$ and another one starts in $y$ and ends in $x_1$. So, since $\vec{\mathcal{P}}(G)$ is transitive, $(x_1,y)\in A$ and also $(y,x_1)\in A$, that means that $x_1\diamond y$, a contradiction.  
		\end{proof}
		
		It follows immediately, by the proposition above, that a directed cycle with maximum order in $\vec{\mathcal{P}}(G)$ has length $\phi(o(x))$ with $x$ an element having $\phi(o(x))$ maximum.
		As a consequence, we have that, in $\mathcal{P}(G)$, a cycle with maximum order has at least length $\phi(o(x))$ for $x$ an element with maximum $\phi(o(x))$. With the next proposition we improve this lower bound.
	
		\begin{prop}
		\label{orderBound}	Let $G$ be a group with $|G|\geq3$. Let $\mathcal{C}$ be a cycle of $\mathcal{P}(G)$ with maximum length. Then 
			$$ \ell(\mathcal{C})\geq \max_{x\in G}\{o(x)\}:=M_{o(G)}.$$
		\end{prop}
		\begin{proof}
			Let $x\in G$. Consider the subgraph of $\mathcal{P}(G)$ induced by $\langle x \rangle$ that is, by Remark \ref{InducedPG}, $\mathcal{P}(\langle x \rangle)$.
			By Proposition \ref{C_nHamilton}, $\mathcal{P}(\langle x \rangle)$ is hamiltonian. Then there exists a cycle $\mathcal{C}'$ of $\mathcal{P}(G)$ with length $o(x)$.
			Therefore $o(x)\leq \ell (\mathcal{C})$. It follows the thesis.
		\end{proof}
	
		From now on, for a group $G$, we use $M_{o(G)}$ instead of $\max_{x\in G}{o(x)}$. 
	
		Recall that, since $\phi$ is not a monotone function, it can happen that an element $x$ of maximum order is not an element having maximum $\phi(o(x))$. For example consider a group $G$ for which the orders of the elements are $\{1,2,3,5,6\}$. Such a group has maximum length of a directed cycle equals to $\phi(o(5))=4$. But the lower bound for the maximum length of a cycle is given by the presence of elements of order $6$.  
		
		We want to further improve the lower bound. To achieve this result we need the following lemma.
		\begin{lemma}
			\label{N-classiinCicli}Let $\mathcal{C}$ be a cycle of $\mathcal{P}(G)$. Then there exists a cycle $\mathcal{C}'$ of $\mathcal{P}(G)$ with $V_{\mathcal{C}'}\supseteq V_{\mathcal{C}}$ and such that, for all $x\in V_{\mathcal{C}'}$ $[x]_{\mathtt{N}}\subseteq V_{\mathcal{C}'}$.
		\end{lemma}
		\begin{proof}
			Let $\mathcal{C}=x_1, x_2, \dots, x_k, x_1$.
			and $Y_{\mathcal{C}}:=\{x\in V_{\mathcal{C}} \, | \, [x]_{\mathtt{N}}\nsubseteq V_{\mathcal{C}}\}$.
			We proceed by induction on $|Y_{\mathcal{C}}|$. If $|Y_{\mathcal{C}}|=0$, then we take $\mathcal{C}'=\mathcal{C}$. Assume the statement true for $|Y_{\mathcal{C}}|\leq n$. Suppose that there exists $i \in [k]$ such that $[x_i]_{\mathtt{N}}\nsubseteq V_{\mathcal{C}}$. After renaming, if necessary, we can consider $i=1$. Let $X:=[x_1]_{\mathtt{N}}\setminus V_{\mathcal{C}} $. Then $|X|\geq1$. In particular $|[x_1]_{\mathtt{N}}|\ne1$. 
			Let $X:=\{\bar{x}_1,...,\bar{x}_l\}$. 
			Since the graph induced by $[x_1]_{\mathtt{N}}$ is a complete graph, then, for all $i,j\in [l]$, $\bar{x}_i$ and $\bar{x}_j$ are adjacent, and we also have that $\bar{x}_i$ and $x_1$ are adjacent.
			Furthermore we have that $\bar{x}_l$ is adjacent to $x_{2}$, because $x_{2}$ is adjacent to $x_1$ and $N[\bar{x}_l]=N[x_1]$. 
			Thus consider $$\mathcal{C}':=x_1,\bar{x}_1,\bar{x}_2,\dots,\bar{x}_l, x_2, \dots, x_k,x_1.$$
			$\mathcal{C}'$ is a cycle such that $V_{\mathcal{C}}\cup [x_i]_{\mathtt{N}}\subseteq V_{\mathcal{C}'}$. Now $|Y_{\mathcal{C}'}|\leq |Y_{\mathcal{C}}|-1$. Therefore, by inductive hypothesis we obtain the thesis.
		\end{proof}
	
	
	In \cite{PowerGraphsofPowerGraoups}, Acharyya and Williams use a directed version of the quotient graph of a group $G$ to find a the longest directed path of $\vec{\mathcal{P}}(G)$. We want to use a similar idea, so we need the following definition extrapolated from \cite{PowerGraphsofPowerGraoups}.
	
	\begin{definition}
		\rm Let $\Gamma$ be a graph and $\sim$ be an equivalence relation on $V_{\Gamma}$. The \emph{$\sim$-weight function} is defined as follow:
		\begin{align*}
			w_{\sim}:[V_{\Gamma}]_{\sim}\longrightarrow \mathbb{N} \quad \, &\\
			[x]_{\sim}\longmapsto |[x]_{\sim}|.
		\end{align*}
	
		Furthermore, if $X\subseteq [V_{\Gamma}]_{\sim}$, then we define $w_{\sim}(X):= \sum_{x\in X}w_{\sim}(x)$.
	\end{definition}
	
	Let's have a look at an example of computation of $\mathtt{N}$-weight. Consider $G=\langle g \rangle$, then
	\[
	w_{\mathtt{N}}([g]_{\mathtt{N}})=
	\begin{cases}
		\phi(o(g))+1 & \text{if $o(g)$ is not a prime power} \\
		|G| & \text{if $o(g)$ is a prime power}
	\end{cases}
	\]
	Otherwise, if we consider $G=D_n$ with presentation give in \eqref{Dn_presentation}, then $w_{\mathtt{N}}(b)=1$ and 
	
	\[
	w_{\mathtt{N}}([a]_{\mathtt{N}})=
	\begin{cases}
		\phi(o(a)) & \text{if $o(a)$ is a not prime power} \\
		o(a)-1 & \text{if $o(a)$ is a prime power}
	\end{cases}
	\]
	
	Indeed $[b]_{\mathtt{N}}= \{b\}$ and $[a]_{\mathtt{N}}=\{\text{generators of } \langle a \rangle \}$ if $o(a)$ is not a prime power, instead $[a]_{\mathtt{N}}=\langle a \rangle \setminus \{1\}$ if $o(a)$ is a prime power. Then 
	\[
	w_{\mathtt{N}}(\{[b]_{\mathtt{N}}, [a]_{\mathtt{N}}\})=
	\begin{cases}
		\phi(o(a))+1 & \text{if $o(a)$ is a not prime power} \\
		o(a) & \text{if $o(a)$ is a prime power}
	\end{cases}
	\]
	
	 Denoted $\Omega$ the set of all paths in $\mathcal{P}^*(G)/\mathtt{N}$, then $$w_G:=\max_{\mathcal{W}\in \Omega} w_{\mathtt{N}}(\mathcal{W})$$.

	\begin{prop}
		\label{Weigth-length}Let $G$ be a group.
		If $\mathcal{C}$ is the cycle of $\mathcal{P}(G)$ with maximum length, then
		
		$$ \ell(\mathcal{C})\geq w_G+1. $$
		
		Moreover, if equality holds, then, denoted $\mathcal{W}=c_1, \dots, c_k$ a path in $\mathcal{P}^*(G)/\mathtt{N}$ with maximum $\mathtt{N}$-weight, for all $i\in [k-1]\setminus \{1\}$ such that $w_{\mathtt{N}}(c_i)\geq 2$, there is no vertex $y\in N_{\mathcal{P}^*(G)/\mathtt{N}}[c_i] \setminus V_{\mathcal{W}}$ such that $w_{\mathtt{N}}(y)\geq 2$. 
	
	\end{prop}
	\begin{proof}
		Let $\mathcal{W}=c_1, \dots, c_k$ a path in $\mathcal{P}^*(G)/\mathtt{N}$ with maximum $\mathtt{N}$-weight. Then $$w_{\mathtt{N}}(\mathcal{W})=\max_{\mathcal{W}\in \Omega} w_{\mathtt{N}}(\mathcal{W})=w_G.$$ Let $n_i=w_{\mathtt{N}}(c_i)$ for all $i\in [k]$.
		Denote $c_i=\{x_1^i, \dots, x_{n_i}^i\}$ for all $i\in [k]$.		
		We have that $\{x_j^i, x_t^{i+1}\}\in E_{\mathcal{P}(G)}$ for all $i\in[k-1]$, $j\in [n_i]$ and $t\in [n_{i+1}]$ because $\{c_i,c_{i+1}\}\in E_{\mathcal{P}^*(G)/\mathtt{N}}$. Moreover, since $c_i$ is an $\mathtt{N}$-class for all $i\in [k]$, then $\{x_j^i,x_t^i\}\in E_{\mathcal{P}(G)}$ for all $j, t\in[n_i]$. Then, since $1$ is a star vertex of $\mathcal{P}(G)$, $$\mathcal{C}':= 1,x_1^1,\dots, x_{n_1}^1, x_1^2, \dots, x_{n_2}^2, \dots, x_1^k, \dots, x_{n_k}^k,1$$ is a cycle with length $w_{\mathtt{N}}(\mathcal{W})+1$. This proves $$w_G+1\leq \ell(\mathcal{C}). $$
		
		Assume next that equality holds. We have that the cycle $\mathcal{C}'$ above is a cycle of maximum length.
		Suppose, by contradiction, that there exists an index $i\in [k-1]\setminus \{1\}$ such that:
		\begin{enumerate}
			\item $w_{\mathtt{N}}(c_i)\geq 2$;
			
			\item There exists a vertex $y\in N_{\mathcal{P}^*(G)/\mathtt{N}}[c_i] \setminus V_{\mathcal{W}}$ such that $w_{\mathtt{N}}(y)\geq 2$.
		\end{enumerate}
		
		We denote $[y]_{\mathtt{N}}:=\{{\bf y_1,\dots, y_n}\}$.
		We reach a contradiction due to the presence of a cycle $\mathcal{C}''$ with $\ell(\mathcal{C}'')>\ell(\mathcal{C}')$.
		We define $\mathcal{C}''$ as follows:
		$$\mathcal{C}'':= 1,x_1^1,\dots, x_{n_1}^1,\dots,  x_1^i, {\bf y_1, \dots, y_n}, x_2^i, \dots, x_{n_i}^i, \dots, x_1^k, \dots, x_{n_k}^k,1.$$

	\end{proof}
	
	This new lower bound is not always better than the one in Proposition \ref{orderBound}. For example, think at $C_n$ or $D_n$. 
	By Proposition \ref{C_nHamilton}, we have that, for $\mathcal{P}(C_n)$, the maximum length for a cycle is $n$.
	For $\mathcal{P}(D_n)$ we have that there is a cycle $\mathcal{C}$ of length $n$ since, considering the presentation given in \eqref{Dn_presentation}, $o(a)=n$. Now we have $D_n\setminus V_{\mathcal{C}}=D_n\setminus \langle a \rangle$. We also have that every element $x\in D_n\setminus \langle a \rangle$ has order $2$. In particular $x$ cannot be a vertex of a cycle of $\mathcal{P}(D_n)$ because it is joined only with the identity. Therefore $\mathcal{C}$ has maximum length since no vertices of $\mathcal{P}(D_n)$ not in $\mathcal{C}$ can be in a cycle. 
	We have proven that the cycle with maximum length in the power graphs of $C_n$ and $D_n$ has length $n$, that is exactly the order of an element with maximum order.
	Anyway in those cases we have that $w_{\mathtt{N}}(\mathcal{W})=n-1$ for $\mathcal{W}$ a path with maximum $\mathtt{N}$-weight in the {\scriptsize CTR} of the proper power graph of those two groups.
	
	\begin{Oss}
		Let $G$ be a group. Then $$ M_{o(G)}\leq w_G+1.$$ 
	\end{Oss}
	\begin{proof}
		Let $x\in G$ be an element of order $M_{o(G)}$.
		Let $\mathcal{W}$ be an hamiltonian path of $\mathcal{P}^*(\langle x \rangle)/\diamond$. Such a path exists by \cite[Theorem 11]{CyclicSubgroupGraphofFiniteGroup}. Since $\mathcal{W}$ is hamiltonian, then $$w_{\mathtt{N}}(\mathcal{W})=w_{\mathtt{N}}(V_{\mathcal{P}^*(\langle x \rangle)/\diamond})=|V_{\mathcal{P}^*(\langle x \rangle)}|=o(x)-1.$$ Now, $\mathcal{P}^*(\langle x \rangle )/\diamond=\mathcal{P}^*(\langle x \rangle )/\mathtt{N}$ and $\mathcal{P}^*(\langle x \rangle )/\mathtt{N}$ is a subgraph of $\mathcal{P}^*(G)/\mathtt{N}$. Hence $$ w_{\mathtt{N}}(\mathcal{W})\leq w_G.$$
		Therefore $$ M_{o(G)}=o(x)=w_{\mathtt{N}}(\mathcal{W})+1\leq w_G+1.$$
	\end{proof}
	
	Let's analyse the examples in Figure \ref{C_2xC_10} and Figure \ref{C_2xC_4} where the two lower bounds are different. In Figure \ref{C_2xC_10} we see $\mathcal{P}^*(G)$ and $\mathcal{P}^*(G)/\mathtt{N}$ for $G=C_2\times C_{10}$. The colours represent the $\mathtt{N}$-classes. Remember that, since $G$ is abelian and $\mathcal{S}=\{1\}$, the $\mathtt{N}$-classes are equals to the $\diamond$-classes.
	Take a look at the red path $\mathcal{W}$ in $\mathcal{P}^*(G)/\mathtt{N}$. It is a path with maximum weight, and $w_{\mathtt{N}}(\mathcal{W})=w_G=14$. Note that $$M_{o(G)}=10<w_G+1=15.$$ Thus the lower bound in Proposition \ref{Weigth-length} is better than the one in Proposition \ref{orderBound}. The cycle $\mathcal{C}'$, defined in the proof of Proposition \ref{Weigth-length}, is composed by the two red paths highlighted in $\mathcal{P}^*(G)$ joined by the green edge and closed joining both the end vertices, the purple one and the pink one, with the identity. Hence the cycle obtained has length $w_G+1=15$.
	Note that for this path $\mathcal{W}$ the condition for the equality does not hold. Indeed $\mathcal{P}(C_2\times C_{10})$ is hamiltonian. An hamiltonian cycle of $\mathcal{P}(C_2\times C_{10})$ is obtained as follows:
	\begin{itemize}
		\item Start from the identity.
		\item Move to the purple vertex.
		\item Follow the red path until you reach to the first blue vertex.
		\item Now follow the cyan path until you reach the blue vertex joined with the previous blue vertex with the green edge.
		\item Follow the red path with this new blue vertex as an end vertex; you will arrive in the pink vertex.
		\item Move to the identity and close the cycle. 
	\end{itemize}

	\begin{figure}
		\centering
		\begin{tikzpicture}[line cap=round,line join=round,>=triangle 45,x=1.0cm,y=1.0cm]
			\clip(-1.5,-1.) rectangle (12.5,6.);
			\draw [line width=1.pt] (1.5034163897239292,-0.2887114013232972)-- (0.5575991480232942,0.03598806788138642);
			\draw [line width=1.pt,color=ffqqqq] (1.5034163897239292,-0.2887114013232972)-- (-0.231541361373099,0.6502007805710548);
			\draw [line width=1.pt] (1.5034163897239292,-0.2887114013232972)-- (-0.7784895194955253,1.487367258833583);
			\draw [line width=1.pt] (1.5034163897239292,-0.2887114013232972)-- (1.0170550863212076,5.580796648144843);
			\draw [line width=1.pt] (1.5034163897239292,-0.2887114013232972)-- (-1.023975006636324,2.4567675247729133);
			\draw [line width=1.pt] (1.5034163897239292,-0.2887114013232972)-- (-0.9413956611639915,3.453352017779583);
			\draw [line width=1.pt] (1.5034163897239292,-0.2887114013232972)-- (-0.5397002365110217,4.369125344434639);
			\draw [line width=1.pt,color=ffqqqq] (1.5034163897239292,-0.2887114013232972)-- (0.13758133511471993,5.10484925510777);
			\draw [line width=1.pt,color=qqffff] (3.449233631424563,0.03598806788138598)-- (0.5575991480232942,0.03598806788138642);
			\draw [line width=1.pt,color=qqffff] (3.449233631424563,0.03598806788138598)-- (4.785322298943383,1.4873672588335813);
			\draw [line width=1.pt] (3.449233631424563,0.03598806788138598)-- (-0.7784895194955253,1.487367258833583);
			\draw [line width=1.pt] (3.449233631424563,0.03598806788138598)-- (4.9482284406118495,3.4533520177795802);
			\draw [line width=1.pt] (3.449233631424563,0.03598806788138598)-- (-0.9413956611639915,3.453352017779583);
			\draw [line width=1.pt] (3.449233631424563,0.03598806788138598)-- (3.8692514443331403,5.104849255107769);
			\draw [line width=1.pt] (3.449233631424563,0.03598806788138598)-- (0.13758133511471993,5.10484925510777);
			\draw [line width=1.pt] (3.449233631424563,0.03598806788138598)-- (2.0034163897239297,5.745391238425577);
			\draw [line width=1.pt] (4.238374140820957,0.650200780571053)-- (-0.7784895194955253,1.487367258833583);
			\draw [line width=1.pt] (4.238374140820957,0.650200780571053)-- (4.546533015958881,4.369125344434638);
			\draw [line width=1.pt] (4.238374140820957,0.650200780571053)-- (0.13758133511471993,5.10484925510777);
			\draw [line width=1.pt,color=ffqqqq] (4.238374140820957,0.650200780571053)-- (2.503416389723929,-0.2887114013232972);
			\draw [line width=1.pt] (4.238374140820957,0.650200780571053)-- (0.5575991480232942,0.03598806788138642);
			\draw [line width=1.pt,color=ffqqqq] (4.238374140820957,0.650200780571053)-- (5.030807786084182,2.4567675247729115);
			\draw [line width=1.pt] (4.238374140820957,0.650200780571053)-- (-0.9413956611639915,3.453352017779583);
			\draw [line width=1.pt] (4.238374140820957,0.650200780571053)-- (2.9897776931266518,5.580796648144843);
			\draw [line width=1.pt] (4.785322298943383,1.4873672588335813)-- (-0.9413956611639915,3.453352017779583);
			\draw [line width=1.pt] (4.785322298943383,1.4873672588335813)-- (2.0034163897239297,5.745391238425577);
			\draw [line width=1.pt] (4.785322298943383,1.4873672588335813)-- (0.5575991480232942,0.03598806788138642);
			\draw [line width=1.pt,color=qqffff] (4.785322298943383,1.4873672588335813)-- (4.9482284406118495,3.4533520177795802);
			\draw [line width=1.pt] (4.785322298943383,1.4873672588335813)-- (3.8692514443331403,5.104849255107769);
			\draw [line width=1.pt] (4.785322298943383,1.4873672588335813)-- (-0.7784895194955253,1.487367258833583);
			\draw [line width=1.pt] (4.785322298943383,1.4873672588335813)-- (0.13758133511471993,5.10484925510777);
			\draw [line width=1.pt] (5.030807786084182,2.4567675247729115)-- (0.13758133511471993,5.10484925510777);
			\draw [line width=1.pt] (5.030807786084182,2.4567675247729115)-- (-0.9413956611639915,3.453352017779583);
			\draw [line width=1.pt] (5.030807786084182,2.4567675247729115)-- (2.503416389723929,-0.2887114013232972);
			\draw [line width=1.pt] (5.030807786084182,2.4567675247729115)-- (-0.7784895194955253,1.487367258833583);
			\draw [line width=1.pt] (5.030807786084182,2.4567675247729115)-- (2.9897776931266518,5.580796648144843);
			\draw [line width=1.pt] (5.030807786084182,2.4567675247729115)-- (0.5575991480232942,0.03598806788138642);
			\draw [line width=1.pt,color=ffqqqq] (5.030807786084182,2.4567675247729115)-- (4.546533015958881,4.369125344434638);
			\draw [line width=1.pt,color=ffqqqq] (4.546533015958881,4.369125344434638)-- (2.9897776931266518,5.580796648144843);
			\draw [line width=1.pt] (4.546533015958881,4.369125344434638)-- (0.13758133511471993,5.10484925510777);
			\draw [line width=1.pt] (4.546533015958881,4.369125344434638)-- (-0.9413956611639915,3.453352017779583);
			\draw [line width=1.pt] (4.546533015958881,4.369125344434638)-- (-0.7784895194955253,1.487367258833583);
			\draw [line width=1.pt] (4.546533015958881,4.369125344434638)-- (0.5575991480232942,0.03598806788138642);
			\draw [line width=1.pt] (4.546533015958881,4.369125344434638)-- (2.503416389723929,-0.2887114013232972);
			\draw [line width=1.pt] (3.8692514443331403,5.104849255107769)-- (0.5575991480232942,0.03598806788138642);
			\draw [line width=1.pt] (3.8692514443331403,5.104849255107769)-- (-0.7784895194955253,1.487367258833583);
			\draw [line width=1.pt] (3.8692514443331403,5.104849255107769)-- (-0.9413956611639915,3.453352017779583);
			\draw [line width=1.pt] (3.8692514443331403,5.104849255107769)-- (0.13758133511471993,5.10484925510777);
			\draw [line width=1.pt,color=qqffff] (3.8692514443331403,5.104849255107769)-- (2.0034163897239297,5.745391238425577);
			\draw [line width=1.pt,color=qqffff] (3.8692514443331403,5.104849255107769)-- (4.9482284406118495,3.4533520177795802);
			\draw [line width=1.pt] (2.9897776931266518,5.580796648144843)-- (-0.9413956611639915,3.453352017779583);
			\draw [line width=1.pt,color=ffqqqq] (2.9897776931266518,5.580796648144843)-- (0.5575991480232942,0.03598806788138642);
			\draw [line width=1.pt] (2.9897776931266518,5.580796648144843)-- (2.503416389723929,-0.2887114013232972);
			\draw [line width=1.pt] (2.9897776931266518,5.580796648144843)-- (0.13758133511471993,5.10484925510777);
			\draw [line width=1.pt] (2.9897776931266518,5.580796648144843)-- (-0.7784895194955253,1.487367258833583);
			\draw [line width=1.pt] (2.0034163897239297,5.745391238425577)-- (0.13758133511471993,5.10484925510777);
			\draw [line width=1.pt] (2.0034163897239297,5.745391238425577)-- (-0.9413956611639915,3.453352017779583);
			\draw [line width=1.pt] (2.0034163897239297,5.745391238425577)-- (4.9482284406118495,3.4533520177795802);
			\draw [line width=1.pt,color=qqffff] (2.0034163897239297,5.745391238425577)-- (-0.7784895194955253,1.487367258833583);
			\draw [line width=1.pt] (2.0034163897239297,5.745391238425577)-- (0.5575991480232942,0.03598806788138642);
			\draw [line width=1.pt] (-0.231541361373099,0.6502007805710548)-- (0.5575991480232942,0.03598806788138642);
			\draw [line width=1.pt] (-0.231541361373099,0.6502007805710548)-- (-0.7784895194955253,1.487367258833583);
			\draw [line width=1.pt] (-0.231541361373099,0.6502007805710548)-- (-1.023975006636324,2.4567675247729133);
			\draw [line width=1.pt] (-0.231541361373099,0.6502007805710548)-- (-0.9413956611639915,3.453352017779583);
			\draw [line width=1.pt,color=ffqqqq] (-0.231541361373099,0.6502007805710548)-- (-0.5397002365110217,4.369125344434639);
			\draw [line width=1.pt] (-0.231541361373099,0.6502007805710548)-- (0.13758133511471993,5.10484925510777);
			\draw [line width=1.pt] (-0.231541361373099,0.6502007805710548)-- (1.0170550863212076,5.580796648144843);
			\draw [line width=1.pt] (-0.5397002365110217,4.369125344434639)-- (-1.023975006636324,2.4567675247729133);
			\draw [line width=1.pt] (-0.5397002365110217,4.369125344434639)-- (-0.9413956611639915,3.453352017779583);
			\draw [line width=1.pt] (-0.5397002365110217,4.369125344434639)-- (-0.7784895194955253,1.487367258833583);
			\draw [line width=1.pt] (-0.5397002365110217,4.369125344434639)-- (0.5575991480232942,0.03598806788138642);
			\draw [line width=1.pt] (-0.5397002365110217,4.369125344434639)-- (0.13758133511471993,5.10484925510777);
			\draw [line width=1.pt,color=ffqqqq] (-0.5397002365110217,4.369125344434639)-- (1.0170550863212076,5.580796648144843);
			\draw [line width=1.pt] (1.0170550863212076,5.580796648144843)-- (0.13758133511471993,5.10484925510777);
			\draw [line width=1.pt] (1.0170550863212076,5.580796648144843)-- (-0.9413956611639915,3.453352017779583);
			\draw [line width=1.pt,color=ffqqqq] (1.0170550863212076,5.580796648144843)-- (-1.023975006636324,2.4567675247729133);
			\draw [line width=1.pt] (1.0170550863212076,5.580796648144843)-- (-0.7784895194955253,1.487367258833583);
			\draw [line width=1.pt] (1.0170550863212076,5.580796648144843)-- (0.5575991480232942,0.03598806788138642);
			\draw [line width=1.pt,color=ffqqqq] (9.046193824211727,0.5312289784348669)-- (7.79921422049426,2.0948919433709268);
			\draw [line width=1.pt,color=ffqqqq] (9.046193824211727,0.5312289784348669)-- (8.244256088406889,4.044747767734574);
			\draw [line width=1.pt] (7.79921422049426,2.0948919433709268)-- (10.046193824211727,4.91251524596969);
			\draw [line width=1.pt,color=ffqqqq] (7.79921422049426,2.0948919433709268)-- (11.848131560016565,4.044747767734573);
			\draw [line width=1.pt] (10.046193824211727,4.91251524596969)-- (12.293173427929194,2.094891943370926);
			\draw [line width=1.pt,color=ffqqqq] (11.046193824211727,0.5312289784348669)-- (11.848131560016565,4.044747767734573);
			\draw [line width=1.pt,color=qqffqq] (0.5575991480232942,0.03598806788138642)-- (-0.7784895194955253,1.487367258833583);
			\draw [line width=1.pt] (0.5575991480232942,0.03598806788138642)-- (-0.9413956611639915,3.453352017779583);
			\draw [line width=1.pt] (0.5575991480232942,0.03598806788138642)-- (0.13758133511471993,5.10484925510777);
			\draw [line width=1.pt,color=ffqqqq] (-0.7784895194955253,1.487367258833583)-- (-0.9413956611639915,3.453352017779583);
			\draw [line width=1.pt,color=ffqqqq] (-0.9413956611639915,3.453352017779583)-- (0.13758133511471993,5.10484925510777);
			\begin{scriptsize}
				\draw [fill=ffffqq] (1.5034163897239292,-0.2887114013232972) circle (3.pt);
				\draw [fill=xfqqff] (2.503416389723929,-0.2887114013232972) circle (3.pt);
				\draw [fill=ffxfqq] (3.449233631424563,0.03598806788138598) circle (3.pt);
				\draw [fill=ffqqff] (4.238374140820957,0.650200780571053) circle (3.pt);
				\draw [fill=ffxfqq] (4.785322298943383,1.4873672588335813) circle (3.pt);
				\draw [fill=ffqqff] (5.030807786084182,2.4567675247729115) circle (3.pt);
				\draw [fill=ffffff] (4.9482284406118495,3.4533520177795802) circle (3.pt);
				\draw [fill=ffqqff] (4.546533015958881,4.369125344434638) circle (3.pt);
				\draw [fill=ffxfqq] (3.8692514443331403,5.104849255107769) circle (3.pt);
				\draw [fill=ffqqff] (2.9897776931266518,5.580796648144843) circle (3.pt);
				\draw [fill=ffxfqq] (2.0034163897239297,5.745391238425577) circle (3.pt);
				\draw [fill=ffffqq] (1.0170550863212076,5.580796648144843) circle (3.pt);
				\draw [fill=qqqqff] (0.13758133511471993,5.10484925510777) circle (3.pt);
				\draw [fill=ffffqq] (-0.5397002365110217,4.369125344434639) circle (3.pt);
				\draw [fill=qqqqff] (-0.9413956611639915,3.453352017779583) circle (3.pt);
				\draw [fill=ffcqcb] (-1.023975006636324,2.4567675247729133) circle (3.pt);
				\draw [fill=qqqqff] (-0.7784895194955253,1.487367258833583) circle (3.pt);
				\draw [fill=ffffqq] (-0.231541361373099,0.6502007805710548) circle (3.pt);
				\draw [fill=qqqqff] (0.5575991480232942,0.03598806788138642) circle (3.pt);
				\draw [fill=ffffqq] (9.046193824211727,0.5312289784348669) circle (3.pt);
				\draw [fill=xfqqff] (11.046193824211727,0.5312289784348669) circle (3.pt);
				\draw [fill=ffffff] (12.293173427929194,2.094891943370926) circle (3.pt);
				\draw [fill=ffqqff] (11.848131560016565,4.044747767734573) circle (3.pt);
				\draw [fill=ffxfqq] (10.046193824211727,4.91251524596969) circle (3.pt);
				\draw [fill=ffcqcb] (8.244256088406889,4.044747767734574) circle (3.pt);
				\draw [fill=qqqqff] (7.79921422049426,2.0948919433709268) circle (3.pt);
			\end{scriptsize}
		\end{tikzpicture}
		\caption{$\mathcal{P}^*(C_2\times C_{10})$ and $\mathcal{P}^*(C_2\times C_{10})/\mathtt{N}$.}
		\label{C_2xC_10}
	\end{figure}
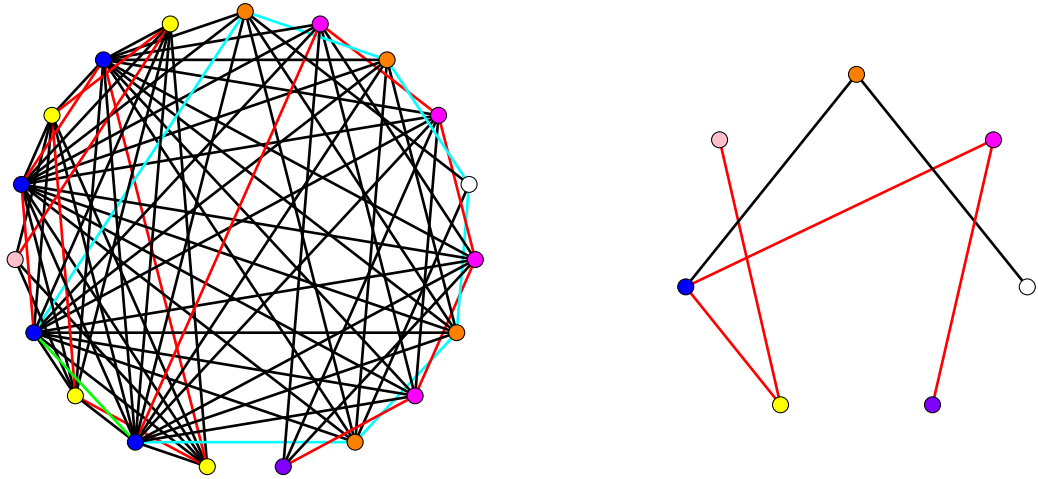

	In Figure \ref{C_2xC_4} there is represented both $\mathcal{P}^*(C_2\times C_4)$ and $\mathcal{P}^*(C_2\times C_4)/\mathtt{N}$. In this case in the quotient graph there is a unique path with maximum weight, that is the $2$-path that goes from the blue vertex to the yellow one. The cycle $\mathcal{C}'$ defined in the proof of the Proposition \ref{Weigth-length} it is also the cycle with maximum length. $\mathcal{C}'$ starts in the identity, goes to a blue vertex and then to the other blue one. Then, after it has passed the white vertex, it goes to a yellow vertex, passes the other yellow vertex and then comes back to the identity. In conclusion, we have that 
	 $$M_{o(G)}=4<w_G+1=6. $$
	
	\begin{figure}
		\centering
		\begin{tikzpicture}[line cap=round,line join=round,>=triangle 45,x=1.0cm,y=1.0cm]
			\clip(-5.5,0.) rectangle (9.5,8.);
			\draw [line width=1.pt] (-3.,1.)-- (-4.870469405576201,3.34549444740409);
			\draw [line width=1.pt] (-3.,1.)-- (-4.202906603707257,6.270278183949561);
			\draw [line width=1.pt] (-4.202906603707257,6.270278183949561)-- (-4.870469405576201,3.34549444740409);
			\draw [line width=1.pt] (1.8704694055762006,3.3454944474040893)-- (-4.870469405576201,3.34549444740409);
			\draw [line width=1.pt] (1.8704694055762006,3.3454944474040893)-- (-1.5,7.571929401302235);
			\draw [line width=1.pt] (-1.5,7.571929401302235)-- (-4.870469405576201,3.34549444740409);
			\draw [line width=1.pt] (5.,2.)-- (4.0729490168751585,4.853169548885461);
			\draw [line width=1.pt] (4.0729490168751585,4.853169548885461)-- (6.5,6.616525305762879);
			\draw [line width=1.pt,dash pattern=on 2pt off 2pt] (3.,-1.)-- (3.,10.);
			\begin{scriptsize}
				\draw [fill=qqqqff] (-3.,1.) circle (3.pt);
				\draw [fill=ffqqff] (0.,1.) circle (3.pt);
				\draw [fill=ffffqq] (1.8704694055762006,3.3454944474040893) circle (3.pt);
				\draw [fill=ffxfqq] (1.2029066037072575,6.27027818394956) circle (3.pt);
				\draw [fill=ffffqq] (-1.5,7.571929401302235) circle (3.pt);
				\draw [fill=qqqqff] (-4.202906603707257,6.270278183949561) circle (3.pt);
				\draw [fill=ffffff] (-4.870469405576201,3.34549444740409) circle (3.pt);
				\draw [fill=qqqqff] (5.,2.) circle (3.pt);
				\draw [fill=ffqqff] (8.,2.) circle (3.pt);
				\draw [fill=ffxfqq] (8.927050983124843,4.85316954888546) circle (3.pt);
				\draw [fill=ffffqq] (6.5,6.616525305762879) circle (3.pt);
				\draw [fill=ffffff] (4.0729490168751585,4.853169548885461) circle (3.pt);
				\draw [fill=ududff] (3.,-1.) circle (3.pt);
				\draw [fill=ududff] (3.,10.) circle (3.pt);
				
			\end{scriptsize}
		\end{tikzpicture}
		\caption{$\mathcal{P}^*(C_2\times C_4)$ and $\mathcal{P}^*(C_2 \times C_4)/\mathtt{N}$}
		\label{C_2xC_4}
	\end{figure}
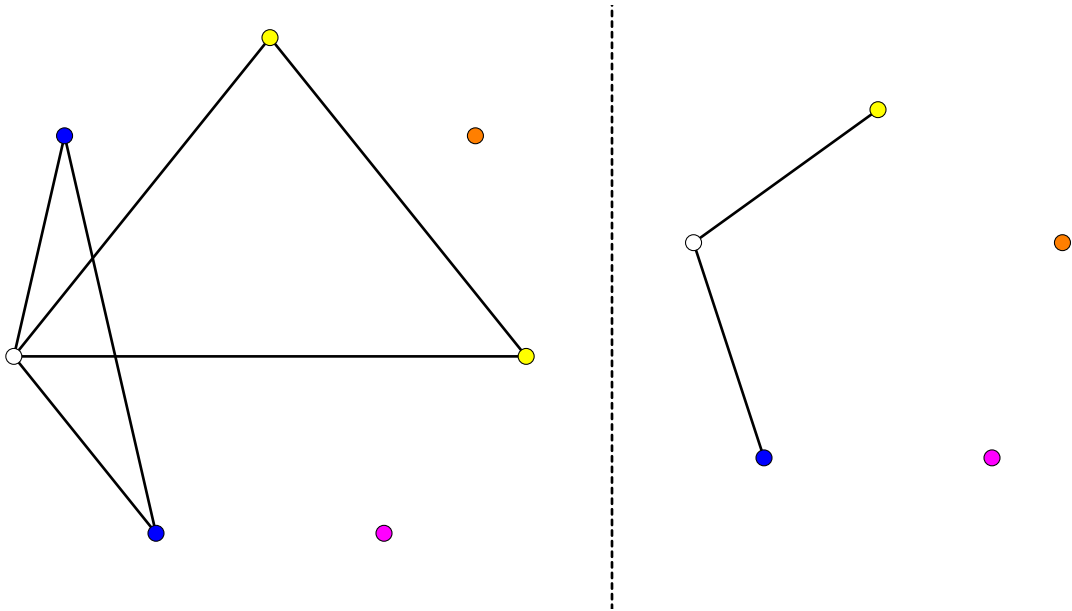

	We conclude this chapter with two questions for future research.
	
	\begin{Prob}
		Do there exists a formula for the maximum length of a cycle/path of a power graph?
	\end{Prob}

	\begin{Prob}
		What can be said about a group with fixed maximum length of a cycle/path of its power graph?
	\end{Prob}

\end{document}